\documentclass[a4paper, 12pt]{amsart}
\usepackage[margin=2.5cm]{geometry}
\usepackage{tikz}
\usepackage{amssymb}
\usepackage[utf8]{inputenc}
\usepackage{amsmath,amssymb,amsthm,amscd, amsrefs, bbm, enumerate, tikz-cd, enumitem}
\usepackage{booktabs} 
\usepackage{verbatim}
\usepackage{dsfont} 
\usepackage{etoolbox}

\newtheorem{defn}{Definition}[subsection]
\newtheorem{thm}[defn]{Theorem}

\newtheorem{prop}[defn]{Proposition}
\newtheorem{cor}[defn]{Corollary}
\newtheorem{lemma}[defn]{Lemma}

\newtheorem{conjec}[defn]{Conjecture}
\newtheorem{rmk}[defn]{Remark}
\begin{document}

\title[The K-theory of the C*-algebras associated to rational functions]{The K-theory of the C*-algebras associated to rational functions}

\author[Hume]{Jeremy B. Hume}
\address{Jeremy B. Hume \\ School of Mathematics and Statistics  \\
University of Glasgow\\ University Place \\ Glasgow Q12 8QQ \\ United Kingdom}
\email{jeremybhume@gmail.com}
\begin{abstract}
We compute the $K$-theory of the three $C^{*}$-algebras associated to a rational function $R$ acting on the Riemann sphere, its Fatou set, and its Julia set. The latter $C^{*}$-algebra is a unital UCT Kirchberg algebra and is thus classified by its $K$-theory. The $K$-theory in all three cases is shown to depend only on the degree of $R$, the critical points of $R$, and the Fatou cycles of $R$. Our results yield new dynamical invariants for rational functions and a $C^{*}$-algebraic interpretation of the Density of Hyperbolicity Conjecture for quadratic polynomials. These calculations are possible due to new exact sequences in $K$-theory we induce from morphisms of $C^{*}$-correspondences.
\end{abstract}

\maketitle

\section{Introduction}\label{introduction}
A rational function $R(z) = \frac{P(z)}{Q(z)}$, where $P(z)$, $Q(z)$ are polynomials with complex co-efficients, can be thought of as a holomorphic map on the Riemann sphere $\hat{\mathbb{C}} = \mathbb{C}\cup\{\infty\}$. We will assume the degree of $R$ is larger than one, so that $R$ is not a homeomorphism or a constant. $R:\hat{\mathbb{C}}\mapsto \hat{\mathbb{C}}$ can be studied as a dynamical system by considering its forward iterations $\{R^{\circ n}\}_{n\in\mathbb{N}_{0}}$ by function composition.
\par
Holomorphic dynamical systems have been studied since the $19^{th}$ century, and we invite the reader to \cite{A94} and \cite{AIR11} for historical overviews. The modern theory was developed in the 1920's, primarily by Fatou \cite{F19} and Julia \cite{J18}, with contributions from Lattés \cite{L18} and Ritt \cite{R20}. The standard reference for the modern theory is \cite{Milnor:Dynamics_in_one_complex_variable}.
\par
Fatou and Julia made the fundamental observation that the sphere partitions into two $R$-invariant domains exhibiting wildly different dynamical phenomena. The first such domain is the \textit{Fatou set} $F_{R}$, where the dynamics is regular. Remarkably, every connected component of $F_{R}$ is eventually mapped by $R$ into a finite cycle of Fatou components, and there can only be four different types of behaviour of $R$ when restricted to a Fatou cycle (see \cite[Theorem~16.1]{Milnor:Dynamics_in_one_complex_variable}).
\par
The complement of $F_{R}$ in $\hat{\mathbb{C}}$ is the \textit{Julia set} and is denoted $J_{R}$. This set is always non-empty (\cite[Theorem~14.1]{Milnor:Dynamics_in_one_complex_variable}) and the dynamics restricted to $J_{R}$ behaves chaotically (\cite[Corollary~14.2]{Milnor:Dynamics_in_one_complex_variable}). There is no corresponding classification of the possible behaviours for $R:J_{R}\mapsto J_{R}$, and these systems are notoriously hard to understand for seemingly simple families of rational functions.
\par
An approach to understanding a dynamical system is to study its space of orbits. One quickly runs into trouble when trying to do this for a rational function $R:J_{R}\mapsto J_{R}$, as each orbit $O(x) = \{y\in J_{R}:\exists\text{ } n,m\in\mathbb{N}: R^{\circ n}(x) = R^{\circ m}(y)\}$ is dense in $J_{R}$ and thus the space of orbits $O_{R,J_{R}}$ is highly singular and contains little information. This is a common occurence for a dynamical system.
\par
An insight from Connes' non-commutative geometry program \cite{C94} is to instead consider the ``non-commutative'' orbit space of a dynamical system, which is a $C^{*}$-algebra constructed from the dual action of the dynamical system, acting on the continuous complex-valued functions of its phase space. In general, the $C^{*}$-algebra of a dynamical system is well behaved, and whenever the orbit space is well behaved it is the same (up to strong Morita equivalence) as the $C^{*}$-algebra of continuous functions on the orbit space (see \cite[Chapter~4]{W07}).
\par
Kajiwara and Watatani in \cite{KW:C*-algebras_associated_with_complex_systems} were the first to study the $C^{*}$-algebra $\mathcal{O}_{R,J_{R}}$ of a rational function $R:J_{R}\mapsto J_{R}$. They proved in \cite{KW:C*-algebras_associated_with_complex_systems} that $\mathcal{O}_{R,J_{R}}$ are unital UCT Kirchberg algebras and are therefore classified up to *-isomorphism (homeomorphism of non-commutative spaces) by their $K$-theory and the class of the unit in $K_{0}$ by the Kirchberg-Phillips Theorem \cite{Phillips:classification}. Calculating the $K$-theory of $\mathcal{O}_{R,J_{R}}$ is therefore of fundamental importance in understanding these $C^{*}$-algebras.
\par
Nekrashevych in \cite{Nek:C-alg_and_self_similar} computed the $K$-theory of $\mathcal{O}_{R,J_{R}}$ in the special case that $R$ is hyperbolic and post-critically finite. While these are very interesting examples, there are only countably many distinct conjugacy classes of such $R$ and they do not capture all of the diverse behaviour of rational dynamics. Moreover, the method Nekrashevych employs to calculate the $K$-theory does not extend past his example class, so new techniques must be developed.
\par
In this paper, we calculate the $K$-theory of $\mathcal{O}_{R,J_{R}}$ for a general rational function using an entirely different approach. We also calculate the $K$-theory of the $C^{*}$-algebras associated to the dynamics of $R$ on $\hat{\mathbb{C}}$ ($\mathcal{O}_{R,\hat{\mathbb{C}}}$) and on $F_{R}$ ($\mathcal{O}_{R,F_{R}}$). A corollary is that, for polynomials of a fixed degree, there are only finitely many isomorphism classes of the $C^{*}$-algebras $\mathcal{O}_{R,J_{R}}$ and two such $C^{*}$-algebras are isomorphic if and only if their corresponding rational functions have the same number of critical points inside their Julia sets and have the same number of Fatou cycles. In general, we express the $K$-theory of $\mathcal{O}_{R,J_{R}}$ in terms of the Fatou cycles and the kernel and co-kernel of a matrix associated to the oriented Herman cycles of $R$ (a type of Fatou cycle) and the location of its critical points relative to these orientations. A new result from our calculations is that these kernel and co-kernel groups, as well as the Fatou cycle length data of $R$, are invariants for the topological conjugacy class of $R$ restricted to its Julia set.
\par
Our approach is to study a category $\textbf{Cor}$ where the objects are $C^{*}$-correspondences and the morphisms are ``intertwiners'' between $C^{*}$-corrrespondences. A special sub-category of $\textbf{Cor}$ has been studied by a number of authors, see for instance \cite[Section~2.4]{MS19} and the references therein. To a correspondence one can always associate a class in a $KK^{0}$ group. We show for the first time this association is functorial. We are also the first to study exactness in this category and its $KK$-theoretic consequences. Our main tools in calculating the $K$-theory of these algebras are derived from these consequences.
\par
Out of all families of rational functions, the quadratic family $f_{c}(z) = z^{2} + c$, $c$ in $\mathbb{C}$, has been studied the most extensively. By the work of Douady and Hubbard \cite{DH85}, we know that there is a deep relationship between the dynamical properties of $f_{c}$ and the location of the parameter $c$ relative to the Mandelbrot set $\mathcal{M}\subseteq\mathbb{C}$. For instance, by \cite[Proposition~11(b)]{DH85}, the set of parameters $c$ for which $f_{c}$ is \textit{structurally stable} on its Julia set (or $J$-stable) is precisely $\mathbb{C}\setminus \partial \mathcal{M}$. One of the most important conjectures of holomorphic dynamics is that $J$-stability of $f_{c}$ is equivalent to hyperbolicity of $f_{c}$. This conjecture is equivalent to the conjecture that hyperbolic quadratics are dense. See \cite[p.~201]{MSS83} for a brief discussion of the importance of this conjecture or \cite{M14}, \cite{Benini:MLC_survey} for surveys on progress made so far. 
\par
We specialize our $K$-theory calculations to the case of the quadratics and show there are only four isomorphism types of $C^{*}$-algebras, and the isomorphism type is dependent on the location of $0$ relative to the filled Julia set $K_{c}$ of $f_{c}$. A consequence of this is that the Density of Hyperbolicity Conjecture is equivalent to the density of quadratics with a certain $C^{*}$-algebra.
\section{Acknowledgements}
This project has received funding from an NSERC Discovery Grant in Canada and the European Research Council (ERC) under the European Union's Horizon 2020 research and innovation programme (grant agreement No. 817597). It also forms part of the PhD thesis of the author at the University of Glasgow, Scotland. Early work on this project was done while the author was a MSc student at the University of Victoria, Canada, and he thanks Ian Putnam for many enlightening conversations about $K$-theory. He would also like to express thanks to Xin Li and Michael Whittaker for helpful comments on drafts of this project.

\section{Statement of Results}
\subsection{$K$-theory of $\mathcal{O}_{R,\hat{\mathbb{C}}}$}
The $C^{*}$-algebra $\mathcal{O}_{R,\hat{\mathbb{C}}}$ of a rational function $R:\hat{\mathbb{C}}\mapsto \hat{\mathbb{C}}$ is the Exel crossed product of $C(\hat{\mathbb{C}})$ by the endomorphism $R^{*}:C(\hat{\mathbb{C}})\mapsto C(\hat{\mathbb{C}})$ and the transfer operator $\Phi$, defined for $f$ in $C(\hat{\mathbb{C}})$ as 
$$\Phi(f)(z) = \sum_{w:R(w) = z}\text{ind}_{R}(w)f(w), \text{ }z\text{ in }\hat{\mathbb{C}},$$
where $\text{ind}_{R}(w)$ is the (positive) local winding number of $R:\hat{\mathbb{C}}\mapsto \hat{\mathbb{C}}$ about $w$. The weightings $\text{ind}_{R}$ are required so that $\Phi$ maps continuous functions to continuous functions. Kajiwara and Watatani first define $\mathcal{O}_{R,\hat{\mathbb{C}}}$ as the Cuntz-Pimsner algebra of the natural bi-module $E_{R,\hat{\mathbb{C}}}$ associated to $R^{*}$ and $\Phi$, but these are equivalent constructions (\cite[Proposition~3.2]{KW:C*-algebras_associated_with_complex_systems}). The $C^{*}$-algebras $\mathcal{O}_{R,X}$ for $X = F_{R}, J_{R}$ are defined similarly and were also first studied in \cite{KW:C*-algebras_associated_with_complex_systems}.
\par
Our first new result is the calculation of the $K$-theory of $\mathcal{O}_{R,\hat{\mathbb{C}}}$. For each $X = F_{R},\hat{\mathbb{C}}, J_{R}$, we shall denote by $C_{R,X}$ the critical points $\{c\in X:\text{ind}_{R}(c) > 1\}$. These are finite sets. Denote $|C_{R,X}| = c_{R,X}$.
\newtheorem*{thm:forC}{Theorem 1 (\ref{CK})}
\begin{thm:forC}
Let $R$ be a rational function of degree $d>1$. Then,
$K_{0}(\mathcal{O}_{R,\hat{\mathbb{C}}})\simeq \mathbb{Z}^{c_{R,\hat{\mathbb{C}}}+1}$, with the class of the unit corresponding to a generator in a minimal generating set for $\mathbb{Z}^{c_{R,\hat{\mathbb{C}}}+1}$, and $K_{1}(\mathcal{O}_{R,\hat{\mathbb{C}}})\simeq\mathbb{Z}$.
\end{thm:forC}
There are rational functions with $J_{R} = \hat{\mathbb{C}}$. Some examples are considered by Lattès in \cite{L18}, and Rees showed in \cite{R86} there is a ``positive measure'' subset of rational functions which have $J_{R} = \hat{\mathbb{C}}$. Thus, we have already determined the $K$-theory of $\mathcal{O}_{R,J_{R}}$ for a large class of examples.
\\
\subsection{$K$-theory of $\mathcal{O}_{R,F_{R}}$}
Our second new result we present is the $K$-theory of $\mathcal{O}_{R,F_{R}}$. 
The \textit{Fatou set} $F_{R}$ is the maximal open set $U$ such that the iterates $\{R^{\circ n}|_{U}\}_{n\in\mathbb{N}}$ are pre-compact in the compact-open topology on $C(U,\hat{\mathbb{C}})$. It can be shown to be \textit{totally invariant} in the sense that $R^{-1}(F_{R}) = F_{R}$ (see \cite[Lemma~4.3]{Milnor:Dynamics_in_one_complex_variable}). For this reason, $\mathcal{O}_{R,F_{R}}$ is an ideal of $\mathcal{O}_{R,\hat{\mathbb{C}}}$.
A maximal connected component $U$ of $F_{R}$ is called a \textit{Fatou component}. It is clear that $R$ maps any Fatou component onto another such component. A finite collection $P$ of Fatou components is a \textit{Fatou cycle} if  we can write $P = \{U_{1},...,U_{n}\}$, where $U_{i}\neq U_{j}$ for all $i\neq j$ and
$R(U_{i}) = U_{i+1}$ for all $i\leq n$, with indices taken modulo $n$. We will denote the set of Fatou cycles of $R$ by $\mathcal{F}_{R}$.
\par
By \cite[Corollary~2]{Shishikura}, $\mathcal{F}_{R}$ must be a finite set.
By \cite[Theorem~1]{NWDSullivan}, every Fatou component is eventually mapped onto an element in a Fatou cycle. Moreover, a Fatou cycle $P$ is either an
\begin{enumerate}
    \item \textit{attracting} cycle: for every $U$ in $P$, $\{R^{\circ nk}:U\mapsto U\}_{k\in\mathbb{N}}$ converges on compact sets to an attracting fixed point in $U$,
    \item \textit{parabolic} cycle: for every $U$ in $P$, $\{R^{\circ nk}:U\mapsto U\}_{n\in\mathbb{N}}$ converges on compact sets to a parabolic (see \cite[Section~10]{Milnor:Dynamics_in_one_complex_variable}) fixed point in $\partial U$,
    \item \textit{Siegel} cycle: for every $U$ in $P$, $R^{\circ n}:U\mapsto U$ is conformally conjugate to an irrational rotation on $\mathbb{D}$, or
    \item \textit{Herman} cycle: for every $U$ in $P$, $R^{\circ n}:U\mapsto U$ is conformally conjugate to an irrational rotation on $\mathbb{A}_{r} = \{z\in\mathbb{C}: 1< |z|< r\}$ for some $r>1$.
\end{enumerate}
See \cite[Theorem~16.1]{Milnor:Dynamics_in_one_complex_variable} for a proof of this fact. This is known as the Fatou-Julia-Sullivan classification of Fatou components, as it was proven in part by Fatou and Julia, and finished by Sullivan in \cite{NWDSullivan} after he adapted the techniques of Kleinian group theory to the study of holomorphic dynamical systems.
\par
We shall denote the set of Herman cycles for $R$ by $\mathcal{H}_{R}$. They play a very special role in the groups for $R:F_{R}\mapsto F_{R}$ and $R:J_{R}\mapsto J_{R}$, as we shall see. Denote $f_{R} = |\mathcal{F}_{R}|$ and $h_{R} = |\mathcal{H}_{R}|$.

\newtheorem*{thm:forF}{Theorem 2 (\ref{FK})}
\begin{thm:forF} Let $R$ be a rational function of degree $d>1$. Then, $K_{0}(\mathcal{O}_{R, F_{R}})\simeq\mathbb{Z}^{c_{R,F_{R}} + f_{R} + h_{R}}$ and $K_{1}(\mathcal{O}_{R,F_{R}})\simeq\mathbb{Z}^{f_{R} + h_{R}}$.

\end{thm:forF}
The proof uses the classification of Fatou components in an essential way. The Herman cycles are counted twice due to the non-triviality of $K^{-1}$ of an annulus. In all other cycle types, any compact subset of the cycle will be eventually mapped into an ``attractor'' with trivial $K^{-1}$ (see Corollary \ref{fatattractor}).
\\
\subsection{$K$-theory of $\mathcal{O}_{R,J_{R}}$}
Our third new and most important result is the $K$-theory of $\mathcal{O}_{R,J_{R}}$. The \textit{Julia set} $J_{R}$ is the complement of the Fatou set. It is a compact uncountable subset of $\hat{\mathbb{C}}$ with no isolated points and is either the entire Riemann sphere, or it has no interior. Moreover, repelling periodic points of $R$ are dense in $J_{R}$, and for every open set $V\subseteq J_{R}$, there is some number $n$ such that $R^{\circ n}(V) = J_{R}$; see \cite[Theorem~14.1]{Milnor:Dynamics_in_one_complex_variable} and \cite[Corollary~14.2]{Milnor:Dynamics_in_one_complex_variable} for proofs of these assertions. Thus, the dynamics restricted to $J_{R}$ behaves chaotically. The above dynamical properties also imply $\mathcal{O}_{R,J_{R}}$ is purely infinite and simple (see \cite[Theorem~3.8]{KW:C*-algebras_associated_with_complex_systems}). For further properties of these algebras, see \cite{IzumiKW:KMS_states_branched}, \cite{KW17} and \cite{I23}.
\\
\par
We now present the $K$-theory of $\mathcal{O}_{R,J_{R}}$.
Let $Q$ be a Herman cycle of length $n$ and $U$ a Fatou component in $Q$. By \cite[Lemma~15.7]{Milnor:Dynamics_in_one_complex_variable} the boundary of $U$ has two connected components. We orient $Q$ by first choosing a boundary component $\partial^{+}U$ of $\partial U$. We then orient every other element $V$ in $Q$ (choosing a component $\partial^{+}V$ of $\partial V$) by declaring $R$ to be orientation preserving on $Q$, in the sense that  $\partial^{+}R^{\circ i}(U) = R^{\circ i}(\partial^{+}U)$ for every $i\leq n$. Since $R^{\circ n}:U\mapsto U$ is conjugate to an irrational rotation, it is homotopic to the identity. Therefore, $R^{\circ n}(\partial^{+}U) = \partial^{+}U$ and thus the orientation given to $Q$ is well defined. Note also that there are only two possible orientations on $Q$ defined in this way. We call $Q$ with such a choice of orientation an \textit{oriented Herman cycle}.
\par
We will assume now that all of our Herman cycles in $\mathcal{H}_{R}$ are oriented. For $x$ in $J_{R}$ and $Q$ in $\mathcal{H}_{R}$, we will let $H_{Q}(x)$ be the number of Fatou components $U$ in $Q$ for which $x$ is in the connected component of $\hat{\mathbb{C}}\setminus U$ containing $\partial^{+}U$. The function $H_{Q}:J_{R}\mapsto\mathbb{Z}$ is continuous and satisfies the property that $H_{Q} - \Phi(H_{Q})$ is constant (Corollary \ref{Hconst}). Moreover, one can show $\{H_{Q}\}_{Q\in\mathcal{H}_{R}}\cup\{1_{J_{R}}\}$ are linear independent over $\mathbb{Z}$ and generate the subgroup $\{f\in C(J_{R},\mathbb{Z}): f - \Phi(f)\text{ is constant}\}$ (Proposition \ref{k0nc}).
\par
We let $H_{R}$ be the following matrix with rows indexed by elements in $C_{R,J_{R}}\cup\{u\}$ and columns indexed by $\mathcal{H}_{R}\cup\{u\}$ ($u$ is an extra index we must add): 
\begin{itemize}
\item $(H_{R})_{c,Q} = H_{Q}(c)$ for all $c$ in $C_{R,J_{R}}$ and $Q$ in $\mathcal{H}_{R}$,
\item $(H_{R})_{c,u} = 1$, for all $c$ in $C_{R,J_{R}}$,
\item $(H_{R})_{u,Q} = \Phi(H_{Q}) - H_{Q}$, for all $Q$ in $\mathcal{H}_{R}$, and
\item $(H_{R})_{u,u} = \text{deg}(R)-1$.
\end{itemize}
$H_{R}$ determines a group homomorphism $H_{R}:\mathbb{Z}[\mathcal{H}_{R}\cup\{u\}]\mapsto \mathbb{Z}[C_{R,J_{R}}\cup\{u\}]$.
Denote by $\omega_{R}$ the greatest common divisor of the Fatou cycle lengths $\{|P|\}_{P\in\mathcal{F}_{R}}$. If $\mathcal{F}_{R}=\emptyset$, then we set $\omega_{R} = 1$.

\newtheorem*{thm:forJ*}{Theorem 3 (\ref{jk})}
\begin{thm:forJ*}
Let $R$ be a rational function of degree $d>1$. Then, $K_{1}(\mathcal{O}_{R,J_{R}})\simeq\text{ker}(H_{R})\oplus\mathbb{Z}/\omega_{R}\mathbb{Z}\oplus\mathbb{Z}^{|f_{R}-1|}$ and $K_{0}(\mathcal{O}_{R,J_{R}})\simeq \text{co-ker}(H_{R})\oplus\mathbb{Z}^{|f_{R} +h_{R} -1|}$, with class of the unit corresponding to the class of $u$ in $\text{co-ker}(H_{R})$.
\end{thm:forJ*}
Interestingly, the matrix $H_{R}$, its kernel, or its co-kernel do not appear anywhere in the complex dynamics literature. Since these two groups are invariants for the topological conjugacy class of $R:J_{R}\mapsto J_{R}$ (see Corollary \ref{apphgroup}), it would be worthwhile to compare it with other invariant combinatorial objects of $R$ constructed from oriented Herman cycles, like its Shishikura tree \cite{S89}.
\par
When $\mathcal{H}_{R} = \emptyset$ (which happens, for instance, for all polynomials), we have $\text{ker}(H_{R}) = 0$ and $\text{co-ker}(H_{R})\simeq \mathbb{Z}^{c_{R, J_{R}}}$ if $C_{R,J_{R}}\neq \emptyset$ and $\text{co-ker}(H_{R})\simeq\mathbb{Z}/(d-1)\mathbb{Z}$ otherwise, where $d$ is the degree of $R$. In the first case, the class of the unit corresponds to a generator in a minimal generating set for $\mathbb{Z}^{c_{R,J_{R}}}$, while in the second case the unit generates $\mathbb{Z}/(d-1)\mathbb{Z}$.
\par
Since a polynomial $R$ always has an attracting cycle of length $1$ at $\infty$, we always have $f_{R}\geq 1$ and $\omega_{R} = 1$. We therefore have a simple characterization for when two polynomials have isomorphic $C^{*}$-algebras:
\newtheorem*{thm:poly}{Theorem 4 (\ref{appcharpol})}
\begin{thm:poly}
Let $R$ and $S$ be non-linear polynomials. Then, $\mathcal{O}_{R,J_{R}}$ is isomorphic to $\mathcal{O}_{S,J_{S}}$ if and only if either
\begin{itemize}
    \item $c_{R,J_{R}} = c_{S,J_{S}} = 0$, $\text{deg}(R) = \text{deg}(S)$, and $f_{R} = f_{S}$, or 
    \item $c_{R,J_{R}} = c_{S,J_{S}}\neq 0$ and $f_{R} = f_{S}$.
\end{itemize}

\end{thm:poly}
\par
Nekrashevych in \cite[Theorem~6.6]{ Nek:C-alg_and_self_similar} computed the $K$-theory of $\mathcal{O}_{R,J_{R}}$ in the special case that $R$ is hyperbolic and post-critically finite. This just means that every critical point is eventually mapped to an attracting periodic orbit. There are only countably many distinct conjugacy classes of such $R$ by Thurston's Rigidity Theorem (see \cite[Theorem~2.2]{pcf}), and for quadratic polynomials there are only two: the class of $z^{2}$ and of $z^{2} - 1$. With the extra structure afforded to such a rational function, Nekrashevych constructs a descending sequence of approximations $(\mathcal{C}_{n})_{n\in\mathbb{N}}$ of the Julia set, each with manageable $K$-theory, computes the $K$-theory for $R:\mathcal{C}_{n+1}\mapsto\mathcal{C}_{n}$, then takes a limit as $n$ approaches $\infty$ to get the $K$-theory for $\mathcal{O}_{R,J_{R}}$. We remark that this approximation method is not viable in general, due to the fact that the Fatou set can be more complicated (the components may no longer be simply connected, and all four of the cycle types might appear, not just attractors). We review our method and our main results on the $C^{*}$-correspondence category in Section \ref{Methods}.
\\
\par
\subsection{Applications}
Now, we present applications. By the \textit{Fatou cycle length data of $R$} we shall mean the tuple $L_{R} = (|P|)_{P\in \mathcal{F}_{R}}$, where the entries are ordered in non-decreasing order. Similarly, we define the \textit{Herman cycle length data of $R$} to be the tuple $T_{R} = (|Q|)_{Q\in \mathcal{H}_{R}}$.
\newtheorem*{app:1}{Corollary 5 (\ref{appconjboot})}
\begin{app:1}
Let $R$ and $S$ be rational functions. If $R$ and $S$ are topologically conjugate on their Julia sets, then $L_{R} = L_{S}$ and $T_{R} = T_{S}$.
\end{app:1}
It is interesting that data from the dynamics of $R$ on its Fatou set is an invariant for its dynamics on the Julia set - this result reflects the rigid nature of rational dynamics.
The above corollary is proven by observing that if $R$ and $S$ are conjugate on their Julia sets, then  $R^{\circ n}$ and $S^{\circ n}$ are conjugate on their Julia sets, for all $n\in\mathbb{N}$, and so our calculations in Section \ref{ratmapj} imply $f_{R^{\circ n}} = f_{S^{\circ n}}$ and $h_{R^{\circ n}} = h_{S^{\circ n}}$. By a lemma concerning elementary number theory (Lemma \ref{number}), these sequences of numbers are equal if and only if $L_{R} =L_{S}$ and $T_{R} = T_{S}$. It is not clear how this result would be proven using a dynamical argument.
\\
\par
We now provide a more detailed description of the $K$-theory for a quadratic polynomial. Every quadratic polynomial is conjugate to one of the form $f_{c}(z) = z^{2} + c$, $z$ in $\hat{\mathbb{C}}$, for some $c$ in $\mathbb{C}$. Note that its only critical points are $0$ and $\infty$, with $\infty$ being a super-attracting fixed point inside the Fatou set. Denote $J_{f_{c}} = J_{c}$ and $\mathcal{O}_{f_{c},J_{c}} = \mathcal{O}_{c,J}$. Define the \textit{filled Julia set} of $f_{c}$ to be the points $z$ for which the orbit $\{f^{\circ n}_{c}(z)\}_{n\in\mathbb{N}}$ is bounded in $\mathbb{C}$ and denote this set $K_{c}$. It can be shown that $\partial K_{c} = J_{c}$ and that $K_{c}$ is the union of $J_{c}$ along with the bounded Fatou components of $f_{c}$ (see \cite[Lemma~9.4]{Milnor:Dynamics_in_one_complex_variable}).
\newtheorem*{app:2}{Corollary 6 (\ref{appclassquad1})}
\begin{app:2}
There are four isomorphism types for $\mathcal{O}_{c,J}$, dependent on the location of $0$ relative to the filled Julia set.
\begin{enumerate}
    \item[] \textbf{Case 0} $(0\notin K_{c}):$ Then, $K_{1}(\mathcal{O}_{c,J}) = K_{0}(\mathcal{O}_{c,J}) = 0$.
    \par
    \item[] \textbf{Case 1}  $(0\in\text{int}(K_{c})):$ Then, $K_{1}(\mathcal{O}_{c,J})\simeq K_{0}(\mathcal{O}_{c,J})\simeq\mathbb{Z}$ and $[1_{J_{c}}] = 0$.
    
    \item[] \textbf{Case 2} $(0\in \partial K_{c} = J_{c}$,  $\text{int}(K_{c})\neq\emptyset):$ Then $K_{1}(\mathcal{O}_{c,J})\simeq\mathbb{Z}$, $K_{0}(\mathcal{O}_{c,J})\simeq\mathbb{Z}^{2}$ and $[1_{J_{c}}]$ is a generator in a minimal generating set for $\mathbb{Z}^{2}$.
    
    \item[] \textbf{Case 3} $(0\in \partial K_{c} =J_{c}$, $\text{int}(K_{c}) = \emptyset):$ Then $K_{1}(\mathcal{O}_{c,J}) = 0$, $K_{0}(\mathcal{O}_{c,J})\simeq\mathbb{Z}$ and $[1_{J_{c}}]$ is a generator.
\end{enumerate}
\end{app:2}
Case 0,1,3 above are realized by $c = 1, 0, -2$, respectively, while case 2 is realized by $c = \frac{e^{2\pi i\varphi}}{2}(1 + \frac{e^{2\pi i\varphi}}{2})$, where $\varphi$ is the golden ratio.
\par
We now identify the corresponding $C^{*}$-algebras to each case, and for more details see the discussion below Corollary \ref{bolicchar}.
Case $0$ corresponds to the Cuntz algebra $\mathcal{O}_{2}$ considered first in \cite{C77}. Case $1$ corresponds to the 2-adic ring $C^{*}$-algebra $\mathcal{Q}_{2}$ studied by Larsen and Li in \cite{LL12}, and the author thanks Chris Bruce for this identification. It also appears in other contexts; see \cite[remark~3.2]{LL12}. To our knowledge, case $2$ corresponds to a $C^{*}$-algebra that has not been studied in the literature. It can be shown it is a $C^{*}$-algebra of a certain graph, see the figure above Corollary \ref{mandelconjeccor}. We denote it $\mathcal{Q}_{2,\infty}$ as it seems to share properties of both $\mathcal{Q}_{2}$ and $\mathcal{O}_{\infty}$. Case $3$ corresponds to the Cuntz algebra $\mathcal{O}_{\infty}$ considered in \cite{C77} also.
\par
A consequence of the above Corollary is that $[1_{J_{c}}] = 0$ if and only if $f_{c}$ is either hyperbolic or parabolic. This allows us to re-state the Density of Hyperbolicity Conjecture for quadratics in terms of the $K$-theory for quadratics. Recall that the Mandelbrot set is defined as $\mathcal{M} = \{c\in\mathbb{C}:0\in K_{c}\}$ ( $= \{c\in\mathbb{C}: K_{0}(\mathcal{O}_{c,J})\neq 0\}$).
\newtheorem*{app:3}{Corollary 7 (\ref{mandelconjeccor})}
\begin{app:3}
The Density of Hyperbolicity Conjecture is true if and only if
$\mathcal{H}':=\{c\in\mathcal{M}: [1_{J_{c}}] = 0 \text{ in }K_{0}(\mathcal{O}_{c,J})\} = \{c\in \mathcal{M}:\mathcal{O}_{c,J}\simeq\mathcal{Q}_{2}\}$ is dense in $\mathcal{M}$.
\end{app:3}
We do not claim this makes the conjecture any easier, though it is interesting to see it in $C^{*}$-algebraic terms.
\section{Methods}
\label{Methods}
\subsection{Motivation}
We now motivate our approach to calculating the $K$-theory of $\mathcal{O}_{R,J_{R}}$. The situation is that we have a short exact sequence of Cuntz-Pimsner algebras $$
\begin{tikzcd}
0 \arrow[r] & {\mathcal{O}_{R,F_{R}}} \arrow[r] & {\mathcal{O}_{R,\hat{\mathbb{C}}}} \arrow[r] & {\mathcal{O}_{R,J_{R}}} \arrow[r] & 0
\end{tikzcd}$$ and the Pimsner-Voiculescu 6-term exact sequence (\cite[Theorem~4.9]{Pimsner:Generalizing_Cuntz-Krieger}) of $\mathcal{O}_{R,F_{R}}$ and $\mathcal{O}_{R,\hat{\mathbb{C}}}$ can be calculated. In these cases, the exact sequences split, so it determines for us the $K$-theory of $\mathcal{O}_{R,F_{R}}$ and $\mathcal{O}_{R,\hat{\mathbb{C}}}$. The problem is that this splitting is not natural, which makes calculating the $K$-theory of $\mathcal{O}_{R,J_{R}}$ from the $6$-term exact sequence determined by the above extension of $C^{*}$-algebras rather difficult. Our observation is that not only do we have a short exact sequence of $C^{*}$-algebras, we also have a short exact sequence of their defining $C^{*}$-correspondences
$$
\begin{tikzcd}
0 \arrow[r] & {(E_{R,F_{R}},\alpha_{F_{R}})} \arrow[r] & {(E_{R,\hat{\mathbb{C}}},\alpha_{\hat{\mathbb{C}}})} \arrow[r] & {(E_{R,J_{R}},\alpha_{J_{R}})} \arrow[r] & 0,
\end{tikzcd}$$
and this will induce exact sequences at the level of the building blocks of $K$-theory determined by the Pimsner-Voiculescu $6$-term exact sequences. Thus, we circumvent the need to work with an un-natural splitting. This is our motivation for considering the category of $C^{*}$-correspondences, which we now present along with our main results about it.
\\
\subsection{Morphisms of $C^{*}$-correspondences}
Recall that if $A_{1}$, $A_{2}$ are $C^{*}$-algebras, then an \textit{$A_{1}$-$A_{2}$ correspondence} ${}_{A_{1}}(E,\alpha)_{A_{2}}$ is a left action $\alpha$ of $A_{1}$ as endomorphisms of a Hilbert $A_{2}$-module $E$. Given such a correspondence, we can always restrict the left action $\alpha$ to the ideal $J_{(E,\alpha)}$ consisting of elements of $A_{1}$ which act by compact operators of $E$. After restriction this correspondence defines a class $K(E,\alpha)$ in $KK^{0}(J_{(E,\alpha)},A_{2})$. We define a category for which the assignment $(E,\alpha)\mapsto K(E,\alpha)$ is functorial.
\par
A \textit{morphism of correspondences} in our sense is a triple $(\varphi_{1}, S, \varphi_{2}): {}_{A_{1}}(E,\alpha)_{A_{2}}\mapsto {}_{B_{1}}(F,\beta)_{B_{2}}$, where $\varphi_{1}:A_{1}\mapsto A_{2}$, $\varphi_{2}:A_{2}\mapsto B_{2}$ are $*$-homomorphisms and $S:E\mapsto F$ is a linear map that satisfies compatibility criterion with $\varphi_{1}$, $\varphi_{2}$ comparable to that of a co-variant representation of a Hilbert bi-module; see Definition \ref{morph_of_corr}. The sub-category where $A_{1} = A_{2}$, $B_{1} = B_{2}$, and $\varphi_{1} = \varphi_{2}$ has been studied by many authors, usually with applications to Cuntz-Pimsner algebras in mind. See for instance \cite[Section~2.4]{MS19} and the references therein. However, $KK$-theory consequences have not been considered and this is our main focus. Our first main tool is that the assignment $(E,\alpha)\mapsto K(E,\alpha)$ is functorial, though we will not present it like this until Remark \ref{intertwiner}.

\newtheorem*{thm:cor_funct}{Proposition 8 (\ref{morphintertwine})}
\begin{thm:cor_funct}
Suppose $(\varphi_{1},S,\varphi_{2}): {}_{A_{1}}(E,\alpha){}_{A_{2}}\mapsto {}_{B_{1}}(F,\beta){}_{B_{2}}$ is a morphism of correspondences. Then, $K(E,\alpha)\hat{\otimes}[\varphi_{2}] = [\varphi_{1}]\hat{\otimes}K(F,\beta)$, where $\hat{\otimes}$ is the Kasporov product and $[\varphi_{i}]$, for $i=1,2$, is the class of the $*$-homomorphism $\varphi_{i}$ in $KK^{0}$. 
\end{thm:cor_funct}
We now present our second main tool. If $(\varphi_{1},S,\varphi_{2}): {}_{A_{1}}(E,\alpha){}_{A_{2}}\mapsto {}_{B_{1}}(F,\beta){}_{B_{2}}$ is a morphism, then it restricts to  a morphism $(\varphi_{1},S,\varphi_{2}):{}_{J_{(E,\alpha)}}(E,\alpha){}_{A_{2}}\mapsto {}_{J_{(F,\beta)}}(F,\beta){}_{B_{2}}$. We call this functor $J:\textbf{Cor}\mapsto \textbf{Cor}$.
\par
We say a sequence 
$$
\begin{tikzcd}
{{}_{A_{1}}(E,\alpha){}_{A_{2}}} \arrow[r, "{(\varphi_{1}, S, \varphi_{2})}"] & {{}_{B_{1}}(F,\beta){}_{B_{2}}} \arrow[r, "{(\psi_{1}, T, \psi_{2})}"] & {{}_{C_{1}}(G,\gamma){}_{C_{2}}}
\end{tikzcd}$$
is \textit{$J$-exact} if it is exact after applying the functor $J$ (Definition \ref{morphexact}).
\newtheorem*{thm:cor_exact}{Proposition 9 (\ref{morphnatural})}
\begin{thm:cor_exact}
If
$$\begin{tikzcd}
0 \arrow[r] & {}_{A_{1}}{(E,\alpha)}{}_{A_{2}} \arrow[r, "{(i_{1}, I,i_{2})}"] & {}_{B_{1}}{(F,\beta)}{}_{B_{2}} \arrow[r, "{(q_{1},Q,q_{2})}"] & {}_{C_{1}}{(G,\gamma)}{}_{C_{2}} \arrow[r] & 0
\end{tikzcd}$$ is a  $J$-exact sequence of correspondences, then $\delta_{J_{(F,\beta)}}\hat{\otimes} K(E,\alpha) = K(G,\gamma)\hat{\otimes}\delta_{B_{2}}$, where $\delta_{J_{(F,\beta)}}$, $\delta_{B_{2}}$ are the classes of the short exact sequences
$$
\begin{tikzcd}
0 \arrow[r] & {J_{(E,\alpha)}} \arrow[r, "i_{1}"] & {J_{(F,\beta)}} \arrow[r, "q_{1}"] & {J_{(G,\gamma)}} \arrow[r] & 0 \\
0 \arrow[r] & A_{2} \arrow[r, "i_{2}"]            & B_{2} \arrow[r, "q_{2}"]           & C_{2} \arrow[r]            & 0,
\end{tikzcd}$$

respectively, in $KK^{1}$.
\end{thm:cor_exact}
This result can be thought of as an extension of the usual naturality result for an extension class in $KK^{1}$ and is proven as such.
\par
Applying Proposition 5 and Proposition 6 above to $K$-theory yields the following diagram, which is our main tool in this paper.
\newtheorem*{thm:main_dia}{Corollary 10 (\ref{Maindiagram})}
\begin{thm:main_dia}

If
$$\begin{tikzcd}
0 \arrow[r] & {}_{A}{(E,\alpha)}{}_{A} \arrow[r, "{(i, I,i)}"] & {}_{B}{(F,\beta)}{}_{B} \arrow[r, "{(q,Q,q)}"] & {}_{C}{(G,\gamma)}{}_{C} \arrow[r] & 0
\end{tikzcd}$$ is a  $J$-exact sequence of correspondences,
then the following diagram commutes

$$    
\begin{tikzcd}[sep=small]
& K_{0}(J_{(E,\alpha)})  \arrow[rr] \arrow[dd] & & K_{0}(J_{(F,\beta)}) \arrow[rr] \arrow[dd] & & K_{0}(J_{(G,\gamma)}) \arrow[dl] \arrow[dd]\\
K_{1}(J_{(G,\gamma)})  \arrow[dd] \arrow[ur]& & K_{1}(J_{(F,\beta)})\arrow[dd, crossing over] \arrow[ll, crossing over] & & K_{1}(J_{(E,\alpha)}) \arrow[dd, crossing over] \arrow[ll, crossing over]\\
& K_{0}(A) \arrow[rr] & & K_{0}(B) \arrow[rr] & & K_{0}(C)\arrow[dl]\\
K_{1}(C) \arrow[ur]& & K_{1}(B) \arrow[ll] \arrow[from=uu, crossing over] \arrow[ll] & & K_{1}(A) \arrow[ll]\\
\end{tikzcd}
$$
The top and bottom horizontal faces are the 6-term exact sequences of $K$-theory associated to the respective extensions of $C^{*}$-algebras, and the vertical maps are of the form $\iota - \hat{\otimes}_{i}K(H,\eta)$, where $\iota:K_{i}(J_{(H,\eta)})\mapsto K_{i}(D)$ is the map induced by inclusion and ${}_{D}(H,\eta){}_{D} = {}_{A}(E,\alpha){}_{A}, {}_{B}(F,\beta){}_{B}, {}_{C}(G,\gamma){}_{C}$.
\end{thm:main_dia}
The result in the paper is more generally stated for non self-correspondences; see Corollary \ref{Maindiagram}. We can also get similar diagrams for $K$-homology (with arrows reversed) and more generally for $KK^{*}(-, A)$, $KK^{*}(A,-)$, where $A$ is any seperable nuclear $C^{*}$-algebra.
\par
We can extract exact sequences involving the $K$-theory building blocks of $\mathcal{O}_{R,F_{R}}$, $\mathcal{O}_{R,\hat{\mathbb{C}}}$ and $\mathcal{O}_{R,J_{R}}$ by applying the Snake Lemma strategically to certain vertical faces of the above diagram. This is done in Section \ref{ratmapj}, and it is how we calculate the $K$-theory of $\mathcal{O}_{R,J_{R}}$.
\subsection{Branched functions}
Inspired by Kajiwara and Watatani's construction of a $C^{*}$-correspondence from a rational function, we define $C^{*}$-correspondences from \textit{branched functions}. A function $F:X\mapsto Y$ is branched if at every point in its range, it admits a system of inverse branches; see Definition \ref{branchedfunction}. This represents a broad and interesting class of functions, including all holomorphic functions between Riemann surfaces. In particular, we can for the first time define a $C^{*}$-algebra from an arbitrary holomorphic dynamical system without excluding its branched points.
\par
No similar definition of a branched function exists in the literature. Our reason for introducing branched functions here is that we will be working with holomorphic functions restricted to a variety of subspaces, and we would like a unified definition for $C^{*}$-correspondences associated to these to make applying our results on morphisms of $C^{*}$-correspondences to such restrictions easier. 
\par
We present two results that will be useful for us when calculating $K$-theory. Denote the $C^{*}$-correspondence of a branched function $F:X\mapsto Y$ by $(E_{F,X},\alpha_{X})$.
\newtheorem*{branched}{Proposition 11 (\ref{classnaturality})}
\begin{branched}
If $F:X\mapsto Y$ is a branched function and $U\subseteq X$, $V\subseteq Y$ are open sets such that $F(U)\subseteq V$, then $F:U\mapsto V$ is branched and the inclusion functions induce a morphism ${}_{C_{0}(X)}(E_{F,X}, \alpha_{X}){}_{C_{0}(Y)}\mapsto {}_{C_{0}(U)}(E_{F,U},\alpha_{U}){}_{C_{0}(V)}$
\end{branched}
We also characterize when a restriction of a branched function to a closed set yields a short $J$-exact sequence. For a branched function $F:X\mapsto Y$, denote by $C_{F,X} = \{x\in X: F\text{ is not locally injective at }x\}$.
\newtheorem*{branchedr}{Proposition 12 (\ref{holmorph1})}
\begin{branchedr}
Let $F:U\mapsto V$ be a branched function and $Y$ a closed subset of $V$. Denote $F^{-1}(Y) = X$. Then, inclusion and restriction induce morphisms $(E_{F,U\setminus X},\alpha_{U\setminus X})\mapsto (E_{U},\alpha_{U})$ and $(E_{U},\alpha_{U})\mapsto (E_{X},\alpha_{X})$, respectively. $C_{F,U}\cap X = C_{F,X}$ if and only if the sequence 
$$
\begin{tikzcd}
0 \arrow[r] & {(E_{F,U\setminus X},\alpha_{U\setminus X})} \arrow[r] & {(E_{F,U},\alpha_{U})} \arrow[r] & {(E_{F,X},\alpha_{X})} \arrow[r] & 0
\end{tikzcd}$$
is $J$-exact.
\end{branchedr}
As a special case, any holomorphic dynamical system $R:M\mapsto M$, where $M$ is a connected Riemann surface, will determine a $J$-exact sequence of correspondences via the restriction to its Julia set. See Section \ref{holcor} for more information on these correspondences.
\section{Organization}
In Section \ref{background}, we review the necessary background required to understand this paper. We review Hilbert modules, $C^{*}$-correspondences, Cuntz-Pimsner algebras, $K$-theory, $KK$-theory and holomorphic dynamical systems. In Section \ref{morphcorr} we define the category of $C^{*}$-correspondences and study its relation to $KK$-theory.
We define branched functions in Section \ref{holcorrespondence}, construct their associated $C^{*}$-correspondences and study them within the context of our category.
We specialize to branched functions that are restrictions of holomorphic functions in Section \ref{holcor}, prove some regularity properties about their $C^{*}$-correspondences and describe their action on $K$-theory.
Sections \ref{ratmapc}, \ref{ratmapf} and \ref{ratmapj} calculate the $K$-theory of $\mathcal{O}_{R,\hat{\mathbb{C}}}$, $\mathcal{O}_{R,F_{R}}$ and $\mathcal{O}_{R,J_{R}}$, respectively.
Section \ref{app} is devoted to proving applications of our $K$-theory results.

\section{Background}\label{background}
\subsection{$C^{*}$-correspondences}
\label{correspondence} We briefly review the concepts of Hilbert $C^{*}$-modules and $C^{*}$-correspondences, as well as some of their auxiliary notions. Our review of the standard definitions closely follow the exposition of Jensen and Thomsen in \cite[Section~1]{Jensen}. See \cite{Lance} for a more detailed treatment of the theory of Hilbert $C^{*}$-modules.
\par
Let $A$ be a $C^{*}$-algebra. A \textit{right Hilbert $A$-module} is a complex vector space $E$ with a right $A$-module structure,
together with a right $\mathbb{C}$-linear map
$\langle\cdot,\cdot\rangle:E\times E\mapsto A$ that is also
\begin{enumerate}
\item right $A$-linear: $\langle e_{1}, e_{2}\cdot a\rangle = \langle e_{1}, e_{2}\rangle b$, for all $e_{1}$, $e_{2}$ in $E$ and $a$ in $A$,
\item Hermitian: $\langle e_{1}, e_{2}\rangle^{*} = \langle e_{2}, e_{1}\rangle$, for all $e_{1}$, $e_{2}$ in $E$,
\item positive: $\langle e, e\rangle\geq 0$, for all $e$ in $E$,
\item faithful: $\langle e, e\rangle\neq 0$, for all non-zero $e$ in $E$,
\item and $E$ is complete with respect to the norm $\|e\|:=\|\sqrt{\langle e, e\rangle }\|$.
\end{enumerate}
We will call $\langle\cdot,\cdot\rangle$ an \textit{$A$-inner product} if it satisfies these conditions.
All $C^{*}$-algebras and Hilbert modules in this paper will be assumed \textit{seperable}.
\par
We say a Hilbert $A$-module is \textit{full} if the image of the $B$-inner product linearly spans a dense sub-*-algebra in $A$. In general, the closed linear span of $\langle\cdot,\cdot\rangle$ is an ideal in $A$, which we shall denote by $\langle E\rangle$.
\par
A right $A$-module endomorphism $T:E\mapsto E$ is said to be \textit{adjointable} if there exists another right $A$-module endomorphism $T^{*}:E\mapsto E$ such that 
$$\langle T(e_{1}),e_{2}\rangle = \langle e_{1}, T^{*}(e_{2})\rangle\text{ for all }e_{1},e_{2}\text{ in }E.$$
$T^{*}$ is called the $\textit{adjoint}$ of $T$. The collection of all bounded, adjointable right $A$-module endomorphisms is denoted $\mathcal{B}(E)$. It is a $C^{*}$-algebra with pointwise addition of endomorphisms as addition, composition as multiplication, the adjoint as the involution, and the operator norm induced from $(E,\|\cdot\|)$ as the norm. 
\par
For each $e_{1},e_{2}$ in $E$, define the endomorphism $\theta_{e_{1},e_{2}}:E\mapsto E$ as
$$\theta_{e_{1},e_{2}}(e_{3}) = e_{1}\cdot\langle e_{2}, e_{3}\rangle, \text{ for all }e_{3}\text{ in }E.$$
The closed linear span of all such endomorphisms is denoted as $\mathcal{K}(E)$, the \textit{compact operators of $E$}, and can be shown to be an ideal of $\mathcal{B}(E)$.
\par
To a Hilbert $A$-module $E$, we may construct its \textit{linking algebra} $$\mathcal{L}(E) = \{\begin{pmatrix}
k & e\\
f & a 
\end{pmatrix}:k\in \mathcal{K}(E), e,f\in E, a\in \langle E\rangle \},$$
which is a $C^{*}$-algebra with involution $\begin{pmatrix}
k & e\\
f & a 
\end{pmatrix}^{*} = \begin{pmatrix}
k^{*} & f\\
e & a^{*} 
\end{pmatrix}$, multiplication $$\begin{pmatrix}
k_{1} & e_{1}\\
f_{1} & a_{1} 
\end{pmatrix}\cdot \begin{pmatrix}
k_{2} & e_{2}\\
f_{2} & a_{2} 
\end{pmatrix} = \begin{pmatrix} 
k_{1}k_{2} + \theta_{e_{1}, f_{2}} & k_{1}(e_{2}) + e_{1}a_{2}\\
k^{*}_{2}(f_{1}) + f_{2}a_{1}^{*} & \langle f_{1}, e_{2}\rangle + a_{1}a_{2}
\end{pmatrix},$$ and addition  $$\begin{pmatrix}
k_{1} & e_{1}\\
f_{1} & a_{1} 
\end{pmatrix} + \begin{pmatrix}
k_{2} & e_{2}\\
f_{2} & a_{2} 
\end{pmatrix} = \begin{pmatrix}
k_{1} + k_{2} & e_{1} +e_{2}\\
f_{1} + f_{2} & a_{1} + a_{2} 
\end{pmatrix}.$$
See \cite[Lemma~3.19]{RW98} for a description of the $C^{*}$-norm on $\mathcal{L}(E)$.
$\mathcal{K}(E)$ and $\langle E\rangle$ are full and hereditary sub-$C^{*}$-algebras of $\mathcal{L}(E)$ with respect to the obvious corner embeddings (\cite[Theorem~3.19]{RW98}).
\par
Let $E$ be a Hilbert $A$-module, $F$ a Hilbert $B$-module, and $\varphi:A\mapsto B$ a $*$-homomorphism. A \textit{$\varphi$-twisted morphism from $E$ to $F$} is a linear map $S:E\mapsto F$ satisfying $\langle S(e), S(f)\rangle= \varphi(\langle e, f\rangle)$ for all $e,f$ in $E$. It follows that $S(e\cdot a) = S(e)\cdot\varphi(a)$, for all $e$ in $E$ and $a$ in $A$.
\par
It can be shown, using the linking algebra and a standard amplification trick, that, for any elements $\{e_{i}\}_{i=1}^{n}, \{f_{i}\}_{i=1}^{n}\subseteq E$, the norm of $\sum_{i=1}^{n}\theta_{e_{i},f_{i}}$ is equal to the norm of the matrix $\sqrt{F}\sqrt{E}$ in $M_{n}(A)$, where the $(i,j)$ entry of $E$, $F$, is $\langle e_{i}, e_{j}\rangle$, $\langle f_{i}, f_{j}\rangle$, respectively.
\par
It follows that
a $\varphi$-twisted morphism $S$ induces a well defined $*$-homomorphism 
$\hat{S}:\mathcal{K}(E)\mapsto\mathcal{K}(F)$, determined by the equation
$$\hat{S}(\sum_{i=1}^{n}\theta_{e_{i},f_{i}}) = \sum_{i=1}^{n}\theta_{S(e_{i}),S(f_{i})},\text{ for any } \{e_{i}\}_{i=1}^{n}, \{f_{i}\}_{i=1}^{n}\subseteq E.$$
It is then an easy check that $\mathcal{L}(S):\mathcal{L}(E)\mapsto \mathcal{L}(F)$, defined for $\begin{pmatrix}
k & e\\
f & a 
\end{pmatrix}$ in $\mathcal{L}(E)$ as 
$$\mathcal{L}(S)(\begin{pmatrix}
k & e\\
f & a 
\end{pmatrix}) = \begin{pmatrix}
\hat{S}(k) & S(e)\\
S(f) & \varphi(a) 
\end{pmatrix}$$ is a $*$-homomorphism. It follows that any $\varphi$-twisted morphism has closed image.
\par
Let $A$ and $B$ be $C^{*}$-algebras. an \textit{$A$-$B$ $C^{*}$-correspondence} is a right Hilbert $B$-module $E$ together with a $*$-homomorphism $\alpha:A\mapsto \mathcal{B}(E)$. We will call $\alpha$ the \textit{action of $A$ on $E$} and will frequently use the notation $a\cdot e = \alpha(a)(e)$, $a$ in $A$, $e$ in $E$. A correspondence will be denoted ${}_{A}(E,\alpha)_{B}$. We will frequently write ${}_{A}(E,\alpha)_{B} = (E,\alpha)$ if the algebras $A$ and $B$ are specified. $(E,\alpha)$ is \textit{faithful} if $\alpha$ is injective, \textit{full} if $E$ is full, and \textit{non-degenerate} if the linear span of $A\cdot E$ is dense in $E$. 
\par
We shall denote $J_{(E,\alpha)}:= \alpha^{-1}(\mathcal{K}(E))$. Note that this is an ideal of $A$.
\par
We can compose an $A$-$B$ $C^{*}$-correspondence  $(E,\alpha)$ with a $B$-$C$ $C^{*}$-correspondence $(F,\beta)$ to obtain an $A$-$C$ $C^{*}$-correspondence $(E\otimes_{B} F, \alpha\otimes_{B}\text{id})$. This construction is called the \textit{balanced (or internal) tensor product} and appears, for instance, in \cite[Section~1.2.3]{Jensen}. We review the construction briefly.
\par
Let $E\otimes F$ be the vector space tensor product of $E$ and $F$. Clearly this is still a right $C$-module, and carries the left $A$ action $\alpha\otimes\text{id}$.  We define a map $\langle\cdot,\cdot\rangle: E\otimes F\times E\otimes F\mapsto C$  first on basic tensors as $$\langle e_{1}\otimes f_{1}, e_{2}\otimes f_{2}\rangle := \langle f_{1}, \langle e_{1},e_{2}\rangle\cdot f_{2}\rangle, \text{ for all }e_{1},e_{2}\text{ in }E, \text{ }f_{1},f_{2}\text{ in }F,$$
and extend $\langle \cdot,\cdot\rangle $ linearly in the right-entry, anti-linearly in the left-entry to all of $E\otimes F\times E\otimes F$. $\langle\cdot,\cdot\rangle$ satisfies the axioms $(1)$-$(3)$ for a $C$-inner product, but not necessarily $(4)$ and $(5)$. 
\par
However, if we quotient out by the $C$-sub-module $\{g\in E\otimes F:\langle g,g\rangle = 0\}$ and complete with respect to the quotient norm of $\|\cdot\|:= \|\sqrt{\langle \cdot,\cdot\rangle}\|$, then we obtain a vector space $E\otimes_{B}F$ such that the right $C$-module structure and $\langle\cdot,\cdot\rangle$ on $E\otimes F$ pass down to a Hilbert $C$-module structure, and the left $A$ action passes down to a $*$-homomorphism
$\alpha\otimes_{B}\text{id}:A\mapsto \mathcal{B}(E\otimes_{B}F)$. $(E\otimes_{B}F, \alpha\otimes_{B}\text{id})$ is therefore an $A$-$C$ $C^{*}$-correspondence.
\subsection{Cuntz-Pimsner algebras}
\label{cpalg}
We now review the Cuntz-Pimsner algebra construction. We refer the reader to \cite{Pimsner:Generalizing_Cuntz-Krieger} for more details.
\par
Let $A$ be a $C^{*}$-algebra and $(E,\alpha)$ a full, faithful, $A$ - $A$ $C^{*}$-correspondence. Denote $E^{\otimes_{A} 0} = A$, $E^{\otimes_{A} 1} = E$, and define, for $n$ in $\mathbb{N}$, $E^{\otimes_{A} n}:= E\otimes_{A}E^{\otimes_{A}(n-1)}$. The \textit{Fock module of $E$} is the orthogonal sum $\sum_{n=0}^{\infty}E^{\otimes_{A} n} = \mathcal{F}_{E}$. It is itself an $A$-$A$ correspondence, and we shall denote the left action by $\alpha_{\mathcal{F}}$.
\par
For each $e$ in $E$, we can define an adjointable, bounded Fock module endomorphism $T_{e}:\mathcal{F}_{E}\mapsto \mathcal{F}_{E}$ as
$$T_{e}(g) = e\otimes_{A} g,\text{ for }g\text{ in }\mathcal{F}_{E}.$$
The $C^{*}$-algebra generated by $\{T_{e}\}_{e\in E}$ is called the \textit{Toeplitz algebra} of $E$, and is denoted $\mathcal{T}_{E}$. As an abuse of notation, for $a$ in $A$, we shall write $\alpha_{\mathcal{F}}(a) = a$. Then, note that $T^{*}_{e}T_{f} = \langle e, f\rangle$ for all $e,f$ in $E$, so, by full-ness of $E$,  $\mathcal{T}_{E}$ contains all of $A$. Moreover, $a\cdot T_{e} = T_{a\cdot e}$, $T_{e}\cdot a =  T_{e\cdot a}$, and $T_{e} + T_{f} = T_{e+f}$ for all $e, f$ in $E$ and $a$ in $A$. By \cite[Theorem~3.4]{Pimsner:Generalizing_Cuntz-Krieger}, $\mathcal{T}_{E}$ is the universal $C^{*}$-algebra with operators $T_{e}$, $e$ in $E$, and $a$ in $A$ satisfying the above relations.
\par
Let $J_{E} = J_{(E,\alpha)}$ and consider the Hilbert $J_{E}$-module $G_{E} = \{x\in\mathcal{F}_{E}:\langle x, x\rangle\in J_{E}\}$. Then, $\mathcal{K}(G_{E})$ (considered as a sub-algebra of $\mathcal{K}(\mathcal{F}_{E})$) is an ideal of $\mathcal{B}(\mathcal{F}_{E})$ contained in $\mathcal{T}_{E}$ (\cite[Theorem~3.13]{Pimsner:Generalizing_Cuntz-Krieger}). We shall denote this ideal by $\mathcal{I}_{E}$. It can be described in the following two ways:
\par
First, the map sending, for each pair $e,f$ in $E$, a compact operator $\theta_{e,f}$ in $\mathcal{K}(E)$ to $T_{e}T_{f}^{*}$ extends to a $*$-embedding $\phi:\mathcal{K}(E)\mapsto \mathcal{T}_{E}$, so, for every $a$ in $J_{E}$,
$a - \phi(\alpha(a))$ is in $\mathcal{T}_{E}$, and is in fact in $\mathcal{I}_{E}$. Moreover, such operators generate the ideal $\mathcal{I}_{E}$.
\par
Second, consider, for each $k$ in $\mathbb{N}_{0}$ the projection $P_{k}:\mathcal{F}_{E}\mapsto\mathcal{F}_{E}$ onto the Hilbert $A$-sub-module $\sum_{n=0}^{k}E^{\hat{\otimes}n}$. It can be shown that $T$ is in $\mathcal{I}_{E}$ if and only if $T$ is in $\mathcal{T}_{E}$ and $\lim_{k\to\infty}\|P_{k}T-T\|=0$; see \cite[Theorem~3.13]{Pimsner:Generalizing_Cuntz-Krieger}, \cite[Definition~1.1]{Pimsner:Generalizing_Cuntz-Krieger}, and condition $(4)$ above \cite[Remark~3.9]{Pimsner:Generalizing_Cuntz-Krieger} for the three descriptions of the ideal above, respectively.
\par
The \textit{Cuntz-Pimsner algebra of $E$} is the quotient $\mathcal{T}_{E}/\mathcal{I}_{E}$, and is denoted $\mathcal{O}_{E}$. $A\subseteq\mathcal{O}_{E}$ is often called the \textit{co-efficient algebra}.

\subsection{$K$-theory} 
\label{k}
We will assume the reader has basic knowledge of $K$-theory for $C^{*}$-algebras, including the 6-term exact sequence associated to an extension. A good reference for the background required is \cite{Rordam:intro_K-theory}. We mention that no $K$-theory class in the $K_{0}$ or $K_{1}$ group of a non-commutative $C^{*}$-algebra is ever computed explicitly in this paper (other than the class of the unit in $K_{0}$). For a locally compact Hausdorff space $X$, we shall write $K_{i}(C_{0}(X)) = K^{-i}(X)$, $i=0,1$, to keep certain notations from becoming too cumbersome, but we will still be using the operator $K$-theory picture for $X$. When $A$ is a sub-$C^{*}$-algebra of $B$, we will often denote by $j$ or $i$ the inclusion map $A\mapsto B$.

\subsection{$KK$-theory}
\label{kk}
We briefly review the (un-graded) $KK$-theory that is used in this paper. See \cite{Jensen} for more details about the general theory.
\par
Let $A$ and $B$ be $C^{*}$-algebras, and recall that a \textit{Kasparov $A$ - $B$ bi-module} is a triple $(E,\alpha, T)$, where $(E,\alpha)$ is an $A$ - $B$ correspondence and $T:E\mapsto E$ is a bounded adjointable endomorphism such that, for all $a$ in $A$, the operators $T\alpha(a) - \alpha(a)T$, $(T - T^{*})\alpha(a)$, and $(T^{2} - 1)\alpha(a)$ are in $\mathcal{K}(E)$.
\par
It is clear that the structure of a Kasparov $A$ - $B$ bi-module is preserved under direct sum. There is a notion of \textit{homotopy} between two Kasparov bi-modules such that, if we consider the set of all homotopy equivalence classes $KK^{0}(A,B)$, then this is an abelian group under the direct sum operation, with $0$ being the class of the zero $A$ - $B$ bi-module. We shall denote the homotopy class of $\mathcal{E} = (E,\alpha, F)$ by $[\mathcal{E}]$.
\par
There are natural isomorpishms $KK^{i}(\mathbb{C}, A)\simeq K_{i}(A)$ and $KK^{i}(A,\mathbb{C})\simeq K^{i}(A)$, for $i=0,1$, making $KK$-theory a generalization of $K$-theory and $K$-homology.
\par
A class $[\mathcal{E}]$ in $KK^{0}(A,B)$ determines mappings of $K$-theory $\hat{\otimes}_{i}[\mathcal{E}]:K_{i}(A)\mapsto K_{i}(B)$, for $i = 0,1$. This is a special case of the Kasparov product $\hat{\otimes}:KK^{i}(A,B)\times KK^{j}(B,C)\mapsto KK^{i+j}(A, C)$, for $i,j$ in $\{0,1\}$, (index $i+j$ is taken mod 2) which is why the image of a class $g$ in $K_{i}(A)$ under $\hat{\otimes}_{i}[\mathcal{E}]$ will be denoted $g\hat{\otimes}_{i}[\mathcal{E}]$. For example, the induced map $\hat{\otimes}_{i}[\varphi]$ of a $*$-homomorphism $\varphi$ is equal to the usual induced map $\varphi_{*}$ on $K$-theory.
\par
In this paper, we will mostly only consider Kasparov bi-modules of the form $\mathcal{E} = (E,\alpha, 0)$, so $\alpha(A)\subseteq\mathcal{K}(E)$ necessarily. We now describe the operations we will have to consider on such bi-modules.
\par
Suppose $\psi:B\mapsto B'$ and  $\varphi:A'\mapsto A$ are $*$-homomorphism of $C^{*}$-algebras. The \textit{pullback of $\mathcal{E}$ by $\varphi$} is the Kasparov $A'$ - $B$ bi-module $\varphi^{*}(\mathcal{E}) = (E,\alpha\circ\varphi, 0)$. The \textit{pushforward of $\mathcal{E}$ by $\psi$} is the Kasparov $A$ - $B'$ bi-module $\psi_{*}(\mathcal{E}) = (E\otimes_{B} B', \alpha\otimes_{B}\text{id}_{B'}, 0)$, where $B'$ is regarded as the $B$ - $B'$ correspondence with inner product $\langle b_{1}, b_{2}\rangle  = b^{*}_{1}b_{2}$, right $B'$-action $b_{1}\cdot b_{2} = b_{1}b_{2}$, defined for all $b_{1},b_{2}$ in $B'$, and the left $B$-action $b\cdot b' = \psi(b)b'$, defined for $b$ in $B$ and $b'$ in $B'$. See \cite[Lemma~1.2.8]{Jensen} for the proof that $\psi_{*}(\mathcal{E})$ is a Kasparov bi-module. We will denote $E\otimes_{B}B' = E\otimes_{\psi}B'$ and $\alpha\otimes_{B}\text{id} = \alpha\otimes_{\psi}\text{id}$ so that the module structure is clear. It can be shown that $[\varphi^{*}(\mathcal{E})] = [\varphi]\hat{\otimes}[\mathcal{E}]$ and $[\psi_{*}(\mathcal{E})] = [\mathcal{E}]\hat{\otimes}[\psi]$.
\par
We will say two Kasparov $A$ - $B$ bi-modules $\mathcal{E} = (E,\alpha,0)$, $\mathcal{F} = (E,\beta,0)$ are \textit{isomorphic} if there is a unitary $U:E\mapsto F$ such that $U\alpha U^{*} = \beta$.
\par
When $\psi$ as above is surjective, consider the quotient space $E/E^{\psi}=:E_{\psi}$ of $E$ by the sub-module $E^{\psi} = \{e\in E:\psi(\langle e, e\rangle ) = 0\}$. Let $q:E\mapsto E_{\psi}$ denote the  quotient map. It can be checked that $\psi_{*}\mathcal{E}$ is isomorphic to the bi-module structure on $(E_{\psi},\alpha_{\psi})$ satisfying $\langle q(e), q(f)\rangle = \psi(\langle e, f\rangle)$ and $\alpha_{\psi}(a)q(e)\psi(b) = q(\alpha(a)eb)$ for all $e,f$ in $E$, $a$ in $A$ and $b$ in $B$.
\par
For $t$ in $[0,1]$, let $ev_{t}:C([0,1],B)\mapsto B$ be the $*$-homomorphism sending $f$ to $ev_{t}(f) = f(t)$. Two bi-modules $\mathcal{E} = (E,\alpha,0)$, $\mathcal{F} = (E,\beta,0)$ are \textit{homotopic} if there is a Kasparov $A$ - $C([0,1],B)$ bi-module $\mathcal{H} = (H, \eta, 0)$ such that $(ev_{0})_{*}\mathcal{H}$ is isomorphic to $\mathcal{E}$ and $(ev_{1})_{*}\mathcal{H}$ is isomorphic to $\mathcal{F}.$
\par
As an example of a $KK^{0}$ class, an $A$ - $B$ correspondence $(E,\alpha)$ determines a Kasparov $J_{(E,\alpha)}$ - $B$ bi-module $J(E,\alpha): =(E,\alpha|_{J_{{(E,\alpha)}}},0)$ and hence a class $K(E,\alpha) = [J(E,\alpha)]$ in $KK^{0}(J_{(E,\alpha)}, B)$.
\par
The class of a short exact sequence of $C^{*}$-algebras
$$
\begin{tikzcd}
0 \arrow[r] & A \arrow[r, "i"] & B \arrow[r, "q"] & C \arrow[r] & 0
\end{tikzcd}$$
in $KK^{1}(C,A)$ will be denoted $\delta_{B}$. Recall that this class is natural in the sense that if 
$$
\begin{tikzcd}
0 \arrow[r] & A_{1} \arrow[r, "i_{1}"] \arrow[d, "\alpha"] & B_{1} \arrow[r, "q_{1}"] \arrow[d, "\beta"] & C_{1} \arrow[r] \arrow[d, "\gamma"] & 0 \\
0 \arrow[r] & A_{2} \arrow[r, "i_{2}"]                     & B_{2} \arrow[r, "q_{2}"]                    & C_{2} \arrow[r]                     & 0
\end{tikzcd}$$
is a commutative diagram of $*$-homomorphisms with exact rows, then $\delta_{B_{1}}\hat{\otimes}[\alpha] = [\gamma]\hat{\otimes}\delta_{B_{2}}$. See the paragraph above \cite[Remark~2.5.1]{CS86} for an explanation. 
\par
When we consider the action of $\delta_{B}$ on $K$-theory, often we will denote $\hat{\otimes}_{1}\delta_{B} =: \delta$ and call this the \textit{index map}, and denote $\hat{\otimes}_{0}\delta_{B}=:\text{exp}$ and call this the \textit{exponential map}. These are the usual index and exponential maps appearing in the 6-term exact sequence of $K$-theory associated to the above short exact sequence of $C^{*}$-algebras.
\par
If $(E,\alpha)$ is a faithful and full $A$ - $A$ correspondence, we will let $\delta_{E}$ denote the class of the extension
$$
\begin{tikzcd}
0 \arrow[r] & \mathcal{I}_{E} \arrow[r] & \mathcal{T}_{E} \arrow[r] & \mathcal{O}_{E} \arrow[r] & 0
\end{tikzcd}$$ in $KK^{1}(\mathcal{O}_{E}, \mathcal{I}_{E})$ and recall the $\mathcal{I}_{E}$-$J_{(E,\alpha)}$ equivalence bi-module $G_{E}$ defined in section \ref{cpalg}. Let $\delta_{E,PV} = \delta\hat{\otimes}[G_{E}]$.
For a Cuntz-Pimsner algebra $\mathcal{O}_{E}$
Pimsner in \cite{Pimsner:Generalizing_Cuntz-Krieger} determined the following relationship between the $K$-theory of $J_{(E,\alpha)}$, $A$, and $\mathcal{O}_{E}$ which will be crucial in computing the $K$-theory of a rational function. See \cite[Theorem~4.9]{Pimsner:Generalizing_Cuntz-Krieger}.
\begin{prop}
\label{cpexact}
Let $(E,\alpha)$ be a faithful and full $A$-$A$ correspondence. Let $i:A\mapsto \mathcal{O}_{E}$ and $\iota:J_{E}\mapsto A$ be the inclusions. Then, we have the following $6$-term exact sequence of $K$-theory:
$$
\begin{tikzcd}[sep = large]
K_{0}(J_{E}) \arrow[r, "{\iota - \hat{\otimes}_{0}K(E,\alpha)}"] & K_{0}(A) \arrow[r, "i_{*}"] & K_{0}(\mathcal{O}_{E}) \arrow[d, "\hat{\otimes}_{0}\delta_{E,PV}"]            \\
K_{1}(\mathcal{O}_{E}) \arrow[u, "\hat{\otimes}_{1}\delta_{E,PV}"]                & K_{1}(A) \arrow[l, "i_{*}"] & K_{1}(J_{E}). \arrow[l, "{\iota - \hat{\otimes}_{1}K(E,\alpha)}"]
\end{tikzcd}$$ The analogous 6-term exact sequence also holds for $K$-homology (with arrows reversed).
\end{prop}
We will call this the \textit{Pimsner-Voiculescu 6-term exact sequence of $K$-theory}. We will often denote $\hat{\otimes}_{1}\delta_{E,PV} =:\delta_{PV}$ and $\hat{\otimes}_{0}\delta_{E,PV} = :\text{exp}_{PV}$.
\subsection{Holomorphic dynamical systems}
\label{holdynsys}
A \textit{holomorphic dynamical system} is a holomorphic function $R:M\mapsto M$ defined on a Riemann surface $M$. By the Uniformization Theorem, the universal covering space $\tilde{M}$ of a connected Riemann surface $M$ is isomorphic either to the Riemann sphere $\hat{\mathbb{C}}$, the plane $\mathbb{C}$ or the unit disk $\mathbb{D}$ and the holomorphic map $R$ lifts to a holomorphic map $\tilde{R}:\tilde{M}\mapsto\tilde{M}$. 
\par
The dynamics of holomorphic maps defined on Riemann surfaces with $\tilde{M}\simeq \mathbb{D}$ is essentially trivial by the Pick Theorem (see \cite[Theorem~2.11]{Milnor:Dynamics_in_one_complex_variable}). 
\par
If $\tilde{M}\simeq\hat{\mathbb{C}}$, then $M\simeq\hat{\mathbb{C}}$ and $R$ is a rational function. This case will be the main focus of the paper and the standard reference for the general theory is \cite{Milnor:Dynamics_in_one_complex_variable}. 
\par
As for the case $\tilde{M}\simeq\mathbb{C}$, we have either $M\simeq\mathbb{C}$, $M\simeq\mathbb{C}\setminus\{0\}=:\mathbb{C}^{*}$ or $M\simeq\mathbb{C}/\Lambda$, where $\Lambda\subseteq\mathbb{C}$ is some lattice. If $M\simeq\mathbb{C}/\Lambda$, then there is $\alpha$ in $\Lambda$ and $\beta$ in $\mathbb{C}$ such that $\tilde{R}(z) = \alpha z +\beta$ for all $z$ in $\mathbb{C}$ (see \cite[Theorem~6.1]{Milnor:Dynamics_in_one_complex_variable}. Thus $R$ is either constant, a homeomorphism, or an expanding local homeomorphism, and the dynamics can be described relatively easily. On the contrary, the dynamical behaviour of $R$ when $M\simeq \mathbb{C}$ or $M\simeq\mathbb{C}^{*}$ can be extremely complicated and less is known than the case when $M\simeq\hat{\mathbb{C}}$. If $R$ is defined on either of these Riemann surfaces and cannot be extended to a rational function, it is often called \textit{transcendental}. We refer the reader to the surveys \cite{B93} and \cite{S98}. Further references can also be found in \cite[Section~6]{Milnor:Dynamics_in_one_complex_variable}.
\par
The dynamics of $R$ partitions $M$ into two sets. A point $z$ in $M$ is in the \textit{Fatou set} $F_{R}$ if and only if there is a neighbourhood $U$ of $z$ for which $\{R^{n}:U\mapsto M\}_{n\in\mathbb{N}}$ is a pre-compact set in the compact open-topology on $C(U,M)$. $F_{R}$ is open and invariant, in the sense that $R^{-1}(F_{R}) = F_{R}$ (see \cite[Section~4]{Milnor:Dynamics_in_one_complex_variable}). The dynamics restricted to $F_{R}$ is generally well understood; see for instance \cite[Section~16]{Milnor:Dynamics_in_one_complex_variable} for a description of the possible dynamical behaviour when $R$ is a rational function, and \cite{B93} in the case that $R$ is an entire function on $M = \mathbb{C}$. 
\par
The main difference between the Fatou set in the rational case and the transcendental case is the presence of ``wandering domains'', i.e. a sequence of $\{U_{n}\}_{n\in\mathbb{N}}$ of pairwise disjoint connected components in the Fatou set such that $R(U_{n})\subseteq U_{n+1}$. It is first proved in \cite{NWDSullivan} that a rational function does not admit wandering domains, while the first example of a wandering domain for a transcendental function was found in \cite{B76}.
\par
The complement of the Fatou set is the Julia set and is denoted $J_{R}$. When $M$ is either $\hat{\mathbb{C}}$, $\mathbb{C}$ or $\mathbb{C}^{*}$ and $\text{deg}(R) > 1$, $J_{R}$ is non-empty, totally invariant ($R^{-1}(J_{R}) = J_{R}$), contains no isolated points and is the closure of the repelling periodic points of $R$ (see \cite{B93}). The Julia set is where the dynamics behaves chaotically and is the main object of study of a holomorphic dynamical system.

\section{Morphisms of $C^{*}$-correspondences}
\label{morphcorr}
We show that $C^{*}$-correspondences are the objects of a category in which a morphism $(E,\alpha)\mapsto (F,\beta)$ in this category intertwines $(E,\alpha)$ with $(F,\beta)$ in a suitable sense. We show in Proposition \ref{morphintertwine} that such a morphism induces an intertwining of the $KK^{0}$-classes $K(E,\alpha)$ and $K(F,\beta)$. We then consider types of exactness in this category and show that a short $J$-exact sequence induces an intertwining of certain extension classes in $KK$-theory (Proposition \ref{morphnatural}). This can be thought of as a generalization of naturality of extension classes associated to exact sequences of $C^{*}$-algebras.  
\begin{defn}\label{morph_of_corr} A
 \textit{morphism} from an $A_{1}$ - $A_{2}$  $C^{*}$-correspondence $(E,\alpha)$ to a $B_{1}$ - $B_{2}$ correspondence $(F,\beta)$ is a triple $(\varphi_{1},S,\varphi_{2})$ of maps, where $\varphi_{1}: A_{1}\mapsto B_{1}$, $\varphi_{2}:A_{2}\mapsto B_{2}$ are *-homomorphisms, and $S:E\mapsto F$ is a linear map satisfying the additional identities 
\begin{enumerate}
\item $\langle S(e_{1}), S(e_{2})\rangle = \varphi_{2}(\langle e_{1}, e_{2}\rangle)$ for all $e_{1}$, $e_{2}$ in $E$ (i.e. $S$ is a $\varphi_{2}$-twisted morphism), 
\item $S(a\cdot e) = \varphi(a)\cdot S(e)$, for all $a$ in $A_{1}$, $e$ in $E$, 
\item $\varphi_{1}(J_{(E,\alpha)})\subseteq J_{(F,\beta)}$, and
\item $\hat{S}(\alpha(a)) = \beta(\varphi_{1}(a))$ for all $a$ in $J_{(E,\alpha)}$.
\end{enumerate}
\end{defn}
Condition $(4)$ above is automatically implied  by $(1)$-$(3)$ if $S$ is surjective, or if $\alpha^{-1}(\mathcal{K}(E)) = \emptyset$.
\par
The special case of the above definition when $A_{1} = A_{2}$, $B_{1} = B_{2}$ and $\varphi_{1} = \varphi_{2}$ appears often in the literature under the name ``co-variant morphism'' between self-correspondences, see \cite[Section~2.4]{MS19} and the references therein. The focus is usually on its functorial properties with respect to the Cuntz-Pimsner algebra construction, which we will review briefly, but here we will focus on its $KK$-theoretic consequences in the general case when $A_{1}$, $A_{2}$ and $B_{1},$ $B_{2}$ are not necessarily the same.
\par
It is easy to see that if $(\varphi_{1}, S, \varphi_{2}): (E,\alpha)\mapsto (F,\beta)$ and $(\psi_{1}, T, \psi_{2}): (F,\beta)\mapsto (G,\gamma)$ are morphisms, then their composition $(\psi_{1}, T, \psi_{2})\circ (\varphi_{1}, S, \varphi_{2}) = (\psi_{1}\circ \varphi_{1}, T\circ S, \psi_{2}\circ\varphi_{2})$ is a morphism, and that, given an $A_{1}$ - $A_{2}$ correspondence $(E,\alpha)$, $\text{id}_{(E,\alpha)} = (\text{id}_{A_{1}}, \text{id}_{E}, \text{id}_{A_{2}})$ is an identity morphism relative to this composition operation. Thus, $C^{*}$-correspondences form a category with morphisms as above.
\par
Now, suppose $(\varphi, S, \varphi)$ is a morphism from an $A$ - $A$ correspondence $(E,\alpha)$ to a $B$ - $B$ correspondence $(F,\beta)$ (both assumed to be full and faithful). The universal property of Toeplitz algebras (\cite[Theorem~3.4]{Pimsner:Generalizing_Cuntz-Krieger}) imply there is a unique $*$-homomorphism $\mathcal{T}(S):\mathcal{T}_{E}\mapsto\mathcal{T}_{F}$ defined by the equation
$$\mathcal{T}(S)(T_{e}) = T_{S(e)}, \text{ for all }e\text{ in }E.$$
\par
Note that $\mathcal{T}(S)\circ \phi = \phi\circ\hat{S}$, so the fact that $\hat{S}\circ \alpha(a) = \beta(\varphi(a))$, for all $a$ in $J_{E}$, implies

$\mathcal{T}(S)(a - \phi(\alpha(a))) = \varphi(a) - \phi(\beta(\varphi(a)))$, for all $a$ in $J_{E}$. Therefore, $\mathcal{T}(S)(\mathcal{I}_{E})\subseteq \mathcal{I}_{F}$, and hence $\mathcal{T}(S)$ passes down to a $*$-homomorphism at the level of Cuntz-Pimsner algebras, which we shall denote by $\mathcal{O}(S):\mathcal{O}_{E}\mapsto\mathcal{O}_{F}$.
\par
We show now that a morphism between two correspondences intertwines their induced classes in $KK^{0}$.
\begin{prop}
\label{morphintertwine}
Suppose $(\varphi_{1},S,\varphi_{2}): {}_{A_{1}}(E,\alpha){}_{A_{2}}\mapsto {}_{B_{1}}(F,\beta){}_{B_{2}}$ is a morphism of correspondences. Then, $K(E,\alpha)\hat{\otimes}[\varphi_{2}] = [\varphi_{1}]\hat{\otimes}K(F,\beta)$. 
\end{prop}
\begin{proof}
It is routine to see the mapping $T:E\otimes_{\psi} B_{2}\mapsto F$ defined on a basic tensor $e\otimes_{\psi} b$, for $e$ in $E$ and $b$ in $B_{2}$, as $T(e\otimes_{\psi} b) = S(e)b$ is a unitary onto its image $G = \overline{\text{span}\{S(E)\cdot B_{2}}\}$. Under $T$, the left action $\alpha\otimes_{\varphi_{2}}\text{id}$ on $E\otimes_{\varphi_{2}} B_{2}$ is identified with the action defined, for all $a$ in $A_{1}$ and $g$ in $G$, as $\gamma(a)\cdot g:= \beta(\varphi_{1}(a))g$.
\par
Let us show $J_{(E,\alpha)}\subseteq J_{(G,\gamma)}$. By the definition of morphism, for every $a$ in $J_{(E,\alpha)}$ and $\epsilon > 0$, there are $\{e_{i}\}_{i=1}^{n},\{e'_{i}\}_{i=1}^{n}\subseteq E_{1}$ such that $\|\beta(\varphi_{1}(a)) - \sum_{i=1}^{n}\theta_{S(e_{i}), S(e'_{i})}\|_{\mathcal{B}(E_{2})}\leq\epsilon$. If $\{u_{\lambda}\}_{\lambda\in \Lambda}$ is an approximate unit in $B_{2}$, then there is $\lambda_{0}$ in $\Lambda$ such that $\|\sum_{i=1}^{n}\theta_{S(e_{i}), S(e'_{i})} - \sum_{i=1}^{n}\theta_{S(e_{i})\sqrt{u_{\lambda_{0}}}, S(e'_{i})\sqrt{u_{\lambda_{0}}}}\|_{\mathcal{B}(F)}\leq\varepsilon$. Therefore, $\|\gamma(a) - \sum_{i=1}^{n}\theta_{S(e_{i})\sqrt{u_{\lambda_{0}}}, S(e'_{i})\sqrt{u_{\lambda_{0}}}}\|_{\mathcal{B}(G)}\leq \|\alpha_{2}(\varphi(a)) - \sum_{i=1}^{n}\theta_{S(e_{i})\sqrt{u_{\lambda_{0}}}, S(e'_{i})\sqrt{u_{\lambda_{0}}}}\|_{\mathcal{B}(F)}\leq 2\varepsilon$. As $\sum_{i=1}^{n}\theta_{S(e_{i})u_{\lambda_{0}}, S(e'_{i})u_{\lambda_{0}}}$ is in $\mathcal{K}(G)$ and $\epsilon$ was arbitrary, we may conclude that $a$ is in $J_{(G,\gamma)}$. Hence, $J_{(E,\alpha)}\subseteq J_{(G,\gamma)}$.
\par
Consider the Hilbert $C([0,1],B_{2})$-module $H = \{h\in C([0,1],F): h(0)\in G\}$ with operations defined point-wise as 
\begin{enumerate}
    \item $g\cdot b(t):= g(t)b(t)$, for all $h$ in $H$, $b$ in $B_{2}$ and $t$ in $[0,1]$,
    \item $\langle h_{1},h_{2}\rangle(t) := \langle h_{1}(t),h_{2}(t)\rangle$, for all $h_{1},h_{2}$ in $H$ and $t$ in $[0,1]$.
\end{enumerate}
Since $\beta(\varphi_{1}(A_{1}))\cdot G\subseteq G$, we may define the left action $\eta$ on $H$ as $\eta(a)\cdot h(t):= \beta(\varphi(a))\cdot h(t)$, for all $a$ in $A_{1}$, $h$ in $H$, and $t$ in $[0,1]$.
\par
$G$ embeds into $H$ as the constant functions and $\eta$ restricted to $G$ is equal to $\gamma$, so $\mathcal{K}(G)\subseteq \mathcal{K}(H)$ and $J_{(G,\gamma)}\subseteq J_{(H,\eta)}$. Therefore, $J_{(E,\alpha)}\subseteq J_{(H,\eta)}.$
\par
Thus, $(H,\eta|_{J_{(E,\alpha)}})$ is a Kasparov $J_{(E,\alpha)}-C([0,1],B_{2})$ module such that $(ev_{0})_{*}\mathcal{E}_{H}\simeq (G,\gamma)\simeq (\varphi_{2})_{*}(E,\alpha|_{(J_{(E,\alpha)}})$ and $(ev_{1})_{*}\mathcal{E}_{H}\simeq \varphi_{1}^{*}(F,\beta|_{J_{(F,\beta)}})$, proving the Proposition.
\end{proof}
\begin{rmk}\label{intertwiner}
    Suppose $a$ is in $KK(A_{1},A_{2})$, $b$ is in $KK(B_{1}, B_{2})$. We will say a pair $(x_{1},x_{2})$ in $KK(A_{1},B_{1})\times KK(A_{2}, B_{2})$ intertwines $a$ with $b$ if $a\hat{\otimes}x_{2} = x_{1}\hat{\otimes}b$. We call $(x_{1},x_{2})$ an intertwiner and write $(x_{1},x_{2}):a\mapsto b$ suggestively. Clearly if $(x_{1},x_{2}):a\mapsto b$ and $(y_{1},y_{2}):b\mapsto c$ are intertwiners, then $(x_{1},x_{2})\hat{\otimes}(y_{1},y_{2}) := (x_{1}\hat{\otimes}y_{1}, x_{2}\hat{\otimes} y_{2}):a\mapsto c$ is an intertwiner.
    Proposition \ref{morphintertwine} can be re-stated as saying that if $(\varphi_{1}, S, \varphi_{2}):(E,\alpha)\mapsto (F,\beta)$ is a morphism, then $([\varphi_{1}], [\varphi_{2}]):K(E,\alpha)\mapsto K(F,\beta)$ is an intertwiner. Moreover, it is easy to see that the assignment $(\varphi_{1}, S, \varphi_{2})\mapsto ([\varphi_{1}], [\varphi_{2}])$ is functorial, but we shall not make use of this fact.
\end{rmk}

Let us specialize to the case of a morphism between full and faithful self correspondences and show that such a morphism induces intertwinings of the $KK$-classes involved in the Pimsner-Voicelescu 6-term exact sequences of the respective Cuntz-Pimsner algebras (Proposition \ref{cpexact}).
\begin{prop}
\label{functcp}
    Let $(\varphi, S, \varphi):{}_{A}(E,\alpha){}_{A}\mapsto {}_{B}(F,\beta){}_{B}$ be a morphism between two full and faithful $C^{*}$-correspondences. Then, 
    \begin{enumerate}
    \item $[\varphi]\hat{\otimes} (\iota - K(F,\beta))) = (\iota - K(E,\alpha))\hat{\otimes}[\varphi]$,
    \item $[\varphi]\hat{\otimes}[i] =[i]\hat{\otimes}[\mathcal{O}(S)]$ and
    \item $[\mathcal{O}(S)]\hat{\otimes}\delta_{PV} = \delta_{PV}\hat{\otimes}[\varphi]$.
    \end{enumerate}
\end{prop}
\begin{proof}
$(\varphi,\mathcal{O}(S),\mathcal{O}(S)):i\mapsto i$ and $(\varphi, \varphi, \varphi):\iota\mapsto \iota$ and $(\varphi, S, \varphi):(E,\alpha)\mapsto (F,\beta)$ are morphisms, so $(1)$ and $(3)$ follow from Proposition \ref{morphintertwine}.
\par
The diagram
$$
\begin{tikzcd}
0 \arrow[r] & \mathcal{I}_{E} \arrow[r] \arrow[d, "\mathcal{T}(S)"] & \mathcal{T}_{E} \arrow[r] \arrow[d, "\mathcal{T}(S)"] & \mathcal{O}_{E} \arrow[r] \arrow[d, "\mathcal{O}(S)"] & 0 \\
0 \arrow[r] & \mathcal{I}_{F} \arrow[r]                             & \mathcal{T}_{F} \arrow[r]                             & \mathcal{O}_{F} \arrow[r]                             & 0
\end{tikzcd}$$ commutes, so naturality of $\delta$ implies $\delta_{E}\hat{\otimes}[\mathcal{T}(S)] = [\mathcal{O}(S)]\hat{\otimes}\delta_{F}$. 
\par
Since $(\varphi, S, \varphi)$ is a morphism, it is easy to see that $(\mathcal{T}(S), S^{\infty}, \varphi):G_{E}\mapsto G_{F}$ is a morphism, where $S^{\infty}:\mathcal{F}_{E}\mapsto \mathcal{F}_{F}$ is defined on a basic tensor $t= e_{1}\otimes_{A} e_{2}...\otimes_{A} e_{n}$ in $E^{\otimes_{A} n}$ as $S^{\infty}(t) = S(e_{1})\otimes_{B} S(e_{2})\otimes...\otimes S(e_{n})$. Proposition \ref{morphintertwine} then implies $[\mathcal{T}(S)]\hat{\otimes}[G_{F}] = [G_{E}]\hat{\otimes}[\varphi]$. Combining this equality with naturality of the index map and $\delta_{H, PV} =\delta_{H}\hat{\otimes}[G_{H}]$, for $H=E,F$, we have that $[\mathcal{O}(S)]\hat{\otimes}\delta_{F,PV} = \delta_{E,PV}\hat{\otimes}[\varphi]$.
\end{proof}
We now formulate a type of exactness in this category.
\begin{defn}
\label{morphexact}
Let $(E,\alpha)$ be an $A_{1}-A_{2}$ correspondence, $(F,\beta)$ a $B_{1}-B_{2}$ correspondence and $(G,\gamma)$ a $C_{1}-C_{2}$ correspondence. Morphisms $(\varphi_{1}, S, \varphi_{2}):(E,\alpha)\mapsto (F,\beta)$ and $(\psi_{1}, T, \psi_{2}):(F,\beta)\mapsto (G,\gamma)$ are said to be \textit{$J$-exact} if

$$
\begin{tikzcd}
J_{(E,\alpha)} \arrow[r, "\varphi_{1}"]                                          & J_{(F,\beta)} \arrow[r, "\psi_{1}"]                                          & J_{(G,\gamma)} \\
E \arrow[r, "S"]                                                  & F \arrow[r, "T"]                                                  & G\\
A_{2} \arrow[r, "\varphi_{2}"]                                          & B_{2} \arrow[r, "\psi_{2}"]                                          & C_{2}
\end{tikzcd}$$
are exact.
\end{defn}
\begin{rmk}
    Given a morphism $(\varphi_{1}, S, \varphi_{2}):(E,\alpha)\mapsto (F,\beta)$ from an $A_{1}$-$A_{2}$ correspondence to a $B_{1}$-$B_{2}$ correspondence, its \textit{image} is the $\varphi_{1}(A_{1})$ - $\varphi_{2}(A_{2})$ correspondence $\text{im}(\varphi_{1}, S, \varphi_{2}) = (S(E),\beta|_{\varphi_{1}(A_{1})})$ and its kernel is the $\text{ker}(\varphi_{1})$ - $\text{ker}(\varphi_{2})$ correspondence $\text{ker}(\varphi_{1}, S, \varphi_{2})=(\text{ker}(S), \alpha|_{\text{ker}(\varphi_{1})})$. 
    \par
    It is easy to see that the inclusion of the kernel and image correspondences into $(E,\alpha)$, $(F,\beta)$, respectively and the restriction of the co-domain of $(\varphi_{1}, S, \varphi_{2})$ to the image are morphisms. Moreover, they satisfy the universal property for images and kernels of morphisms in a category. The categorically correct notion of exactness is therefore a pair of morphisms $(\varphi_{1}, S, \varphi_{2}):(E,\alpha)\mapsto (F,\beta)$ and $(\psi_{1}, T, \psi_{2}):(F,\beta)\mapsto (G,\gamma)$ such that $\text{im}(\varphi_{1}, S, \varphi_{2}) = \text{ker}(\psi_{1}, T, \psi_{2})$. This is the case if and only if
    $$\begin{tikzcd}
A_{1} \arrow[r, "\varphi_{1}"]                                          & B_{1} \arrow[r, "\psi_{1}"]                                          & C_{1} \\
E \arrow[r, "S"]                                                  & F \arrow[r, "T"]                                                  & G\\
A_{2} \arrow[r, "\varphi_{2}"]                                          & B_{2} \arrow[r, "\psi_{2}"]                                          & C_{2}
\end{tikzcd}$$
are exact.
    \par
    Unfortunately an exact pair of morphisms is not necessarily $J$-exact. Suppose $B$ is a $C^{*}$-algebra with a non-trivial ideal $I$ and equip $I$, $B$ and $B/I$ with their identity Hilbert module structures. Let $\alpha:B\mapsto \mathcal{M}(I)$ and $\beta:B\mapsto \mathcal{M}(B)$ be the standard actions by multipliers and let $\gamma:0\mapsto \mathcal{M}(B/I)$. Then,
    denoting $i:I\mapsto B$ the inclusion and $q:B\mapsto B/I$ the quotient map,
    $$
    \begin{tikzcd}
    {(I,\alpha)} \arrow[r, "{(\text{id}_{B}, i,i)}"] & {(B,\beta)} \arrow[r, "{(0,q,q)}"] & {(B/I,\gamma)}
    \end{tikzcd}$$
    is exact in the categorical sense, but is not $J$-exact.
\end{rmk}
The following result is an extension of naturality of the index class of extensions.
\begin{prop}
\label{morphnatural}
If
$$\begin{tikzcd}
0 \arrow[r] & {}_{A_{1}}{(E,\alpha)}{}_{A_{2}} \arrow[r, "{(i_{1}, I,i_{2})}"] & {}_{B_{1}}{(F,\beta)}{}_{B_{2}} \arrow[r, "{(q_{1},Q,q_{2})}"] & {}_{C_{1}}{(G,\gamma)}{}_{C_{2}} \arrow[r] & 0
\end{tikzcd}$$ is a  $J$-exact sequence of correspondences, then $\delta_{J_{(F,\beta)}}\hat{\otimes} K(E,\alpha) = K(G,\gamma)\hat{\otimes}\delta_{B_{2}}$.
\end{prop}
\begin{proof}
In the notation of remark \ref{intertwiner}, we must show $(K(G,\gamma), K(E,\alpha)):\delta_{J_{(F,\beta)}}\mapsto \delta_{B_{2}}$ is an intertwiner. By the Rieffel correspondence (\cite[Proposition~3.24]{RW98}), exactness of 
$$
\begin{tikzcd}
0 \arrow[r] & E \arrow[r, "I"] & F \arrow[r, "Q"] & G \arrow[r] & 0
\end{tikzcd}$$ implies exactness of
$$
\begin{tikzcd}
0 \arrow[r] & \mathcal{K}(E) \arrow[r, "\hat{I}"] & \mathcal{K}(F) \arrow[r, "\hat{Q}"] & \mathcal{K}(G) \arrow[r]   & 0 \\
0 \arrow[r] & \langle E\rangle \arrow[r, "i_{2}"] & \langle F\rangle \arrow[r, "q_{2}"] & \langle G\rangle \arrow[r] & 0.
\end{tikzcd}$$ Hence, the following diagram commutes and has exact rows (the vertical maps labelled with $i$ are the respective inclusions):
$$
\begin{tikzcd}
0 \arrow[r] & {J_{(E,\alpha)}} \arrow[d, "\alpha"] \arrow[r, "i_{1}"]                                                    & {J_{(F,\beta)}} \arrow[d, "\beta"] \arrow[r, "q_{1}"]                                                     & {J_{(G,\gamma)}} \arrow[d, "\gamma"] \arrow[r]                                                  & 0 \\
0 \arrow[r] & \mathcal{K}(E) \arrow[r, "\hat{I}"] \arrow[d, "i^{E}_{K\mapsto L}"]                                        & \mathcal{K}(F) \arrow[r, "\hat{Q}"] \arrow[d, "i^{F}_{K\mapsto L}"]                                       & \mathcal{K}(G) \arrow[r] \arrow[d, "i^{G}_{K\mapsto L}"]                                        & 0 \\
0 \arrow[r] & \mathcal{L}(E) \arrow[r, "\mathcal{L}(I)"]                                                                 & \mathcal{L}(F) \arrow[r, "\mathcal{L}(Q)"]                                                                & \mathcal{L}(G) \arrow[r]                                                                        & 0 \\
0 \arrow[r] & \langle E\rangle  \arrow[u, "i^{E}_{\langle \rangle\mapsto L}"'] \arrow[r, "i_{2}"] \arrow[d, "i^{A_{2}}"] & \langle F\rangle  \arrow[u, "i^{F}_{\langle\rangle\mapsto L}"'] \arrow[r, "q_{2}"] \arrow[d, "i^{B_{2}}"] & \langle G\rangle \arrow[u, "i^{G}_{\langle\rangle\mapsto L}"'] \arrow[r] \arrow[d, "i^{C_{2}}"] & 0 \\
0 \arrow[r] & A_{2} \arrow[r, "i_{2}"]                                                                                   & B_{2} \arrow[r, "q_{2}"]                                                                                  & C_{2} \arrow[r]                                                                                 & 0.
\end{tikzcd}$$
Therefore, naturality of $\delta$ implies 

\begin{enumerate}
\item $([\gamma], [\alpha]):\delta_{J_{(F,\beta)}}\mapsto \delta_{\mathcal{K}(F)}$,
\item $([i^{G}_{K\mapsto L}], [i^{E}_{K\mapsto L}]):\delta_{\mathcal{K}(F)}\mapsto \delta_{\mathcal{L}(F)}$,
\item $([i^{G}_{\langle\rangle\mapsto L}], [i^{E}_{\langle\rangle\mapsto L}]):\delta_{\langle F\rangle}\mapsto \delta_{L(F)}$, and
\item $([i^{C_{2}}], [i^{A_{2}}]):\delta_{\langle F\rangle}\mapsto \delta_{B_{2}}$
\end{enumerate}
 are all intertwiners.
\par
For $H = E,F, G$, the inclusion $i^{H}_{\langle\rangle\mapsto L}:\langle H\rangle\mapsto \mathcal{L}(H)$ is full and hereditary, and therefore its class $[i^{H}_{\langle\rangle\mapsto L}]$ in $KK^{0}$ is invertible. Therefore, $([i^{G}_{\langle\rangle\mapsto L}]^{-1}, [i^{E}_{\langle\rangle\mapsto L}]^{-1}):\delta_{L(F)}\mapsto \delta_{\langle F\rangle}$ is an intertwiner. Composing all the intertwiners above yields an intertwiner $(g,e):\delta_{J_{(F,\beta)}}\mapsto \delta_{B_{2}}$, where $g = [\gamma]\hat{\otimes}[i^{G}_{K\mapsto L}]\hat{\otimes}[i^{G}_{\langle\rangle\mapsto L}]^{-1}\hat{\otimes}i^{C_{2}}$ and $e = [\alpha]\hat{\otimes}[i^{E}_{K\mapsto L}]\hat{\otimes}[i^{E}_{\langle\rangle\mapsto L}]^{-1}\hat{\otimes}i^{A_{2}}$. Let us show $e = K(E,\alpha)$ and $g = K(G,\gamma)$.
\par
For $H = E,F,G$, denote by $M(H)$ the $\mathcal{K}(H)$ - $\langle H\rangle$ Kasparov bi-module $H$. $(i^{H}_{\langle \rangle\mapsto L})_{*}M(H)$ is isomorphic as a bi-module to $\begin{pmatrix}
    \mathcal{K}(H) & H\\
    0 & 0
\end{pmatrix}$ with left action of $k$ in $\mathcal{K}(H)$ by the matrix $i^{H}_{K\mapsto L}(k) = \begin{pmatrix}
    k & 0\\
    0 & 0
\end{pmatrix}$ and right $\mathcal{L}(H)$-module structure inherited by the identity Hilbert module structure on $\mathcal{L}(H)$. The isomorphism $U:H\hat{\otimes}_{i^{H}_{\langle\rangle\mapsto L}}\mathcal{L}(H)\mapsto \begin{pmatrix}
    \mathcal{K}(H) & H\\
    0 & 0
\end{pmatrix}$ on a basic tensor $t = e_{1}\hat{\otimes} \begin{pmatrix}
    k & e_{2}\\
    f & a
\end{pmatrix}$ is $U(t) = \begin{pmatrix}
    0 & e_{1}\\
    0 & 0
\end{pmatrix}\cdot \begin{pmatrix}
    k & e_{2}\\
    f & a
\end{pmatrix} = \begin{pmatrix}
    \theta_{e_{1},f} & e_{1}\cdot a\\
    0 & 0
\end{pmatrix}$.
\par
Let $D = \{f\in C([0,1],\mathcal{L}(H)):f(1)\in \begin{pmatrix}
    \mathcal{K}(H) & H\\
    0 & 0
\end{pmatrix}\}$ with the right $C([0,1],\mathcal{L}(H))$-module structure inherited by the identity Hilbert module structure on $C([0,1],\mathcal{L}(H))$. Let $\eta:\mathcal{K}(H)\mapsto\mathcal{K}(D)$ be the embedding that sends $k$ in $\mathcal{K}(H)$ to the constant function $\eta(k)(t) = \begin{pmatrix}
    k & 0\\
    0 & 0
\end{pmatrix}$, $t$ in $[0,1]$. Then, $(D,\eta)$ is a Kasparov $\mathcal{K}(H)$-$C([0,1],\mathcal{L}(H))$ bi-module such that $(ev_{0})_{*}(D,\eta)\simeq (\mathcal{L}(H), i^{H}_{K\mapsto L})$ and $(ev_{1})_{*}(D,\eta)\simeq (\begin{pmatrix}
    \mathcal{K}(H) & H\\
    0 & 0
\end{pmatrix}, i^{H}_{K\mapsto L})\simeq\\(i^{H}_{\langle\rangle\mapsto L})_{*}M(H)$. Hence, $[M(H)]\hat{\otimes}[i^{H}_{\langle\rangle\mapsto L}] = [i^{H}_{K\mapsto L}]$, so that $[M(H)] = [i^{H}_{\langle\rangle\mapsto L}]\hat{\otimes}[i^{H}_{K\mapsto L}]^{-1}$.
\par
Now, it is routine to see that $[\alpha]\hat{\otimes}[M(E)]\hat{\otimes}[i^{A_{2}}] = K(E,\alpha)$ and $[\gamma]\hat{\otimes}[M(F)]\hat{\otimes}[i^{C_{2}}] = K(F,\gamma)$, so the calculation directly above implies $e = K(E,\alpha)$ and $g = K(G,\gamma)$.
\end{proof}
The intertwinings we have proven above yield commutative diagrams of their induced maps on $K$-theory and $K$-homology (and more generally on $KK^{*}(A,-)$ and $KK^{*}(-, A)$, $A$ nuclear). We record the diagram that will be of use to us in this paper.
\begin{cor}
\label{Maindiagram}
If
$$\begin{tikzcd}
0 \arrow[r] & {}_{A_{1}}{(E,\alpha)}{}_{A_{2}} \arrow[r, "{(i_{1}, I,i_{2})}"] & {}_{B_{1}}{(F,\beta)}{}_{B_{2}} \arrow[r, "{(q_{1},Q,q_{2})}"] & {}_{C_{1}}{(G,\gamma)}{}_{C_{2}} \arrow[r] & 0
\end{tikzcd}$$ is a  $J$-exact sequence of correspondences,
then the following diagram commutes

$$    
\begin{tikzcd}[sep=small]
& K_{0}(J_{(E,\alpha)})  \arrow[rr] \arrow[dd] & & K_{0}(J_{(F,\beta)}) \arrow[rr] \arrow[dd] & & K_{0}(J_{(G,\gamma)}) \arrow[dl] \arrow[dd]\\
K_{1}(J_{(G,\gamma)})  \arrow[dd] \arrow[ur]& & K_{1}(J_{(F,\beta)})\arrow[dd, crossing over] \arrow[ll, crossing over] & & K_{1}(J_{(E,\alpha)}) \arrow[dd, crossing over] \arrow[ll, crossing over]\\
& K_{0}(A_{2}) \arrow[rr] & & K_{0}(B_{2}) \arrow[rr] & & K_{0}(C_{2})\arrow[dl]\\
K_{1}(C_{2}) \arrow[ur]& & K_{1}(B_{2}) \arrow[ll] \arrow[from=uu, crossing over] \arrow[ll] & & K_{1}(A_{2}) \arrow[ll]\\
\end{tikzcd}
$$
The top and bottom horizontal faces are the 6-term exact sequences of $K$-theory associated to the respective extensions of $C^{*}$-algebras, and the vertical maps are the respective maps of the form $\hat{\otimes}_{i}K(H,\eta)$, $(H,\eta) = (E,\alpha), (F,\beta), (G,\gamma)$.
\par
If $A_{1} = A_{2}=:A$, $B_{1} = B_{2}=:B$ and $C_{1} = C_{2}=:C$ then the above diagram commutes with the vertical maps of the form $\iota - \hat{\otimes}_{i}K(H,\eta)$, where $\iota:K_{i}(J_{(H,\eta)})\mapsto K_{i}(D)$ is the map induced by inclusion, and ${}_{D}(H,\eta){}_{D} = {}_{A}(E,\alpha){}_{A}, {}_{B}(F,\beta){}_{B}, {}_{C}(G,\gamma){}_{C}$.
\end{cor}
\begin{proof}
This is a direct application of Proposition \ref{morphintertwine} and \ref{morphnatural}.
\end{proof}
\par
Let us show a $J$-exact sequence of self-correspondences induces a short exact sequence of Cuntz-Pimsner algebras.
\begin{prop}\label{J-exactness.C.P.} If $$\begin{tikzcd}
0 \arrow[r] & {}_{A}{(E,\alpha)}{}_{A} \arrow[r, "{(i, I,i)}"] & {}_{B}{(F,\beta)}{}_{B} \arrow[r, "{(q,Q,q)}"] & {}_{C}{(G,\gamma)}{}_{C} \arrow[r] & 0
\end{tikzcd}$$ is a $J$-exact sequence of morphisms of full and faithful correspondences, then the following diagram commutes and has exact rows and columns:
$$
\begin{tikzcd}
            & 0 \arrow[d]                                                    & 0 \arrow[d]                                                    & 0 \arrow[d]                                  &   \\
0 \arrow[r] & \mathcal{I}_{E} \arrow[r, "\mathcal{T}(I)"] \arrow[d, "i_{E}"] & \mathcal{I}_{F} \arrow[r, "\mathcal{T}(Q)"] \arrow[d, "i_{F}"] & \mathcal{I}_{G} \arrow[r] \arrow[d, "i_{G}"] & 0 \\
0 \arrow[r] & \mathcal{T}_{E} \arrow[r, "\mathcal{T}(I)"] \arrow[d, "q_{E}"] & \mathcal{T}_{F} \arrow[r, "\mathcal{T}(Q)"] \arrow[d, "q_{F}"] & \mathcal{T}_{G} \arrow[r] \arrow[d, "q_{G}"] & 0 \\
0 \arrow[r] & \mathcal{O}_{E} \arrow[r, "\mathcal{O}(I)"] \arrow[d]          & \mathcal{O}_{F} \arrow[r, "\mathcal{O}(Q)"] \arrow[d]          & \mathcal{O}_{G} \arrow[r] \arrow[d]          & 0 \\
            & 0                                                              & 0                                                              & 0                                            &  
\end{tikzcd}$$
\end{prop}
\begin{proof}
We first show the sequence of Toeplitz algebras is exact. For $T$ in $\mathcal{T}_{E}$, $\mathcal{T}(I)(T)$ restricted to the $A$-module $\sum_{n=0}^{\infty}I(E)^{\otimes_{B} n}\simeq\mathcal{F}_{E}$ is conjugate to $T$. Therefore, $\mathcal{T}(I)$ is injective.
\par
$\mathcal{T}(Q)$ surjects the generators $T_{F} = \{T_{f}:f\in F\}$ onto the generators $T_{G}$, and is therefore surjective. 
\par
By functoriality of $\mathcal{T}(-)$, we have that $\mathcal{T}(Q)\circ\mathcal{T}(I) = \mathcal{T}(Q\circ I) = 0$. It remains to show $\mathcal{T}(Q):\mathcal{T}_{F}/\mathcal{T}(I)(\mathcal{T}_{E})\mapsto\mathcal{T}_{G}$ is injective. For $T$ in $\mathcal{T}_{F}$, denote by $\overline{T}$ its image in $\mathcal{T}_{F}/\mathcal{T}(I)(\mathcal{T}_{E})$. For $c$ in $C$ and $g$ in $G$, let $b$ in $B$ and $f$ in $F$ be such that $q(b) = c$ and $Q(f) = g$. Let $s(c) = \overline{b}$ and $\pi(g) = \overline{T_{f}}$. $s:C\mapsto \mathcal{T}_{F}/\mathcal{T}(I)(\mathcal{T}_{E})$ and $\pi:G\mapsto \mathcal{T}_{F}/\mathcal{T}(I)(\mathcal{T}_{E})$ are well defined since $\mathcal{T}(I)(\mathcal{T}_{E})\cap B = i(A)$ and $\mathcal{T}(I)(\mathcal{T}_{E})\cap T_{F} = T_{I(E)}$, and a different choice of representative for $s(c)$, $\pi(g)$ differs by an element of $i(A)$, $\mathcal{T}_{I(E)}$, respectively.
\par
It is easy to see that $(\pi, s)$ is a Toeplitz representation. By the universal property of Toeplitz algebras, $(\pi, s)$ induces a $*$-homomorphism $\hat{\pi}:\mathcal{T}_{G}\mapsto \mathcal{T}_{F}/\mathcal{T}(I)(\mathcal{T}_{E})$ that is the inverse of $\mathcal{T}(Q):\mathcal{T}_{F}/\mathcal{T}(I)(\mathcal{T}_{E})\mapsto\mathcal{T}_{G}$.
\par
Now, we show the top row sequence is exact. For $H = E,F,G$, denote by $P_{k}^{H}$ the projection of the Fock module $\sum_{n=0}^{\infty}H^{\otimes n}$ onto $\sum_{n=0}^{k} H^{\otimes n}$ and let $T$ be in $\mathcal{T}_{E}$. By the second characterization of $\mathcal{I}_{F}$ in Section \ref{cpalg}, $\mathcal{T}(I)(T)$ is in $\mathcal{I}_{F}$ if and only if $P_{k}^{F}\mathcal{T}(I)(T)$ converges to $\mathcal{T}(I)(T)$ in the operator norm as $k$ tends to $\infty$.
\par
The restrictions of $P_{k}^{F}$, for $k$ in $\mathbb{N}$, and $\mathcal{T}(I)(T)$ to the sub-module $\sum_{n=0}^{\infty}I(E)^{\otimes n}$ are identical to $P^{E}_{k}$ and $T$, respectively, so if $\mathcal{T}(I)(T)$ is in $\mathcal{I}_{E}$, then $P_{k}^{E}\circ T$ converges to $T$ in the operator norm. Hence, $\mathcal{T}(I)^{-1}(\mathcal{I}_{F})\subseteq \mathcal{I}_{E}$. The reverse inclusion is obvious.
\par
The equality $\mathcal{T}(I)^{-1}(\mathcal{I}_{F}) = \mathcal{I}_{E}$ together with exactness of the middle row implies $\mathcal{T}(I)(\mathcal{I}_{E}) = \text{ker}(\mathcal{T}_{Q})\cap \mathcal{I}_{F}$.
\par
It remains to show $\mathcal{T}(Q):\mathcal{I}_{F}\mapsto \mathcal{I}_{G}$ is surjective. By $J$-exactness, $q:\mathcal{J}_{(F,\beta)}\mapsto\mathcal{J}_{(E,\alpha)}$ is surjective. Therefore, for every generator for $\mathcal{I}_{G}$ of the form $c - \phi(\gamma(c))$, $c$ in $\mathcal{J}_{(G,\gamma)}$, there is $b$ in $\mathcal{J}_{F,\beta)}$ such that $q(b) = c$. Hence $\mathcal{T}(Q)(b - \phi(\beta(b))) = c - \phi(\gamma(c))$. Therefore, $\mathcal{T}(Q):\mathcal{I}_{F}\mapsto \mathcal{I}_{G}$ is surjective.
\par
The above commutative diagram has exact columns by definition and the top two rows are exact from what we have shown. Exactness of the bottom row then follows from a diagram chase.
\end{proof}

\section{Correspondences from branched functions}
\label{holcorrespondence}
Along the way to proving our main result, we will be working with the dynamics of holomorphic functions restricted to a variety of subspaces, so we need to abstract the $C^{*}$-correspondence construction of Kajiwara and Watatani \cite{KW:C*-algebras_associated_with_complex_systems} to allow for these restrictions. It turns out that the only essential fact for constructing correspondences in a similar fashion is that the function has a system of inverse branches, which we now make precise.
\begin{defn}
\label{branchedfunction}
Let $F:X\mapsto Y$ be a continuous function between locally compact Hausdorff spaces $X$ and $Y$. We will say $F$ is a \textit{branched function} if there is a function $\text{ind}_{F}:X\mapsto \mathbb{N}$ such that for every $x$ in $X$, there is a neighbourhood $U$ of $x$, a neighbourhood $V$ of $F(x)$, and functions $\{F_{i}^{-1}:V\mapsto U\}_{i=1}^{\text{ind}_{F}(x)}$ satisfying
\begin{enumerate}
\item $F\circ F_{i}^{-1} = \text{id}_{V}$, for all $i\leq \text{ind}_{F}(x)$,
\item $F_{i}^{-1}(F(x)) = x$ and $F_{i}^{-1}$ is continuous at $F(x)$, for all $i\leq \text{ind}_{F}(x)$,
\item $U = \bigcup_{i=1}^{\text{ind}_{F}(x)}F_{i}^{-1}(V)$, and
\item for all $u$ in $U$, $\text{ind}_{F}(u) = |\{i\leq \text{ind}_{F}(x): u\in F^{-1}_{i}(V)\}|$.
\end{enumerate}
$\text{ind}_{F}$ is called the index function and $\{F_{i}^{-1}:V\mapsto U\}_{i=1}^{d}$ are called inverse branches of $F$ centered at $x$.
\end{defn}
We remark that branched functions are in abundance. For instance, take a sheet of paper (allowing for arbitrary shape and cuts) and fold it in such a way that the crease lines and boundary map to the boundary of the folded paper. This is a branched function. 
\par
A special case of this construction is to take a Sierpinksi triangle and fold the outer three equilateral triangles into the middle. We may rotate the resulting Sierpinksi triangle by 180 degrees and dilate by 3 to get a surjective self map if we wish. The index function will be 1 everywhere except 3 at the intersection points of the three outer triangles.
\par
As in \cite[Lemma~2.1]{KW:C*-algebras_associated_with_complex_systems}, these branching properties can be used to show a canonical transfer operator of $F$ maps $C_{c}(X)$ functions to $C_{c}(Y)$ functions. 
\par
First, note that if $\{F^{-1}_{i}:V\mapsto U\}_{i=1}^{\text{ind}_{F}(x)}$ are inverse branches centered at $x$ in $X$, then $U\cap F^{-1}(F(x)) = \{x\}$. Hence, $F^{-1}(y)$ is discrete, for all $y$ in $Y$.
\begin{prop}
\label{transfer}
Let $F:X\mapsto Y$ be a branched function. If $f$ is in $C_{c}(X)$, then the function $\Phi(f):Y\mapsto\mathbb{C}$ defined, for $y$ in $Y$, as $$\Phi(f)(y) = \sum_{x\in F^{-1}(y)}\text{ind}_{F}(x)f(x)$$ is in $C_{c}(Y)$.
\end{prop}
\begin{proof}
Let $\text{supp}(f) = K$ and $y$ in $Y$. We show $\Phi(f)$ is continuous at $y$. Since $F$ is branched, $F^{-1}(y)$ is discrete and hence $F^{-1}(y)\cap K$ is finite. Hence, $\Phi(f)(y)<\infty$. Write $F^{-1}(y)\cap K = \{x_{1},...,x_{k}\}$ and let $\{U_{i}\}_{i=1}^{k}$ be a collection of pairwise disjoint open sets such that $U_{i}\cap F^{-1}(F(x)) = \{x_{i}\}$ and $U_{i}$ is a co-domain of inverse branches $\{F_{ij}^{-1}:F(U_{i})\mapsto U_{i}\}_{j=1}^{\text{ind}(x_{i})}$ centered at $x_{i}$, for $i\leq k$. By a compactness argument, there is an open neighbourhood $V\subseteq \bigcap_{i=1}^{k}F(U_{i})$ of $y$ such that $F^{-1}(V)\cap K= \bigcup_{i=1}^{k} F^{-1}(V)\cap U_{i}\cap K$. By property $(3)$ of the inverse branches, we have $\bigcup_{i=1}^{k} F^{-1}(V)\cap U_{i}\cap K = \bigcup_{i=1}^{k}\bigcup_{j=1}^{\text{ind}(x_{i})}F^{-1}_{ij}(V)\cap K$
\par
Hence, for every $v$ in $V$, we have $$\sum_{x\in F^{-1}(v)}\text{ind}_{F}(x)f(x) = \sum_{x\in \bigcup_{i=1}^{k}\bigcup_{j=1}^{\text{ind}(x_{i})}F^{-1}_{ij}(v)}\text{ind}_{F}(x)f(x).$$ By property $(4)$ of the inverse branches, we have $$\sum_{x\in \bigcup_{i=1}^{k}\bigcup_{j=1}^{\text{ind}(x_{i})}F^{-1}_{ij}(v)}\text{ind}_{F}(x)f(x) = \sum_{i=1}^{k}\sum_{j=1}^{\text{ind}_{F}(x_{i})}f\circ F^{-1}_{ij}(v).$$ Since $F_{ij}^{-1}$ is continuous at $y$ for all $i\leq k$ and $j\leq \text{ind}_{F}(x_{i})$, it follows that $\Phi(f)|_{V} = \sum_{i=1}^{k}\sum_{j=1}^{\text{ind}_{F}(x_{i})}f\circ F_{ij}^{-1}$ is continuous at $y$.
\par
As $y$ in $Y$ was arbitrary, we may conclude $\Phi(f)$ is continuous. Note that $\text{supp}(\Phi(f)) = F(\text{supp}(f))$, and therefore $\Phi(f)$ is in $C_{c}(Y)$.
\end{proof}
To a branched function $F:X\mapsto Y$, we can associate a $C_{0}(X)-C_{0}(Y)$ correspondence $(E_{F,X},\alpha_{X})$ in a similar way as in \cite{KW:C*-algebras_associated_with_complex_systems}. We show this now. Let $\tilde{E}_{F,X} = C_{c}(X)$, equipped with a right 
$C_{0}(Y)$-module structure defined, for $\psi$ in $\tilde{E}_{F,X}$ and $g$ in $C_{0}(Y)$, as 
$$(\psi\cdot g)(x) = \psi(x)g(F(x)),\text{ }x\text{ in }X.$$ 
$\tilde{E}_{F,X}$ is
equipped with the $C_{0}(Y)$-valued inner product defined, for $\psi, \varphi$ in $\tilde{E}_{F,X}$, as
$$\langle\psi,\varphi\rangle = \Phi(\overline{\psi}\varphi).$$
For all $\psi$ in $\tilde{E}_{F,X}$ we have
$\|\psi\|\leq \|\psi\|_{2}$, where $\|\cdot\|$ denotes the sup norm and $\|\psi\|_{2}:= \|\sqrt{\langle \psi,\psi\rangle}\|$.
If $\text{deg}(F):=\text{sup}_{y\in Y}|F^{-1}(y)|<\infty$, then we also have
$\|\psi\|_{2}\leq \sqrt{\text{deg}(\tilde{F})}\|\psi\|$. In this case, the Hilbert $C_{0}(Y)$-module completion of $\tilde{E}_{F,X}$ is $E_{F,X} = C_{0}(X)$ and the inner product formula and right action extend naturally.
\par
In general, the completion $E_{F,X}$ is a strict subspace of $C_{0}(X)$ that is invariant under the right action of $C_{0}(Y)$ defined above and equipped with the above inner product. We identify this now. 
First, for $i=1,2$, we let 
$$C_{i}(X) = \{\psi\in C_{0}(X): \Phi(|\psi|^{i})\in C_{0}(Y)\}.$$
\begin{lemma}
\label{transferconvergence}
Let $F:X\mapsto Y$ be a branched function. Then, $C_{1}(X)$ is a subspace of $C_{0}(X)$ that is hereditary in the sense that if $\varphi,\psi$ are in $C_{0}(X)$ with $\psi$ in $C_{1}(X)$ and $0\leq \varphi\leq \psi$, then $\varphi$ is in $C_{1}(X)$. Moreover, $C_{1}(X)$ is complete with respect to the norm $\|\cdot\|_{1}:= \|\Phi(|\cdot|)\|$.
\end{lemma}
\begin{proof}
The fact that $C_{1}(X)$ is a subspace will follow from the hereditary property. Suppose $0\leq \varphi\leq \psi$, $\psi$ is in $C_{1}(X)$ and $\varphi$ is in $C_{0}(X)$. First, we show $\Phi(\varphi)$ is continuous.
\par
Fix $y$ in $Y$ and $\epsilon > 0$. Since $\Phi(\psi)(y) < \infty$, there is $k$ in $\mathbb{N}$ and $\{x_{i}\}^{k}_{i=1}\subseteq F^{-1}(y)$ such that $x_{i}\neq x_{j}$ for all $i\neq j\leq k$ and $$|\Phi(\psi)(y) - \sum_{i=1}^{k}\text{ind}_{F}(x_{i})\psi(x_{i})| <\epsilon/7.$$ For each $i\leq k$, let $\{F^{-1}_{ij}:V\mapsto U_{i}\}^{\text{ind}_{F}(x_{i})}_{j=1}$ be inverse branches of $F$ centered at $x_{i}$. By continuity of these branches and $\Phi(\psi)$ at $y$, we may assume the neighbourhood $V$ of $y$ is chosen such that, for all $\tilde{y}$ in $V$, we have

$$|\sum_{i=1}^{k}\text{ind}_{F}(x_{i})\psi(x_{i}) - \sum_{i=1}^{k}\sum_{j=1}^{\text{ind}_{F}(x_{i})}\psi\circ F^{-1}_{ij}(\tilde{y})| <\epsilon/7,$$ 

$$|\Phi(\psi)(y) -\Phi(\psi)(\tilde{y})| < \epsilon/7,$$

and \begin{equation}
\tag{*}
\label{*}
    |\sum_{i=1}^{k}\sum_{j=1}^{\text{ind}_{F}(x_{i})}\varphi\circ F^{-1}_{ij}(y)- \sum_{i=1}^{k}\sum_{j=1}^{\text{ind}_{F}(x_{i})}\varphi\circ F^{-1}_{ij}(\tilde{y})| < \epsilon/7.
\end{equation}

Let $U = \bigcup_{i=1}U_{i}$. For any $f$ in $C_{0}(X)$ and $\tilde{y}$ in $V$, we have $$\sum_{x\in F^{-1}(\tilde{y})\cap U}\text{ind}_{F}(x)f(x) = \sum_{i=1}^{k}\sum_{j=1}^{\text{ind}_{F}(x_{i})}f\circ F^{-1}_{ij}(\tilde{y})$$. Therefore, the first three inequalities imply, by the triangle inequality, that $$\sum_{x\in F^{-1}(\tilde{y})\setminus U}\text{ind}_{F}(x)\psi(x) =  |\Phi(\psi)(\tilde{y}) - \sum_{x\in F^{-1}(\tilde{y})\cap U}\text{ind}_{F}(x)\psi(x)| < 3\epsilon/7,$$ for all $\tilde{y}$ in $V$.
\par
Since $0\leq \varphi\leq \psi$, it follows that \begin{equation}
    \tag{**}
    \label{**}
    |\Phi(\varphi)(y) - \sum_{x\in F^{-1}(\tilde{y})\cap U}\text{ind}_{F}(x)\varphi(x)|\leq \sum_{x\in F^{-1}(\tilde{y})\setminus U}\text{ind}_{F}(x)\psi(x) < 3\epsilon/7,
\end{equation} for all $\tilde{y}$ in $V$. So, by the triangle inequality and the inequalities \eqref{*} and \eqref{**}, we have 
$$|\Phi(\varphi)(y) - \Phi(\varphi)(\tilde{y})| < \epsilon,$$ for all $\tilde{y}$ in $V$. This shows $\Phi(\varphi)$ is continuous.
\par
$0\leq \Phi(\varphi)\leq \Phi(\psi)$ and the fact that $\Phi(\psi)$ vanishes at infinity implies $\Phi(\varphi)$ does too. Therefore, $\Phi(\varphi)$ is in $C_{0}(Y)$.
\par
We now show $C_{1}(X)$ is complete with respect to $\|\cdot\|_{1}$. Let $\{\psi_{n}\}_{n\in\mathbb{N}}\subset C_{0}(X)$ such that $\sum_{n\in\mathbb{N}}\|\psi_{n}\|_{1} <\infty$. Since $\|\cdot\|\leq \|\cdot\|_{1}$, $\sum_{n\in \mathbb{N}}\psi_{n}$ and $\sum_{n\in\mathbb{N}}|\psi_{n}|$ converge (with respect to $\|\cdot\|$) to elements in $C_{0}(X)$.
\par
$0\leq |\sum_{n\in\mathbb{N}}\psi_{n}|\leq \sum_{n\in\mathbb{N}}|\psi_{n}|$ so, by the hereditary property of $C_{1}(X)$, to prove that $\sum_{n\in\mathbb{N}}\psi_{n}$ is in $C_{1}(X)$ it suffices to show $\sum_{n\in\mathbb{N}}|\psi_{n}|$ is in $C_{1}(X)$.
\par
$\|\sum_{n\leq k}\Phi(|\psi_{n}|) - \sum_{n\in\mathbb{N}}\Phi(|\psi_{n}|)\|\leq \sum_{n > k}\|\psi_{n}\|_{1}$ and so the hypothesis implies \\$\Phi(\sum_{n\in\mathbb{N}}|\psi_{n}|) = \sum_{n\in\mathbb{N}}\Phi(|\psi_{n}|)$ is in $C_{0}(Y)$. Hence, $C_{1}(X)$ is complete.
\end{proof}
\begin{prop}
Let $F:X\mapsto Y$ be a branched function. Then,
the Hilbert $C_{0}(Y)$-module completion of $\tilde{E}_{F,X}$ is isomorphic to $C_{2}(X)$.
\end{prop}
\begin{proof}
Let us first show $C_{2}(X)$ is complete. Let $\{\psi_{n}\}_{n\in\mathbb{N}}\subset C_{0}(X)$ such that $\sum_{n\in\mathbb{N}}\|\psi_{n}\|_{2} <\infty$. Since $\|\cdot\|\leq \|\cdot\|_{2}$, $\sum_{n\in \mathbb{N}}\psi_{n}$ and $\sum_{n\in\mathbb{N}}|\psi_{n}|$ converge (with respect to  $\|\cdot\|$) to elements in $C_{0}(X)$. 
\par
We must show $|\sum_{n\in\mathbb{N}}\psi_{n}|^{2}$ is in $C_{1}(X)$. By the Cauchy-Schwarz inequality, $\|\overline{\psi_{n}}\psi_{m}\|_{1}\leq \|\psi_{n}\|_{2}\|\psi_{m}\|_{2}$ for all $n,m$ in $\mathbb{N}$. Hence, $|\sum_{n\in\mathbb{N}}\psi_{n}|^{2} = \sum_{n,m\in\mathbb{N}}\overline{\psi_{n}}\psi_{m}$ converges absolutely with respect to $\|\cdot\|_{1}$. Lemma \ref{transferconvergence} then implies $|\sum_{n\in\mathbb{N}}\psi_{n}|^{2}$ is in $C_{1}(X)$.
\par
Lastly, we must show $C_{c}(X)$ is dense in $C_{2}(X)$ relative to $\|\cdot\|_{2}$. It suffices to show positive elements in $C_{2}(X)$ can be approximated arbitrarily by elements in $C_{c}(X)$.
\par
Let $0\leq \varphi$ be such that $\psi:= \varphi^{2} $ is in $C_{1}(X)$. Given $\epsilon > 0$, let $K\subseteq Y$ be a compact set such that $\Phi(\psi)(\tilde{y}) < \epsilon^{2}/2$ for all $\tilde{y}$ in $X\setminus K$. As in Lemma \ref{transferconvergence}, For each $y$ in $K$ there is a pre-compact open set $U_{y}\subseteq X$ such that $y$ is in $F(U_{y}) =: V_{y}$ and  $$|\Phi(\psi)(\tilde{y}) - \sum_{x\in F^{-1}(\tilde{y})\cap U_{y}}\text{ind}_{F}(x)\psi(x)| < \epsilon^{2}/2$$ for all $\tilde{y}$ in $V_{y}$. Let $\{y_{i}\}^{d}_{i=1}\subseteq K$ be such that $K\subseteq \bigcup_{i=1}^{d}V_{y_{i}}$. Then, $U: = \bigcup_{i=1}^{d}U_{y_{i}}$ is a pre-compact open set such that 
$$|\Phi(\psi)(\tilde{y}) - \sum_{x\in F^{-1}(\tilde{y})\cap U}\text{ind}_{F}(x)\psi(x)| < \epsilon^{2}$$ for all $y$ in $Y$.
\par
Choose $\phi$ in $C_{c}(X)$ such that $0\leq \phi\leq 1$ and $\phi(x) = 1$ for all $x$ in $U$. From the above inequality, we have 
$$ |\Phi(|\varphi - \sqrt{\phi}\varphi|^{2})(\tilde{y})|\leq \sum_{x\in F^{-1}(\tilde{y})\setminus U}\text{ind}_{F}(x)\psi(x) < \epsilon^{2} $$ for all $\tilde{y}$ in $Y$. Hence, $\|\varphi - \sqrt{\phi}\varphi\|_{2}\leq \epsilon$.
\end{proof}
\par
$E_{F,X} = C_{2}(X)$ has a left action $\alpha_{X}$ of $C_{0}(X)$ by
Hilbert module endomorphisms given, for $f$ in $C_{0}(X)$ and $\psi$ in $E_{F,X}$, as 
$$(f\cdot \psi)(x) = f(x)\psi(x), \text{ }x\text{ in }X.$$ 
The pair $(E_{F,X},\alpha_{X})$ is a $C_{0}(X)$-$C_{0}(Y)$ correspondence. It is always injective and non-degenerate. If $F$ is surjective, it is full.
\par
If $F:X\mapsto Y$ and $G:Y\mapsto Z$ are branched functions with index functions $\text{ind}_{F}$ and $\text{ind}_{G}$, then $G\circ F$ is a branched function with index function $\text{ind}_{G\circ F} := (\text{ind}_{G}\circ F)\text{ind}_{F}$; the inverse branches of $F\circ G$ are the composite of inverse branches for $F$ and $G$.
\par
For $F:X\mapsto Y$ a continuous function, the \textit{critical points of $F$} will be the set of points in $X$ for which $F$ is \textit{not} a local homeomorphism at, in the sense that there is no open neighbourhood $U$ about $x$ such that $F(U)$ is open in $Y$ and $F:U\mapsto F(U)$ is a homeomorphism. We shall denote this set by $C_{F,X}$. If $F:X\mapsto Y$ is an open map (which is the case for a branched function), then this is a closed subset.
\par
We can think of a branched function $F:X\mapsto Y$ and its canonical transfer operator $\Phi$ as defining a topological quiver in the sense of \cite{MT05}; in the notation of that paper, $E^{1} = X$, $E^{0} = Y$ and $r = s = F$, with $r$-system $\{\Phi(-)(y)\}_{y\in Y}$. Then, the Hilbert $C_{0}(Y)$-module constructed in \cite[Section~3.1]{MT05} is isomorphic to $E_{F,X}$. \cite[Theorem~3.11]{MT05} then can be applied to show $J_{(E_{F,X},\alpha_{X})} = C_{0}(X\setminus C_{F,X})$. This is a generalization of \cite[Proposition~2.5]{KW:C*-algebras_associated_with_complex_systems}.
\par
We will denote the trivially graded Kasparov $C_{0}(X\setminus C_{F,X})-C_{0}(Y)$ bi-module \\$(E_{F,X}, \alpha_{X}|_{C_{0}(X\setminus C_{F,X})},0)$ by $\mathcal{E}_{F,X}$, and denote its class in $KK^{0}(C_{0}(X\setminus C_{F,X}), C_{0}(Y))$ by $[\mathcal{E}_{F,X}]$. We shall denote the class of the inclusion $i:C_{0}(X\setminus C_{F,X})\mapsto C_{0}(X)$ in $KK^{0}(C_{0}(X\setminus C_{F,X}), C_{0}(X))$ by $\iota$.
\par
We remark that there is an ambiguity in our notations $E_{F,X}$ and $\mathcal{E}_{F,X}$, since we do not specify the co-domain of $F$, but it will be clear from the context what it is.
\par
We now show that the class $[\mathcal{E}_{F,X}]$ behaves well with restrictions of $F$.
If $U$ is an open set in $X$, let $i_{U} = i:C_{0}(U)\mapsto C_{0}(X)$, $\iota_{U} = i:C_{0}(U\setminus C_{F,X})\mapsto C_{0}(X\setminus C_{F,X})$ denote the respective inclusions. For an open set $V$ of $Y$ such that $F(U)\subseteq V$ it is clear that $F:U\mapsto V$ is also a branched function.
\begin{prop}
\label{classnaturality}
If $F:X\mapsto Y$ is a branched function and $U$ is an open set in $X$, then $(\iota_{U})^{*}[\mathcal{E}_{F,X}] = (i_{V})_{*}[\mathcal{E}_{F,U}]$, where $V$ is any open set in $Y$ such that $F(U)\subseteq V$. Moreover, $(i_{U}, i_{U}, i_{V}):(E_{F,U},\alpha_{U})\mapsto (E_{F,X},\alpha_{X})$ is a morphism of correspondences
\end{prop}
\begin{proof}
The proof that $(i_{U}, i_{U}, i_{V}):(E_{F,U},\alpha_{U})\mapsto (E_{F,X},\alpha_{X})$ is a morphism of correspondences is straightforward, and we omit it. Proposition \ref{morphintertwine} implies $(\iota_{U})^{*}[\mathcal{E}_{F,X}] = (i_{V})_{*}[\mathcal{E}_{F,U}]$.
\end{proof}
If $X = Y$ and $F:X\mapsto X$ is a surjective branched function, we will denote the corresponding Cuntz-Pimsner algebra and its
Toeplitz algebra extension by $\mathcal{O}_{F,X}$ and $\mathcal{T}_{F,X}$, respectively. These algebras are conjugacy invariants for the pair $(F, \text{ind}_{F})$.
\begin{prop}
\label{conjugacyiso1}
Suppose $F:X\mapsto X$ and $G:Y\mapsto Y$ are surjective branched functions, and $\varphi:X\mapsto Y$ is a homeomorphism such that $\varphi\circ F = G\circ\varphi$ and $\text{ind}_{G}\circ\varphi = \text{ind}_{F}$. Then, $(\varphi^{*},\varphi^{*},\varphi^{*}):E_{G,Y}\mapsto E_{F,X}$ is an isomorphism of correspondences. In particular, the induced $*$-homomorphisms $\mathcal{T}(\varphi^{*}):\mathcal{T}_{G,Y}\mapsto \mathcal{T}_{F,X}$,
$\mathcal{O}(\varphi^{*}):\mathcal{O}_{G,Y}\mapsto\mathcal{O}_{F,X}$ are $*$-isomorphisms.
\end{prop}
\begin{proof}
The proof is straightforward, and we omit it.
\end{proof}
We now consider exact sequences of branched functions. If $F:U\mapsto V$ is branched and $Y\subseteq V$ is closed, letting $F^{-1}(Y) = X$, $F|_{X}:X\mapsto Y$ is also branched, with $\text{ind}_{F|_{X}} = \text{ind}_{F}|_{X}$ and inverse branches equal to the inverse branches of $F$ restricted to $Y$.
\begin{prop}
\label{holmorph1}
Let $F:U\mapsto V$ be a branched function, and $Y$ a closed subset of $V$. Denote $F^{-1}(Y) = X$, $i_{1}:C_{0}(U\setminus X)\mapsto C_{0}(U)$, $i_{2}:C_{0}(V\setminus Y)\mapsto C_{0}(V)$ the inclusions and $r_{1}:C_{0}(U)\mapsto C_{0}(X)$, $r_{2}:C_{0}(V)\mapsto C_{0}(Y)$ the restrictions. Then, $(i_{1},i_{1},i_{2}): (E_{F,U\setminus X},\alpha_{U\setminus X})\mapsto (E_{U},\alpha_{U})$ and $(r_{1},r_{1},r_{2}):(E_{U},\alpha_{U})\mapsto (E_{X},\alpha_{X})$ are morphisms. $C_{F,U}\cap X = C_{F,X}$, if and only if the sequence 
$$
\begin{tikzcd}
0 \arrow[r] & {(E_{F,U\setminus X},\alpha)} \arrow[r, "{(i_{1},i_{1},i_{2})}"] & {(E_{F,U},\alpha)} \arrow[r, "{(r_{1},r_{1},r_{2})}"] & {(E_{F,X},\alpha)} \arrow[r] & 0
\end{tikzcd}$$
is $J$-exact.
\end{prop}
\begin{proof}
Since $F^{-1}(Y) = X$, it is clear that
\begin{enumerate}
\item $i_{1}(f\cdot\psi) = i_{1}(f)\cdot i_{1}(\psi)$ for all $f$ in $C_{0}(U\setminus X)$, $\psi$ in $E_{F,U\setminus X}$,
\item $\langle i_{1}(\psi), i_{1}(\varphi)\rangle = i_{2}(\langle\psi,\varphi\rangle)$, for all $\psi,\varphi$ in $E_{F,U\setminus X}$, and 
\item $i_{1}(\psi\cdot g) = i_{1}(\psi)\cdot i_{2}(g)$, for all $\psi$ in $E_{F,U\setminus X}$ and $g$ in $C_{0}(V\setminus Y)$.
\end{enumerate}
Recall that for a branched function $G:A\mapsto B$, $J_{(E_{G,A},\alpha_{A})} = C_{0}(A\setminus C_{G,A})$.
Since $U\setminus X$ is open in $U$, we have that $C_{F,U\setminus X} = C_{F, U}\cap (U\setminus X)$ and therefore $i(C_{0}((U\setminus X)\setminus C_{F,U\setminus X}))\subseteq C_{0}(U\setminus C_{F,U})$. To finish proving $(i_{1},i_{1},i_{2})$ is a morphism, it suffices to show $\hat{i_{1}}(\alpha_{U\setminus X}(f)) = \alpha_{U}(i_{1}(f))$ for all $f$ in $C_{0}((U\setminus X)\setminus C_{F,U\setminus X})$.
\par
It suffices to prove this for a positive function $f$ compactly supported in an open domain $W$ of $U$ such that $W\cap (X\cup C_{F,U}) = \emptyset$ and $F:W\mapsto F(W)$ is injective. In this case, $\alpha_{U\setminus X}(f) = \theta_{\sqrt{f},\sqrt{f}}$ and $\alpha_{U}(i_{1}(f)) = \theta_{\sqrt{i_{1}(f)},\sqrt{i_{1}(f)}}$. By definition, $\hat{i_{1}}(\theta_{\sqrt{f},\sqrt{f}}) = \theta_{i_{1}(\sqrt{f}), i_{1}(\sqrt{f})}$. Hence, $\hat{i_{1}}(\alpha_{U\setminus X}(f)) = \alpha_{U}(i(f))$.
\par
Similarly, it is clear from $F^{-1}(Y) = X$ that 

\begin{enumerate}
\item $r_{1}(f\cdot\psi) = r_{1}(f)\cdot r_{1}(\psi)$
\item $\langle r_{1}(\psi), r_{1}(\varphi)\rangle = r_{2}(\langle\psi,\varphi\rangle)$, for $f$ in $C_{0}(U)$, $\psi,\varphi$ in $E_{F,U}$, and 
\item $r_{1}(\psi\cdot g) = r_{1}(\eta)\cdot r_{2}(g)$, for all $\psi$ in $E_{F,U}$ and $g$ in $C_{0}(V)$.
\end{enumerate}
Therefore, $r_{1}:C_{2}(U)\mapsto C_{2}(X)$ is a $r_{2}$-twisted morphism of Hilbert modules. Every  twisted morphism of Hilbert modules has closed image (it extends isometrically to a *-homomorphism of the linking algebra, which has closed image), and therefore $r_{1}(C_{c}(U)) = C_{c}(X)$ implies $r_{1}(C_{2}(U)) = C_{2}(X)$.
\par
It is general that $C_{F,X}\subseteq C_{F,U}\cap X$, so we always have $r(C_{0}(U\setminus C_{F,U}))\subseteq C_{0}(X\setminus C_{F,X})$. By surjectivity of $r_{1}:C_{2}(U)\mapsto C_{2}(X)$ it is immediate that $\hat{r_{1}}(\alpha_{U}(f)) = \alpha_{X}(r_{1}(f))$
\par
$J$-exactness of 
$$
\begin{tikzcd}
0 \arrow[r] & {(E_{F,U\setminus X},\alpha)} \arrow[r, "{(i_{1},i_{1},i_{2})}"] & {(E_{F,U},\alpha)} \arrow[r, "{(r_{1},r_{1},r_{2})}"] & {(E_{F,X},\alpha)} \arrow[r] & 0
\end{tikzcd}$$

is equivalent to exactness of the sequences

$$
\begin{tikzcd}
0 \arrow[r] & C_{0}(U\setminus X) \arrow[r, "i_{1}"]                            & C_{0}(U) \arrow[r, "r_{1}"]                    & C_{0}(X) \arrow[r]                    & 0\\
0 \arrow[r] & {C_{0}(U\setminus (X\cup C_{F,U\setminus X}))} \arrow[r, "i_{1}"] & {C_{0}(U\setminus C_{F,U})} \arrow[r, "r_{1}"] & {C_{0}(X\setminus C_{F,X})} \arrow[r] & 0 \\
0 \arrow[r] & C_{2}(U\setminus X) \arrow[r, "i_{1}"]                            & C_{2}(U) \arrow[r, "r_{1}"]                    & C_{2}(X) \arrow[r]                    & 0\\
0 \arrow[r] & C_{0}(V\setminus Y) \arrow[r, "i_{2}"]                            & C_{0}(V) \arrow[r, "r_{2}"]                    & C_{0}(Y) \arrow[r]                    & 0.
\end{tikzcd}$$
Exactness of the top and bottom sequence is by hypothesis. Since $r_{1}(C_{2}(U)) = C_{2}(X)$ and $C_{2}(U)\cap C_{0}(U\setminus X) = C_{2}(U\setminus X)$, the second sequence from the bottom one is exact. Exactness of the sequence second from the top is equivalent to $C_{F,U}\cap X = C_{F,X}$.
\end{proof}
\section{Correspondences from holomorphic functions}
\label{holcor}
We now consider branched functions which are the restrictions of holomorphic functions and prove some extra regularity properties about them.
\par
Let $M$ and $N$ be second countable Riemann surfaces and $X\subseteq M$, $Y\subseteq N$ be closed subspaces. A function $F:X\mapsto Y$ is a \textit{holomorphic branched function} if there is a holomorphic function $\tilde{F}:M\mapsto N$ such that $\tilde{F}|_{X}  = F$ and $\tilde{F}^{-1}(Y) = X$.
$F$ is necessarily an open map, since $\tilde{F}$ is an open map and $\tilde{F}^{-1}(Y) = X$.
\par
Moreover, because $\tilde{F}$ is holomorphic, for every $u$ in $M$, there is a neighbourhood $U$ of $u$, bi-holomorphisms $\phi:U\mapsto \mathbb{D}_{r}$, $\psi:F(U)\mapsto\mathbb{D}_{r^{n}}$, for some $r> 0$ and $n$ in $\mathbb{N}$, such that $\psi\circ\tilde{F}\circ\phi^{-1}(z) = z^{n}$ for all $z$ in $\mathbb{D}_{r}$. Since $z^{n}$ is branched, it follows that $\tilde{F}$ is branched, with $n = \text{ind}_{\tilde{F}}(u)$. Hence, $\tilde{F}|_{X} = F$ is branched. Note that $\{u\in M:\text{ind}_{\tilde{F}}(u) > 1\} = C_{\tilde{F},M}$ and this set is countable and discrete.
\par
We will now show $J$-exactness of extensions of holomorphic branched functions is automatic, provided that its domain contains no isolated points. First, we need a lemma.
\begin{lemma}
\label{lemdyn}
Let $F:X\mapsto Y$ be a holomorphic branched function and assume $Y$ contains no isolated points. For $x$ in $X$, let $d_{F}(x) = d$ be the maximal number for which there exists a sequence $\{x_{n,j}\}_{n\in\mathbb{N},j\leq d}$ satisfying the properties
\begin{enumerate}
    \item $\lim_{n\to\infty}x_{n,i} = x$ for any $i\leq d$,
    \item $F(x_{n,i}) = F(x_{n,j}) := x'_{n}\neq F(x)$ for any $i,j\leq d$, $n$ in $\mathbb{N}$,
    \item $F^{-1}(x'_{n}) = \bigcup_{i=1}^{d}\{x_{n,i}\}$, for any $n$ in $\mathbb{N}$, and
    \item $x_{n,i}\neq x_{n,j}$ for any two distinct $i,j\leq d$, for any $n$ in $\mathbb{N}$.
\end{enumerate}
Then, $\text{ind}_{F}(x) = d_{F}(x)$.
\end{lemma}
\begin{proof}
    For $x$ in $X$, let $\{F_{i}^{-1}:V\mapsto U\}_{i=1}^{\text{ind}_{F}(x)}$ be inverse branches centered at $x$. Since $Y$ contains no isolated points, there is a sequence $x'_{n}$ in $V\setminus F(x)$ converging to $F(x)$.
    Since $\{x\in X:\text{ind}_{F}(x) > 1\}$ is discrete, we may assume $U$ is small enough so that $\text{ind}_{F}(u) = 1$ for all $u$ in $U\setminus \{x\}$. Hence, $\tilde{F}_{i}^{-1}(x'_{n}):= x_{n,i}$ must satisfy $x_{n,i}\neq x_{n,j}$ for all $i\neq j\leq \text{ind}_{F}(x)$ and $n$ in $\mathbb{N}$. By continuity of the inverse branches at $F(x)$, we have $\lim_{n\to\infty}x_{n,i} = x$ for all $i\leq \text{ind}_{F}(x)$.
    Lastly, we have $\bigcup_{i=1}^{\text{ind}_{F}(x)}x_{n,i} = F^{-1}(x'_{n})$, so that the collection $\{x_{n,i}\}_{i\leq \text{ind}_{F}(x)}$ satisfy $(1)$-$(4)$ in the hypothesis of the lemma. Hence, $\text{ind}_{F}(x)\leq d_{F}(x)$.
    \par
    If $\{x_{n,i}\}_{i\leq d_{F}(x)}$ is a collection satisfying $(1)$-$(4)$ above with $F(x_{n,i}):=x'_{n}$, then $\bigcup_{i=1}^{d_{F}(x)}x_{n,i}\\ \subseteq U\setminus x$ eventually, so that $F^{-1}(x'_{n}) = \{F_{i}^{-1}(x'_{n})\}_{i=1}^{\text{ind}_{F}(x)}\subseteq \bigcup_{i=1}^{d_{F}(x)}x_{n,i} = F^{-1}(x'_{n})$ eventually. Since $x_{n,i}\neq x_{n,j}$ for all $i\neq j\leq d_{F}(x)$, it follows that $d_{F}(x)\leq \text{ind}_{F}(x)$.
\end{proof}
\begin{cor}
\label{holmorph}
    Let $F:M\mapsto N$ be a holomorphic branched function, and suppose $Y\subseteq N$ is a closed set that contains no isolated points. Let $F^{-1}(Y) =: X $. Then, $C_{F,M}\cap X = C_{F,X}$, so that the sequence
    $$\begin{tikzcd}
0 \arrow[r] & {(E_{F,M\setminus X},\alpha)} \arrow[r, "{(i_{1},i_{1},i_{2})}"] & {(E_{F,M},\alpha)} \arrow[r, "{(r_{1},r_{1},r_{2})}"] & {(E_{F,X},\alpha)} \arrow[r] & 0
\end{tikzcd}$$
is $J$-exact.
\end{cor}
\begin{proof}
If $x$ is in $C_{F,M}\cap X$, then $\text{ind}_{F}(x) > 1$. By Lemma \ref{lemdyn}, $d_{F|_{X}}(x) = \text{ind}_{F|_{X}}(x) := \text{ind}_{F}(x) > 1$. It is easy to see that $\{x\in X:d_{F|_{X}}(x) > 1\} = C_{F,X}$, so that $x$ is in $C_{F,X}$. Hence, $C_{F,X} = C_{F,M}\cap X$. Proposition \ref{holmorph1} then implies the above sequence of correspondences is exact.
\end{proof}
\begin{cor}
\label{conjugacyiso}
If $F:X\mapsto X$ and $G:Y\mapsto Y$ are holomorphic branched functions for which $X$ and $Y$ contain no isolated points and $\varphi:X\mapsto Y$ is a homeomorphism such that $G\circ \varphi = \varphi\circ F$, then $\text{ind}_{G}\circ\varphi = \text{ind}_{F}$ and hence $(\varphi^{*}, \varphi^{*}, \varphi^{*}):(E_{G,Y},\alpha_{Y})\mapsto (E_{F,X}, \alpha_{X})$ is an isomorphism.
\end{cor}
\begin{proof}
It is easy to see that $d_{G}\circ \varphi = d_{F}$, so Lemma \ref{lemdyn} and Proposition \ref{conjugacyiso1} imply the Corollary.
\end{proof}
\subsection{$K$-theoretic properties of holomorphic branched functions}
\label{kholdyn}
In this section we study the induced mapping on $K$-theory associated to the $C^{*}$-correspondence of a holomorphic branched function. The results here will be crucial later in Section \ref{ratmapj}.
\par
Our main example of a holomorphic branched function will be a holomorphic dynamical system $R:M\mapsto M$ and its restriction to its Fatou set $F_{R}$ and its Julia set $J_{R}$. Recall from Section \ref{holdynsys} that when $M$ is either $\hat{\mathbb{C}}$, $\mathbb{C}$ or $\mathbb{C}^{*}$ and $\text{deg}(R) > 1$, $J_{R}$ is non-empty, totally invariant ($R^{-1}(J_{R}) = J_{R}$), contains no isolated points and is a closed subset of $M$. Corollary \ref{holmorph} implies we have a $J$-exact sequence
$$
\begin{tikzcd}
0 \arrow[r] & {(E_{R,F_{R}},\alpha_{F_{R}})} \arrow[r, "{(i,i,i)}"] & {(E_{R,M},\alpha_{M})} \arrow[r, "{(r,r,r)}"] & {(E_{R,J_{R}},\alpha_{J_{R}})} \arrow[r] & 0.
\end{tikzcd}$$

We now provide a concrete description of $\iota - \hat{\otimes}[\mathcal{E}_{R,J_{R}}]$ acting on $K^{0}$. 
\par
For a locally compact Hausdorff space $W$, let $\text{Tr}:K^{0}(W)\mapsto C_{0}(W,\mathbb{Z})$ be the homomorphism (of additive groups) defined for $g$ in $K^{0}(W)$ and $w$ in $W$ as $\text{Tr}(g)(w) = (ev_{w})_{*}(g)$, where $ev_{w}:C_{0}(W)\mapsto\mathbb{C}$ is evaluation at $w$. We can identify $K_{0}(\mathbb{C})$ with $\mathbb{Z}$ via the trace map. 
\par
We first show $ \hat{\otimes}[\mathcal{E}_{R,J_{R}}]$ acting on $K^{0}$ can be identified with the transfer operator $\Phi$ using the trace map $\text{Tr}$.
\begin{prop}
If $F:X\mapsto Y$ is a holomorphic branched function and $X,Y$ are closed  and proper subsets of either $\mathbb{C}^{*}$, $\mathbb{C}$ or $\hat{\mathbb{C}}$, then
\label{Triso}
$\text{Tr}:K^{0}(X\setminus C_{F,X})\mapsto C_{0}(X\setminus C_{F,X},\mathbb{Z})$, $\text{Tr}:K^{0}(Y)\mapsto C(Y,\mathbb{Z})$ are isomorphisms, and $\text{Tr}\circ (\hat{\otimes}_{0}[\mathcal{E}_{F,X}]) = (\Phi)\circ \text{Tr}$.
\end{prop}
\begin{proof}
When $X$ is a compact, connected  proper subspace of the Riemann sphere, $\hat{\mathbb{C}}\setminus X$ is a non-empty disjoint union of simply connected open sets, and so $K^{-1}(\hat{\mathbb{C}}\setminus X) = 0$ and $i_{*}:K^{0}(\hat{\mathbb{C}}\setminus X)\mapsto K^{0}(\hat{\mathbb{C}})$ maps onto $\mathbb{Z}\cdot \beta_{\hat{\mathbb{C}}}$ (by Corollary \ref{mainbott}). Therefore, the 6-term exact sequence of $K$-theory associated to 
$$
\begin{tikzcd}
0 \arrow[r] & C_{0}(\hat{\mathbb{C}}\setminus X) \arrow[r] & C(\hat{\mathbb{C}}) \arrow[r] & C(X) \arrow[r] & 0
\end{tikzcd}$$
implies $K^{0}(X) = \mathbb{Z}[1_{X}]\simeq\mathbb{Z} $, and in particular $\text{Tr}:K^{0}(X)\mapsto C(X,\mathbb{Z})$ is an isomorphism.
\par
Now, when $X$ is a finite disjoint union of compact connected sets in $\hat{\mathbb{C}}$, the above result implies $\text{Tr}:K^{0}(X)\mapsto C(X,\mathbb{Z})$ is an isomorphism.
\par
In general, if $X$ is a compact, proper set of $\hat{\mathbb{C}}$, then we can write $X = \bigcap_{n\in\mathbb{N}} X_{n}$, where, for every $n$ in $\mathbb{N}$, $X_{n}$ is a finite disjoint union of compact connected sets such that $X_{n+1}\subseteq X_{n}$. The diagram 
$$ 
\begin{tikzcd}
K^{0}(X_{n}) \arrow[d, "\text{Tr}"'] \arrow[r, "r_{*}"] & K^{0}(X_{n+1}) \arrow[d, "\text{Tr}"] \\
{C(X_{n},\mathbb{Z})} \arrow[r, "r"]                & {C(X_{n+1},\mathbb{Z})}            
\end{tikzcd}$$
commutes, for all $n$ in $\mathbb{N}$, where $r$ is the restriction map, and the vertical maps are isomorphisms. Therefore, the limit map $\text{Tr}:K^{0}(X)\mapsto C(X,\mathbb{Z})$ is an isomorphism.
\par
Now, suppose $C$ is a finite set contained in $X$. $K^{-1}(C) = 0$, so the following diagram has exact rows and commutes:
$$
\begin{tikzcd}
0 \arrow[r] & {K^{0}(X\setminus C)} \arrow[r] \arrow[d, "\text{Tr}"] & K^{0}(X) \arrow[r] \arrow[d, "\text{Tr}"] & {K^{0}(C)} \arrow[d, "\text{Tr}"] \\
0 \arrow[r] & {C_{0}(X\setminus C,\mathbb{Z})} \arrow[r]             & {C(X,\mathbb{Z})} \arrow[r]               & {C(C,\mathbb{Z}).}                
\end{tikzcd}$$
Since the two right-most vertical maps are isomorphisms, a diagram chase implies $\text{Tr}:K^{0}(X\setminus C)\mapsto C_{0}(X\setminus C,\mathbb{Z})$ is an isomorphism.
\par
Now if $X$ and $Y$ are closed and proper subsets of either $\mathbb{C}^{*}$, $\mathbb{C}$ or $\hat{\mathbb{C}}$, then $X = \overline{X}\setminus F_{X}$ and $Y = \overline{Y}\setminus F_{Y}$, where the closure is in $\hat{\mathbb{C}}$ and $F_{X}$, $F_{Y}$ are finite sets. By properness of $X$ and $Y$, we have $\overline{X}\neq\hat{\mathbb{C}}$ and $\overline{Y}\neq\hat{\mathbb{C}}$, so the above result applies to see that $\text{Tr}:K^{0}(X)\mapsto C_{0}(X,\mathbb{Z})$  and $\text{Tr}:K^{0}(Y)\mapsto C_{0}(Y,\mathbb{Z})$ are isomorphisms.
\par
Since $C_{F,X}$ is closed and discrete, we have $K^{-1}(C_{F,X}) = 0$ and $\text{Tr}:K^{0}(C_{F,X})\mapsto C_{0}(C_{F,X},\mathbb{Z})$ is an isomorphism. By a similar diagram chase to that above, we have that $\text{Tr}:K^{0}(X\setminus C_{F,X})\mapsto C_{0}(X\setminus C_{F,X},\mathbb{Z})$ is an isomorphism.
\par
Note that every element $f$ in $C_{0}(X,\mathbb{Z})$ is compactly supported, so that $\Phi(f)$ is well defined.
For $y$ in $Y$, $(\text{ev}_{y})_{*}\mathcal{E}_{F,X}$ is represented by the Hilbert $\mathbb{C}$-module with orthogonal basis $\{\delta_{x}\}_{x\in F^{-1}(y)}$ and inner product satisfying $\langle\delta_{x},\delta_{x}\rangle = \text{ind}_{F}(x)$, for all $x$ in $F^{-1}(y)$. The left action is defined for $f$ in $C_{0}(X\setminus C_{F,X})$ and $x$ in $F^{-1}(y)$ as $f\cdot \delta_{x} = f(x)\delta_{x}$. Therefore, $(\text{ev}_{y})_{*}\mathcal{E}_{F,X} = \sum_{x\in F^{-1}(y)}(\text{ev}_{x})_{*}$, where $(ev_{x})_{*}$ is thought of as a class in $KK^{0}(C_{0}(X\setminus C_{F,X}),\mathbb{C})$ and the sum is the direct sum operation of Kasparov bi-modules. When $\text{ind}_{F}(x) > 1$, $x$ is in $C_{F,X}$ and consequently $(ev_{x})_{*} = 0$ in $KK^{0}(C_{0}(X\setminus C_{F,X}),\mathbb{C})$. So, we can write $(ev_{y})_{*}\mathcal{E}_{F,X}$ as $\sum_{x\in F^{-1}(y)}\text{ind}_{F}(x)(ev_{x})_{*}$ and the diagram
$$
\begin{tikzcd}
{K^{0}(X\setminus C_{F,X})} \arrow[r, "{\hat{\otimes}_{0}[\mathcal{E}_{F,X}]}"] \arrow[d, "\text{Tr}"'] & K^{0}(Y) \arrow[d, "\text{Tr}"] \\
{C_{0}(X\setminus C_{F,X},\mathbb{Z})} \arrow[r, "\Phi"]                                            & {C_{0}(Y,\mathbb{Z}).}          
\end{tikzcd}$$
commutes.
\end{proof}
For a holomorphic branched function $F:X\mapsto Y$ with $X,Y$ closed and proper subsets of either $\mathbb{C}^{*}$, $\mathbb{C}$ or $\hat{\mathbb{C}}$,
we will identify $K^{0}(X\setminus C_{F,X})$, $K^{0}(X)$ and $K^{0}(Y)$ with $C_{0}(X\setminus C_{F,X},\mathbb{Z})$, $C_{0}(X,\mathbb{Z})$ and $C_{0}(Y,\mathbb{Z})$, respectively. 
\par
If $R:M\mapsto M$ is a holomorphic function such that $J_{R}\neq M$ and $M$ is either $\mathbb{C}^{*}$, $\mathbb{C}$ or $\hat{\mathbb{C}}$, by Corollary \ref{Maindiagram} and Proposition \ref{Triso}, we have a commutative diagram
$$
\begin{tikzcd}
{C_{0}(J_{R}\setminus C_{R,J_{R}},\mathbb{Z})} \arrow[r, "\text{exp}"] \arrow[d, "\Phi"'] & {K^{-1}(F_{R}\setminus C_{R,F_{R}})} \arrow[d, "{\hat{\otimes}_{1}[\mathcal{E}_{R,F_{R}}]}"] \\
{C_{0}(J_{R},\mathbb{Z})} \arrow[r, "\text{exp}"]                                             & K^{-1}(F_{R}).                                                                           
\end{tikzcd}$$
The left-most vertical map extends to the group homomorphism $\Phi:C_{0}(J_{R},\mathbb{Z})\mapsto C_{0}(J_{R},\mathbb{Z})$ and the top horizontal map extends to the exponential map $\text{exp}:C_{0}(J_{R},\mathbb{Z})\mapsto K^{-1}(F_{R}\setminus C_{R,F_{R}})$ from the short exact sequence

$$
\begin{tikzcd}
0 \arrow[r] & {C_{0}(F_{R}\setminus C_{R,F_{R}}))} \arrow[r, "i"] & {C_{0}(M\setminus C_{R,F_{R}})} \arrow[r, "r"] & C_{0}(J_{R}) \arrow[r] & 0.
\end{tikzcd}$$

We will show the extension of the diagram above commutes. This result will be used to compute the kernel and co-kernel of $\text{id} - \Phi:C(J_{R},\mathbb{Z})\mapsto C(J_{R},\mathbb{Z})$ when $R$ is rational (Proposition \ref{k0nc}). First, we we make a definition and prove two lemmas.
\par
If $U$ is a simply connected open proper subset of $\hat{\mathbb{C}}$ and $z$ is in $U$, the short exact sequence
$$
\begin{tikzcd}
0 \arrow[r] & C_{0}(U\setminus \{z\})) \arrow[r] & C_{0}(U) \arrow[r, "\text{ev}_{z}"] & \mathbb{C} \arrow[r] & 0
\end{tikzcd}$$
yields an isomorphism $\text{exp}:\mathbb{Z} = K^{0}(\{z\})\mapsto K^{-1}(U\setminus z).$ If $1_{z}$ denotes the characteristic function on the point $\{z\}$, we will let $v_{z} =: \text{exp}(1_{z})$
\par
Let $F_{d}:\mathbb{D}\mapsto \mathbb{D}$ be the mapping defined, for $z$ in $\mathbb{D}$, as $F_{d}(z) = z^{d}$.
\begin{lemma}
\label{zd}
$\hat{\otimes}_{1}[\mathcal{E}_{F_{d},\mathbb{D}\setminus\{0\}}]:K^{-1}(\mathbb{D}\setminus\{0\})\mapsto K^{-1}(\mathbb{D}\setminus\{0\})$ maps $v_{0}$ to $d\cdot v_{0}$.
\end{lemma}
\begin{proof}
If $d=1$, then $F_{d} = \text{id}_{\mathbb{D}}$ and the lemma follows.
\par
Assume $d>1$. Denote $(F_{d})^{*}:C_{0}(\mathbb{D})\mapsto C_{0}(\mathbb{D})$ by $\varphi_{d}$. By naturality of $\text{exp}$, the diagram 

$$
\begin{tikzcd}
K^{0}(\{0\}) \arrow[r, "(\varphi_{d})_{*}"] \arrow[d, "\text{exp}"] & K^{0}(\{0\}) \arrow[d, "\text{exp}"] \\
K^{-1}(\mathbb{D}\setminus\{0\}) \arrow[r, "(\varphi_{d})_{*}"]     & K^{-1}(\mathbb{D}\setminus\{0\})    
\end{tikzcd}$$
commutes. The top horizontal map is equal to the identity. Hence, $(\varphi_{d})_{*}(v_{0}) = v_{0}$. Therefore, to prove the lemma, it suffices to show $\varphi_{d}^{*}[\mathcal{E}_{F_{d},\mathbb{D}\setminus\{0\}}]$ is equal to $d\cdot \text{id}_{C_{0}(\mathbb{D}\setminus\{0\})}$ in $KK^{0}(C_{0}(\mathbb{D}\setminus\{0\}),C_{0}(\mathbb{D}\setminus\{0\})).$
\par
Let $\omega$ be a $d^{th}$ root of unity, and consider for $0\leq j\leq d-1$ the linear map $P_{j}:C_{0}(\mathbb{D}\setminus\{0\})\mapsto C_{0}(\mathbb{D}\setminus\{0\})$ defined for
$\psi$ in $C_{0}(\mathbb{D}\setminus\{0\})$ as $P_{j}(\psi)(z) = \frac{1}{d}\sum_{k=1}^{d-1}\omega^{-jk}\psi(\omega^{k}z)$, $z$ in $\mathbb{D}\setminus\{0\}$. If $\varphi$ is a function in the image of $P_{j}$, then it
satisfies $\varphi(\omega z) = \omega^{j}\varphi(z)$, for all $z$ in $\mathbb{D}\setminus\{0\}$, so $P_{j}(\varphi) = \varphi$, and hence $P_{j}^{2} = P_{j}$.
\par
We also have, for any $z$ in $\mathbb{D}\setminus\{0\}$, that $\sum_{j=0}^{d-1}P_{j}(\psi)(z) = \sum_{k=0}^{d-1}(\frac{1}{d}\sum_{j=0}^{d-1}\omega^{-jk})\psi(\omega^{k}z)$.
$\frac{1}{d}\sum_{j=0}^{d-1}\omega^{-jk} = 1$ if $j=0$, and is zero otherwise. Hence, $\sum_{j=0}^{d-1}P_{j} = \text{id}$.
\par
We show $P_{j}$ is a Hilbert $C_{0}(\mathbb{D}\setminus\{0\})$-module endomorphism of $E_{F_{d},\mathbb{D}\setminus\{0\}}$ commuting with the left action $\alpha_{\mathbb{D}\setminus\{0\}}\circ \varphi_{d}$. Since $\omega^{d} = 1$, we have, for any $a$ in $C_{0}(\mathbb{D}\setminus\{0\})$, $\psi$ in $E_{F_{d}, \mathbb{D}\setminus\{0\}}$ and $z$ in $\mathbb{D}\setminus\{0\}$,
that 
$$P_{j}(\psi\cdot a)(z) = \frac{1}{d}\sum_{k=1}^{d-1}\omega^{-jk}\psi(\omega^{k}z)a((\omega^{k}z)^{d}) = (\frac{1}{d}\sum_{k=1}^{d-1}\omega^{-jk}\psi(\omega^{k}z))a(z^{d}) = (P_{j}(\psi)\cdot a)(z).$$
We also have $P_{j}(a\cdot\psi) = a\cdot P_{j}(\psi)$, since the left and right actions are equal for $\varphi_{d}^{*}\mathcal{E}_{F_{d}, \mathbb{D}\setminus\{0\}}$. 
\par
Since $\sum_{j=0}^{d-1}P_{j} = \text{id}$, to show that
$P_{j}^{*} = P_{j}$ (hence $P_{j}$ is adjointable), we only need to show that for all $\psi_{1}, \psi_{2}$ in $E_{F_{d}, \mathbb{D}\setminus\{0\}}$ and $i\neq j$, we have $\langle P_{i}(\psi_{1}), P_{j}(\psi_{2})\rangle = 0$. Write $\varphi_{1} = P_{i}(\psi_{1})$ and $\varphi_{2} = P_{j}(\psi_{2})$. For $z$ in $\mathbb{D}\setminus\{0\}$, we have 
$\langle \varphi_{1}, \varphi_{2} \rangle(F_{d}) = \sum_{k=0}^{d-1}\overline{\varphi_{1}}(\omega^{k}z)\varphi_{2}(\omega^{k}z).$ Since $\varphi_{1}(\omega^{k}z) = \omega^{ik}\varphi_{1}(z)$ and
$\varphi_{2}(\omega^{k}z) = \omega^{jk}\varphi_{1}(z)$, we have that 
$\sum_{k=0}^{d-1}\overline{\varphi_{1}}(\omega^{k}z)\varphi_{2}(\omega^{k}z) = (\sum_{k=0}^{d-1}\omega^{(j-i)k})\overline{\varphi_{1}}(z)\varphi_{2}(z) = 0$.
\par
We have shown that $\{P_{i}\}_{i=0}^{d-1}$ is a collection of mutually orthogonal projections of the Kasparov $C_{0}(\mathbb{D}\setminus\{0\})-C_{0}(\mathbb{D}\setminus\{0\})$ bi-module $\varphi_{d}^{*}\mathcal{E}_{F_{d},\mathbb{D}\setminus\{0\}}$ that sum to $\text{id}$.
Hence, $[\varphi_{d}^{*}\mathcal{E}_{F_{d}, \mathbb{D}\setminus\{0\}}] = \sum_{i=0}^{d-1}[P_{i}\varphi_{d}^{*}\mathcal{E}_{F_{d}, \mathbb{D}\setminus\{0\}}]$. 
\par
Consider the unitary $U$ of $\varphi_{d}^{*}\mathcal{E}_{F_{d}, \mathbb{D}\setminus\{0\}}$ defined for $\psi$ in $E_{F_{d}, \mathbb{D}\setminus\{0\}}$ as
$U(\psi)(z) = \frac{z}{|z|}\psi(z)$, $z$ in $\mathbb{D}\setminus\{0\}$. Then, one checks that $UP_{i}U^{*} = P_{i+1}$ for $0\leq i\leq d-1$ ($P_{d} = P_{0}$). Therefore, 
$[\varphi_{d}^{*}\mathcal{E}_{F_{d}, \mathbb{D}\setminus\{0\}}] = d[P_{0}\varphi_{d}^{*}\mathcal{E}_{F_{d}, \mathbb{D}\setminus\{0\}}]$. We show $[P_{0}\varphi_{d}^{*}\mathcal{E}_{F_{d}, \mathbb{D}\setminus\{0\}}] = \text{id}$.
\par
Note that $[\mathcal{E}_{F_{0},\mathbb{D}\setminus\{0\}}] = [\text{id}_{C_{0}(\mathbb{D}\setminus\{0\})}]$, so it suffices to show $[P_{0}\varphi_{d}^{*}\mathcal{E}_{F_{d}, \mathbb{D}\setminus\{0\}}] = [\mathcal{E}_{F_{0},\mathbb{D}\setminus\{0\}}]$.
Consider the map $S =\frac{1}{\sqrt{d}}(F_{d})^{*}:E_{z, \mathbb{D}\setminus\{0\}}\mapsto E_{F_{d}, \mathbb{D}\setminus\{0\}}$. It is obvious $S$ is $C_{0}(\mathbb{D}\setminus\{0\})$ linear with respect to the left and right actions on each bi-module. For $\psi_{1}, \psi_{2}$
in $E_{z, \mathbb{D}\setminus\{0\}}$, we have $\langle S(\psi_{1}), S(\psi_{2})\rangle(z) = \frac{1}{d}\sum_{w:w^{d} = z}\overline{\psi_{1}}(w^{d})\psi_{2}(w^{d}) = \overline{\psi_{1}}(z)\psi_{2}(z) = \langle\psi_{1}, \psi_{2}\rangle(z).$ So,
$S$ preserves the inner products. 
\par
Every function $\varphi$ in the image of $S$ satisfies $\varphi(\omega z) = \varphi(z)$ for all $z$ in $\mathbb{D}\setminus\{0\}$, 
so $\text{im}(S)\subseteq \text{im}(P_{0})$. For $\varphi$ in $E_{F_{d}, \mathbb{D}\setminus\{0\}}$, let $\Phi(\varphi)(z) = \frac{1}{\sqrt{d}}\sum_{w:w^{d} = z}\varphi(w)$. Then, $S\Phi = P_{0}$, so if $\varphi$ is in $\text{im}(P_{0})$, we have
$S\Phi(\varphi) = \varphi$. Therefore, $\text{im}(S) = \text{im}(P_{0})$ and $S:E_{F_{0},\mathbb{D}\setminus\{0\}}\mapsto P_{0}E_{F_{d},\mathbb{D}\setminus\{0\}}$ is an isomorphism of Kasparov bi-modules. Hence, $[P_{0}\varphi_{d}^{*}\mathcal{E}_{F_{d}, \mathbb{D}\setminus\{0\}}] = [\mathcal{E}_{F_{0},\mathbb{D}\setminus\{0\}}]$.
\end{proof}
\begin{lemma}
    \label{lem:evk-1}
    Let $M$ be either $\mathbb{C}^{*}$, $\mathbb{C}$ or $\hat{\mathbb{C}}$ and $X$ a proper closed subset of $M$. Then, for every non-zero $v$ in $K^{-1}(M\setminus X)$, there is $x$ in $X$ such that $(i_{x})_{*}(v)\neq 0$, where $i_{x}:C_{0}(M\setminus X)\mapsto C_{0}(M\setminus x)$ is the inclusion.
\end{lemma}
\begin{proof}
We may assume $|X|\geq 2$, otherwise the lemma is trivial. If $M = \hat{\mathbb{C}}$, by applying a rotation, we may assume without loss of generality that $\infty$ is in $X$. Then, $M\setminus X = \mathbb{C}\setminus (X\setminus \{\infty\})$ and $X\setminus \{\infty\}$ is closed in $\mathbb{C}$. This reduces the case of $M = \hat{\mathbb{C}}$ to the case where $M = \mathbb{C}$.
\par
If $M = \mathbb{C}^{*}$, then $X\cup\{0\}=:X_{0}$ is closed in $\mathbb{C}$ and for $x$ in $X$, $(i_{x})_{*} = (j_{x})_{*}$, where $j_{x}:C_{0}(\mathbb{C}\setminus X_{0})\mapsto C_{0}(\mathbb{C}\setminus \{x, 0\})$ is the inclusion. Let $k_{x}:C_{0}(\mathbb{C}\setminus X_{0})\mapsto C_{0}(\mathbb{C}\setminus \{x\})$ and $l_{x}:C_{0}(\mathbb{C}\setminus \{x,0\})\mapsto C_{0}(\mathbb{C}\setminus \{x\})$ denote the inclusions. For $x$ in $X$, we have $l_{x}\circ j_{x} = k_{x}$ and $k_{0} = l_{0}\circ j_{x}$. By functorality of $K$-theory it follows that if $(k_{x})_{*}(v)\neq 0$ for some $v$ in $K^{-1}(\mathbb{C}^{*}\setminus X))$, then we must have $(j_{x})_{*}(v)\neq 0$. This reduces the case of $M = \mathbb{C}^{*}$ to the case $M = \mathbb{C}$.
\par
So, assume $X$ is a closed and proper subset of $\mathbb{C}$. For $x$ in $X$, the diagram

$$
\begin{tikzcd}
0 \arrow[r] & C_{0}(\mathbb{C}\setminus X) \arrow[r] \arrow[d, "i_{x}"] & C_{0}(\mathbb{C}) \arrow[r, "r_{X}"] \arrow[d, equal] & C_{0}(X) \arrow[r] \arrow[d, "\text{ev}_{x}"] & 0 \\
0 \arrow[r] & C_{0}(\mathbb{C}\setminus \{x\}) \arrow[r]                & C_{0}(\mathbb{C}) \arrow[r]                                         & \mathbb{C} \arrow[r]                          & 0
\end{tikzcd}$$
commutes and has exact rows. Therefore, by naturality of $\text{exp}$, the diagram

$$
\begin{tikzcd}
K^{-1}(\mathbb{C}\setminus X) \arrow[d, "(i_{x})_{*}"] & {C_{0}(X,\mathbb{Z})} \arrow[l, "\text{exp}"'] \arrow[d, "\text{ev}_{x}"] \\
K^{-1}(\mathbb{C}\setminus x)          & K^{-1}(x) = \mathbb{Z}. \arrow[l, "\text{exp}"']                          
\end{tikzcd}$$
Since $K^{-1}(\mathbb{C}) = 0$ the horizontal maps are surjections. $K^{0}(\mathbb{C})\mapsto K^{0}(x)$ and $(r_{X})_{*}:K^{0}(\mathbb{C})\mapsto C_{0}(X,\mathbb{Z})$ both are the zero maps and hence the horizontal maps above are also injections. 
\par
For every $f\neq 0$ in $C_{0}(X,\mathbb{Z})$ there is $x$ in $X$ such that $ev_{x}(f)\neq 0$. Using this and that the horizontal maps above are isomorphisms, the lemma follows.
\end{proof}

\begin{prop}
\label{ncintertwine}
Let $R:M\mapsto M$ be a holomorphic dynamical system such that $\text{deg}(R) > 1$, where $M$ is either $\mathbb{C}^{*}$, $\mathbb{C}$ or $\hat{\mathbb{C}}$. If $J_{R}\neq M$ then the diagram
$$
\begin{tikzcd}
{C_{0}(J_{R},\mathbb{Z})} \arrow[r, "\text{exp}"] \arrow[d, "\Phi"'] & {K^{-1}(F_{R}\setminus C_{R,F_{R}})} \arrow[d, "{\hat{\otimes}_{1}[\mathcal{E}_{R,F_{R}}]}"] \\
{C_{0}(J_{R},\mathbb{Z})} \arrow[r, "\text{exp}"]                        & K^{-1}(F_{R})                                                                           
\end{tikzcd}$$
commutes.
\end{prop}
\begin{proof}

First, we show for every $x$ in $J_{R}$, the diagram

$$
\begin{tikzcd}[sep = huge]
{C_{0}(R^{-1}(x),\mathbb{Z})} \arrow[d, "\text{exp}"] \arrow[r, "\Phi_{x}"]                                                 & {\mathbb{Z} = K^{0}(x)} \arrow[d, "\text{exp}"] \\
{K^{-1}(M\setminus (R^{-1}(x)\cup C_{R,M}))} \arrow[r, "{\hat{\otimes}_{1}[\mathcal{E}_{R,M\setminus R^{-1}(x)}]}"] & K^{-1}(M\setminus \{x\})                     
\end{tikzcd}$$
commutes, where $\Phi_{x}:C(R^{-1}(F),\mathbb{Z})\mapsto \mathbb{Z}$ is the group homomorphism sending $1_{z}$, $z$ in $R^{-1}(x)$, to $\text{ind}_{R}(z)1_{x}$.
\par
$R^{-1}(x)$ and $C_{R,M}$ are closed and discrete, so for every $z$ in $R^{-1}(x)$, there is a simply connected neighbourhood $U_{z}$ of $z$ such that $U_{z}\cap (R^{-1}(x)\cup C_{R,M}) = \{x\}$, $V_{z}:= R(U_{z})$ is simply connected, with bi-holomorphisms $\varphi:U_{z}\mapsto\mathbb{D}$, $\psi:V_{z}\mapsto\mathbb{D}$ such that $\varphi(z) = 0$, $\psi(x) = 0$ and, for all $w$ in $\mathbb{D}$, $\psi\circ R\circ\varphi^{-1}(w) = w^{\text{ind}_{R}(z)}$.
\par
Lemma \ref{zd} implies $v_{0}\hat{\otimes}_{1}[\mathcal{E}_{\psi\circ R\circ\varphi^{-1},\mathbb{D}\setminus\{0\}}] = \text{ind}_{R}(z)\cdot v_{0}$. 
\par
Naturality of $\text{exp}$ implies $((\varphi^{-1})^{*})_{*}(v_{z}) = v_{0}$ and $(\psi^{*})_{*}(v_{0}) = v_{x}$, so $v_{z}\hat{\otimes}_{1}[\mathcal{E}_{R,U_{z}\setminus\{z\}}] = ((\psi^{-1})^{*})_{*}(v_{0}\hat{\otimes}_{1}[\mathcal{E}_{\psi\circ R\circ\varphi^{-1},\mathbb{D}\setminus\{0\}}]) = \text{ind}_{R}(z)\cdot v_{x}$.
\par
Naturality of $\text{exp}$ and Proposition \ref{classnaturality} then imply $\text{exp}(1_{z})\hat{\otimes}_{1}[\mathcal{E}_{R,M\setminus R^{-1}(x)}] = \text{ind}_{R}(z)v_{x} = \text{exp}(\Phi_{x}(1_{x}))$. As $z$ in $R^{-1}(x)$ was arbitrary, the above identity implies the diagram
$$
\begin{tikzcd}[sep = huge]
{C_{0}(R^{-1}(x),\mathbb{Z})} \arrow[d, "\text{exp}"] \arrow[r, "\Phi_{x}"]                                                 & {\mathbb{Z} = K^{0}(x)} \arrow[d, "\text{exp}"] \\
{K^{-1}(M\setminus (R^{-1}(x)\cup C_{R,M}))} \arrow[r, "{\hat{\otimes}_{1}[\mathcal{E}_{R,U\setminus R^{-1}(F)}]}"] & K^{-1}(M\setminus \{x\})                     
\end{tikzcd}$$
commutes, as claimed.
\par
Let $x$ be in $J_{R}$, and let $i_{x}: K^{-1}(M\setminus (J_{R}\cup C_{R,F_{R}}))\mapsto K^{-1}(
M\setminus (R^{-1}(x)\cup C_{R,F_{R}}))$ and $j_{x}:K^{-1}(M\setminus J_{R})\mapsto K^{-1}(M\setminus \{x\})$ be the maps induced by the respective inclusions. Let $r_{x}:C(J_{R},\mathbb{Z})\mapsto C(R^{-1}(x),\mathbb{Z})$ and $ev_{x}:C(J_{R},\mathbb{Z})\mapsto\mathbb{Z} $ be the restriction/evaluation maps. By the above commutative diagram, for every $g$ in $C(J_{R},\mathbb{Z})$, we have

$ 0 = \text{exp}(r_{x}(g))\hat{\otimes}_{1}[\mathcal{E}_{R,M\setminus R^{-1}(x)}] - \text{exp}(\Phi_{x}(r_{x}(g))).$ 
Note that $\Phi_{x}\circ r_{x} = ev_{x}\circ\Phi$, and naturality of $\text{exp}$ implies $\text{exp}\circ r_{x} = i_{x}\circ\text{exp}$ and $\text{exp}\circ ev_{x} = j_{x}\circ\text{exp}$.

Therefore, $\text{exp}(r_{x}(g))\hat{\otimes}_{1}[\mathcal{E}_{R,M\setminus R^{-1}(x)}] - \text{exp}(\Phi_{x}(r_{x}(g))) = i_{x}(\text{exp}(g))\hat{\otimes}_{1}[\mathcal{E}_{R,M\setminus R^{-1}(x)}] - j_{x}(\text{exp}(\Phi(g))).$ Proposition \ref{classnaturality} implies $i_{x}(\text{exp}(g))\hat{\otimes}_{1}[\mathcal{E}_{R,M\setminus R^{-1}(x)}] = j_{x}(\text{exp}(g)\hat{\otimes}_{1}[\mathcal{E}_{R,F_{R}}])$. So, putting these equalities together, we have $0 = j_{x}(\text{exp}(g)\hat{\otimes}_{1}[\mathcal{E}_{R,M\setminus J_{R}}] - \text{exp}(\Phi(g)))$. As $x$ in $J_{R}$ was arbitrary, Lemma \ref{lem:evk-1} implies $\text{exp}(g)\hat{\otimes}_{1}[\mathcal{E}_{R,F_{R}}] =\text{exp}(\Phi(g))$.
\end{proof}
\subsection{Appendix: Bott projections}
\label{sectionbott}
To compute certain maps in the exact sequences we will work with in this paper, it will be useful to ``orient" the $K^{0}$ (and $K^{-1}$) groups for special subsets of $\hat{\mathbb{C}}$ by finding minimal generating sets that behave well with maps like inclusion and bi-holomorphisms. In this section, given a connected open set $U$ of $\hat{\mathbb{C}}$, we describe such a canonical generator $\beta_{U}$ for $K^{0}(U)$, the \textit{Bott projection of $U$}. When $U = \mathbb{C}$ it is the usual Bott projection. We could not find these results in the literature, but we do not claim any originality as they are likely folklore. The reader may wish to skip this section if they believe Corollary \ref{mainbott}.
\par
Our method is to construct the Bott projection for the open unit disk $\mathbb{D}$, prove some properties about it, then bootstrap up to the general construction (Corollary \ref{mainbott}).
\par
First, let $U$ be a simply connected open set of the complex plane $\mathbb{C}$, and let $\gamma$ be a Jordan curve inside it. 
Denote by $U^{-}_{\gamma}$ the connected component of
$U\setminus\text{im}(\gamma)$ whose closure in $\mathbb{C}$ intersects the boundary $\partial U$, and denote the other component by $U^{+}_{\gamma}$.
\par
For a continuous function $a:\text{im}(\gamma)\mapsto U_{\gamma}^{+}$, let $u_{a,\gamma}:\text{im}(\gamma)\mapsto S^{1}$ be defined, for $z$ in 
$\text{im}(\gamma)$, as $u_{a,\gamma}(z) = \frac{z-a(z)}{|z-a(z)|}$. Any two functions $a,b:\text{im}(\gamma)\mapsto U_{\gamma}^{+}$ are homotopic, 
since $U_{\gamma}^{+}$ is homeomorphic to $\mathbb{D}$. It follows then, by compactness of $\text{im}(\gamma)$, that $u_{a,\gamma}$ and $u_{b,\gamma}$
are homotopic as elements in $C(\text{im}(\gamma), S^{1})$. Hence, the class $u_{\gamma} := [u_{a,\gamma}]$ in $K^{-1}(\text{im}(\gamma))$ is independent of $a$.
\par
Let $\delta_{\gamma}:K^{-1}(\text{im}(\gamma))\mapsto K^{0}(U_{\gamma}^{+})$ be the index map from the 6-term exact sequence associated to the short exact sequence
$$
\begin{tikzcd}
0 \arrow[r] & C_{0}(U_{\gamma}^{+}) \arrow[r] & C_{0}(\overline{U_{\gamma}^{+}}) \arrow[r] & C(\text{im}(\gamma)) \arrow[r] & 0.
\end{tikzcd}$$
Let $\iota_{\gamma}:K^{0}(U_{\gamma}^{+})\mapsto K^{0}(U)$ be the map induced from the inclusion 
$i:C_{0}(U_{\gamma}^{+})\mapsto C_{0}(U)$.
\begin{prop}
\label{simplybott}
For every simply connected open set $U$ properly contained in $\hat{\mathbb{C}}$, there is a unique generator $\beta_{U}$ in $K^{0}(U)\simeq\mathbb{Z}$ such that
\begin{enumerate}
\item if $U\subseteq\mathbb{C}$, then $\beta_{U} = \iota_{\gamma}\delta_{\gamma}(u_{\gamma})$ for any Jordan curve $\gamma$ in $U$,
\item if $T:U\mapsto V$ is a bi-holomorphism, where $V$ is an open subset of $\hat{\mathbb{C}}$, then $((T^{-1})^{*})_{*}\beta_{U} = \beta_{V}$, and
\item if $U\subseteq V$, where $V$ is a simply connected open set properly contained in $\hat{\mathbb{C}}$, and $i:C_{0}(U)\mapsto C_{0}(V)$ is the inclusion, then $i_{*}\beta_{U} = \beta_{V}$.
\end{enumerate}
\end{prop}
\begin{proof}
We first prove $(1)$ and $(2)$ of the Proposition for $U = \mathbb{D} = V$, and then bootstrap to the general case.
\par
First, we show $\iota_{\gamma}\delta_{\gamma}(u_{\gamma}) = \iota_{\eta}\delta_{\eta}(u_{\eta})$ for any two Jordan curves in $\mathbb{D}$. To prove this, it suffices to show 
$\iota_{\gamma}\delta_{\gamma}(u_{\gamma}) = \iota_{\varepsilon_{r}}\delta_{\varepsilon_{r}}(u_{\varepsilon_{r}})$, where $\varepsilon_{r}(t) = r e^{2\pi i t}$, for $t$ in $[0,1]$, for any $r < 1$ such that $|\gamma| < r$.
\par
Let $\mathbb{A}_{r,\gamma} = \mathbb{D}_{r}\setminus \overline{\mathbb{D}_{\gamma}^{+}}$. Let 
$\delta_{\mathbb{A}}:K^{-1}(\text{im}(\gamma))\oplus K^{-1}(\text{im}(\varepsilon_{r}))\mapsto K^{0}(\mathbb{A}_{r,\gamma})$ be the index map from the 6-term exact 
sequence associated to 
$$
\begin{tikzcd}
0 \arrow[r] & {C_{0}(\mathbb{A}_{r,\gamma})} \arrow[r, "i"] & {C_{0}(\overline{\mathbb{A}_{r,\gamma}})} \arrow[r, "r"] & C(\text{im}(\gamma)\sqcup \text{im}(\varepsilon_{r})) \arrow[r] & 0.
\end{tikzcd}$$
Choose a point $z_{0}$ in $\mathbb{D}_{\gamma}^{+}$, and define $u:\overline{\mathbb{A}_{r,\gamma}}\mapsto S^{1}$, for $z$ in $\overline{\mathbb{A}_{r,\gamma}}$, as $u(z) = \frac{z-z_{0}}{|z-z_{0}|}$.
Since $\mathbb{D}_{\gamma}^{+}\subseteq \mathbb{D}_{\varepsilon_{r}}^{+} = \mathbb{D}_{r}$, it follows that $r_{*}[u] = [u_{z_{0},\gamma}]\oplus [u_{z_{0},\varepsilon_{r}}]$. Hence, by exactness,
$\iota_{\mathbb{A}}\delta_{\mathbb{A}}(u_{\gamma}) = -\iota_{\mathbb{A}}\delta_{\mathbb{A}}(u_{\varepsilon_{r}})$, where 
$\iota_{\mathbb{A}}:K^{0}(\mathbb{A}_{r,\gamma})\mapsto K^{0}(\mathbb{D})$ is the map induced from inclusion.
\par
Let $\delta:K^{-1}(\text{im}(\gamma))\mapsto K^{0}(\mathbb{D}^{+}_{\gamma})\oplus K^{0}(\mathbb{A}_{r,\gamma})$ be the index map from the 6-term exact sequence associated to
$$
\begin{tikzcd}
0 \arrow[r] & {C_{0}(\mathbb{D}^{+}_{\gamma}\sqcup\mathbb{A}_{r,\gamma})} \arrow[r] & C_{0}(\mathbb{D}_{r}) \arrow[r] & C(\text{im}(\gamma)) \arrow[r] & 0.
\end{tikzcd}$$
By naturality of the index map, we have that $\delta(u_{\gamma}) = \delta_{\gamma}(u_{\gamma})\oplus \delta_{\mathbb{A}}(u_{\gamma})$, and so exactness implies
$\iota_{\gamma}\delta_{\gamma}(u_{\gamma}) + \iota_{\mathbb{A}}\delta_{\mathbb{A}}(u_{\gamma}) = 0$. Hence,
$\iota_{\gamma}\delta_{\gamma}(u_{\gamma}) = \iota_{\mathbb{A}}\delta_{\mathbb{A}}(u_{\varepsilon_{r}})$. By naturality of the index map, 
$\iota_{\mathbb{A}}\delta_{\mathbb{A}}(u_{\varepsilon_{r}}) = \iota_{\varepsilon_{r}}\delta_{\varepsilon_{r}}(u_{\varepsilon_{r}})$, proving the claim.
\par
Now, we show $\beta_{\mathbb{D}}:=\iota_{\varepsilon_{r}}\delta_{\varepsilon_{r}}(u_{\varepsilon_{r}})$ generates $K^{0}(\mathbb{D})$. 
Since $C_{0}(\mathbb{D}\setminus\mathbb{D}_{r})$ is contractible to $0$, the 6-term exact sequence of $K$-theory implies the map $\iota_{\varepsilon_{r}}:K^{0}(\mathbb{D}_{r})\mapsto K^{0}(\mathbb{D})$ 
is an isomorphism. Similarily, $C_{0}(\overline{\mathbb{D}_{r}}\setminus\{0\})$ being contractible to $0$ implies 
$K^{0}(\mathbb{\overline{D}}_{r}) = \mathbb{Z}[1_{\overline{\mathbb{D}_{r}}}]$. Hence, $i_{*}:K^{0}(\mathbb{D}_{r})\mapsto K^{0}(\overline{\mathbb{D}_{r}})$
must be zero. It then follows by exactness that $\iota_{\varepsilon_{r}}\delta_{\varepsilon_{r}}:K^{0}(\text{im}(\varepsilon_{r}))\mapsto K^{0}(\mathbb{D})$ is an isomorphism. 
Clearly the class of $u_{0,\varepsilon_{r}}(z) = \frac{z}{r}$, $z$ in $r S^{1}$, generates $K^{-1}(\text{im}(\varepsilon_{r}))$, and therefore 
$\beta_{\mathbb{D}} = \iota_{\varepsilon_{r}}\delta_{\varepsilon_{r}}(u_{\varepsilon_{r}})$ generates $K^{0}(\mathbb{D})$. Hence, $(1)$ is proven for $U = \mathbb{D}$.
\par
Every bi-holomorphism $T:\mathbb{D}\mapsto\mathbb{D}$ is of the form $T(z) = e^{2\pi i\theta}\frac{z-a}{1-\overline{a}z}$, for any $z$ in $\mathbb{D}$, for some $a$ in $\mathbb{D}$, $\theta$ in $\mathbb{R}$, and is therefore homotopic to the identity. 
Therefore, $((T^{-1})^{*})_{*}\beta_{\mathbb{D}} = \beta_{\mathbb{D}}$. This proves $(2)$ in the case that $U = V = \mathbb{D}$.
\par
Now we prove the Proposition in the case that $U$ and $V$ are properly contained in $\mathbb{C}$. By the Riemann mapping Theorem, it follows that there is a 
bi-holomorphism $T:\mathbb{D}\mapsto U$. By naturality of the index map, the diagram
$$
\begin{tikzcd}
K^{-1}(\text{im}(\gamma)) \arrow[r, "T_{*}"] \arrow[d, "\delta_{\gamma}"]      & K^{-1}(\text{im}(T(\gamma))) \arrow[d, "\delta_{T(\gamma)}"] \\
K^{0}(\mathbb{D}^{+}_{\gamma}) \arrow[r, "T_{*}"] \arrow[d, "\iota_{\gamma}"] & K^{0}(U^{+}_{\gamma}) \arrow[d, "\iota_{T(\gamma)}"]        \\
K^{0}(\mathbb{D}) \arrow[r, "T_{*}"]                                          & K^{0}(U)                                                   
\end{tikzcd}$$
commutes. Hence, to prove $(1)$ and $((T^{-1})^{*})_{*}\beta_{\mathbb{D}} = \beta_{U}$, it suffices to prove
$T_{*}u_{\gamma} = u_{T(\gamma)}$ for any Jordan curve $\gamma$ in $\mathbb{D}$.
\par
Since $T$ is a bi-holomorphism and $\text{im}(\gamma)$ is compact, for every $\epsilon > 0$, there is a continuous map $b:\text{im}(T(\gamma))\mapsto U^{+}_{T(\gamma)}$ such that

$$\bigg|\frac{T^{-1}(z) - T^{-1}(b(z))}{z-b(z)}\cdot\frac{|z-b(z)|}{|T^{-1}(z) - T^{-1}(b(z))|}\text{ } -\text{ } \frac{(T^{-1})'(z)}{|(T^{-1})'(z)|}\bigg| <\epsilon, \text{ } \text{for all }z\text{ in } \text{im}(T(\gamma)).$$

Let $a = T^{-1}\circ b\circ T$. Note that, for all $w$ in $\text{im}(T(\gamma))$, $u_{a,\gamma}\circ T^{-1}(w) = \frac{T^{-1}(w) - T^{-1}(b(w))}{|T^{-1}(w) - T^{-1}(b(w))|}$ and 
$u_{b,T(\gamma)}(w) = \frac{w - b(w)}{|w - b(w)|}$. Hence, $\| (u_{a,\gamma}\circ T^{-1})\cdot (u_{b,T(\gamma)})^{-1} - \frac{(T^{-1})'}{|(T^{-1})'|}\| <\epsilon$. So, if we choose $\epsilon < 2$,
then, by \cite[Lemma~2.1.3~(iii)]{Rordam:intro_K-theory}, $(u_{a,\gamma}\circ T^{-1})\cdot (u_{b,T(\gamma)})^{-1}$ is homotopic to $\frac{(T^{-1})'}{|(T^{-1})'|}$ as elements in $C(\text{im}(T(\gamma)),S^{1})$.
\par
The domain of $\frac{(T^{-1})'}{|(T^{-1})'|}:\text{im}(T(\gamma))\mapsto S^{1}$ extends continuously to $U^{+}_{T(\gamma)}$ and is therefore
homotopic, as an element in $C(\text{im}(T(\gamma)),S^{1})$, to a constant. Hence, $T_{*}u_{\gamma} - u_{T(\gamma)} = 
[(u_{a,\gamma}\circ T^{-1})(u_{b,T(\gamma)})^{-1}] = 0$ in $K^{-1}(\text{im}(T(\gamma)))$.
\par
Now, suppose $U$, $V$ are simply connected proper open sets of $\mathbb{C}$. Let $T_{U}:\mathbb{D}\mapsto U$ and $T_{V}:\mathbb{D}\mapsto V$ be bi-holomorphisms. 
If $T:U\mapsto V$ is a bi-holomorphism, then $T_{*}\beta_{U} = (T_{V})_{*}(T^{-1}_{V})_{*}T_{*}(T_{U})_{*}\beta_{\mathbb{D}} = (T_{V})_{*}\beta_{\mathbb{D}} = \beta_{V}$, proving $(2)$ in this case.
\par
If $U\subseteq V$, let $\gamma$ be a Jordan curve in $U$. Denote $\gamma$ by $\gamma_{U}$, $\gamma_{V}$ when thinking of it as a curve in $U$, $V$, respectively. The commutative diagram
$$
\begin{tikzcd}
K^{-1}(\text{im}(\gamma_{U})) \arrow[d, "\delta_{\gamma_{U}}"] \arrow[r, no head, equal] & K^{-1}(\text{im}(\gamma_{V})) \arrow[d, "\delta_{\gamma_{V}}"] \\
K^{0}(U^{+}_{\gamma_{U}}) \arrow[r, no head, equal] \arrow[d, "\iota_{\gamma_{U}}"]             & K^{0}(V^{+}_{\gamma_{V}}) \arrow[d, "\iota_{\gamma_{V}}"]     \\
K^{0}(U) \arrow[r, "i_{*}"]                                                              & K^{0}(V)                                                     
\end{tikzcd}$$
along with the fact that $\iota_{\gamma_{U}}\delta_{\gamma_{U}}(u_{\gamma_{U}}) =\beta_{U}$ and $ \iota_{\gamma_{V}}\delta_{\gamma_{V}}(u_{\gamma_{V}}) = \beta_{V}$ implies $i_{*}\beta_{U} = \beta_{V}$. This proves $(3)$ for
such $U$ and $V$. 
\par
The only case remaining for $(1)$ is when $U = \mathbb{C}$. Since $C_{0}(\mathbb{C})$ is the inductive limit of the proper simply connected open sets of $\mathbb{C}$ (ordered by inclusion) and
$i_{*}\beta_{U} = \beta_{V}$ when $U\subseteq V$ , by continuity of $K^{0}$, we have that $\beta_{\mathbb{C}}:= i_{*}\beta_{U}$ generates $K^{0}(\mathbb{C})$ and is independent of $U$, where $U$ is any simply connected
proper open set of $\mathbb{C}$. Every Jordan curve $\gamma$ in $\mathbb{C}$ is eventually contained in a proper simply connected open set $U$ of $\mathbb{C}$, so naturality of the index map, and the fact that $i_{*}\beta_{U} = \beta_{\mathbb{C}}$,
implies $\iota_{\gamma}\delta_{\gamma}(u_{\gamma}) = \beta_{\mathbb{C}}$. This proves $(1)$, and $(2)$ in the case that $V = \mathbb{C}$.
\par
If $T:\mathbb{C}\mapsto\mathbb{C}$ is a bi-holomorphism, it is affine linear, and hence homotopic to the identity. Therefore, $T_{*}\beta_{\mathbb{C}} = \beta_{\mathbb{C}}$, and so $(3)$ is proven when $U = V = \mathbb{C}$.
\par
Now, for an arbitrary open simply connected set $U$ properly contained in $\hat{\mathbb{C}}$, let $T:\hat{\mathbb{C}}\mapsto \hat{\mathbb{C}}$ be a Mobius 
transformation such that $T^{-1}(U)\subseteq\mathbb{C}$ and define $\beta_{U} = T_{*}\beta_{T^{-1}(U)}$. By the properties of $\beta_{T^{-1}(U)}$ proven above, it is clear that $\beta_{U}$ is independent of $T$, and that 
$(2)$ and $(3)$ of the Proposition are satisfied.
\end{proof}
\begin{cor}
\label{hatbott}
$\beta_{\hat{\mathbb{C}}}:= i_{*}(\beta_{U})$ in $K^{0}(\hat{\mathbb{C}})$ is independent of the choice of simply connected proper open set $U$ in $\hat{\mathbb{C}}$. $\beta_{\hat{\mathbb{C}}}$ and $[1_{\hat{\mathbb{C}}}]$ form a minimal generating set for $K^{0}(\hat{\mathbb{C}})\simeq\mathbb{Z}^{2}.$
\end{cor}
\begin{proof}
Let $U$ and $V$ be simply connected open sets properly contained in $\hat{\mathbb{C}}$. Choose simply connected sets $U'\subseteq U$ and $V'\subseteq V$ such that $U'\cup V'$ is contained in a simply connected proper open set $W$ of $\hat{\mathbb{C}}$. The commutative diagram 
$$
\begin{tikzcd}
K^{0}(U') \arrow[r] \arrow[rd] & K^{0}(U) \arrow[rd] &                         \\
                               & K^{0}(W) \arrow[r]  & K^{0}(\hat{\mathbb{C}}) \\
K^{0}(V') \arrow[r] \arrow[ru] & K^{0}(V) \arrow[ru] &                        
\end{tikzcd}$$
and Proposition \ref{simplybott} $(3)$ implies $i_{*}(\beta_{U}) = i_{*}(\beta_{V}):= \beta_{\hat{\mathbb{C}}}$.
\par
The short exact sequence

$$
\begin{tikzcd}
0 \arrow[r] & K^{0}(\mathbb{C}) \arrow[r, "i_{*}"] & K^{0}(\hat{\mathbb{C}}) \arrow[r] & K^{0}(\infty) \arrow[r] & 0
\end{tikzcd}$$
implies $\beta_{\hat{\mathbb{C}}}$ and $[1_{\hat{\mathbb{C}}}]$ form a minimal generating set for $K^{0}(\hat{\mathbb{C}})\simeq\mathbb{Z}^{2}$.
\end{proof}
\begin{prop}
\label{bottopenconnect}
Let $U$ be a connected open set properly contained in $\hat{\mathbb{C}}$. Then, $K^{0}(U)\simeq\mathbb{Z}$, with a unique generator $\beta_{U}:= i_{*}(\beta_{V})$ in $K^{0}(U)$, for any simply connected open set $V\subseteq U$.
\end{prop}
\begin{proof}
Since $U\neq\hat{\mathbb{C}}$, and $((T^{-1})^{*})_{*}\beta_{V} = \beta_{T(V)}$ for any fractional linear transformation $T$ and simply connected open set $V$ of $\hat{\mathbb{C}}$, by applying a Mobius transformation to $U$, it suffices to prove the Proposition in the case that $U$ is an open set in $\mathbb{C}$.
\par
First, we show that $K^{-1}(\hat{\mathbb{C}}\setminus U) = 0$ and $q_{*}:K^{0}(\hat{\mathbb{C}})\mapsto K^{0}(\hat{\mathbb{C}}\setminus U)$ sends $\beta_{\hat{\mathbb{C}}}$ to $0$, where $q$ is the restriction map. 
\par
We claim that $U$ can be written as a countable union $U = \bigcup_{n\in\mathbb{N}}U_{n}$ where, for each $n$ in $\mathbb{N}$, $U_{n}$ is a finite union of bounded open rectangles $\mathcal{U}_{n} = \{R^{n}_{1},...,R^{n}_{k_{n}}\}$ such that
\begin{enumerate}
    \item $U_{n}$ is connected,
    \item The closure of $U_{n}$ in $\mathbb{C}$ is contained in $U$,
    \item $U_{n}\subseteq U_{n+1}$, and
    \item for any $i,j\leq k_{n}$, $\overline{R^{n}_{i}}\cap\overline{R^{n}_{j}}$ is either empty or a closed rectangle.
\end{enumerate}
Finding a collection $\mathcal{U}_{n} = \{R^{n}_{1},...,R^{n}_{k_{n}}\}_{n\in\mathbb{N}}$ satisfying $(1)$ and $(2)$ is easy. Suppose $\mathcal{U}_{l}$ also satisfy $(3)$ and $(4)$ for $l\leq n-1$. 
\par
Since $\overline{U_{n}}$ is compact and contained in $U$, for each $i\leq k_{n}$ we can dilate $R^{n}_{i}$ to a larger rectangle $\tilde{R}^{n}_{i}$ so that $\{\tilde{R}^{n}_{1},...,\tilde{R}^{n}_{k_{n}}\} = \mathcal{U}'_{n}$ satisfies $(1)$, $(2)$ and $(4)$. Denote $U'_{n} = \bigcup_{i\leq k_{n}}\tilde{R}^{n}_{i}$.
\par
Now, for each $i\leq k_{n}$, cover  $\tilde{R}^{n}_{i}\setminus \overline{R^{n}_{i}}$ by a finite collection of open rectangles $\mathcal{S}_{i} = \{S_{i,1},...,S_{i,m_{i}}\}$ in $U$, and, for $l\geq n+1$, let $U'_{l} = U_{l}\cup\bigcup_{i,j}S_{i,j}$ and $\mathcal{U}'_{l} = \mathcal{U}_{l}\cup\bigcup_{i}\mathcal{S}_{i}$. $\mathcal{U}'_{n+1}$ satisifes $(1)$ and $(2)$, and $U_{n}$ satisfies $(3)$. So, we may induct this process to replace a collection $\{\mathcal{U}_{n}\}_{n\in\mathbb{N}}$ satisfying $(1)$ and $(2)$ with a collection satisfying $(1)-(4)$.
\par
Fix such a collection $\{\mathcal{U}_{n}\}_{n\in\mathbb{N}}$. Let $X$ be a maximal connected component of $\hat{\mathbb{C}}\setminus U_{n}$. We show that $X$ has a (non-empty) simply connected interior with a boundary that is a Jordan curve. 
\par
For each $i\leq k_{n}$, write $\partial R^{n}_{i} = \bigcup_{j=1}^{4}L_{i,j}$, where $L_{i,j}$ is a maximal horizontal or vertical line. Let $x$ be in $\partial X.$ Then, $x$ is in a line $L_{i,j}$ for some $i,j$. Let $\gamma_{x}:[0,t]\mapsto \partial X$ be the unit speed curve following the line $L_{i,j}$ in the counter-clockwise direction, starting at $\gamma_{x}(0) = x$ and ending once the line intersects a new line at $\gamma_{x}(t)$. By condition $(4)$ of $\mathcal{U}_{n}$, this new line is unique. Continue defining $\gamma_{x}$ on this new line, and inductively on the next new lines until $\gamma_{x}$ intersects itself at $\gamma_{x}(t_{*})$. By condition $(4)$ of $\mathcal{U}_{n}$, this first intersection point of $\gamma_{x}$ must be $x$. Since $\gamma_{x}:[0,t_{*}]\mapsto \partial X$ is piece-wise linear with no intersections other than at the endpoints, it is a Jordan curve. 
\par
By the Jordan curve Theorem, $\gamma_{x}$ separates $\hat{\mathbb{C}}$ into two simply connected components, one which contains $U_{n}$, and the other, denoted $A_{x}$, which does not. By maximality of $X$, we must have $\overline{A_{x}}\subseteq X$. By $(4)$ and connectedness of $\partial X$ (which follows from $(1)$), we must have $\partial X = \gamma_{x}([0,t_{*}]) = \partial\overline{A_{x}}$ and hence $\overline{A_{x}} = X$. We have shown $X$ has a non-empty simply connected interior with boundary that is a Jordan curve.
\par
Therefore, by Carathéodory's Theorem, there is a homeomorphism $\varphi:X\mapsto\overline{\mathbb{D}}$, which is holomorphic on the interior of $X$. Hence $K^{-1}(X) = 0$. Since $\hat{\mathbb{C}}\setminus U_{n}$ is a finite disjoint union of its maximal connected components, it follows that $K^{-1}(\hat{\mathbb{C}}\setminus U_{n}) = 0$. 
By Corollary \ref{hatbott}, $i_{*}:K^{0}(A)\mapsto K^{0}(\hat{\mathbb{C}})$ sends $\beta_{A}$, where $A$ is a simply connected component as above, to $\beta_{\hat{\mathbb{C}}}$, and so $q_{*}(\beta_{\hat{\mathbb{C}}}) = j_{*}(\beta_{A})$ in the direct summand $K^{0}(X)$ of $K^{0}(\hat{\mathbb{C}}\setminus U_{n})$, where $j = i:C_{0}(A)\mapsto C(X)$ is the inclusion. Since $\varphi$ is a bi-holomorphism on $A$, Proposition \ref{simplybott} $(2)$ implies $((\varphi^{-1})^{*})_{*}j_{*}(\beta_{A}) = j_{*}(\beta_{\mathbb{D}})$, which is $0$ in $K^{0}(\overline{\mathbb{D}})$. Hence, $q_{*}(\beta_{\mathbb{C}}) = 0$ in $K^{0}(\hat{\mathbb{C}}\setminus U_{n})$.
\par
$C(\hat{\mathbb{C}}\setminus U)$ is the inductive limit of the restriction maps $r_{n}:C(\hat{\mathbb{C}}\setminus U_{n})\mapsto C(\hat{\mathbb{C}}\setminus U_{n+1})$, so, by continuity of $K^{-1}$ and $K^{0}$, it follows that $K^{-1}(\hat{\mathbb{C}}\setminus U) = 0$ and $q_{*}(\beta_{\hat{\mathbb{C}}}) = 0$ in $K^{0}(\hat{\mathbb{C}}\setminus U)$. Note also that $q_{*}([1_{\hat{\mathbb{C}}}]) = [1_{\hat{\mathbb{C}}\setminus U}].$
\par
Hence, the 6 term exact sequence of $K$-theory associated to 
$$
\begin{tikzcd}
0 \arrow[r] & C_{0}(U) \arrow[r, "i"] & C(\hat{\mathbb{C}}) \arrow[r, "q"] & C(\hat{\mathbb{C}}\setminus U) \arrow[r] & 0
\end{tikzcd}$$
implies that $i_{*}:K^{0}(U)\mapsto K^{0}(\hat{\mathbb{C}})$ is an isomorphism onto $\mathbb{Z}\beta_{\hat{\mathbb{C}}}.$ Let $\beta_{U}$ be the unique generator of $K^{0}(U)$ such that $i_{*}(\beta_{U}) = \beta_{\hat{\mathbb{C}}}.$
\par
By Corollary \ref{hatbott}, $j_{*}(\beta_{V}) = \beta_{\hat{\mathbb{C}}}$ for any simply connected open set $V\subseteq U$, where $j:C_{0}(V)\mapsto C(\hat{\mathbb{C}})$ is the inclusion. Hence, $i_{*}(\beta_{V}) = \beta_{U}$.
\end{proof}
\begin{cor}
\label{mainbott}
For any connected open set $U$ of $\hat{\mathbb{C}}$, there is a generator $\beta_{U}$ in $K^{0}(U)$ such that, for $V\subseteq\hat{\mathbb{C}}$ another connected open set,
\begin{enumerate}
    \item if $T:U\mapsto V$ is a bi-holomorphism, then $((T^{-1})^{*})_{*}\beta_{U} = \beta_{V}$, and
    \item if $U\subseteq V$ and $i:C_{0}(U)\mapsto C_{0}(V)$ is the inclusion, then $i_{*}\beta_{U} = \beta_{V}$.
\end{enumerate}
Moreover, if $U$ is not entirely $\hat{\mathbb{C}}$, then $K^{0}(U) = \mathbb{Z}\beta_{U}\simeq\mathbb{Z}$. If $U = \hat{\mathbb{C}}$, then $\beta_{\hat{\mathbb{C}}}$ and  $[1_{\hat{\mathbb{C}}}]$ form a minimal generating set for $K^{0}(\hat{\mathbb{C}})\simeq\mathbb{Z}^{2}$.
\end{cor}
\begin{proof}
The generating properties of $\beta_{U}$ are contained in Corollary \ref{hatbott} and Proposition \ref{bottopenconnect}. 
\par
Let $U'\subseteq U$ be a simply connected open set, $U\subseteq V$ and $T:U\mapsto T(U)$ a bi-holomorphism. By Proposition \ref{bottopenconnect}, the maps induced by inclusion send $\beta_{U'}$ to $\beta_{U}$, $\beta_{U'}$ to $\beta_{V}$, and $\beta_{T(U')}$ to $\beta_{T(U)}$. Hence, the commutative diagrams
$$
\begin{tikzcd}
K^{0}(U') \arrow[r, "T_{*}"] \arrow[d] & K^{0}(T(U')) \arrow[d] & K^{0}(U') \arrow[d] \arrow[rd] &          \\
K^{0}(U) \arrow[r, "T_{*}"]            & K^{0}(T(U))            & K^{0}(U) \arrow[r, "i_{*}"]    & K^{0}(V)
\end{tikzcd}$$
and Proposition \ref{simplybott} $(2)$ imply the Corollary.
\end{proof}

\section{The $K$-theory of a rational function acting on the Riemann sphere}
\label{ratmapc}
We now compute the $K$-theory for an arbitrary rational function acting on the Riemann sphere, which follows easily from Corollary \ref{mainbott}, Proposition \ref{classnaturality}, and the Pimsner-Voiculescu 6-term exact sequence. 
\par
We first record an easy lemma that will be used in this section as well as in Section \ref{ratmapf}.
\begin{lemma}
\label{localE}
Let $R$ be a rational function and $U$ an open set such that $\tilde{R} = R:U\mapsto R(U) = V$ is a homeomorphism. Then, $\beta_{U}\hat{\otimes}_{0}[\mathcal{E}_{\tilde{R},U}] = \beta_{V}$.
\end{lemma}
\begin{proof}
For $f$ in $C_{0}(U)$,  $\psi, \varphi$ in $E_{\tilde{R}, U}$, and $g$ in $C_{0}(V)$, we have $f\cdot\psi = f\psi$, $\langle\psi,\varphi\rangle = (\overline{\psi}\varphi)\circ (\tilde{R})^{-1}$, and 
$\varphi\cdot g = \varphi(g\circ \tilde{R})$. Therefore, the class of $\mathcal{E}_{\tilde{R}, U}$ in $KK^{0}(C_{0}(U), C_{0}(V))$ is equal to the class of the *-isomorphism
$(\tilde{R}^{-1})^{*}$.
\par
$\tilde{R}:U\mapsto V$ is a bi-holomorphism, so by Corollary \ref{mainbott}, we have $((\tilde{R}^{-1})^{*})_{*}\beta_{U} = \beta_{V}$.
\end{proof}
We now compute one of the maps appearing in the Pimsner-Voiculescu 6-term exact sequence of $K$-theory for $\mathcal{O}_{R,\hat{\mathbb{C}}}$.
\begin{prop}
\label{Cmap}
Let $R$ be a rational function. Then,
$\hat{\otimes}[\mathcal{E}_{R, \hat{\mathbb{C}}}] = \iota$ as mappings $K^{0}(\hat{\mathbb{C}}\setminus C_{R,\hat{\mathbb{C}}})\mapsto K^{0}(\hat{\mathbb{C}})$.
\end{prop}
\begin{proof}
Let $U$ be a connected open set in $\hat{\mathbb{C}}$ such that $U\cap C_{R,\hat{\mathbb{C}}} = \emptyset$ and $\tilde{R} = R:U\mapsto R(U) =: V$ is a homeomorphism. By Proposition \ref{classnaturality}, $(\iota_{U})^{*}[\mathcal{E}_{R,\hat{\mathbb{C}}}] = (i_{V})_{*}[\mathcal{E}_{\tilde{R},U}]$.
\par
Corollary \ref{mainbott} implies that $(\iota_{U})_{*}:K^{0}(U)\mapsto K^{0}(\hat{\mathbb{C}}\setminus C_{R,\hat{\mathbb{C}}})$ is an isomorphism. Therefore, to prove the Proposition, it suffices to show $\hat{\otimes}(\iota_{U})^{*}[\mathcal{E}_{R,\hat{\mathbb{C}}}] = \hat{\otimes}(\iota_{U})^{*}\iota$ as mappings $K^{0}(U)\mapsto K^{0}(\hat{\mathbb{C}})$.
\par
Note that 
\begin{enumerate}
    \item $\hat{\otimes}(\iota_{U})^{*}\iota = (i_{U})_{*}$, and
    \item $\hat{\otimes}(\iota_{U})^{*}[\mathcal{E}_{R,\hat{\mathbb{C}}}] = (i_{V})_{*}\circ(\hat{\otimes}_{0}[\mathcal{E}_{\tilde{R},U}])$ (from above).
\end{enumerate}
By Corollary \ref{mainbott}, both $(i_{U})_{*}$ and $(i_{V})_{*}\circ(\hat{\otimes}_{0}[\mathcal{E}_{\tilde{R},U}])$ send $\beta_{U}$ to $\beta_{\hat{\mathbb{C}}}$.
\end{proof}
The above calculation is enough information to determine the $K$-theory of $\mathcal{O}_{R,\hat{\mathbb{C}}}$.
\begin{thm}
\label{CK}
$\delta_{PV}:K_{1}(\mathcal{O}_{R,\hat{\mathbb{C}}})\mapsto K^{0}(\hat{\mathbb{C}}\setminus C_{R,\hat{\mathbb{C}}})$ is an isomorphism, and
$$ 
\begin{tikzcd}
0 \arrow[r] & K^{0}(\hat{\mathbb{C}}) \arrow[r, "i_{*}"] & {K_{0}(\mathcal{O}_{R,\hat{\mathbb{C}}})} \arrow[r, "\text{exp}_{PV}"] & K^{-1}(\hat{\mathbb{C}}\setminus C_{R,\hat{\mathbb{C}}}) \arrow[r] & 0
\end{tikzcd}$$ is a short exact sequence. Consequently, $K_{0}(\mathcal{O}_{R,\hat{\mathbb{C}}})\simeq \mathbb{Z}^{|C_{R,\hat{\mathbb{C}}}|+1}$, with the class of the unit a generator in a minimal generating set for $K_{0}(\mathcal{O}_{R,\hat{\mathbb{C}}})$, and $K_{1}(\mathcal{O}_{R,\hat{\mathbb{C}}})\simeq\mathbb{Z}$.
\end{thm}
\begin{proof}
By Proposition \ref{Cmap}, and the fact that $K^{-1}(\hat{\mathbb{C}}) = 0$, we may fill in the Pimsner-Voiculescu 6-term exact sequence as follows:
$$ 
\begin{tikzcd}
K^{0}(\hat{\mathbb{C}}\setminus C_{R,\hat{\mathbb{C}}})) \arrow[r, "0"]            & K^{0}(\hat{\mathbb{C}}) \arrow[r, "i"] & {K_{0}(\mathcal{O}_{R,\hat{\mathbb{C}}})} \arrow[d, "exp"] \\
{K_{1}(\mathcal{O}_{R,\hat{\mathbb{C}}})} \arrow[u, "\delta"] & 0 \arrow[l]                            & K^{-1}(\hat{\mathbb{C}}\setminus C_{R,\hat{\mathbb{C}}}). \arrow[l]             
\end{tikzcd}$$
Exactness of the above diagram imply the first two claims of the Proposition.
\par
$K_{1}(\mathcal{O}_{R,\hat{\mathbb{C}}})\simeq\mathbb{Z}$ then follows from Corollary \ref{mainbott}. By the paragraph preceding Proposition \ref{genmappingk1} (this forward reference won't cause any circular arguments), we have that $K^{-1}(\hat{\mathbb{C}}\setminus C_{R,\hat{\mathbb{C}}})\simeq\mathbb{Z}^{|C_{R,\hat{\mathbb{C}}}| - 1}$, and so the short exact sequence in the Corollary splits. Hence, by Corollary \ref{mainbott} in the case $U = \hat{\mathbb{C}}$, we have $K_{0}(\mathcal{O}_{R,\hat{\mathbb{C}}})\simeq K^{0}(\hat{\mathbb{C}})\oplus K^{-1}(\hat{\mathbb{C}}\setminus C_{R,\hat{\mathbb{C}}})\simeq\mathbb{Z}^{|C_{R,\hat{\mathbb{C}}}|+1}$, with the class of the unit a generator in a minimal generating set.
\end{proof}

\section{The $K$-theory of a rational function acting on its Fatou set}
\label{ratmapf}
In this section we compute the $K$-theory of a rational function $R$ acting on its Fatou set. As in the previous section, we do so by calculating the kernel and co-kernel of $\iota - \hat{\otimes}_{i}[\mathcal{E}_{R,F_{R}}]$, $i=0,1$. 
\par
The case for $i=0$ follows similar techniques as in the previous section once we understand how $R$ permutes its Fatou components. Most of the section will be dedicated to $i=1$. The fact that any compact set of $F_{R}$ is eventually mapped into a global ``attractor'' with manageable $K^{-1}$ (essentially Corollary \ref{fatattractor}) makes this calculation possible.
\par
Let's write $F_{R} = \bigcup_{x\in X} U_{x}$, where $U_{x}$ are the maximal connected components of $F_{R}$ and $X$ is a (countable) indexing set. Since $R^{-1}(F_{R}) = F_{R}$, it is easy to see that $R$
maps $U_{x}$ onto another component $U_{\sigma(x)}$, for every $x$ in $X$.
\par
For each $x$ in $X$, denote $\beta_{U_{x}} = f_{x}$ and $\beta_{U_{x}\setminus C_{R,F_{R}}} = e_{x}$. By Corollary \ref{mainbott} we can (and will) identify $K^{0}(F_{R}\setminus C_{R,F_{R}})$ and $K^{0}(F_{R})$  with $\bigoplus_{x\in X}\mathbb{Z}[e_{x}]$ and  $\bigoplus_{x\in X}\mathbb{Z}[f_{x}]$, respectively, via the inclusion maps.
\begin{prop}
\label{localEFK0}
$\hat{\otimes}_{0}[\mathcal{E}_{R, F_{R}}]$ sends $e_{x}$ to $f_{\sigma(x)}$, for all $x$ in $X$.
\end{prop}
\begin{proof}
For each $x$ in $X$, let $U$ be an open set in $U_{x}$ such that $\tilde{R} = R:U\mapsto R(U) = V$ is a homeomorphism. Then, by Proposition \ref{classnaturality}, we have $(\iota_{U})^{*}[\mathcal{E}_{R,U_{x}}] = (i_{V})_{*}[\mathcal{E}_{\tilde{R},U}]$, where $\iota_{U} = i:C_{0}(U)\mapsto C_{0}(U_{x}\setminus C_{R,F_{R}})$ and $i_{V} = i:C_{0}(V)\mapsto C_{0}(U_{\sigma(x)})$ are the inclusions. 
\par
Hence,
$e_{x}\hat{\otimes}_{0}[\mathcal{E}_{R,U_{x}}] = ((\iota_{U})_{*}\beta_{U})\hat\otimes [\mathcal{E}_{R,U_{x}}] = (i_{V})_{*}(\beta_{U}\hat{\otimes}\mathcal{E}_{\tilde{R},U})$. By Lemma \ref{localE}, $\beta_{U}\hat{\otimes}\mathcal{E}_{\tilde{R},U} \\= \beta_{V}$, and, by its definition, $f_{\sigma(x)} = (i_{V})_{*}\beta_{V}$. Therefore, $e_{x}\hat{\otimes}_{0}[\mathcal{E}_{R,U_{x}}] = f_{\sigma(x)}$ for all $x$ in $X$.
\par
Since 
$(\iota_{x})^{*}\mathcal{E}_{R, F_{R}} = (i_{\sigma(x)})_{*}\mathcal{E}_{R, U_{x}}$ for all $x$ in $X$, where
$\iota_{x} = i:C_{0}(U_{x}\setminus C_{R,\hat{\mathbb{C}}})\mapsto C_{0}(F_{R})$ and $i_{\sigma(x)} = i:C_{0}(U_{\sigma(x)})\mapsto C_{0}(F_{R})$ are the inclusions,
we have that, under the identifications made previous to the Corollary,
$\hat{\otimes}_{0}[\mathcal{E}_{R,F_{R}}] = \oplus_{x\in X}\hat{\otimes}_{0}[\mathcal{E}_{R, U_{x}}]$ as mappings from $\bigoplus_{x\in X}\mathbb{Z}[e_{x}]$ to $\bigoplus_{x\in X}\mathbb{Z}[f_{x}]$, proving the result.
\end{proof}
We will call a finite subset $P = \{x_{1},...,x_{n}\}\subseteq X$ a \textit{cycle} if $\sigma(P) = P$ and $\sigma:P\mapsto P$ is minimal, and so, equivalently, $P$ is a periodic orbit of $\sigma:X\mapsto X$. It easy to see that two distinct cycles are disjoint.
We will denote the
collection of cycles to be $\mathcal{F}_{R}$. 
\par
By \cite[Corollary~2]{Shishikura}, $\mathcal{F}_{R}$ must be a finite set. By \cite[Theorem~1]{NWDSullivan}, for every $x$ in $X$, there is $k$ in $\mathbb{N}$ and a cycle $P$ in $\mathcal{F}_{R}$ such that $\sigma^{k}(x)$ is in $P$. So,
$X = \bigsqcup_{P\in\mathcal{F}_{R}}\bigcup_{n\in\mathbb{N}}\sigma^{-n}(P)$.

\begin{cor}
\label{Fkernel}
The mapping $\eta =\iota - \hat{\otimes}_{0}[\mathcal{E}_{R, F_{R}}]:K^{0}(F_{R}\setminus C_{R,\hat{\mathbb{C}}})\mapsto K^{0}(F_{R})$ has kernel generated by the elements
$e_{P} := \sum_{x\in P}e_{x}$, where $P$ is a cycle in $\mathcal{F}_{R}$.
\par
For each cycle $P$, choose an $x_{P}$ in $P$. The subgroup of $K^{0}(F_{R})$ generated by the elements $\{f_{x_{P}}\}_{P\in\mathcal{F}_{R}}$ maps isomorphically onto the co-kernel of $\eta$, via the quotient map.
\end{cor}
\begin{proof}
For each $P\in\mathcal{F}_{R}$, denote $X_{P} = \bigcup_{n\in\mathbb{N}}\sigma^{-n}(P)$, and $E_{P}$, $F_{P}$ the subgroups of $K^{0}(F_{R}\setminus C_{R,\hat{\mathbb{C}}})$, $K^{0}(F_{R})$ generated by 
$\{e_{x}\}_{x\in X_{P}}$, $\{f_{x}\}_{x\in X_{P}}$, respectively. By Corollary \ref{mainbott} and Proposition \ref{localEFK0}, $\eta(e_{x}) = f_{x} - f_{\sigma(x)} $ for all $x$ in $X$, so $\eta^{-1}(F_{P}) = E_{P}$ for all $P$ in $\mathcal{F}_{R}$. Denote $\eta_{P} = \eta:E_{P}\mapsto F_{P}$. 
By the above commentary we have $X_{P}\cap X_{P'} = \emptyset$ for distinct cycles $P$, $P'$ and $\bigsqcup_{P\in\mathcal{F}_{R}}X_{P} = X$, so (canonically) $K^{0}(F_{R}\setminus C_{R,\hat{\mathbb{C}}})\simeq \bigoplus_{P\in\mathcal{F}_{R}} E_{P}$, $K^{0}(F_{R})\simeq \bigoplus_{P\in\mathcal{F}_{R}}F_{P}$ and $\eta\simeq \oplus_{P\in\mathcal{F}_{R}}\eta_{P}$. Hence, it suffices to show
$\text{ker}(\eta_{P}) = \mathbb{Z}[e_{P}]$ and $\mathbb{Z}[f_{x}]$ maps isomorphically onto $\text{co-ker}(\eta_{P})$ via the quotient map, for any cycle $P$ in $\mathcal{F}_{R}$ and any $x$ in $P$.
\par
Since $\sigma:P\mapsto P$ is a bijection, we have that $\eta(e_{P}) = \sum_{x\in P}f_{x} - \sum_{x\in P}f_{\sigma(x)}  = 0. $ Suppose $g = \sum_{x\in F} a_{x}e_{x}$, where $F$ is a finite set of $X_{P}$ containing $P$, and $0 = \eta(g) = \sum_{x\in F}a_{x}f_{x} - \sum_{x\in F}a_{x}f_{\sigma(x)}$. Then, $\sum_{x\in F\setminus \sigma(F)} a_{x}f_{x} = 0$, so we may conclude that 
$F' = \{x\in F: a_{x}\neq 0\}$ satisfies $\sigma(F') = F'$. $P$ is the only non-empty cycle contained in $X_{P}$, so either $F' = \emptyset$ $(g = 0)$ or $F' = P$. Let's assume the latter, and write $P = \{x_{1},...,x_{n}\}$, where $\sigma(x_{i}) = x_{i+1}$ mod $n$, for all $i\leq n$. We have that
$0 = \eta(g) =   \sum_{x\in P}a_{x}f_{x} - \sum_{x\in P}a_{x}f_{\sigma(x)} = \sum_{i=1}^{n} (a_{i} - a_{i-1})f_{x_{i}}$, so $a_{i} - a_{i-1} = 0$ mod $n$. Hence, $g = a\cdot e_{P}$ for $a = a_{1}$. Therefore, $\text{ker}(\eta_{P}) = \mathbb{Z}[e_{P}]$.
\par
Now, it remains to show for any $x$ in $P$ and $n$ in $\mathbb{Z}\setminus\{0\}$, $nf_{x}$ is not in $\text{im}(\eta_{P})$, and $\mathbb{Z}[f_{x}] +\text{im}(\eta_{P}) = F_{P}$. Let $\varphi:F_{P}\mapsto\mathbb{Z}$ be the homomorphism satisfying $\varphi(f_{y}) = 1$ for all $y$ in $X_{P}$. Since $\varphi\circ \eta_{P}(e_{y}) = 0$ for all $y$ in $X_{P}$, we have
$\varphi(\text{im}(\eta_{P})) = 0$. Since $0\neq n = \varphi(nf_{x})$, we have that $nf_{x}$ is not in $\text{im}(\eta_{P})$. For any $y$ in $X_{P}$, $0 = \eta(e_{y}) = f_{y} - f_{\sigma(y)}$ mod $\text{im}(\eta_{P})$, so $f_{y} = f_{\sigma^{k}(y)}$ mod $\text{im}(\eta_{P})$ for all $k$ in $\mathbb{N}$. By definition of $X_{P}$, there is a $k$ in $\mathbb{N}$ such that $\sigma^{k}(y) = x$. Hence,
$\mathbb{Z}[f_{x}] + \text{im}(\eta_{P}) = F_{P}$.
\end{proof}
We will now determine the kernel and co-kernel of $\iota - \hat{\otimes}_{1}[\mathcal{E}_{R,F_{R}}]$ acting on $K^{-1}$. For this, we will need to know more about the structure of Fatou components.
\begin{thm}
\label{fatcomclass}
Let $P$ be a cycle of Fatou components of length $n$ for a rational function $R$. Then, either $P$ is an 
\begin{enumerate}
    \item \textit{attracting} cycle: for every $x$ in $P$, $\{R^{\circ nk}:U_{x}\mapsto U_{x}\}_{k\in\mathbb{N}}$ converges on compact sets to an attracting fixed point in $U_{x}$,
    \item \textit{parabolic} cycle: for every $x$ in $P$, $\{R^{\circ nk}:U_{x}\mapsto U_{x}\}_{n\in\mathbb{N}}$ converges on compact sets to a parabolic (see \cite[Section~10]{Milnor:Dynamics_in_one_complex_variable}) fixed point in $\partial U_{x}$,
    \item \textit{Siegel} cycle: for every $x$ in $P$, $R^{\circ n}:U_{x}\mapsto U_{x}$ is conformally conjugate to an irrational rotation on $\mathbb{D}$, or
    \item \textit{Herman} cycle: for every $x$ in $P$, $R^{\circ n}:U_{x}\mapsto U_{x}$ is conformally conjugate to an irrational rotation on $\mathbb{A}_{r} = \{z\in\mathbb{C}: 1< |z|< r\}$ for some $r>1$.
\end{enumerate}
\end{thm}
\begin{proof}
See \cite[Theorem~16.1]{Milnor:Dynamics_in_one_complex_variable} for a proof of this Theorem.
\end{proof}
All we need to know from the above Theorem is the following Corollary:
\begin{cor}
\label{fatattractor}
If $P$ is an attracting, parabolic, or Siegel cycle for a rational function $R$, then there is an open set $A_{P}$ contained in a finite union $B$ of pair-wise disjoint simply connected open subsets of $\bigcup_{x\in P}U_{x}:= U_{P}$ such that $R(A_{P})\subseteq A_{P}$ and for every compact set $K\subseteq\bigcup_{x\in X_{P}}U_{x}$, there is a $k$ in $\mathbb{N}$ such that $R^{k}(K)\subseteq A_{P}$. Moreover, $B$ can be chosen so that $R^{l}(B)\subseteq A_{P}$ for some $l$ in $\mathbb{N}$
\end{cor}
\begin{proof}
Let $n$ denote the length of $P$. Write $P = \{x_{1},...,x_{n}\}$, where $\sigma(x_{i}) = x_{i+1}$ mod $n$. First, suppose $P$ is attracting, and let $\{z_{1},...,z_{n}\}$ be its attracting periodic orbit, where $z_{i}$ is in $U_{x_{i}}$ for $i\leq n$. By Koenig's linearization Theorem (\cite[Theorem~8.2]{Milnor:Dynamics_in_one_complex_variable}) for the geometrically attracting case and Böttcher's Theorem for the super attracting case (\cite[Theorem~9.1]{Milnor:Dynamics_in_one_complex_variable}), for any $i\leq n$, there is a simply connected, open, and pre-compact neighbourhood $V_{i}$ of $z_{i}$ contained in $U_{x_{i}}$ such that $\overline{V_{i}}\subseteq U_{x_{i}}$ and $R^{\circ n}(V_{i})\subseteq V_{i}$. 
\par
Let $A_{P} = \bigsqcup_{i=1}^{n} V_{i}\cap  \bigcap_{j=1}^{n-1}R^{n-j}(V_{i+j})$. Then, $A_{P}$ is contained in the disjoint union of simply connected open sets $\bigsqcup_{i=1}^{n}V_{i} = B$ and $R(A_{P})\subseteq A_{P}$. Since $A_{P}$ is a neighbourhood of the attracting periodic orbit, for every compact set $K\subseteq U_{P}$ there is $k$ in $\mathbb{N}$ such that $R^{\circ k}(K)\subseteq A_{P}$. In particular, this is true for $K = \overline{B}$. Since any compact set of $\bigcup_{x\in X_{P}}U_{x}$ is eventually mapped into $U_{P}$, the result follows in this case.
\par
If $P$ is parabolic, then, by the Parabolic flower Theorem (\cite[Theorem~10.7]{Milnor:Dynamics_in_one_complex_variable}), for every $x$ in $P$ there is an attracting \textit{petal} (see \cite[Definition~10.6]{Milnor:Dynamics_in_one_complex_variable}) $\mathcal{P}_{x}\subseteq U_{x}$ such that for every compact set $K\subseteq U_{x}$, there is $k$ in $\mathbb{N}$ such that $R^{\circ nk}(K)\subseteq \mathcal{P}_{x}$. By definition, $\mathcal{P}_{x_{1}}$ is a simply connected open set which is mapped homeomorphically into itself by $R^{\circ n}$. Therefore, $A_{P} = \bigcup_{i=0}^{n-1}R^{\circ i}(\mathcal{P}_{x_{1}}) = B$ satisfies the conclusion of the Corollary.
\par
If $P$ is a Siegel cycle, then we just let $B = A_{P} = U_{P}$.
\end{proof}
We now explain how we orient the components of a Herman cycle. This will be crucial later on to 
calculate the kernel and co-kernel of the connecting maps between the groups associated to $R:F_{R}\mapsto F_{R}$ and that of $R:\hat{\mathbb{C}}\mapsto \hat{\mathbb{C}}$, which we will need to know if we are to calculate the groups associated to $R:J_{R}\mapsto J_{R}$ from our exact sequences relating all three. For this reason, we must describe part of the kernel and co-kernel of $\iota - \hat{\otimes}_{1}[\mathcal{E}_{R,F_{R}}]:K^{-1}(F_{R}\setminus C_{R,F_{R}})\mapsto K^{-1}(F_{R})$ in terms of this orientation, which we now define.
\par
Let $R$ be a rational function, and suppose $Q$ is a Herman cycle for $R$. Choose a component $x_{Q}$ in $Q$. The boundary $\partial U_{x_{Q}}$ has two connected components (\cite[Lemma~15.7]{Milnor:Dynamics_in_one_complex_variable}); choose one to denote $\partial^{+}U_{x_{Q}}$ and call it the \textit{interior boundary}. Denote the other component $\partial^{-}U_{x_{Q}}$ and call it the \textit{exterior boundary}.
\par
From the 6-term exact sequence of $K$-theory associated to
$$
\begin{tikzcd}
0 \arrow[r] & C_{0}(U_{x_{Q}}) \arrow[r] & C(\overline{U_{x_{Q}}}) \arrow[r] & C(\partial^{+}U_{x_{Q}}\bigsqcup\partial^{-}U_{x_{Q}}) \arrow[r] & 0,
\end{tikzcd}$$
there is a group homomorphism $\text{exp}:K^{0}(\partial^{+}U_{x_{Q}})\oplus K^{0}(\partial^{-}U_{x_{Q}})\mapsto K^{-1}(U_{x_{Q}})$. Since $\overline{U_{x_{Q}}}$, $\partial^{+}U_{x_{Q}}$, and $\partial^{-}U_{x_{Q}}$ are connected, compact and proper subsets of $\hat{\mathbb{C}}$, by Proposition \ref{Triso}, their $K^{0}$ groups are isomorphic to $\mathbb{Z}$ and are generated by $[1_{\overline{U_{x_{Q}}}}], [1_{\partial^{+}U_{x_{Q}}}], [1_{\partial^{-}U_{x_{Q}}}]$, respectively.
Hence, the kernel of $\text{exp}$ is generated by $[1_{\partial^{+}U_{x_{Q}}}] + [1_{\partial^{-}U_{x_{Q}}}]$, the image is generated by $\text{exp}([1_{\partial^{+}U_{x_{Q}}}])$,
and $\text{exp}([1_{\partial^{+}U_{x_{Q}}}]) = -\text{exp}([1_{\partial^{-}U_{x_{Q}}}])$. $U_{x_{Q}}$ is homeomorphic to an open annulus, so $K^{-1}(U_{x_{Q}}))$ is isomorphic to $\mathbb{Z}$, and the image of $\text{exp}$ is isomorphic to $n\cdot\mathbb{Z}$ for some $n>0$.
From the 6-term exact sequence associated to 

$$
\begin{tikzcd}
0 \arrow[r] & C_{0}(\hat{\mathbb{C}}\setminus\overline{U_{x_{Q}}}) \arrow[r] & C(\hat{\mathbb{C}}) \arrow[r] & C(\overline{U}_{x_{Q}}) \arrow[r] & 0,
\end{tikzcd}$$
it is easy to see $K^{-1}(\overline{U_{x_{Q}}})$ contains no isotropy. Therefore, $\text{exp}$ must be surjective, and so $u_{x_{Q}}:= \text{exp}([1_{\partial^{+}U_{x_{Q}}}])$ is a choice of generator for $K^{-1}(U_{x_{Q}})$.
\par
We now orient the boundaries of the rest of the cycle elements as follows. Suppose $Q$ is length $n$, and $0\leq k\leq n-1$. Let $\partial^{+}U_{\sigma^{k}(x_{Q})} = R^{k}(\partial^{+}U_{x_{Q}})$, $\partial^{-}U_{\sigma^{k}(x_{Q})} = R^{k}(\partial^{-}U_{\sigma^{k}(x_{Q}}),$ and $u_{\sigma^{k}(x_{Q})} = \text{exp}([1_{\partial^{+}U_{\sigma^{k}(x_{Q})}}])$.
\par
$R^{\circ n}:U_{x_{Q}}\mapsto U_{x_{Q}}$ is conjugate to an irrational rotation, and is therefore homotopic to the identity. Hence, the induced map $(R^{\circ n*})_{*}$ on $K^{-1}(U_{x_{Q}})$ is equal to $\text{id}$. Therefore, by naturality of $\text{exp}$, we have a commutative diagram

$$
\begin{tikzcd}
{\mathbb{Z}[1_{\partial^{+}U_{x_{Q}}}]} \arrow[d, "\text{exp}"] & {\mathbb{Z}[1_{R^{\circ n}(\partial^{+}U_{x_{Q}})}]} \arrow[l, "(R^{\circ n*})_{*}"'] \arrow[d, "\text{exp}"] \\
K^{-1}(U_{x_{Q}})                                                & K^{-1}(U_{x_{Q}}). \arrow[l, "\text{id}"]                                                                      
\end{tikzcd}$$
Hence, $R^{\circ n}(\partial^{+}U_{x_{Q}}) = \partial^{+}U_{x_{Q}}.$
\par
Similarly, for any $x$ in $Q$, the commutative diagram

$$
\begin{tikzcd}
0 \arrow[r] & C_{0}(U_{x}) \arrow[r]                            & C_{0}(\overline{U_{x}}) \arrow[r]                            & C(\partial^{+}U_{x}\bigsqcup\partial^{-}U_{x}) \arrow[r]                                    & 0 \\
0 \arrow[r] & C_{0}(U_{\sigma(x)}) \arrow[u, "R^{*}"] \arrow[r] & C_{0}(\overline{U_{\sigma(x)}}) \arrow[u, "R^{*}"] \arrow[r] & C(\partial^{+}U_{\sigma(x)}\bigsqcup\partial^{-}U_{\sigma(x)}) \arrow[u, "R^{*}"] \arrow[r] & 0
\end{tikzcd}$$

implies, by naturality of $\text{exp}$, a commutative diagram

$$
\begin{tikzcd}
{\mathbb{Z}[1_{\partial^{+}U_{x}}]} \arrow[d, "\text{exp}"] & {\mathbb{Z}[1_{\partial^{+}U_{\sigma(x)}}]} \arrow[l, "(R^{*})_{*}"'] \arrow[d, "\text{exp}"] \\
K^{-1}(U_{x})                                                & K^{-1}(U_{\sigma(x)}), \arrow[l, "(R^{*})_{*}"]                                                
\end{tikzcd}$$
and so $((R^{\circ -1})^{*})_{*}u_{x} = u_{\sigma(x)}$.
\par
We will call the choice of generators $\{u_{x}\}_{x\in Q}$ for $K^{-1}(U_{Q}) = \bigoplus_{x\in Q}K^{-1}(U_{x})$ (or equivalently a choice of boundary components) an \textit{orientation for $Q$}, and denote $u_{Q} = \sum_{x\in Q}u_{x}$. Note that there are only two possible choices of such an orientation for $Q$. We will call $Q$ equipped with a choice of orientation an \textit{oriented Herman cycle}.
\par
We will also need to understand the relationship between $K^{-1}(F_{R}\setminus C_{R,F_{R}})$ and $K^{-1}(F_{R})$, which is the domain and co-domain of $\hat{\otimes}_{1}[\mathcal{E}_{R,F_{R}}] - \iota$, respectively. The following lemma is all we need.
\begin{lemma}
\label{k1open}
Let $U$ be a proper open set of $\hat{\mathbb{C}}$ and $D\subseteq U$ a finite set. Then, there are open sets $V$ and $W$ contained in $U$ such that
\begin{enumerate}
    \item $V\cap D = \emptyset$, and $V$ contains any connected component of $U$ not intersecting $D$,
    \item $W$ is a disjoint union of simply connected open sets which contain $D$,
    \item $i_{*}:K^{-1}(V)\mapsto K^{-1}(U)$ is an isomorphism, and
    \item $i_{*} + i_{*}:K^{-1}(W\setminus \tilde{D})\oplus K^{-1}(V)\mapsto K^{-1}(U\setminus \tilde{D})$ is an isomorphism, for any $\tilde{D}\subseteq D$.
\end{enumerate}
Denote the image of $i_{*}:K^{-1}(W\setminus D)\mapsto K^{-1}(U\setminus D)$ to be $G(U,D)$.
Assuming $|\hat{\mathbb{C}}\setminus U|\geq 2$, $G(U,D)$ has the property that if $\tilde{W}$ is any open set in $U$ that contains $D$ and is the disjoint union of simply connected open sets, then $i_{*}:K^{-1}(\tilde{W}\setminus D)\mapsto K^{-1}(U\setminus D)$ maps isomorphically onto $G(U,D)$.
\end{lemma}
\begin{proof}
Let $U_{1},...,U_{k}$ be the maximal connected components of $U$ that contain $D$. For each $i\leq k$, let $L_{i}$ be the image of a smooth non-self-intersecting curve $\gamma_{i}:[0,1]\mapsto \hat{\mathbb{C}}$ which passes through $U_{i}\cap D$ and $L_{i}\cap U_{i} = \gamma_{i}([0,1))$. We may also assume $\gamma_{i}(0)$ is in $D$, for all $i\leq k$. Let $W_{i}$ be a simply connected open neighbourhood of $\gamma_{i}([0,1))$ in $U_{i}$ (which can be found, for instance, by applying the tubular neighbourhood Theorem to the embedding $\gamma_{i}((0,1))\subseteq U_{i})$. Then,
$W'_{i} = W_{i}\setminus \gamma_{i}([0,1))$ is also simply connected ($\hat{\mathbb{C}}\setminus W'_{i}$ and $W'_{i}$ are both connected).
\par
Let $V = U\setminus (\bigcup_{i=1}^{k}L_{i})$, $W = \bigcup_{i=1}^{k}W_{i}$, and $W' = \bigcup_{i=1}^{k}W'_{i}$. Then, $V\cap D = \emptyset$, $V$ contains all connected components of $U$ not intersecting $D$, and $D\subseteq W$.
\par
Since $U\setminus V$ is homeomomorphic to the disjoint union of $k$ half-open intervals, $C_{0}(U\setminus V)$ is contractible to $0$, so by the 6-term exact sequence of $K$-theory associated to the extension 
$$
\begin{tikzcd}
0 \arrow[r] & C_{0}(V) \arrow[r] & C_{0}(U) \arrow[r] & C_{0}(U\setminus V)) \arrow[r] & 0,
\end{tikzcd}$$
we have that $i_{*}:K^{-1}(V)\mapsto K^{-1}(U)$ is an isomorphism.
\par
Let $\tilde{D}$ be a subset of $D$. Since $V\cup (W\setminus \tilde{D}) = U\setminus \tilde{D}$ and $V\cap (W\setminus \tilde{D}) = W'$ , by \cite[Theorem~4.19]{Karoubi:K-theory}, we have the following exact sequence:

$$
\begin{tikzcd}
K^{0}(W') \arrow[r, "i_{*}\oplus-i_{*}"] & K^{0}(W\setminus \tilde{D})\oplus K^{0}(V) \arrow[r, "i_{*}+i_{*}"]  & K^{0}(U\setminus \tilde{D})                        \\
K^{-1}(U\setminus \tilde{D}) \arrow[u]                       & K^{-1}(W\setminus \tilde{D})\oplus K^{-1}(V) \arrow[l, "i_{*}+i_{*}"] & K^{-1}(W'). \arrow[l, "i_{*}\oplus-i_{*}"]
\end{tikzcd}$$
By Corollary \ref{mainbott}, $K^{0}(W')$, $K^{0}(W\setminus \tilde{D})$ are free abelian groups generated by $\{\beta_{W'_{i}}\}_{i=1}^{k}$, $\{\beta_{W_{i}\setminus \tilde{D}}\}_{i=1}^{k}$, and $i_{*}(\beta_{W'_{i}}) = \beta_{W_{i}\setminus \tilde{D}}$ for all $i\leq k$. Therefore, the left-most horizontal map in the above diagram is injective. Exactness then implies that $i_{*} + i_{*}:K^{-1}(W\setminus \tilde{D})\oplus K^{-1}(V)\mapsto K^{-1}(U\setminus \tilde{D})$ is surjective. Since $W'$ is the disjoint union of simply connected open sets, $K^{-1}(W') = 0$. Exactness of the diagram then implies $i_{*} + i_{*}:K^{-1}(W\setminus \tilde{D})\oplus K^{-1}(V)\mapsto K^{-1}(U\setminus \tilde{D})$ is injective.
\par
We now show $G(U,D)$ has the stated property.
\par
If $W_{1}$ and $W_{2}$ are open proper subsets of $\hat{\mathbb{C}}$ that are the disjoint union of simply connected sets and $D$ is a finite set such that $D\subseteq W_{1}\subseteq W_{2}$, then $K^{-1}(W_{j}) = 0$, $j=0,1$, and by Corollary \ref{mainbott}, $i_{*}:K^{0}(W_{j}\setminus D)\mapsto K^{0}(W_{j})$ is an isomorphism, for $j=0,1$. Therefore, by naturality of the exponential maps associated to the short exact sequences

$$
\begin{tikzcd}
0 \arrow[r] & C_{0}(W_{j}\setminus D) \arrow[r] & C_{0}(W_{j}) \arrow[r] & C(D) \arrow[r] & 0,
\end{tikzcd}$$
we have a commutative diagram

$$
\begin{tikzcd}
K^{0}(D) \arrow[r, equal] \arrow[d, "\text{exp}"] & K^{0}(D) \arrow[d, "\text{exp}"] \\
K^{-1}(W_{1}\setminus D) \arrow[r,"i_{*}"]                               & K^{-1}(W_{2}\setminus D)         
\end{tikzcd}$$
 with vertical arrows being isomorphisms. Hence,
$i_{*}:K^{-1}(W_{1}\setminus D))\mapsto K^{-1}(W_{2}\setminus D)$ is an isomorphism.
\par
Now, suppose $\tilde{W}$ is the disjoint union of simply connected open sets contained in $U$ and containing $D$. Since $|\hat{\mathbb{C}}\setminus U|\geq 2$, we may regard $W$ and $\tilde{W}$ as the disjoint unions of simply connected, proper open sets of the complex plane $\mathbb{C}$. Therefore, $W\cap\tilde{W}$ is also the disjoint union of simply connected open sets contained in $U$ and containing $D$ ($\hat{\mathbb{C}}\setminus W$ and $\hat{\mathbb{C}}\setminus\tilde{W}$ are connected, contain the common point $\infty$, and hence $\hat{\mathbb{C}}\setminus W\cap\tilde{W} = \hat{\mathbb{C}}\setminus W\cup \hat{\mathbb{C}}\setminus\tilde{W}$ is connected). We have a commututative diagram
$$
\begin{tikzcd}
K^{-1}(\tilde{W}\cap W\setminus D)) \arrow[r] \arrow[d] & K^{-1}(W\setminus D) \arrow[d] \\
K^{-1}(\tilde{W}\setminus D) \arrow[r]                  & K^{-1}(U\setminus D).          
\end{tikzcd}$$
The top horizontal and left-most vertical map have been shown to be isomorphisms, so commutativity implies $i_{*}:K^{-1}(\tilde{W}\setminus D)\mapsto K^{-1}(U\setminus D)$ maps injectively onto $G(U,D)$.
\end{proof}
We can now describe the kernel and co-kernel of $\iota - \hat{\otimes}_{1}[\mathcal{E}_{R,F_{R}}]:K^{-1}(F_{R}\setminus C_{R,F_{R}})\mapsto K^{-1}(F_{R})$.
\begin{prop}
\label{Fkernel1}
Let $R$ be a rational function, denote $\mathcal{H}_{R}$ to be the set of Herman cycles for $R$, and fix an orientation for every $P$ in $\mathcal{H}_{R}$. The mapping $\gamma =\iota - \hat{\otimes}_{1}[\mathcal{E}_{R,F_{R}}]:K^{-1}(F_{R}\setminus C_{R,F_{R}})\mapsto K^{-1}(F_{R})$ has kernel generated by the elements $u_{P}:= \sum_{x\in P}u_{x}$, where $P$ is a Herman cycle, and the subgroup $G(F_{R},C_{R,F_{R}})$ of $K^{-1}(F_{R}\setminus C_{R,F_{R}})$.
\par
For each $P$ in $\mathcal{H}_{P}$, choose an $x_{P}$ in $P$. The subgroup of $K^{-1}(F_{R})$ generated by the elements $\{u_{x_{P}}\}_{P\in\mathcal{H}_{P}}$ maps isomorphically onto the co-kernel of $\gamma$ via the quotient map.
\end{prop}
\begin{proof}
For each cycle $P$ in $\mathcal{F}_{R}$, We will denote $U_{P} = \bigcup_{x\in P}U_{x}$, $X_{P} = \bigcup_{n\in\mathbb{N}_{0}}\sigma^{-n}(P)$, and $F(P) = \bigcup_{x\in X_{P}}U_{x}$. Then, $F(P)\cap F(P')=\emptyset$ for distinct cycles $P$, $P'$, $F_{R} = \bigcup_{P\in\mathcal{F}_{R}}F(P)$, and $R^{-1}(F(P)) = F(P)$ for any cycle $P$. In particular,
we can identify $K^{-1}(F_{R})$, $K^{-1}(F_{R}\setminus C_{R,F_{R}})$ with $\bigoplus_{P\in\mathcal{F}_{R}}K^{-1}(F(P))$, $\bigoplus_{P\in\mathcal{F}_{R}}K^{-1}(F(P)\setminus C_{R,F(P)})$, respectively, via the inclusion maps, and (by Proposition \ref{classnaturality}) $(\iota_{F(P)})^{*}[\mathcal{E}_{R,F_{R}}] = (i_{F(P)})_{*}[\mathcal{E}_{R,F(P)}]$
for any cycle $P$ in $\mathcal{F}_{R}$. Therefore, to prove Proposition \ref{Fkernel1}, it suffices to show the mapping
$\gamma_{P} = \iota - \hat{\otimes}_{1}[\mathcal{E}_{R,F(P)}]:K^{-1}(F(P)\setminus C_{R,F(P)})\mapsto K^{-1}(F(P))$ has kernel generated by $G(F(P),C_{R,F(P)})$, along with $u_{P}$ if $P$ is a Herman cycle, and co-kernel generated by $u_{x_{P}}$ if $P$ is a Herman cycle, with co-kernel 0 otherwise.
\par
Let $A_{P}$ and $B$ be as in Proposition \ref{fatattractor}, and denote $B = A_{P} = U_{P}$ when $P$ is a Herman cycle. Let $\{V'_{n}\}_{n\in\mathbb{N}}$ be a family of pre-compact open sets in $F(P)$ such that $\overline{V'_{n}}\subseteq F(P)$, $V'_{n}\subseteq V'_{n+1}$ for all $n$ in $\mathbb{N}$, and $\bigcup V'_{n} = F(P).$ By Proposition \ref{fatattractor}, for every $n$ in $\mathbb{N}$ there is a $k_{n}$ in $\mathbb{N}$ such that $R^{k_{n}}(V'_{n}\cup B)\subseteq A_{P}$. Define $V_{n} = (\bigcup_{i=0}^{k_{n}}R^{i}(V'_{n}\cup B))\cup A_{P}$ for all $n$ in $\mathbb{N}$. Then, $\bigcup_{n\in\mathbb{N}} V_{n} = F(P)$ and, for all $n$, $V_{n}$ has the properties 
\begin{enumerate}
    \item $A_{P}\subseteq B\subseteq V_{n}\subseteq V_{n+1}$,
    \item $R(V_{n})\subseteq V_{n}$, and
    \item $R^{k_{n}}(V_{n})\subseteq A_{P}$.
\end{enumerate}
By $(1)$ above, $K^{-1}(F(P)\setminus C_{R, F(P)})$ and $K^{-1}(F(P))$ are the inductive limits of the (maps induced by) the inclusions $(\iota_{n})^{*}:K^{-1}(V_{n}\setminus C_{R,F(P)})\mapsto K^{-1}(V_{n+1}\setminus C_{R,F(P)})$, and $(i_{n})_{*}:K^{-1}(V_{n})\mapsto K^{-1}(V_{n+1})$, respectively. Denote $R_{n} = R:V_{n}\mapsto V_{n}$.
By Proposition \ref{classnaturality}, we have that $(\iota_{n})^{*}[\mathcal{E}_{R_{n+1},V_{n+1}}] = (i_{n})_{*}[\mathcal{E}_{R_{n}, V_{n}}]$ for all $n$ in $\mathbb{N}$, and
$(\nu_{n})^{*}[\mathcal{E}_{R,F(P)}] = (\mu_{n})_{*}[\mathcal{E}_{R_{n}, V_{n}}]$, where $\nu_{n} = i:C_{0}(V_{n}\setminus C_{R,F(P)})\mapsto C_{0}(F(P)\setminus C_{R, F(P)})$ and $\mu_{n} =i: C_{0}(V_{n})\mapsto C_{0}(F(P))$, for all $n$ in $\mathbb{N}$. Therefore, $\gamma_{P}$ is the inductive limit of the maps $\gamma_{n} = \iota - \hat{\otimes}_{1}[\mathcal{E}_{R_{n},V_{n}}]:K^{-1}(V_{n}\setminus C_{R,F(P)})\mapsto K^{-1}(V_{n})$, so it suffices to show, for all $n$ in $\mathbb{N}$, that $\gamma_{n}$ has kernel generated by $G(V_{n},C_{R,F(P)}\cap V_{n})$, along with $u_{P}$ if $P$ is a Herman cycle, and co-kernel generated by $u_{x_{P}}$ if $P$ is a Herman cycle, with co-kernel 0 otherwise.
\par
First, we prove the following lemma.
\begin{lemma}
\label{kernellemma}
for any open set $A$ of $F_{R}$, and finite set $D\subseteq A$, $G(A, D\cup (A\cap C_{R,F_{R}}))$ is in the kernel of $\hat{\otimes}_{1}[\mathcal{E}_{R, A\setminus D}]:K^{-1}(A\setminus (D\cup C_{R,F_{R}}))\mapsto K^{-1}(F_{R})$.
\end{lemma}
\begin{proof}
Denote $D' = D\cup (A\cap C_{R,F_{R}})$ Since $R$ is a rational function, for every $c$ in $D'$, we can find local co-ordinates about $\phi$, $\psi$ about $c$, $R(c)$, respectively such that $\phi(c)= 0$, $\psi(R(c)) = 0$, and $\psi(R(\phi^{-1}(z))) = z^{m}$, for all $z$ in a neighbourhood of $0$, for some $m$ in $\mathbb{N}$. Therefore, for every $c$ in $D'$, there is a simply connected open set $W_{c}$ containing $c$ and contained in $A$ such that $W_{c}\cap W_{c'} = \emptyset$ for all distinct $c$, $c'$ in $D'$ and $R(W_{c})$ is simply connected.
\par
By Lemma \ref{k1open},
$\sum (j_{c})_{*}:\bigoplus_{c\in D'}K^{-1}(W_{c}\setminus c)\mapsto G(A, D')$ is an isomorphism, where $j_{c} = i:C_{0}(W_{c}\setminus c)\mapsto C_{0}(A\setminus D')$. By Proposition \ref{classnaturality}, we have $(j_{c})^{*}[\mathcal{E}_{R,A\setminus D}] = (i_{R(W_{c})})_{*}[\mathcal{E}_{R, W_{c}}]$. Since $K^{-1}(R(W_{c})) = 0,$ it follows that $\hat{\otimes}(i_{R(W_{c})})_{*}[\mathcal{E}_{R, W_{c}}]= 0$ for all $c$ in $D$.
This proves the lemma.
\end{proof}
Denote $C_{n} = V_{n}\cap C_{R, F(P)}$. As a special case of Lemma \ref{kernellemma}, we have that $G(V_{n}, C_{n})$ is in the kernel of $\hat{\otimes}_{1}[\mathcal{E}_{R_{n}, V_{n}}]$. The diagram

$$
\begin{tikzcd}
{G(V_{n}, C_{n})} \arrow[r]        & K^{-1}(V_{n})      \\
K^{-1}(W\setminus C_{n}) \arrow[u] \arrow[r] & K^{-1}(W) \arrow[u]
\end{tikzcd}$$
commutes, where $W$ is any disjoint union of simply connected open sets containing $C_{n}$ and contained in $V_{n}$, and the maps are induced by inclusion. The left vertical map is an isomorphism by Lemma \ref{k1open}, and the bottom right group is zero. Hence $G(V_{n},C_{n})$ is also in the kernel of $\iota$, so that $G(V_{n},C_{n})\subseteq\text{ker}(\gamma_{n}).$
\par
We now determine the rest of the kernel, but first we set some notation. For $0\leq i\leq k_{n}$, let $V^{i} = V_{n}\setminus \bigcup_{j=0}^{k_{n}-i}R^{-j}(C_{n})$. Then, $R(V_{i})\subseteq V_{i+1}$, for all $0\leq i\leq k_{n}-1$. let $\tilde{R}_{i} = R:V^{i}\mapsto V^{i+1}$. 
\par
Denote by $V$ the (same labelled) open set from Lemma \ref{k1open} applied in the case that  $U = V_{n}$, $D = \bigcup_{i=0}^{k_{n}}F^{-1}(C_{n})$, and, for every $0\leq i\leq k_{n}$, let $\varepsilon_{i}:K^{-1}(V)\mapsto K^{-1}(V^{i})$ be the (map induced by) inclusion. Then, by Lemma \ref{k1open}, $\varepsilon_{i}+j_{*}:K^{-1}(V)\oplus G(V_{n}, \bigcup_{j=0}^{k_{n}-i}F^{-j}(C_{n}))\mapsto K^{-1}(V^{i})$ is an isomorphism, for every $0\leq i\leq k_{n}$.
\par
For every $0\leq i\leq k_{n}$, we shall denote $\Psi = [\mathcal{E}_{\tilde{R}_{i},V^{i}}]$ and
$\iota = j_{*}:K^{-1}(V^{i})\mapsto K^{-1}(V_{n})$. It will always be clear the domain of these maps, so no confusion from this notation ambiguity will arise. Moreover, for $i,m$ in $\mathbb{N}$ such that $i+m\leq k_{n}$, we will denote $\Psi^{m} = $\\ $ [\mathcal{E}_{\tilde{R}_{i},V^{i}}]\hat{\otimes}[\mathcal{E}_{\tilde{R}_{i+1},V^{i+1}}]...\hat{\otimes}[\mathcal{E}_{\tilde{R}_{i+m-1},V^{i+m-1}}].$
\par
Now, by Proposition \ref{classnaturality}, we have, for every $0\leq i\leq k_{n}$ and $u$ in $K^{-1}(V)$, that $\varepsilon_{k_{n}}(u)\hat{\otimes}_{1}[\mathcal{E}_{R_{n},V_{n}}] = \iota(\varepsilon_{i}(u)\hat{\otimes}_{1}\Psi)$. So, if $\varepsilon_{k_{n}}(u)$ is in the kernel of $\gamma_{n}$, then $\iota(\varepsilon_{i}(u) - \varepsilon_{i-1}(u)\hat{\otimes}\Psi)$\\$ = 0$ for all $1\leq i\leq k_{n}$. It is easy to see from Lemma \ref{k1open} that, for every $0\leq i\leq k_{n}$, the kernel of $\iota:K^{-1}(V^{i})\mapsto K^{-1}(V_{n})$ is precisely $G(V_{n},\bigcup_{j=0}^{k_{n}-i}F^{-j}(C_{n})).$ So, we may conclude that $\varepsilon_{i}(u) - \varepsilon_{i-1}(u)\hat{\otimes}\Psi$ is in $G(V_{n},\bigcup_{j=0}^{k_{n}-i}F^{-j}(C_{n}))$.
\par
We now prove that $\varepsilon_{i}(u) - \varepsilon_{0}(u)\hat{\otimes}\Psi^{i}$ is in $G(V_{n},\bigcup_{j=0}^{k_{n}-i}F^{-j}(C_{n}))$ for all $1\leq i\leq k_{n}$. From directly above, we know this is true for $i = 1$. Suppose we know it to be true for $i\leq k_{n}-1$. By Lemma \ref{kernellemma}, we then have that
$0 = (\varepsilon_{i}(u) - \varepsilon_{0}(u)\hat{\otimes}\Psi^{i})\hat{\otimes}\Psi = \varepsilon_{i}(u)\hat{\otimes}\Psi - \varepsilon_{0}(u)\hat{\otimes}\Psi^{i+1}$. From directly above, we then know that $(\varepsilon_{i+1}(u) - \varepsilon_{i}(u)\hat{\otimes}\Psi) + (\varepsilon_{i}(u)\hat{\otimes}\Psi - \varepsilon_{0}(u)\hat{\otimes}\Psi^{i+1}) = \varepsilon_{i+1}(u) - \varepsilon_{0}(u)\hat{\otimes}\Psi^{i+1}$ is in $G(V_{n},\bigcup_{j=0}^{k_{n}-i-1}F^{-j}(C_{n}))$. By induction, the result holds. In particular, $\varepsilon_{k_{n}}(u) - \varepsilon_{0}(u)\hat{\otimes}\Psi^{k_{n}}$ is in $G(V_{n}, C_{n})$. 
\par
We will now show for any $g$ in $K^{-1}(V^{k_{n}})$, there is $b$ in $K^{-1}(B\setminus C_{n})$ such that $g\hat{\otimes}\Psi^{k_{n}} = i_{*}(b)$, where $j$ is the inclusion map.
\par
Denote $X_{i} = R^{i}(V_{n})\setminus \bigcup_{j=0}^{k_{n}-i}R^{-j}(C_{n})$ for $0\leq i\leq k_{n}$, $F_{i} = R:X_{i}\mapsto X_{i+1}$ for $0\leq i\leq k_{n}-1$, and $\Delta^{i} = [\mathcal{E}_{F_{0},X_{0}}]\hat{\otimes}[\mathcal{E}_{F_{1},X_{1}}]\hat{\otimes}...\hat{\otimes}[\mathcal{E}_{F_{i-1},X_{i-1}}]$ for $1\leq i\leq k_{n}$. We claim $j_{*}\Delta^{m} = \Psi^{m}$, where $j = i:C_{0}(X_{m})\mapsto C_{0}(V^{m})$ for all $1\leq m\leq k_{n}$.
\par
We prove this by induction. Note that $X_{0} = V^{0}$. By Proposition \ref{classnaturality}, $\Psi = j_{*}\Delta$. Now suppose the claim is true for $m\leq k_{n}-1.$ Then, $\Psi^{m+1} = (j_{*}\Delta^{m})\hat{\otimes}[\mathcal{E}_{\tilde{R}_{m},V^{m}}]$. By the properties of the Kasparov product (see the second last paragraph of Section \ref{kk}),  $(j_{*}\Delta^{m})\hat{\otimes}[\mathcal{E}_{\tilde{R}_{m},V^{m}}] = \Delta^{m}\hat{\otimes}[j^{*}\mathcal{E}_{\tilde{R}_{m},V^{m}}]$. Since $R(X_{m})\subseteq X_{m+1}\subseteq V^{m+1}$, by Proposition \ref{classnaturality}, we have that $j^{*}[\mathcal{E}_{\tilde{R}_{m},V^{m}}] = j_{*}[\mathcal{E}_{F_{m}, X_{m}}]$. Therefore $\Psi^{m+1} = \Delta^{m}\hat{\otimes}j_{*}\Delta = j_{*}\Delta^{m+1}$. By induction, we have proven the claim.
\par
In particular, $\Psi^{k_{n}} = j_{*}\Delta^{k_{n}}$ for $j = i:C_{0}(X_{k_{n}})\mapsto C_{0}(V_{n}\setminus C_{n})$. Since $X_{k_{n}}\subseteq A_{P}\setminus C_{n}\subseteq B\setminus C_{n}$, we can factor $j_{*}$ as $j_{*} = i_{*}\circ \tilde{i}_{*}$, for $\tilde{i}_{*} = K^{-1}(X_{k_{n}})\mapsto K^{-1}(B\setminus C_{n})$ and $i_{*} = K^{-1}(B\setminus C_{n})\mapsto K^{-1}(V_{n}\setminus C_{n})$. Therefore, for any $g$ in $K^{-1}(V^{k_{n}}),$ $b = g\hat{\otimes}\tilde{i}_{*}\Delta^{k_{n}}$ satisfies $g\hat{\otimes}\Psi^{k_{n}} = i_{*}(b)$.
\par
So far, we have shown for every $u$ in $K^{-1}(V)$ such that $\varepsilon_{k_{n}}(u)$ is in $\text{ker}(\gamma_{n})$, there is $b$ in $K^{-1}(B\setminus C_{n})$ such that $\varepsilon_{k_{n}}(u) -i_{*}(b)$ is in $G(V_{n},C_{n})$. We must now separate the analysis into two cases, the case when $P$ is a Herman cycle, and the case when it isn't.
\par
Suppose $P$ is not a Herman cycle. Then, for $i_{*}:K^{-1}(B\setminus C_{n})\mapsto K^{-1}(V_{n}\setminus C_{n})$, we can factor $\iota\circ i_{*}$ as $\iota\circ i_{*} = (i_{1})_{*}(i_{2})_{*}$, where $i_{2} = i:C_{0}(B\setminus C_{n})\mapsto C_{0}(B)$. Since in this case $B$ is the disjoint union of simply connected sets, $K^{-1}(B) = 0$. Hence, $\iota(i_{*}(b)) = 0$, which implies $\iota(\varepsilon_{k_{n}}(u)) = 0$. By Lemma \ref{k1open}, $\iota\circ\varepsilon_{k_{n}} = i:K^{-1}(V)\mapsto K^{-1}(V_{n})$ is injective, and therefore $u = 0$. So, when $P$ is not a Herman cycle, the kernel of $\gamma_{n}$ is equal to $G(V_{n},C_{n})$.
\par
Now, suppose $P$ is a Herman cycle. Then $B\setminus C_{n} = U_{P}$. Moreover, $U_{P}$ does not intersect $D = \bigcup_{i=0}^{k_{n}}F^{-i}(C_{n})$ (otherwise an irrational rotation would contain a critical point), and so by $(1)$ of Lemma \ref{k1open}, $U_{P}\subseteq V$. Therefore, $i_{*}(b) = \varepsilon_{k_{n}}(j_{*}(b))$, where $j:C_{0}(U_{P})\mapsto C_{0}(V)$ is the inclusion. So, we can write $\varepsilon_{k_{n}}(u) - i_{*}(b) = \varepsilon_{k_{n}}(u - j_{*}(b))$. Since the intersection of the image of $\varepsilon_{k_{n}}$ and $G(V_{n},C_{n})$ is zero, it follows that $\varepsilon_{k_{n}}(u) = i_{*}(b)$. We can write $i_{*}(b) = \sum_{x\in P}a_{x}u_{x}$, for some $a_{x}$ in $\mathbb{Z}$, $x$ in $P$. Since $R_{P} = R:U_{P}\mapsto U_{P}$ is a homeomorphism, the class of $[\mathcal{E}_{R_{P},U_{P}}]$ is equal to the class of $(R_{P}^{-1})^{*}$. Hence, by the definition of $u_{x}$, for $x$ in $P$, we have $u_{x}\hat{\otimes}_{1}[\mathcal{E}_{R_{n}, V_{n}}] = u_{\sigma(x)}$ for all $x$ in $P$. Therefore,
$0 = \iota(\varepsilon_{k_{n}}(u)) - \varepsilon_{k_{n}}(u)\hat{\otimes}_{1}[\mathcal{E}_{R_{n}, V_{n}}] = \sum_{x\in P}(a_{x} - a_{\sigma(x)})u_{x}$. Hence, $a_{x} = a_{y} =: a$ for all $x, y$ in $P$. Therefore, the kernel of $\gamma_{n}$ in the Herman cycle case is equal to $G(V_{n}, C_{n}) +\mathbb{Z}\cdot u_{P}$.
\par
We will now determine the co-kernel of $\gamma_{n}$.
\par
First, note that for any $0\leq i\leq k_{n}-1$ and $g$ in $K^{-1}(V^{i})$, we have $\iota(g\otimes_{i}\Psi) = j_{*}(g)\hat{\otimes}_{1}[\mathcal{E}_{R_{n},V_{n}}]$, where $j:C_{0}(V^{i})\mapsto C_{0}(V_{n}\setminus C_{n})$ is the inclusion. Therefore, for any $0\leq i\leq k_{n}-1$ and $g$ in $K^{-1}(V^{i})$, $\iota(g - g\hat{\otimes}\Psi)$ is in the image of $\gamma_{n}$.
\par
Let $v$ be in $K^{-1}(V_{n})$. Then by Lemma \ref{k1open}, there is $u$ in $K^{-1}(V)$ such that $\iota(\varepsilon_{k_{n}}(u)) = v$. By the above note, $\iota(\varepsilon_{k_{n}-i}(u)\hat{\otimes}\Psi^{i} - \varepsilon_{k_{n}-(i+1)}(u)\Psi^{(i+1)})$ is in the image of $\gamma_{n}$ for all $0\leq i\leq k_{n}-1$. Therefore, $v$ is equal to $\iota(\varepsilon_{k_{n}}(u)) +\sum_{i=0}^{k_{n}-1}\iota(\varepsilon_{k_{n}-(i+1)}(u)\Psi^{(i+1)} - \varepsilon_{k_{n}-i}\hat{\otimes}\Psi^{i}) = \iota(\varepsilon_{0}(u)\hat{\otimes}\Psi^{k_{n}})$ modulo the image of $\gamma_{n}$.
\par
We have already shown while describing the kernel that $\iota(\varepsilon_{0}(u)\hat{\otimes}\Psi^{k_{n}}) = 0$ if $P$ is not a Herman cycle. Therefore, in this this case, the co-kernel of $\gamma_{n}$ is zero. 
\par
If $P$ is a Herman cycle, then $ \iota(\varepsilon_{0}(u)\hat{\otimes}\Psi^{k_{n}}) = \sum_{x\in P}a_{x}u_{x}$, for some $a_{x}$ in $\mathbb{Z}$, $x$ in $P$. Since $\gamma_{n}(u_{x}) = u_{x} - u_{\sigma(x)}$ for all $x$ in $P$, it follows that $ \sum_{x\in P}a_{x}u_{x}$ (and hence $v$) is equivalent to $a\cdot u_{x_{P}}$ modulo the image of $\gamma_{n}$, where $a:= \sum_{x\in P}a_{x}$.
\par
We now show that for any $a$ in $\mathbb{Z}$, $a\cdot u_{x_{P}}$ is not in the image of $\gamma_{n}$, which will complete the proof of this Proposition.
\par
First, assume $u$ is an element of $K^{-1}(V\setminus U_{P})$ such that $\iota(\varepsilon_{i}(u) - \varepsilon_{i-1}(u)\hat{\otimes}\Psi) = w$ for some $w$ in $K^{-1}(U_{P})$. Therefore, for every $0\leq i\leq k_{n}-1$, there is $h_{i}$ in $G(V_{n}, \bigcup_{j=0}^{i}F^{-j}(C_{n}))$ such that
$\varepsilon_{k_{n}-i}(u) - \varepsilon_{k_{n}-(i+1)}(u)\hat{\otimes}\Psi + h_{i} = w$. Since $K^{-1}(U_{P})$ is invariant under the action of $\Psi$, it follows that 
$\varepsilon_{k_{n}}(u) - \varepsilon_{0}(u)\hat{\otimes}\Psi^{k_{n}} + h_{0} = \sum_{i=0}^{k_{n}-1} \varepsilon_{k_{n}-i}(u)\hat{\otimes}\Psi^{i} - \varepsilon_{k_{n}-(i+1)}(u)\Psi^{(i+1)} = w'$ for some $w'$ in $K^{-1}(U_{P})$. $\varepsilon_{0}(u)\hat{\otimes}\Psi^{k_{n}}$ is also in $K^{-1}(U_{P})$, and therefore $\varepsilon_{k_{n}}(u)$ is in $(K^{-1}(U_{P}) + G(V_{n}, C_{n}))\cap\varepsilon_{k_{n}}(K^{-1}(V\setminus U_{P})) =\{0\}.$ Hence, $u = 0$ whenever $u$ is in $K^{-1}(V\setminus U_{P})$ and $\iota(\varepsilon_{i}(u) - \varepsilon_{i-1}(u)\hat{\otimes}\Psi)$ is in $K^{-1}(U_{P})$.
\par
This implies that if $\iota(g) - g\hat{\otimes}_{1}[\mathcal{E}_{R_{n}, V_{n}}] = a\cdot u_{x_{P}}$ for some $g$ in $K^{-1}(V_{n}\setminus C_{n})$ and $a$ in $\mathbb{Z}$, then $\iota(g) - g\hat{\otimes}_{1}[\mathcal{E}_{R_{n}, V_{n}}] = \sum_{x\in P}(a_{x} - a_{\sigma(x)})u_{x}$ for some $a_{x}$ in $\mathbb{Z}$, $x$ in $P$. Therefore, $a = \sum_{x\in P}(a_{x} - a_{\sigma(x)}) = 0$.
\end{proof}
We can now compute the $K$-theory of $R$ acting on $F_{R}$.
\begin{thm}
\label{FK}
Let $R$ be a rational function. Denote by $\mathcal{F}_{R}$ the set of Fatou cycles for $R$ and $\mathcal{H}_{R}$ the set of Herman cycles for $R$. Let $G(F_{R},C_{R,F_{R}})\subseteq K^{-1}(F_{R}\setminus C_{R,F_{R}})$ be the group in Lemma \ref{k1open} applied to the case $U = F_{R}$, $D = C_{R,F_{R}}$.
\par
Let $\delta_{PV}:K_{1}(\mathcal{O}_{R, F_{R}})\mapsto K^{0}(F_{R}\setminus C_{R,F_{R}})$ and $\text{exp}_{PV}:K_{0}(\mathcal{O}_{R, F_{R}})\mapsto K^{-1}(F_{R}\setminus C_{R,F_{R}})$ be
the same-labelled maps appearing in the Pimsner-Voiculescu 6-term exact sequence for $R:F_{R}\mapsto F_{R}$. Then, we have short exact sequences

$$
\begin{tikzcd}
0 \arrow[r] & {\bigoplus_{P\in\mathcal{F}_{R}}\mathbb{Z}[f_{x_{P}}]} \arrow[r, "i_{*}"]   & {K_{0}(\mathcal{O}_{R,F_{R}})} \arrow[r, "\text{exp}_{PV}"] & {G(F_{R}, C_{R,F_{R}})\oplus\bigoplus_{Q\in\mathcal{H}_{R}}\mathbb{Z}\cdot u_{Q}} \arrow[r] & 0 \\
0 \arrow[r] & \bigoplus_{Q\in\mathcal{H}_{R}}\mathbb{Z}\cdot u_{x_{P}} \arrow[r, "i_{*}"] & {K_{1}(\mathcal{O}_{R,F_{R}})} \arrow[r, "\delta_{PV}"]     & {\bigoplus_{P\in\mathcal{F}_{R}}\mathbb{Z}[e_{P}]} \arrow[r]                          & 0,
\end{tikzcd}$$

where $x_{P}$ is a choice of an element in the cycle $P$, for all $P$ in $\mathcal{F}_{R} $ (or $\mathcal{H}_{R}$ ). Hence, $K_{0}(\mathcal{O}_{R, F_{R}})\simeq\mathbb{Z}^{|\mathcal{F}_{R}| + |\mathcal{H}_{R}| + |C_{R,F_{R}}|}$ and $K_{1}(\mathcal{O}_{R,F_{R}})\simeq\mathbb{Z}^{|\mathcal{F}_{R}| + |\mathcal{H}_{R}|}$.
\end{thm}
\begin{proof}
By Corollary \ref{Fkernel} and Proposition \ref{Fkernel1}, we can fill in the Pimsner-Voiculescu 6-term exact sequence as follows:

$$\begin{tikzcd}
\bigoplus_{P\in\mathcal{F}_{R}}\mathbb{Z}[e_{P}] \arrow[r, "0"]                    & \bigoplus_{P\in\mathcal{F}_{R}}\mathbb{Z}[f_{x_{P}}] \arrow[r, "i_{*}"] & {K_{0}(\mathcal{O}_{R, F_{R}})} \arrow[d, "\text{exp}"] \\
{K_{1}(\mathcal{O}_{R,F_{R}})} \arrow[u, "\delta"] & \bigoplus_{Q\in\mathcal{H}_{R}}\mathbb{Z}\cdot u_{x_{Q}} \arrow[l, "i_{*}"]                               & G(F_{R},C_{R,F_{R}})\oplus\bigoplus_{Q\in\mathcal{H}_{R}}\mathbb{Z}\cdot u_{Q} \arrow[l, "0"]                   
\end{tikzcd}$$
Exactness of the above diagram concludes the proof that we have short exact sequences as claimed.
\par
The 2nd paragraph of Section \ref{ratmapj} implies $G(F_{R},C_{R,F_{R}})\simeq\mathbb{Z}^{|C_{R,F_{R}}|}$, and so both the short exact sequences above are split exact (this forward reference won't cause a circular argument).
\end{proof}
\section{The $K$-theory of a rational function acting on its Julia set}
\label{ratmapj}
In this section we compute the kernel and co-kernel of $\iota - \hat{\otimes}[\mathcal{E}_{R,J_{R}}]$ (in both degrees), as well as some related groups. As a Corollary, we will have determined the $K$-theory of a rational function acting on its Julia set. First, we orient some $K^{-1}$ groups.
\par
Let $W$ be a union of pairwise disjoint open simply connected proper sets of $\hat{\mathbb{C}}$, and $D\subseteq W$ a finite set. By Corollary \ref{mainbott}, $i_{*}:K^{0}(W\setminus D)\mapsto K^{0}(W)$ is an isomorphism. Since $W$ is a disjoint union of simply connected open sets, we have that $K^{-1}(W) = 0$. These two facts, imply the exponential map $\text{exp}:K^{0}(D)\mapsto K^{-1}(W\setminus D)$ from the $6$-term exact sequence of $K$-theory associated to 
$$
\begin{tikzcd}
0 \arrow[r] & C_{0}(W\setminus D) \arrow[r] & C_{0}(W) \arrow[r] & C(D) \arrow[r] & 0
\end{tikzcd}$$
is an isomorphism. For each $d$ in $D$, we will denote $\text{exp}([1_{d}]) = v_{d}$. The free basis $\{v_{d}\}_{d\in D}$ for $K^{-1}(W\setminus D)$ will be our canonical choice of generators, or ``orientation''.
\par
Similarly, the short exact sequence
$$
\begin{tikzcd}
0 \arrow[r] & C_{0}(\hat{\mathbb{C}}\setminus D) \arrow[r] & C_{0}(\hat{\mathbb{C}}) \arrow[r] & C(D) \arrow[r] & 0
\end{tikzcd}$$
gives a surjection $\text{exp}:K^{0}(D)\mapsto K^{-1}(\hat{\mathbb{C}})$. In this case Corollary \ref{mainbott} implies $\mathbb{Z}[1_{\hat{\mathbb{C}}}]$ maps isomorphically onto the co-kernel of $i_{*}:K^{0}(\hat{\mathbb{C}}\setminus D)\mapsto K^{0}(\hat{\mathbb{C}})$, via the quotient map. Therefore, the kernel of $\text{exp}$ is $i_{*}[1_{\hat{\mathbb{C}}}] = \sum_{d\in D}[1_{d}]$. Denote $\text{exp}([1_{d}]) = v_{d}$, for $d$ in $D$. $K^{-1}(\hat{\mathbb{C}}\setminus D)$ is then the group generated by $\{v_{d}\}_{d\in D}$ satisfying the relation $\sum_{d\in D}v_{d} = 0$. Therefore, $\{v_{d}\}_{d\in D\setminus \{d'\}}$ is a free basis for $K^{-1}(\hat{\mathbb{C}}\setminus D)$, for any $d'$ in $D$, but we will prefer to work with the whole generating set modulo its relation.
\par
These generators behave well with respect to inclusion in the following sense.
\begin{prop}
\label{genmappingk1}
Let $W_{1}, W_{2}$ be unions of pairwise disjoint simply connected (not necessarily proper) open sets in $\hat{\mathbb{C}}$ such that $W_{1}\subseteq W_{2}$, and suppose $D,K,C$ are finite sets such that $D\subseteq W_{1}$, $K\subseteq W_{1}\setminus D$, and $C\subseteq W_{2}\setminus W_{1}$. Then,
$i_{*}:K^{-1}(W_{1}\setminus (K\cup D))\mapsto K^{-1}(W_{2}\setminus (D\cup C))$ sends
$v_{d}$ to $v_{d}$, for all $d$ in $D$, and sends $v_{k}$ to $0$, for all $k$ in $K$.
\end{prop}
\begin{proof}
By naturality of $\text{exp}$, we have a commutative diagram

$$
\begin{tikzcd}
K^{0}(K\cup D) \arrow[r, "\varphi"] \arrow[d, "\text{exp}"] & K^{0}(D\cup C) \arrow[d, "\text{exp}"] \\
K^{-1}(W_{1}\setminus (K\cup D)) \arrow[r, "i_{*}"]          & K^{-1}(W_{2}\setminus (D\cup C)),         
\end{tikzcd}$$
where $\varphi$ sends $1_{d}$ to $1_{d}$, for all $d$ in $D$, and $1_{k}$ to $0$, for all $k$ in $K$.
\end{proof}
For $c_{1}$ and $c_{2}$ in $C_{R,J_{R}}$, we will write $c_{1}\sim c_{2}$ if $\{c_{1},c_{2}\}$ is contained in a connected subset of $J_{R}$. Clearly $\sim$ is an equivalence relation. For $c$ in $C_{R,J_{R}}$, we will denote its equivalence class by $[c]$ and the group element $\sum_{d\in C_{R,J_{R}}:d\sim c}v_{d}$ in $K^{-1}(\hat{\mathbb{C}}\setminus C_{R,\hat{\mathbb{C}}})$ by $v_{[c]}$.
The collection of distinct equivalence classes will be denoted $[C_{R,J_{R}}]$.
\begin{prop}
\label{iimage}
Let $R$ be a rational function. The image of $i_{*}:K^{-1}(F_{R}\setminus C_{R,F_{R}})\mapsto K^{-1}(\hat{\mathbb{C}}\setminus C_{R,J_{R}})$ is generated by $\{v_{c}\}_{c\in C_{R,F_{R}}}$ together with $\{v_{[c]}\}_{[c]\in [C_{R,J_{R}}]}$.
\end{prop}
\begin{proof}
By Lemma \ref{k1open} in the case that $U = F_{R}$ and $D = C_{R,F_{R}}$ there are open subsets $W$ and $V$ of $F_{R}$ such that $W$ is a disjoint union of simply connnected open sets containing $C_{R,F_{R}}$, $V$ is disjoint from $C_{R,F_{R}}$, and the mappings $i_{*}+i_{*}:K^{-1}(W\setminus C_{R,F_{R}})\oplus K^{-1}(V)\mapsto K^{-1}(F_{R}\setminus C_{R,F_{R}})$, $i_{*}:K^{-1}(V)\mapsto K^{-1}(F_{R})$ are isomorphisms. Therefore, the image of $i_{*}:K^{-1}(F_{R}\setminus C_{R,F_{R}})\mapsto K^{-1}(\hat{\mathbb{C}}\setminus C_{R,\hat{\mathbb{C}}})$ is equal to the image of $i_{*}:K^{-1}(W\setminus C_{R,F_{R}})\mapsto K^{-1}(\hat{\mathbb{C}}\setminus C_{R,\hat{\mathbb{C}}})$ plus the image of $i_{*}:K^{-1}(V)\mapsto K^{-1}(\hat{\mathbb{C}}\setminus C_{R,\hat{\mathbb{C}}})$.
\par
By Proposition \ref{genmappingk1}, the image of $i_{*}:K^{-1}(W\setminus C_{R,F_{R}})\mapsto K^{-1}(\hat{\mathbb{C}}\setminus C_{R,\hat{\mathbb{C}}})$ is generated by $\{v_{c}\}_{c\in C_{R,F_{R}}}$.
\par
We have a commutative diagram
$$ 
\begin{tikzcd}
K^{-1}(V) \arrow[r, "i_{*}"] \arrow[d, "i_{*}"]                   & {K^{-1}(\hat{\mathbb{C}}\setminus C_{R,\hat{\mathbb{C}}})} \arrow[d, "j_{*}"] \\
K^{-1}(F_{R}) \arrow[r, "i_{*}"]                                  & {K^{-1}(\hat{\mathbb{C}}\setminus C_{R,J_{R}})}                               \\
{C(J_{R},\mathbb{Z})} \arrow[u, "\text{exp}"'] \arrow[r, "r"] & {C(C_{R,J_{R}},\mathbb{Z}),} \arrow[u, "\text{exp}"']                         
\end{tikzcd}$$
where the map labelled $r$ is the restriction map, and the maps labelled $i_{*}, j_{*}$ are induced from the inclusions.
\par
Since $K^{-1}(\hat{\mathbb{C}}) = 0$, exactness implies $\text{exp}:C(J_{R},\mathbb{Z})\mapsto K^{-1}(F_{R})$ is surjective. Therefore, commutativity of the above diagram implies the image of $i_{*}:K^{-1}(F_{R})\mapsto K^{-1}(\hat{\mathbb{C}}\setminus C_{R,J_{R}})$ is equal to the image of $\text{exp}\circ r$.
\par
Functions in $C(J_{R},\mathbb{Z})$ must be constant on connected subsets of $J_{R}$, so the image of $r$ is generated by the elements $\{1_{[c]}:= \sum_{d\in [c]}1_{d}\}_{[c]\in [C_{R,J_{R}}]}$. By definition, $v_{[c]} = \text{exp}(1_{[c]})$. Therefore, the image of $i_{*}:K^{-1}(F_{R})\mapsto K^{-1}(\hat{\mathbb{C}}\setminus C_{R,J_{R}})$ is generated by $\{v_{[c]}\}_{[c]\in C_{R,J_{R}}}$.
\par
$i_{*}:K^{-1}(V)\mapsto K^{-1}(F_{R})$ is an isomorphism, so commutativity of the diagram implies the image of $j_{*}\circ i_{*}$ is generated by $\{v_{[c]}\}_{[c]\in C_{R,J_{R}}}$. By Proposition \ref{genmappingk1}, the kernel of $j_{*}$ is generated by $\{v_{c}\}_{c\in C_{R,F_{R}}}$, and $j_{*}(v_{[c]}) = v_{[c]}$ for all $[c]$ in $[C_{R,J_{R}}].$ Therefore, the image of $i_{*}:K^{-1}(W\setminus C_{R,F_{R}})\mapsto K^{-1}(\hat{\mathbb{C}}\setminus C_{R,\hat{\mathbb{C}}})$ plus the image of $i_{*}:K^{-1}(V)\mapsto K^{-1}(\hat{\mathbb{C}}\setminus C_{R,\hat{\mathbb{C}}})$ is generated by $\{v_{c}\}_{c\in C_{R,F_{R}}}$ and $\{v_{[c]}\}_{[c]\in [C_{R,J_{R}}]}$.
\end{proof}

We will now compute the kernel and co-kernel of $\iota - \hat{\otimes}_{1}[\mathcal{E}_{R,J_{R}}]$ acting on $K^{-1}$.
\begin{prop}
\label{k1j}
Let $R$ be a rational function of degree $d>1$. Let $c_{R,J}$ be the size of $C_{R,J_{R}}$, $k_{R,J}$ the size of $[C_{R,J_{R}}]$, $h_{R}$ the number of Herman cycles, $f_{R}$ the number of Fatou cycles, and $\omega_{R}$ the greatest common divisor of their cycle lengths.
\par
\begin{itemize}
    \item If $J_{R} = \hat{\mathbb{C}}$, the mapping $\iota - \hat{\otimes}_{1}[\mathcal{E}_{R,J_{R}}]:K^{-1}(J_{R}\setminus C_{R,J_{R}})\mapsto K^{-1}(J_{R})$ has kernel isomorphic to $\mathbb{Z}^{c_{R,J}-1}$ and co-kernel equal to $0$.
    \item If $J_{R}\neq\hat{\mathbb{C}}$, the mapping $\iota - \hat{\otimes}_{1}[\mathcal{E}_{R,J_{R}}]:K^{-1}(J_{R}\setminus C_{R,J_{R}})\mapsto K^{-1}(J_{R})$ has kernel isomorphic to $\mathbb{Z}^{f_{R} + c_{R,J} - k_{R,J}-1}$ and co-kernel isomorphic to $\mathbb{Z}/\omega_{R}\mathbb{Z}\oplus\mathbb{Z}^{f_{R}-1}$.
\end{itemize}
\end{prop}
\begin{proof}
First, assume $J_{R} = \hat{\mathbb{C}}$. Then $K^{-1}(J_{R}) = 0$, so $\text{ker}(\iota - \hat{\otimes}_{1}[\mathcal{E}_{R,F_{R}}]) = K^{-1}(\hat{\mathbb{C}}\setminus C_{R,J_{R}})\simeq\mathbb{Z}^{c_{R,J} - 1}$ and $\text{co-ker}(\iota - \hat{\otimes}_{1}[\mathcal{E}_{R,F_{R}}]) = 0$.
\par
Now, assume $J_{R}\neq\hat{\mathbb{C}}$. By  Corollary \ref{Maindiagram}, we have a commutative diagram
$$
\begin{tikzcd}
{K^{-1}(J_{R}\setminus C_{R,J})} \arrow[r, "\text{exp}"] \arrow[d, "{\iota - \hat{\otimes}_{1}[\mathcal{E}_{R,J_{R}}]}"] & {K^{0}(F_{R}\setminus C_{R,F_{R}})} \arrow[d, "{\iota - \hat{\otimes}_{0}[\mathcal{E}_{R,F_{R}}]}"] \\
K^{-1}(J_{R}) \arrow[r, "\text{exp}"]                                                                                & K^{0}(F_{R}).                                                                                   
\end{tikzcd}$$
The bottom horizontal map is injective because $K^{-1}(\hat{\mathbb{C}}) = 0$.
\par
By Corollary \ref{Maindiagram}, we have a commutative diagram

$$
\begin{tikzcd}
{K^{0}(F_{R}\setminus C_{R,F_{R}})} \arrow[r, "i_{*}"] \arrow[d, "{\iota - \hat{\otimes}_{0}[\mathcal{E}_{R,F_{R}}]}"] & {K^{0}(\hat{\mathbb{C}}\setminus C_{R,\hat{\mathbb{C}}})} \arrow[d, "{\iota - \hat{\otimes}_{0}[\mathcal{E}_{R,\hat{\mathbb{C}}}]}"] \\
K^{0}(F_{R}) \arrow[r, "i_{*}"]                                                                                   & K^{0}(\hat{\mathbb{C}}).                                                                                                         
\end{tikzcd}$$
The top horizontal map is surjective by Corollary \ref{mainbott}. Putting these two diagrams together yields a commutative diagram

$$
\begin{tikzcd}
            & {K^{-1}(J_{R}\setminus C_{R,J_{R}})} \arrow[r] \arrow[d, "{\iota - \hat{\otimes}_{1}[\mathcal{E}_{R,J_{R}}]}"] & {K^{0}(F_{R}\setminus C_{R,F_{R}})} \arrow[r, "i_{*}"] \arrow[d, "{\iota - \hat{\otimes}_{0}[\mathcal{E}_{R,F_{R}}]}"] & {K^{0}(\hat{\mathbb{C}}\setminus C_{R,\hat{\mathbb{C}}})} \arrow[d, "{\iota - \hat{\otimes}_{0}[\mathcal{E}_{R,\hat{\mathbb{C}}}]}"] \arrow[r] & 0 \\
0 \arrow[r] & K^{-1}(J_{R}) \arrow[r]                                                                                    & K^{0}(F_{R}) \arrow[r, "i_{*}"]                                                                                    & K^{0}(\hat{\mathbb{C}})                                                                                                                    &  
\end{tikzcd}$$
with exact rows. Applying the Snake Lemma to this diagram yields a boundary map $\partial:\text{ker}(\iota - \hat{\otimes}_{0}[\mathcal{E}_{R,\hat{\mathbb{C}}}])\mapsto\text{co-ker}(\iota - \hat{\otimes}_{1}[\mathcal{E}_{R,J_{R}}])$ and an exact sequence

$$
\begin{tikzcd}
{\text{ker}(\iota - \hat{\otimes}_{1}[\mathcal{E}_{R,J_{R}}])} \arrow[r, "\delta"] & {\text{ker}(\iota - \hat{\otimes}_{0}[\mathcal{E}_{R,F_{R}}])} \arrow[r, "i_{*}"]            & {\text{ker}(\iota - \hat{\otimes}_{0}[\mathcal{E}_{R,\hat{\mathbb{C}}}])} \arrow[d, "\partial"]   \\
{\text{co-ker}(\iota - \hat{\otimes}_{0}[\mathcal{E}_{R,\hat{\mathbb{C}}}])}           & {\text{co-ker}(\iota - \hat{\otimes}_{0}[\mathcal{E}_{R,F_{R}}])} \arrow[l, "\tilde{i_{*}}"] & {\text{co-ker}(\iota - \hat{\otimes}_{1}[\mathcal{E}_{R,J_{R}}]),} \arrow[l, "\tilde{\delta}"]
\end{tikzcd}$$
where $\tilde{\delta}$ and $\tilde{i_{*}}$ are the descent maps of $\delta$ and $i_{*}$, respectively.
\par
 By Proposition \ref{Cmap}, $\text{co-ker}(\iota - \hat{\otimes}_{0}[\mathcal{E}_{R,\hat{\mathbb{C}}}]) = K^{0}(\hat{\mathbb{C}})$ and $\text{ker}(\iota - \hat{\otimes}_{0}[\mathcal{E}_{R,\hat{\mathbb{C}}}]) = K^{0}(\hat{\mathbb{C}}\setminus C_{R,\hat{\mathbb{C}}})$. By Corollary \ref{Fkernel}, $\text{co-ker}(\iota - \hat{\otimes}_{0}[\mathcal{E}_{R,F_{R}}]) = \oplus_{P\in\mathcal{F}_{R}}\mathbb{Z}[f_{x_{P}}]$ and  $\text{ker}(\iota - \hat{\otimes}_{0}[\mathcal{E}_{R,F_{R}}]) = \oplus_{P\in\mathcal{F}_{R}}\mathbb{Z}[e_{P}]$, respectively. By Corollary \ref{mainbott}, $\tilde{i}_{*}(f_{x_{P}}) = \beta_{\hat{\mathbb{C}}}$ and $i_{*}(e_{P}) = |P|\beta_{\hat{\mathbb{C}}\setminus C_{R,\hat{\mathbb{C}}}}$ for every $P$ in $\mathcal{F}_{R}$, so the kernel of $i_{*}:\text{ker}(\iota - \hat{\otimes}_{0}[\mathcal{E}_{R,F_{R}}])\mapsto \text{ker}(\iota - \hat{\otimes}_{0}[\mathcal{E}_{R,\hat{\mathbb{C}}}])$ is isomorphic to $\mathbb{Z}^{f_{R}-1}$, with an image equal to $\sum_{P\in\mathcal{F}_{R}}|P|\mathbb{Z} = \omega_{R}\mathbb{Z}$. Therefore, the co-kernel of $i_{*}$ is isomorphic to $\mathbb{Z}/\omega_{R}\mathbb{Z}$. The kernel of $\tilde{i_{*}}:\text{co-ker}(\iota - \hat{\otimes}_{0}[\mathcal{E}_{R,F_{R}}])\mapsto \text{co-ker}(\iota - \hat{\otimes}_{0}[\mathcal{E}_{R,\hat{\mathbb{C}}}])$ is also isomorphic to $\mathbb{Z}^{f_{R}-1}$ by the same reasoning.
\par
Exactness of the above diagram and the above computations imply we have exact sequences

$$
\begin{tikzcd}
0 \arrow[r] & G \arrow[r]            & {\text{ker}(\iota - \hat{\otimes}_{1}[\mathcal{E}_{R,J_{R}}])} \arrow[r, "\delta"] & \mathbb{Z}^{f_{R}-1} \arrow[r] & 0 \\
0 \arrow[r] & \mathbb{Z}/\omega_{R}\mathbb{Z} \arrow[r] & {\text{co-ker}(\iota - \hat{\otimes}_{1}[\mathcal{E}_{R,J_{R}}])} \arrow[r, "\tilde{\delta}"]   & \mathbb{Z}^{f_{R}-1} \arrow[r]   & 0,
\end{tikzcd}$$
where $G$ is the kernel of $\delta:\text{ker}(\iota-\hat{\otimes}_{1}[\mathcal{E}_{R,J_{R}}])\mapsto \text{ker}(\iota-\hat{\otimes}_{0}[\mathcal{E}_{R,F_{R}}]).$ Both sequences split, so $\text{ker}(\iota-\hat{\otimes}_{1}[\mathcal{E}_{R,J_{R}}])\simeq\text{ker}(\text{exp})\oplus\mathbb{Z}^{f_{R}-1}$ and $\text{co-ker}(\iota - \hat{\otimes}_{1}[\mathcal{E}_{R,J_{R}}])\simeq\mathbb{Z}/\omega_{R}\mathbb{Z}\oplus\mathbb{Z}^{f_{R}-1}.$
\par
We now compute $G$. By
Corollary \ref{Maindiagram}, the diagram

$$
\begin{tikzcd}
{K^{-1}(\hat{\mathbb{C}}\setminus C_{R,\hat{\mathbb{C}}})} \arrow[r, "r_{*}"] \arrow[d, "{\iota-\hat{\otimes}_{1}[\mathcal{E}_{R,\hat{\mathbb{C}}}]}"] & {K^{-1}(J_{R}\setminus C_{R,J_{R}})} \arrow[d, "{\iota-\hat{\otimes}_{1}[\mathcal{E}_{R,J_{R}}]}"] \\
K^{-1}(\hat{\mathbb{C}}) \arrow[r, "r_{*}"]                                                                                                        & K^{-1}(J_{R})                                                                                 
\end{tikzcd}$$
commutes. Therefore, $r_{*}(K^{-1}(\hat{\mathbb{C}}\setminus C_{R,\hat{\mathbb{C}}})) =  r_{*}(\text{ker}(\iota - \hat{\otimes}_{1}[\mathcal{E}_{R,\hat{\mathbb{C}}}]))\subseteq \text{ker}(\iota - \hat{\otimes}_{1}[\mathcal{E}_{R,J_{R}}])$. By exactness, $G$ is equal to $r_{*}(K^{-1}(\hat{\mathbb{C}}\setminus C_{R,\hat{\mathbb{C}}}))\cap\text{ker}(\iota - \hat{\otimes}_{1}[\mathcal{E}_{R,J_{R}}]) = r_{*}(K^{-1}(\hat{\mathbb{C}}\setminus C_{R,\hat{\mathbb{C}}}))$.
\par
Exactness implies the kernel of $r_{*}$ is equal the image of $i_{*}:K^{-1}(F_{R}\setminus C_{R,F_{R}})\mapsto K^{-1}(\hat{\mathbb{C}}\setminus C_{R,\hat{\mathbb{C}}})$, which, by Proposition \ref{iimage}, is generated by $\{v_{c}\}_{c\in C_{R,F_{R}}}$ and $\{v_{[c]}\}_{[c]\in[C_{R,J_{R}}]}$. Therefore, $r_{*}(K^{-1}(\hat{\mathbb{C}}))$ is isomorphic to the group generated by $\{v_{c}\}_{c\in C_{R,J_{R}}}$ satisfying the relations $\sum_{d\sim c}v_{d} = 0$, for all $c$ in $C_{R,J_{R}}$. Hence, $G\simeq\mathbb{Z}^{c_{R,J} - k_{R,J}}.$ Therefore, $\text{ker}(\iota - \hat{\otimes}_{1}[\mathcal{E}_{R,J_{R}}])\simeq \mathbb{Z}^{c_{R,J} +f_{R} - k_{R,J} - 1}.$

\end{proof}

Recall from Section \ref{ratmapf} the notion of an orientation for a Herman cycle $Q$. We now compute, for any $x$ in $Q$, $i_{*}:K^{-1}(U_{x})\mapsto K^{-1}(\hat{\mathbb{C}}\setminus C_{R,\hat{\mathbb{C}}})$ using the descriptions of the domains and co-domains provided by their orientations.
\begin{lemma}
\label{index}
Let $R$ be a rational function with an oriented Herman cycle $Q$, and let $x$ be in $Q$. Let $D$ be a finite set not intersecting $U_{x}$, and let $D_{x}^{+}$ be the points in $D$ contained in the connected component of $\hat{\mathbb{C}}\setminus U_{x}$ containing $\partial^{+}U_{x}$. Then, $i_{*}:K^{-1}(U_{x})\mapsto K^{-1}(\hat{\mathbb{C}}\setminus D)$ sends $u_{x}$ to $\sum_{d\in D_{x}^{+}}v_{d}$.
\end{lemma}
\begin{proof}
Let $U_{x}^{+}$ be the union of $U_{x}$ with the connected component of $\hat{\mathbb{C}}\setminus U_{x}$ containing $\partial^{+}U_{x}$. Since the complement of $U_{x}^{+}$ is connected and closed, $U_{x}^{+}$ is a simply connected open set such that $U_{x}^{+}\cap D = D_{x}^{+}$.
\par
The diagram
$$ 
\begin{tikzcd}
0 \arrow[r] & C_{0}(U_{x}^{+}\setminus D_{x}^{+}) \arrow[r]               & C_{0}(U_{x}^{+}) \arrow[r]                                          & C(D_{x}^{+}) \arrow[r]                                    & 0 \\
0 \arrow[r] & C_{0}(U_{x}) \arrow[r] \arrow[u] \arrow[d, no head, equal] & C_{0}(U_{x}^{+}) \arrow[r] \arrow[u, no head, equal] \arrow[d] & C(U_{x}^{+}\setminus U_{x}) \arrow[r] \arrow[u] \arrow[d] & 0 \\
0 \arrow[r] & C_{0}(U_{x}) \arrow[r]                                          & C_{0}(U_{x}\bigsqcup \partial^{+}U_{x}) \arrow[r]                   & C(\partial^{+}U_{x}) \arrow[r]                            & 0
\end{tikzcd}$$
commutes, and has exact rows. Therefore, by naturality of $\text{exp}$, we have a commutative diagram

$$
\begin{tikzcd}
K^{0}(\partial^{+}U_{x}) \arrow[d, "\text{exp}"] & K^{0}(U_{x}^{+}\setminus U_{x}) \arrow[l] \arrow[d, "\text{exp}"] \arrow[r] & K^{0}(D_{x}^{+}) \arrow[d, "\text{exp}"] \\
K^{-1}(U_{x})                                       & K^{-1}(U_{x}) \arrow[l, no head, equal] \arrow[r, "i^{+}_{*}"]          & K^{-1}(U_{x}^{+}\setminus D_{x}^{+}),   
\end{tikzcd}$$
where $i^{+}:C_{0}(U_{x})\mapsto C_{0}(U_{x}^{+}\setminus D^{+}_{x})$ is the inclusion. Hence, $i_{*}^{+}( u_{x}) = i_{*}^{+}(\text{exp}([1_{U_{x}^{+}\setminus U_{x}}])) = \sum_{d\in D_{x}^{+}}\text{exp}([1_{d}]) = \sum_{d\in D_{x}^{+}}v_{d}$.
\par
Let $j_{*}:K^{-1}(U_{x}^{+}\setminus D_{x}^{+})\mapsto K^{-1}(\hat{\mathbb{C}}\setminus D)$ be the map induced from the inclusion. From the above calculation and Proposition \ref{genmappingk1}, we have $i_{*}(u_{x}) = j_{*}(i_{*}^{+}(u_{x})) = \sum_{d\in D_{x}^{+}}v_{d}$.
\end{proof}

For $Q$ in $\mathcal{H}_{R}$ and $x$ in $Q$, let $(\hat{\mathbb{C}}\setminus U_{x})^{+}$, 
$(\hat{\mathbb{C}}\setminus U_{x})^{-}$ denote the connected component of $\hat{\mathbb{C}}\setminus U_{x}$ containing
$\partial^{+}U_{x}$, $\partial^{-}U_{x}$, respectively. Denote $J_{x}^{\pm}:= (\hat{\mathbb{C}}\setminus U_{x})^{\pm}\cap J_{R}$ 
and $\overline{U_{x}}^{\pm}:= U_{x}\cup \partial^{\pm}U_{x}$.
\begin{lemma}
\label{exphr}
For $Q$ in $\mathcal{H}_{R}$ and $x$ in $Q$, the homomorphism $\text{exp}:C(J_{R},\mathbb{Z})\mapsto K^{-1}(F_{R})$ satisfies $\text{exp}(1_{J_{x}^{+}}) = u_{x}$.
\end{lemma}
\begin{proof}
The diagram
$$
\begin{tikzcd}
0 \arrow[r] & {C_{0}(F_{R}\setminus C_{R,F_{R}})} \arrow[r] \arrow[d, "r"] & {C_{0}(\hat{\mathbb{C}}\setminus C_{R,F_{R}})} \arrow[r] \arrow[d, "r"] & C(J_{R}) \arrow[r] \arrow[d, "r"] & 0 \\
0 \arrow[r] & C_{0}(U_{x}) \arrow[r]                                       & C_{0}(\overline{U_{x}}^{+}) \arrow[r]                                   & C(\partial^{+}U_{x}) \arrow[r]    & 0
\end{tikzcd}$$
commutes and has exact rows, where the vertical maps are the restrictions. Naturality of $\text{exp}$ then implies
the diagram
$$
\begin{tikzcd}
{C(J_{R},\mathbb{Z})} \arrow[r, "\text{exp}"] \arrow[d, "r"] & {K^{-1}(F_{R}\setminus C_{R,F_{R}})} \arrow[d, "r_{*}"] \\
{C(\partial^{+}U_{x},\mathbb{Z})} \arrow[r, "\text{exp}"]    & K^{-1}(U_{x})                                          
\end{tikzcd}$$
commutes. We have, by definition, $\text{exp}(1_{\partial^{\pm}U_{x}}) = \pm u_{x}$, so commutativity
of the diagram implies $\text{exp}(1_{J_{x}^{\pm}})$ satisfies $r_{*}(\text{exp}(1_{J_{x}^{\pm}})) = \pm u_{x}$.
Since $U_{x}$ is a connected component of $F_{R}\setminus C_{R,F_{R}}$, $r_{*}$ is the projection
onto the direct summand $K^{-1}(U_{x})$ of $K^{-1}(F_{R}\setminus C_{R,F_{R}})$; let's denote this
projection $q_{U_{x}}$.
\par
Now, let $U\neq U_{x}$ be a Fatou component, and $q_{U}$ denote the projection of
$K^{-1}(F_{R}\setminus C_{R,F_{R}})$ onto the direct summand $K^{-1}(U\setminus C_{R,F_{R}})$. To prove the lemma, it remains
to show $q_{U}(\text{exp}(u_{x})) = 0$. Denote $J_{U} = J_{R}\cup\{U\setminus C_{R,F_{R}}\}$ The diagram
$$
\begin{tikzcd}
0 \arrow[r] & C_{0}(F_{R}\setminus C_{R,F_{R}}) \arrow[r] \arrow[d, "r"]         & C_{0}(\hat{\mathbb{C}}\setminus C_{R,F_{R}}) \arrow[r] \arrow[d, "r"] & C(J_{R}) \arrow[r] \arrow[d, equal] & 0 \\
0 \arrow[r] & C_{0}(U\setminus C_{R,F_{R}}) \arrow[r] \arrow[d, equal] & C_{0}(J_{U}) \arrow[r] \arrow[d, "r"]            & C(J_{R}) \arrow[r] \arrow[d, "r"]                 & 0 \\
0 \arrow[r] & C_{0}(U\setminus C_{R,F_{R}}) \arrow[r]                                & C_{0}(\overline{U}\setminus C_{R,F_{R}}) \arrow[r]                    & C(\partial U) \arrow[r]                           & 0
\end{tikzcd}$$
commutes, and has exact rows, where the vertical maps labelled $r$ are the restrictions. If we denote
the exponential map of the middle and bottom row by $\text{exp}_{U}$ and $\text{exp}_{\partial U}$, respectively, then
naturality of $\text{exp}$ implies $q_{U}\circ \text{exp} = \text{exp}_{U} = \text{exp}_{\partial U}\circ r$.
\par
Either
$\overline{U}\subseteq (\hat{\mathbb{C}}\setminus U_{x})^{+}$ or 
$\overline{U}\subseteq (\hat{\mathbb{C}}\setminus U_{x})^{-}$.
In the first case, $\partial U\cap J_{x}^{-} = \emptyset$. Therefore, 
$r(1_{J^{-}_{x}}) = 0$ and hence
$-q_{U}(\text{exp}(1_{J^{+}_{x}})) = 
q_{U}(\text{exp}(1_{J^{-}_{x}})) = 
\text{exp}_{\partial U}(r(1_{J^{-}_{x}})) = 0$.
In the second case, $r(1_{J^{+}_{x}}) = 0$, so that $q_{U}(\text{exp}(1_{J^{+}_{x}})) = 
\text{exp}_{\partial U}(r(1_{J^{+}_{x}})) = 0$.
\end{proof}
If $Q$ is a oriented Herman cycle and $d$ is a point not in $U_{Q}$, then we let $H_{Q}(d)$ be the number of $x$ in $Q$ for which $d$ is in the connected component of $\hat{\mathbb{C}}\setminus U_{x}$ containing $\partial^{+}U_{x}$. Lemma \ref{index} implies that if $D$ is a finite set not intersecting $U_{Q}$, then $i_{*}:K^{-1}(U_{Q})\mapsto K^{-1}(\hat{\mathbb{C}}\setminus D)$ sends $u_{Q}$ to $\sum_{d\in D}H_{Q}(d)v_{d}.$ The functions $\{H_{Q}\}_{Q\in\mathcal{H}_{R}}$ will also play a role in describing $\text{ker}(\text{id} - \Phi)$, as the following Corollary to the above lemma may suggest. 
\begin{cor}
\label{Hconst}
$H_{Q}$ is locally constant on $\hat{\mathbb{C}}\setminus U_{Q}$, and $H_{Q} - \Phi(H_{Q})$ is constant on $J_{R}$.
\end{cor}
\begin{proof}
$H_{Q} = \sum_{x\in Q}1_{(\hat{\mathbb{C}}\setminus U_{x})^{+}}$ and is therefore locally constant. We have $H_{Q}|_{J_{R}} = \sum_{x\in Q}1_{J^{+}_{x}}$, so, by Lemma \ref{exphr}, we have $\text{exp}(H_{Q}) = u_{Q}$, where $\text{exp} = \text{exp}:C(J_{R},\mathbb{Z})\mapsto K^{-1}(F_{R}\setminus C_{R,F_{R}})$.
\par
Therefore, Proposition \ref{ncintertwine} and \ref{Fkernel1} (respectively) imply $\text{exp}(H_{Q} - \Phi(H_{Q})) = \iota(u_{Q})-u_{Q}\hat{\otimes}_{1}[\mathcal{E}_{R,F_{R}}]$ and $\iota(u_{Q}) - u_{Q}\hat{\otimes}_{1}[\mathcal{E}_{R,F_{R}}] = 0$. Since the kernel of $\text{exp}:C(J_{R},\mathbb{Z})\mapsto K^{-1}(F_{R})$ equals $\mathbb{Z}[1_{J_{R}}]$, it follows that $H_{Q} - \Phi(H_{Q})$ must be constant on $J_{R}$.
\end{proof}
Let $\alpha_{R}:\mathbb{Z}[\mathcal{H}_{Q}]\oplus\mathbb{Z}\mapsto \mathbb{Z}$ be the homomorphism sending $1$ to $1-d$ and $Q$ to $H_{Q} - \Phi(H_{Q})$, for all
$Q$ in $\mathcal{H_{Q}}$. Denote the greatest common divisor of $\{H_{Q} - \Phi(H_{Q})\}_{Q\in\mathcal{H}_{R}}\cup\{1-d\}$ by $a_{R}$. We now compute the kernel and co-kernel of $\text{id}-\Phi:C(J_{R},\mathbb{Z})\mapsto C(J_{R},\mathbb{Z})$
\par
\begin{prop}
\label{k0nc}
Let $R$ be a rational function of degree $d>1$. Then,
\begin{enumerate}
    \item The homomorphism $\varphi_{R}:\mathbb{Z}^{\mathcal{H}_{R}}\oplus\mathbb{Z}\mapsto C(J_{R},\mathbb{Z})$ sending $1$ to $1_{J_{R}}$ and $Q$ to $H_{Q}$, for all $Q$ in $\mathcal{H_{R}}$, is an isomorphism onto $(\text{id} - \Phi)^{-1}(\mathbb{Z}[1_{J_{R}}])$. Hence, $\varphi_{R}$ maps $\text{ker}(\alpha_{R})$ isomorphically onto $\text{ker}(\text{id}-\Phi)\simeq\mathbb{Z}^{h_{R}}$.
    \item The isotropy degree of $1_{J_{R}}$ in $\text{co-ker}(\text{id} - \Phi)$ is $a_{R}$, and  for any choice of $x_{Q}$ in $Q$, for $Q$ in $\mathcal{H}_{R}$, the homomorphism $\phi_{R}:\mathbb{Z}^{\mathcal{H}_{R}}\oplus\mathbb{Z}/a_{R}\mathbb{Z}\mapsto\text{co-kernel}(\text{id}-\Phi)$ sending $1$ to the image of $1_{J_{R}}$ in the co-kernel and $Q$ to the image of $1_{J_{x_{Q}}^{+}}$, for all $Q$ in $\mathcal{H}_{R}$ in the co-kernel is an isomorphism onto $\text{co-ker}(\text{id}-\Phi)$.
\end{enumerate}
\end{prop}

\begin{proof}
First, assume $J_{R} = \hat{\mathbb{C}}$. Then, $C(J_{R},\mathbb{Z}) = \mathbb{Z} 1_{J_{R}}$ and $\mathcal{H}_{R}=\emptyset$. Hence, $(\text{id} - \Phi)^{-1}(\mathbb{Z}[1_{J_{R}}]) = \mathbb{Z}[1_{J_{R}}]$, $a_{R} = 1-d$, and $\text{co-ker}(\text{id} - \Phi) = \mathbb{Z}[1_{J_{R}}]/(1-d)\mathbb{Z}[1_{J_{R}}]$.
\par
Now, assume $J_{R}\neq\hat{\mathbb{C}}$.
$K^{-1}(\hat{\mathbb{C}}) = 0$, so the diagram
$$
\begin{tikzcd}
{K^{-1}(F_{R}\setminus C_{R,F_{R}})} \arrow[d, "{\iota - \hat{\otimes}_{1}[\mathcal{E}_{R,F_{R}}]}"] \arrow[r, "i_{*}"] & {K^{-1}(\hat{\mathbb{C}}\setminus C_{R,F_{R}})} \arrow[d] \\
K^{-1}(F_{R}) \arrow[r, "i_{*}"]                                                                                    & 0                                                        
\end{tikzcd}$$
automatically commutes. By Proposition \ref{iimage}, the subgroup $G(F_{R},C_{R})$ of $K^{-1}(F_{R}\setminus C_{R,F_{R}})$ 
has $\mathbb{Z}$-linear independent generators $\{v_{c}\}_{c\in C_{R,F_{R}}}$ such that $i_{*}(v_{c}) = v_{c}$, for
all $c$ in $C_{R,F_{R}}$. Therefore, $i_{*}:K^{-1}(F_{R}\setminus C_{R,F_{R}})\mapsto K^{-1}(\hat{\mathbb{C}}\setminus C_{R,F_{R}})$ is surjective.
\par
By Proposition \ref{ncintertwine}, the diagram
$$
\begin{tikzcd}
{C(J_{R},\mathbb{Z})} \arrow[r, "\text{exp}"] \arrow[d, "\text{id} - \Phi"] & {K^{-1}(F_{R}\setminus C_{R,F_{R}})} \arrow[d, "{\iota - \hat{\otimes}_{1}[\mathcal{E}_{R,F_{R}}]}"] \\
{C(J_{R},\mathbb{Z})} \arrow[r, "\text{exp}"]                               & K^{-1}(F_{R})                                                                                   
\end{tikzcd}$$
commutes. The bottom map has kernel equal to $\mathbb{Z}[1_{J_{R}}]$. Hence, the diagram
\par
$$
\begin{tikzcd}
            & {C(J_{R},\mathbb{Z})} \arrow[r, "\text{exp}"] \arrow[d, "\widetilde{\text{id} - \Phi}"] & {K^{-1}(F_{R}\setminus C_{R,F_{R}})} \arrow[d, "{\iota - \hat{\otimes}_{1}[\mathcal{E}_{R,F_{R}}]}"] \arrow[r, "i_{*}"] & {K^{-1}(\hat{\mathbb{C}}\setminus C_{R,F_{R}})} \arrow[d] \arrow[r] & 0 \\
0 \arrow[r] & {C(J_{R},\mathbb{Z})/\mathbb{Z}[1_{J_{R}}]} \arrow[r, "\widetilde{\text{exp}}"]         & K^{-1}(F_{R}) \arrow[r]                                                                                             & 0                                                                   &  
\end{tikzcd}$$
commutes, and has exact rows, where $\widetilde{\text{id} - \Phi}$ and $\widetilde{\text{exp}}$ denote the descent
maps of $\text{id} - \Phi$, $\text{exp}$, respectively.
\par
By the Snake Lemma, there is a boundary map $\partial:K^{-1}(\hat{\mathbb{C}}\setminus C_{R,F_{R}})\mapsto 
\text{co-ker}(\text{id} -\Phi)/\mathbb{Z}[1_{J_{R}}]$ making the sequence
$$
\begin{tikzcd}
{(\text{id} - \Phi)^{-1}(\mathbb{Z}[1_{J_{R}}])} \arrow[r, "\text{exp}"] & {\text{ker}(\iota - \hat{\otimes}_{1}[\mathcal{E}_{R,F_{R}}])} \arrow[r, "i_{*}"] & {K^{-1}(\hat{\mathbb{C}}\setminus C_{R,F_{R}})} \arrow[d, "\partial"]                   \\
0                                                                        & {\text{co-ker}(\iota - \hat{\otimes}_{1}[\mathcal{E}_{R,F_{R}}])} \arrow[l]       & {\text{co-ker}(\text{id} - \Phi)/\mathbb{Z}[1_{J_{R}}]} \arrow[l, "\widetilde{\text{exp}}"]
\end{tikzcd}$$
exact.
\par
By Proposition \ref{Fkernel1}, 
$\text{ker}(\iota - \hat{\otimes}_{1}[\mathcal{E}_{R,F_{R}}]) = G(F_{R},C_{R,F_{R}})
\oplus(\bigoplus_{Q\in\mathcal{H}_{R}}\mathbb{Z}[u_{Q}]).$ By Proposition \ref{index}, 
$i_{*}(u_{Q}) = \sum_{c\in C_{R,F_{R}}} H_{Q}(c)v_{c}$, 
for all $Q$ in $\mathcal{H}_{R}$. So, $i_{*}$ is surjective with kernel freely generated by the elements
$w_{Q}:= u_{Q} - \sum_{c\in C_{R,F_{R}}}H_{Q}(c)v_{c}$, for all $Q$ in $\mathcal{H}_{R}$, and $\sum_{c\in C_{R,F_{R}}}v_{c}$. 
Surjectivity of $i_{*}$ implies, by exactness of the above diagram, the sequences
\begin{equation}
\tag{*}
\label{snake}
\begin{tikzcd}
{(\text{id} - \Phi)^{-1}(\mathbb{Z}[1_{J_{R}}])} \arrow[r, "\text{exp}"] & {\mathbb{Z}[\sum_{c\in C_{R,F_{R}}}v_{c}]\bigoplus_{Q\in\mathcal{H_{R}}}\mathbb{Z}[w_{Q}]} \arrow[r]            & 0                                                               &                                   \\
0                                                                            & {\text{co-ker}(\iota - \hat{\otimes}_{1}[\mathcal{E}_{R,F_{R}}])} \arrow[l] & \text{co-ker}(\text{id} - \Phi) \arrow[l, "\tilde{\text{exp}}"] & {\mathbb{Z}[1_{J_{R}}]} \arrow[l]
\end{tikzcd}\end{equation}
are exact.
\par
We now prove $(1)$. It suffices to show $\{H_{Q}\}_{Q\in\mathcal{H}_{R}}\cup\{1_{J_{R}}\}$ is a $\mathbb{Z}$-linear independent generating set for $(\text{id} - \Phi)^{-1}(\mathbb{Z}[1_{J_{R}}])$.
\par
First, assume $C_{R,F_{R}}=\emptyset$. In this case, $w_{Q} = u_{Q}$, for all $Q$ in $\mathcal{H}_{R}$, and exactness implies $\text{exp}$ has kernel equal to $\mathbb{Z}[1_{J_{R}}]$. From above, $\text{exp}:(\text{id} - \Phi)^{-1}(\mathbb{Z}[1_{J_{R}}])\mapsto\bigoplus_{Q\in\mathcal{H}_{R}}\mathbb{Z}[w_{Q}]$ is a surjection, and, by Lemma \ref{exphr}, $H_{Q} = \sum_{x\in Q}1_{J_{x}^{+}}$ satisfies $\text{exp}(H_{Q}) = \sum_{x\in Q}u_{x} = u_{Q} = w_{Q}$. Therefore, the functions $\{H_{Q}\}_{Q\in\mathcal{H}_{R}}\cup\{1_{J_{R}}\}$ generate $(\text{id} - \Phi)^{-1}(\mathbb{Z}[1_{J_{R}}])$ and are $\mathbb{Z}$-linearly independent (since $\{u_{Q}\}_{Q\in\mathcal{H}_{R}}$ are).
\par
Assume $C_{R,F_{R}}\neq\emptyset$. Exactness then implies $\text{exp}$ is injective, so that $\text{exp}:(\text{id} - \Phi)^{-1}(\mathbb{Z}[1_{J_{R}}])\mapsto\bigoplus_{Q\in\mathcal{H}_{R}}\mathbb{Z}[w_{Q}]\oplus\mathbb{Z}[\sum_{c\in C_{R,F_{R}}}v_{c}]$ is an isomorphism.
\par
Naturality implies the diagram
$$
\begin{tikzcd}
{C(J_{R},\mathbb{Z})} \arrow[r, "\text{exp}"] \arrow[rd, "\text{exp}"] & {K^{-1}(F_{R}\setminus C_{R,F_{R}})} \arrow[d, "i_{*}"] \\
                                                                       & K^{-1}(F_{R})                                          
\end{tikzcd}$$
where $i:C_{0}(F_{R}\setminus C_{R,F_{R}})\mapsto C_{0}(F_{R})$ is the inclusion. Lemma \ref{k1open} implies $i_{*}(w_{Q}) = u_{Q}$ for all $Q$ in $\mathcal{H_{R}}$ and $i_{*}(\sum_{c\in C_{R,F_{R}}}v_{c}) = 0$. Therefore, $i_{*}\text{exp}(H_{Q}) = u_{Q}$ implies $\text{exp}(H_{Q}) = w_{Q} + a_{Q}(\sum_{c\in C_{R,F_{R}}}v_{c})$, for some $a_{Q}$ in $\mathbb{Z}$, for all $Q$ in $\mathcal{H}_{R}$. It remains to show $\text{exp}(1_{J_{R}}) = -\sum_{c\in C_{R,F_{R}}}v_{c}$.
\par
The diagram
$$
\begin{tikzcd}
0 \arrow[r] & {C_{0}(F_{R}\setminus C_{R,F_{R}})} \arrow[r] \arrow[d, equal] & {C_{0}(\hat{\mathbb{C}}\setminus C_{R,F_{R}})} \arrow[r] \arrow[d] & C(J_{R}) \arrow[r] \arrow[d]         & 0 \\
0 \arrow[r] & {C_{0}(F_{R}\setminus C_{R,F_{R}})} \arrow[r]                                & C(\hat{\mathbb{C}}) \arrow[r]                                      & {C(J_{R}\cup C_{R,F_{R}})} \arrow[r] & 0 \\
0 \arrow[r] & {C_{0}(F_{R}\setminus C_{R,F_{R}})} \arrow[u,equal] \arrow[r]                      & C_{0}(F_{R}) \arrow[u] \arrow[r]                                   & {C(C_{R,F_{R}})} \arrow[u] \arrow[r] & 0
\end{tikzcd}$$
commutes, and has exact rows. Naturality of $\text{exp}$ then implies the diagram
$$
\begin{tikzcd}
{C(J_{R},\mathbb{Z})} \arrow[r] \arrow[rd, "\text{exp}"'] & {C(J_{R}\cup C_{R,F_{R}},\mathbb{Z})} \arrow[d, "\text{exp}"] & {C(C_{R,F_{R}},\mathbb{Z})} \arrow[l] \arrow[ld, "\text{exp}"] \\
                                                              & K^{-1}(F_{R}\setminus C_{R,F_{R}})                            &                                                               
\end{tikzcd}$$
commutes. By exactness, $\text{exp}(1_{J_{R}} + 1_{C_{R,F_{R}}}) = 0,$ so that $\text{exp}(1_{J_{R}}) = -\text{exp}(1_{C_{R,F_{R}}}) = -\sum_{c\in C_{R,F_{R}}}v_{c}$. This finishes the proof of $(1)$.

\par

Consequently, the subgroup  of $C(J_{R},\mathbb{Z})$ generated by $\{H_{Q}\}_{Q\in\mathcal{H_{R}}}\cup\{1_{J_{R}}\}$ surjects onto $\text{im}(\text{id} - \Phi)\cap\mathbb{Z}[1_{J_{R}}]$ via $\text{id} - \Phi$. Therefore, the isotropy degree of $1_{J_{R}}$ in $\text{co-ker}(\text{id} - \Phi)$ is equal to $a_{R}$, the greatest common divisor of $\{H_{Q} - \Phi(H_{Q})\}_{Q\in\mathcal{H_{R}}}\cup\{1-d\}$.
\par
We now prove $(2)$. By Proposition \ref{Fkernel1}, $\text{co-ker}(\iota - \hat{\otimes}_{1}[\mathcal{E}_{R,F_{R}}])\simeq\mathbb{Z}^{h_{R}}$, and the elements 
$\{u_{x_{Q}}\}_{Q\in\mathcal{H}_{R}}$ under the quotient map form a $\mathbb{Z}$-linear independent generating set. exactness of the bottom sequence in diagram \eqref{snake} then implies we have a  short exact sequence
$$
\begin{tikzcd}
0 & {\sum_{Q\in\mathcal{H}_{R}}\mathbb{Z}[u_{x_{Q}}]} \arrow[l] & \text{co-ker}(\text{id} - \Phi) \arrow[l, "\text{exp}"'] & {\mathbb{Z}[1_{J_{R}}]/a_{R}\mathbb{Z}[1_{J_{R}}]} \arrow[l] & 0. \arrow[l]
\end{tikzcd}$$
This short exact sequence splits and, by Lemma \ref{exphr}, $\text{exp}(1_{J^{+}_{x_{Q}}}) = u_{x_{Q}}$ for all $Q$ in $\mathcal{H}_{R}$. Therefore, $\phi_{R}$ is equal to the direct sum of the inclusion of the kernel of $\widetilde{\text{exp}}$ and a splitting, and is hence an isomorphism.
\end{proof}
For every $Q$ in $\mathcal{H}_{R}$, Corollary \ref{Hconst} implies $H_{Q}(c) = H_{Q}(d)$ for any $c,d$ in $C_{R,J_{R}}$ such that $c\sim d$. Denote $H_{Q}([c]) := H_{Q}(c)$, for all $[c]$ in $[C_{R,J_{R}}]$.
\par
We let $[H_{R}]$ be the following matrix with rows indexed by elements in $[C_{R,J_{R}}]\cup\{u\}$ and columns indexed by $\mathcal{H}_{R}\cup\{u\}$
\begin{itemize}
\item $([H_{R}])_{[c],Q} = H_{Q}(c)$ for all $[c]$ in $[C_{R,J_{R}}]$ and $Q$ in $\mathcal{H}_{R}$,
\item $([H_{R}])_{[c],u} = 1$, for all $[c]$ in $[C_{R,J_{R}}]$,
\item $([H_{R}])_{u,Q} = \Phi(H_{Q}) - H_{Q}$, for all $Q$ in $\mathcal{H}_{R}$, and
\item $([H_{R}])_{u,u} = \text{deg}(R)-1$.
\end{itemize}
\par
We will write $(\text{id}-\Phi)|_{C_{0}(J_{R}\setminus C_{R,J_{R}},\mathbb{Z})} = \iota - \Phi_{0}$.
\begin{prop}
\label{transfer0}
 Let $R$ be a rational function of degree $d>1$. Then,
\begin{enumerate}
\item $(\iota - \Phi_{0})^{-1}(\mathbb{Z}[1_{J_{R}}])\simeq [H_{R}]^{-1}(\mathbb{Z}[u])$ and $\text{ker}(\iota - \Phi_{0})\simeq \text{ker}([H_{R}])$.

\item $\text{co-ker}(\iota - \Phi_{0})\simeq \text{co-ker}([H_{R}])\oplus\mathbb{Z}^{h_{R}}$ and the isomorphism maps the class of $1_{J_{R}}$ to the class of $u$.
\end{enumerate}

\end{prop}
\begin{proof}
First, let's assume $J_{R} = \hat{\mathbb{C}}$. In this case, $\iota - \Phi_{0} = 0$ and $C(J_{R},\mathbb{Z}) = \mathbb{Z}[1_{J_{R}}]$. Therefore, $(\iota - \Phi)^{-1}(\mathbb{Z}[1_{J_{R}}]) = 0$ and $\text{co-ker}(\iota - \Phi_{0}) = \mathbb{Z}[1_{J_{R}}]\simeq\mathbb{Z}$. 
\par
Since $|[C_{R,J_{R}}]| = 1$ and $\mathcal{H}_{R} = \emptyset$, we have that $[H_{R}]:\mathbb{Z}\mapsto \mathbb{Z}\oplus\mathbb{Z}$ is the map sending $u$ to $[c] + u$. Hence, $[H_{R}]^{-1}(\mathbb{Z}[u]) = 0$ and $\text{co-ker}([H_{R}])\simeq\mathbb{Z}$ and is generated by $u$. 
\par
Therefore, $\text{ker}(\iota - \Phi_{0}) = (\iota - \Phi_{0})^{-1}(\mathbb{Z}[1_{J_{R}}]) = 0 = [H_{R}]^{-1}(\mathbb{Z}[u]) = \text{ker}([H_{R}])$ and $\text{co-ker}(\iota - \Phi_{0})\simeq \text{co-ker}([H_{R}])\simeq \mathbb{Z}$ via the map sending the class of $1_{J_{R}}$ to the class of $u$.
\par
Now, we assume throughout the proof that $J_{R}\neq\hat{\mathbb{C}}$. For $g = m\cdot u + \sum_{Q\in\mathcal{H}_{R}}a_{Q}\cdot Q$ in $\mathbb{Z}^{\mathcal{H}_{R}}\oplus\mathbb{Z}$, we have $[H_{R}](g) =-\alpha_{R}(g)u + \sum_{Q\in\mathcal{H}_{R}}(H_{Q}([c]) + m)[c]$. So, $[H_{R}](g)$ is in $\mathbb{Z}[u]$ if and only if $\sum_{Q\in\mathcal{H}_{R}}H_{Q}(c) + m = 0$, for all $c$ in $C_{R,J_{R}}$. 
\par
Let $\varphi_{R}$ be the isomorphism appearing in Proposition \ref{k0nc}. By Proposition \ref{k0nc}, we have $(\iota -\Phi_{0})^{-1}(\mathbb{Z}[1_{J_{R}}]) = (\text{id}-\Phi)^{-1}(\mathbb{Z}[1_{J_{R}}])\cap C_{0}(J_{R}\setminus C_{R,J_{R}},\mathbb{Z}) = \{ m +\sum_{Q\in\mathcal{H}_{R}}a_{Q}H_{Q}: m +\sum_{Q\in\mathcal{H}_{R}}a_{Q}H_{Q}(c) = 0\text{ }\forall c\in C_{R,J_{R}}\}$. Hence, $\varphi_{R}((\iota -\Phi_{0})^{-1}(\mathbb{Z}[1_{J_{R}}])) = [H_{R}]^{-1}(\mathbb{Z}[u])$.
\par
Since $[H_{R}](g) = - \alpha_{R}(g)u$, for all $g$ in $[H_{R}]^{-1}(\mathbb{Z}[u])$, it also follows that $\varphi^{-1}_{R}(\text{ker}(H_{R})) = \varphi^{-1}_{R}([H_{R}]^{-1}\cap\text{ker}(\alpha_{R}))$. 
By Proposition \ref{k0nc} and the above equalities, we have $\varphi^{-1}_{R}([H_{R}]^{-1}\cap\text{ker}(\alpha_{R})) = (\iota - \Phi_{0})^{-1}(\mathbb{Z}[1_{J_{R}}])\cap \text{ker}(\text{id} - \Phi) = \text{ker}(\iota - \Phi_{0})$. This proves $(1)$
\par
We now prove $(2)$. We first determine the image of $i_{*}:K^{-1}(F_{R}\setminus C_{R,F_{R}})\mapsto K^{-1}(\hat{\mathbb{C}}\setminus C_{R,\hat{\mathbb{C}}})$. Naturality of $\text{exp}$ implies the diagram
$$
\begin{tikzcd}
{K^{0}(C_{R,\hat{\mathbb{C}}})} \arrow[r, "q"] \arrow[d, "\text{exp}"]        & {K^{0}(C_{R,J_{R}})} \arrow[d, "\text{exp}"] \\
{K^{-1}(\hat{\mathbb{C}}\setminus C_{R,\hat{\mathbb{C}}})} \arrow[r, "r_{*}"] & {K^{-1}(J_{R}\setminus C_{R,J_{R}})}        
\end{tikzcd}$$
commutes, where the vertical maps are the restrictions. Therefore, $\text{im}(i_{*}) = \text{ker}(r_{*}) = \text{exp}(q^{-1}(\text{exp}^{-1}(0)))$. By exactness,
the kernel of $\text{exp}:C(C_{R,J_{R}},\mathbb{Z})\mapsto K^{-1}(J_{R}\setminus C_{R,J_{R}})$ is the image of the restriction map $r:C(J_{R},\mathbb{Z})\mapsto C(C_{R,J_{R}},\mathbb{Z})$, which is generated by the functions
$\{1_{[c]}\}_{[c]\in[C_{R,J_{R}}]}$. So, $\text{ker}(r_{*}) = \text{exp}(q^{-1}(\sum_{[c]\in[C_{R,J_{R}}]}\mathbb{Z}[1_{[c]}])) = 
\text{exp}(\sum_{[c]\in[C_{R,J_{R}}]}\mathbb{Z}[1_{[c]}] + \sum_{c\in C_{R,F_{R}}}\mathbb{Z}[1_{c}]) = \sum_{[c]\in[C_{R,J_{R}}]}\mathbb{Z}[v_{[c]}] + \sum_{c\in C_{R,F_{R}}}\mathbb{Z}[v_{c}] =:\tilde{W}_{R}$. The diagram
$$
\begin{tikzcd}
{K^{-1}(F_{R}\setminus C_{R,F_{R}})} \arrow[r, "i_{*}"] \arrow[d, "{\iota - \hat{\otimes}_{1}[\mathcal{E}_{R,F_{R}}]}"] & \tilde{W}_{R} \arrow[d] \\
K^{-1}(F_{R}) \arrow[r]                                                                                             & 0                      
\end{tikzcd}$$
automatically commutes and, from the above calculation, the top row map is surjective.  By Corollary \ref{Maindiagram}, the diagram
$$
\begin{tikzcd}
{C_{0}(J_{R}\setminus C_{R,J_{R}},\mathbb{Z})} \arrow[d, "\iota - \Phi_{0}"] \arrow[r, "\text{exp}"] & {K^{-1}(F_{R}\setminus C_{R,F_{R}})} \arrow[d, "{\iota - \hat{\otimes}_{1}[\mathcal{E}_{R,F_{R}}]}"] \\
{C(J_{R},\mathbb{Z})} \arrow[r, "\text{exp}"]                                                        & K^{-1}(F_{R})                                                                                   \end{tikzcd}$$
commutes. The bottom map has kernel equal to $\mathbb{Z}[1_{J_{R}}]$. Hence, the diagram
$$
\begin{tikzcd}
            & {C(J_{R}\setminus C_{R,J_{R}},\mathbb{Z})} \arrow[r, "\text{exp}"] \arrow[d, "\widetilde{\iota - \Phi_{0}}"] & {K^{-1}(F_{R}\setminus C_{R,F_{R}})} \arrow[d, "{\iota - \hat{\otimes}_{1}[\mathcal{E}_{R,F_{R}}]}"] \arrow[r, "i_{*}"] & \tilde{W}_{R} \arrow[d] \arrow[r] & 0 \\
0 \arrow[r] & {C(J_{R},\mathbb{Z})/\mathbb{Z}[1_{J_{R}}]} \arrow[r, "\widetilde{\text{exp}}"]                              & K^{-1}(F_{R}) \arrow[r]                                                                                             & 0                                 &  
\end{tikzcd}$$
commutes, and has exact rows, where $\widetilde{\iota - \Phi}$ and $\widetilde{\text{exp}}$ denote the descent
maps of $\iota - \Phi$, $\text{exp}$, respectively.
\par
By the Snake Lemma, there is a boundary map $\tilde{\partial}:\tilde{W}_{R}\mapsto 
\text{co-ker}(\iota -\Phi)/\mathbb{Z}[1_{J_{R}}]$ making the sequence
$$
\begin{tikzcd}
{(\iota - \Phi_{0})^{-1}(\mathbb{Z}[1_{J_{R}}])} \arrow[r, "\text{exp}"] & {\text{ker}(\iota - \hat{\otimes}_{1}[\mathcal{E}_{R,F_{R}}])} \arrow[r, "i_{*}"] & \tilde{W}_{R} \arrow[d, "\tilde{\partial}"]                                                     \\
0                                                                        & {\text{co-ker}(\iota - \hat{\otimes}_{1}[\mathcal{E}_{R,F_{R}}])} \arrow[l]       & {\text{co-ker}(\iota - \Phi_{0})/\mathbb{Z}[1_{J_{R}}]} \arrow[l, "\widetilde{\text{exp}}"]
\end{tikzcd}$$
exact. We will first determine $i_{*}$ on generators. By Proposition \ref{Fkernel1}, 
$\text{ker}(\iota - \hat{\otimes}_{1}[\mathcal{E}_{R,F_{R}}]) = G(F_{R},C_{R,F_{R}})
\oplus(\bigoplus_{Q\in\mathcal{H}_{R}}\mathbb{Z}[u_{Q}].$ By the definition of $G(F_{R},C_{R})$
and Proposition \ref{genmappingk1}, this group is freely generated by elements $\{v_{c}\}_{c\in C_{R,F_{R}}}$ with the 
property that $i_{*}(v_{c}) = v_{c},$ for all $c$ in $C_{R,F_{R}}.$ We also have, by Proposition \ref{index}, 
$i_{*}(u_{Q}) = \sum_{c\in C_{R,\hat{\mathbb{C}}}} H_{Q}(c)v_{c}$, 
for all $Q$ in $\mathcal{H}_{R}$. So, for every $Q$ in $\mathcal{H}_{R}$, $w_{Q}:= u_{Q} - \sum_{c\in C_{R,F_{R}}}H_{Q}(c)v_{c}$ satisfies $i_{*}(w_{Q}) = 
\sum_{[c]\in [C_{R,J_{R}}]}H_{Q}([c])v_{[c]}$.
\par

Therefore, the image of $i_{*}$ is equal to $\text{im}(\tilde{H}_{R}) +\sum_{c\in C_{R,F_{R}}}\mathbb{Z}[v_{c}]$, where $\tilde{H}_{R}:\mathbb{Z}^{\mathcal{H}_{R}}\mapsto \tilde{W}_{R}$ is the homomorphism
such that $\tilde{H}_{R}(Q) = \sum_{[c]\in[C_{R,J_{R}}]}H_{Q}[c]v_{[c]}$ for all $Q$ in $\mathcal{H}_{R}$.
\par
Let $V_{R}$ be the abelian group with generators $\{[c]\}_{[c]\in [C_{R,J_{R}}]}\cup\{u\}$ satisfying the relation $\sum_{[c]\in [C_{R,J_{R}}]}[c] = (1-d)u$.
\par
Let $\{X_{[c]}\}_{[c]\in[C_{R,J_{R}}]}$ be a clopen partitition of $J_{R}$ such that $X_{[c]}\cap C_{R,J_{R}} = [c]$, for all $[c]$ in $[C_{R,J_{R}}]$. Define the homomorphism
$\partial:V_{R}\mapsto \text{co-ker}(\iota - \Phi_{0})$ on generators as $\partial(u) = \overline{1_{J_{R}}}$ and $\partial([c]) = \overline{1_{X_{[c]}} - \Phi(X_{[c]})}$, for all $[c]$ in $[C_{R,J_{R}}]$. Since $\partial((1-d)u) = \overline{1_{J_{R}} - \Phi(1_{J_{R}})} = \sum_{[c]\in[C_{R,J_{R}}]}\overline{1_{X_{[c]}} - \Phi(1_{X_{[c]}})} = \partial(\sum_{[c]\in[C_{R,J_{R}}]}[c])$, this homomorphism is well-defined. Let $q:\text{co-ker}(\iota - \Phi_{0})\mapsto \text{co-ker}(\iota - \Phi)/\mathbb{Z}[1_{J_{R}}]$ be the quotient map. 
\par
We show that $q(\partial([c])) = \tilde{\partial}(v_{[c]})$ for all $[c]$ in $[C_{R,J_{R}}]$. By the definition of the boundary map from the Snake Lemma, $\tilde{\partial}(v_{[c]}) = \overline{a}$, for any $a$ in $C(J_{R},\mathbb{Z})$ such that $\text{exp}(a) = \iota(w) - w\hat{\otimes}_{1}[\mathcal{E}_{R,F_{R}}]$, for any $w$ in $K^{-1}(F_{R}\setminus C_{R,F_{R}})$ such that $i_{*}(w) = v_{[c]}$.
\par
First, we choose $w$. The diagram

$$
\begin{tikzcd}
0 \arrow[r] & {C_{0}(F_{R}\setminus C_{R,F_{R}})} \arrow[r] \arrow[d]             & {C_{0}(\hat{\mathbb{C}}\setminus C_{R,F_{R}})} \arrow[r] \arrow[d] & C(J_{R}) \arrow[r] \arrow[d, "k"]          & 0 \\
0 \arrow[r] & {C_{0}(\hat{\mathbb{C}}\setminus C_{R,\hat{\mathbb{C}}})} \arrow[r] & C(\hat{\mathbb{C}}) \arrow[r]                                                 & {C(C_{R,\hat{\mathbb{C}}})} \arrow[r] & 0
\end{tikzcd}$$ commutes, where the two left-most vertical maps are inclusion and the rightmost vertical map is restriction to $C(C_{R,J_{R}})$, followed by the inclusion $C(C_{R,J_{R}})\mapsto C(C_{R,\hat{\mathbb{C}}}).$

Naturality of $\text{exp}$ implies the diagram
$$
\begin{tikzcd}
{C(J_{R},\mathbb{Z})} \arrow[r, "k"] \arrow[d, "\text{exp}"] & {C(C_{R,J_{R}},\mathbb{Z})} \arrow[d, "\text{exp}"]        \\
{K^{-1}(F_{R}\setminus C_{R,F_{R}})} \arrow[r, "i_{*}"]      & {K^{-1}(\hat{\mathbb{C}}\setminus C_{R,\hat{\mathbb{C}}})}
\end{tikzcd}$$
commutes. So, $w:= \text{exp}(1_{X_{[c]}})$ satisfies $i_{*}(w) = \text{exp}(1_{[c]}) = v_{[c]}$. 
\par
Let $a = 1_{X_{[c]}} - \Phi(1_{X_{[c]}})$. Then, by Proposition \ref{ncintertwine} and naturality of $\text{exp}$, $\text{exp}(a) = \iota(\text{exp}(1_{X_{[c]}})) - \text{exp}(1_{X_{[c]}})\hat{\otimes}_{1}[\mathcal{E}_{R,F_{R}}] = \iota(w) - w\hat{\otimes}_{1}[\mathcal{E}_{R,F_{R}}].$ Therefore,
$\tilde{\partial}(v_{[c]}) = \overline{1_{X_{[c]}} - \Phi(1_{X_{[c]}})} = q(\partial([c])).$
\par
Since
$$
\begin{tikzcd}
\tilde{W}_{R} \arrow[r, "\tilde{\partial}"] & {\text{co-ker}(\iota - \Phi_{0})/\mathbb{Z}[1_{J_{R}}]} \arrow[r, "\widetilde{\text{exp}}"] & {\text{co-ker}(\iota - \hat{\otimes}_{1}[\mathcal{E}_{R,F_{R}}])} \arrow[r] & 0
\end{tikzcd}$$
is exact and $\text{exp}(1_{J_{R}}) = 0,$ $q^{-1}(\text{im}(\tilde{\partial}))$ equals the kernel of the surjection $\text{exp}:\text{co-ker}(\iota - \Phi_{0})\mapsto \text{co-ker}(\iota - \hat{\otimes}_{1}[\mathcal{E}_{R,F_{R}}])$.  We have $\mathbb{Z}[1_{J_{R}}]\subseteq\text{im}(\partial)$ and, from above, $q(\text{im}(\partial)) = \tilde{\partial}(\sum_{[c]\in[C_{R,J_{R}}]}\mathbb{Z}[v_{[c]}]) = \text{im}(\tilde{\partial}).$ Therefore, $\text{im}(\partial) = q^{-1}(\text{im}(\tilde{\partial})).$ So, 
$$
\begin{tikzcd}
V_{R} \arrow[r, "\partial"] & \text{co-ker}(\iota - \Phi_{0}) \arrow[r, "\text{exp}"] & {{\text{co-ker}(\iota - \hat{\otimes}_{1}[\mathcal{E}_{R,F_{R}}])}} \arrow[r] & 0
\end{tikzcd}$$ is exact.
Let $\tilde{V}_{R}$ be the group generated by elements $\{v_{[c]}\}_{[c]\in [C_{R,J_{R}}]}$ satisfying the relation $\sum_{[c]\in [C_{R,J_{R}}]}v_{[c]} = 0$; this group is canonically isomorphic to $\tilde{W}_{R}/\mathbb{Z}[\{v_{c}\}_{c\in C_{R,F_{R}}}]$ and $V_{R}/\mathbb{Z}[u]$.  Let $\iota:\bigoplus_{Q\in\mathcal{H}_{R}}\mathbb{Z}[w_{Q}]\mapsto V_{R}$ be the homomorphism sending $w_{Q}$ to $\sum_{[c]\in [C_{R,J_{R}}]}H_{Q}([c])[c]$, and let $p:V_{R}\mapsto \tilde{V}_{R}$ be the hommorphism sending $u$ to $0$, $[c]$ to $v_{[c]}$, for all $[c]$ in $[C_{R,J_{R}}]$. 
Since $i_{*}(G(F_{R},C_{R})) = \mathbb{Z}[\{v_{c}\}_{c\in C_{R,F_{R}}}]\subseteq \text{ker}(\tilde{\partial})$ the diagram
$$
\begin{tikzcd}
{\bigoplus_{Q\in\mathcal{H}_{R}}\mathbb{Z}[w_{Q}]} \arrow[r] \arrow[rd] & \tilde{V}_{R} \arrow[r, "\tilde{\partial}"] & {\text{co-ker}(\iota - \Phi_{0})/\mathbb{Z}[1_{J_{R}}]} \\
                                                                        & V_{R} \arrow[r, "\partial"] \arrow[u, "p"'] & \text{co-ker}(\iota - \Phi_{0}) \arrow[u, "q"']        
\end{tikzcd}$$
commutes, and the top row is exact. Hence, $\partial^{-1}(\mathbb{Z}[1_{J_{R}}]) = \partial^{-1}(q^{-1}(0)) = p^{-1}(\tilde{\partial}^{-1}(0)) = p^{-1}(\text{im}(i_{*})) = \text{im}(\iota) +\mathbb{Z}[u].$ 
\par
For each $[c]$ in $[C_{R,J_{R}}]$, let $Y_{[c]} = X_{[c]}\cap H_{Q}^{-1}(H_{Q}([c]))$. Since $Y_{[c]}$ is a clopen set such that $[c]\subseteq Y_{[c]}$, $1_{X_{[c]}\setminus Y_{[c]}}$ is in $C_{0}(J_{R}\setminus C_{R,J_{R}},\mathbb{Z})$. Hence,
$\partial([c]) = \overline{1_{Y_{[c]}} - \Phi(1_{Y_{[c]}})} + \overline{1_{X_{[c]}\setminus Y_{[c]}} - \Phi(1_{X_{[c]}\setminus Y_{[c]}})} = \overline{1_{Y_{[c]}} - \Phi(1_{Y_{[c]}})}$. By construction, if we let $Y = \bigcup_{[c]\in[C_{R,J_{R}}]}Y_{[c]}$, then $H_{Q} = \sum_{[c]\in [C_{R,J_{R}}]}H_{Q}([c])1_{Y_{[c]}} + H_{Q}1_{J_{R}\setminus Y}$, for all $Q$ in $\mathcal{H}_{R}$. Since $H_{Q}1_{J_{R}\setminus Y}$ is in $C_{0}(J_{R}\setminus C_{R,J_{R}},\mathbb{Z}),$ a similar calculation to that as above shows $\partial(\sum_{[c]\in [C_{R,J_{R}}]}H_{Q}([c])[c]) = \overline{H_{Q} - \Phi(H_{Q})}$, for all $Q$ in $\mathcal{H}_{R}$.
\par
Therefore every element $g$ in $V_{R}$ of the form $g = \sum_{[c]\in[C_{R,J_{R}}]}\sum_{Q\in\mathcal{H}_{R}}a_{Q}H_{Q}([c])[c] + \sum_{Q\in\mathcal{H}_{R}}a_{Q}(\Phi(H_{Q}) - H_{Q})$ is in $\text{ker}(\partial)$. Let us show every element in $\text{ker}(\partial)$ is of this form.
\par
Let $g$ be in $\text{ker}(\partial)$. From above, $g$ is in $\text{im}(\iota) +\mathbb{Z}[u]$, so we can write $g = mu + \sum_{Q\in\mathcal{H}_{R}}a_{Q}\sum_{[c]\in [C_{R,J_{R}}]}H_{Q}([c])[c]$ for some $\{a_{Q}\}_{Q\in\mathcal{H}_{R}}\cup\{m\}\subseteq\mathbb{Z}$.
Therefore, $0 = \partial(g) = \overline{m + \sum_{Q\in\mathcal{H}_{R}} H_{Q} - \Phi(H_{Q})}$. Therefore, there are integers $\{b_{Q}\}_{Q\in\mathcal{H}_{R}}\cup\{n\}\subseteq\mathbb{Z}$ such that $n+ \sum_{Q\in\mathcal{H}_{R}}b_{Q}H_{Q}[c] = 0$ for all $c$ in $C_{R,J_{R}}$ and $(1-d)n+ \sum_{Q\in\mathcal{H}_{R}}b_{Q}(H_{Q} - \Phi(H_{Q})) = m + \sum_{Q\in\mathcal{H}_{R}} H_{Q} - \Phi(H_{Q})$.
\par
Using the relation $(d-1)nu = -n\sum_{[c]\in[C_{R,J_{R}}]}[c]$, we may write $g = (m + (d-1)n)u + \sum_{[c]\in[C_{R,J_{R}}]}(\sum_{Q\in\mathcal{H}_{R}}a_{Q}H_{Q}([c]) + n)[c]$. Then, using the fact that $\sum_{Q\in\mathcal{H}_{R}}b_{Q}H_{Q}(c) = -n$ for all $c$ in $C_{R,J_{R}}$, we may write 

$\sum_{[c]\in[C_{R,J_{R}}]}(\sum_{Q\in\mathcal{H}_{R}}a_{Q}H_{Q}([c]) + n)[c] =\sum_{[c]\in[C_{R,J_{R}}]}\sum_{Q\in\mathcal{H}_{R}}(a_{Q}-b_{Q})H_{Q}([c])[c]$. Hence, setting $a'_{Q} = a_{Q} - b_{Q}$, and using that $\sum_{Q\in\mathcal{H}_{R}}a'_{Q}(\Phi(H_{Q}) - H_{Q}) = m + (d-1)n$, we may write $g = \sum_{[c]\in[C_{R,J_{R}}]}\sum_{Q\in\mathcal{H}_{R}}a'_{Q}H_{Q}([c])[c] + \sum_{Q\in\mathcal{H}_{R}}a'_{Q}(\Phi(H_{Q}) - H_{Q})$.
\par
By the above description of $\text{ker}(\partial)$, we have that $\text{ker}(\partial) = \text{im}(\hat{H}_{R})$, where $\hat{H}_{R}:\mathbb{Z}^{\mathcal{H}_{R}}\mapsto V_{R}$ is the homomorphism sending $Q$ to $(\Phi(H_{Q}) - H_{Q})u + \sum_{[c]\in [C_{R,J_{R}}]}H_{Q}([c])[c]$. By Proposition \ref{Fkernel1}, we have $\text{co-kernel}(\iota - \hat{\otimes}_{1}[\mathcal{E}_{R,F_{R}}])\simeq\mathbb{Z}^{h_{R}}$. Hence, we have an exact sequence
$$
\begin{tikzcd}
0 \arrow[r] & \text{im}(\hat{H}_{R}) \arrow[r] & V_{R} \arrow[r, "\partial"] & \text{co-ker}(\iota - \Phi_{0}) \arrow[r] & \mathbb{Z}^{h_{R}} \arrow[r] & 0.
\end{tikzcd}$$
Therefore, $\text{co-kernel}(\iota-\Phi_{0})\simeq V_{R}/\text{im}(\hat{H}_{R})\oplus\mathbb{Z}^{h_{R}}$, via an isomorphim sending the class of $u$ to the class of $1_{J_{R}}$. Since $\text{im}([H_{R}]) = q^{-1}(\text{im}(\hat{H}_{R}))$, where $q: \mathbb{Z}^{[C_{R,J_{R}}]}\oplus\mathbb{Z}/\mapsto V_{R}$ is the quotient map, it follows that $V_{R}/\text{im}(\hat{H}_{R})\simeq\text{co-ker}([H_{R}])$ via an isomorphism mapping the class of $u$ to the class of $u$.
\end{proof}
As a Corollary to our above calculations, we can describe the $K$-theory of $\mathcal{O}_{R,J_{R}}$.
\begin{thm}
\label{jk}
Let $R$ be a rational function of degree $d>1$. Then, $K_{1}(\mathcal{O}_{R,J_{R}})\simeq\text{ker}(H_{R})\oplus\mathbb{Z}/\omega_{R}\mathbb{Z}\oplus\mathbb{Z}^{|f_{R}-1|}$ and $K_{0}(\mathcal{O}_{R,J_{R}})\simeq \text{co-ker}(H_{R})\oplus\mathbb{Z}^{|f_{R} +h_{R} -1|}$, with class of the unit corresponding to the class of basis element $u$ in $\text{co-ker}(H_{R})$.
\end{thm}
\begin{proof}
First, suppose $J_{R} = \hat{\mathbb{C}}$. By Theorem \ref{CK}, we have $K_{1}(\mathcal{O}_{R,J_{R}})\simeq \mathbb{Z}$ and $K_{0}(\mathcal{O}_{R,J_{R}})\simeq\mathbb{Z}^{c_{R,J}+1}$, with class of the unit corresponding to a generator in a minimal generating set for $\mathbb{Z}^{c_{R,J}+1}$.
\par
In this case, $H_{R}$ is the mapping $\mathbb{Z}\mapsto \mathbb{Z}^{C_{R,J_{R}}}\oplus\mathbb{Z}$ sending $u$ to $(d-1)u +\sum_{c\in C_{R, J_{R}}}c$. Hence, $\text{ker}(H_{R}) = 0$ and $\text{co-ker}(H_{R})\simeq\mathbb{Z}^{c_{R,J_{R}}}$, with the class of $u$ corresponding to a generator in a minimal generating set for $\mathbb{Z}^{c_{R,J}}$.
\par
$\mathbb{Z}^{|f_{R} + h_{R} - 1|} = \mathbb{Z}^{|f_{R}-1|} = \mathbb{Z}$ and $\mathbb{Z}/\omega_{R}\mathbb{Z} = 0$, so that $\text{ker}(H_{R})\oplus \mathbb{Z}/\omega_{R}\mathbb{Z}\oplus \mathbb{Z}^{|f_{R}-1|}\simeq \mathbb{Z}$ and $\text{co-ker}(H_{R})\oplus\mathbb{Z}^{|f_{R} +h_{R} -1|}\simeq\mathbb{Z}^{c_{R,J} +1}$, with the class of $u$ corresponding to a generator in a minimal generating set for $\mathbb{Z}^{c_{R,J} +1}$.
\par
Now, let us assume $J_{R}\neq\hat{\mathbb{C}}$. By Proposition \ref{Triso}, we therefore have $\text{ker}(\iota - \hat{\otimes_{0}}[\mathcal{E}_{R,J_{R}}])\simeq \text{ker}(\iota - \Phi_{0})$ and $\text{co-ker}(\iota - \hat{\otimes}_{0}[\mathcal{E}_{R,J_{R}}])\simeq\text{co-ker}(\iota - \Phi_{0})$, with the class of the unit corresponding to the class of $1_{J_{R}}$. By the Pimsner-Voiculescu 6-term exact sequence and the above isomorphisms, we have $K_{1}(\mathcal{O}_{R,J_{R}})\simeq \text{ker}(\iota - \Phi_{0})\oplus\text{co-ker}(\iota - \hat{\otimes}_{1}[\mathcal{E}_{R,J_{R}}])$ and 
$K_{0}(\mathcal{O}_{R,J_{R}})\simeq \text{ker}(\iota - \hat{\otimes}_{1}[\mathcal{E}_{R,J_{R}}])\oplus \text{co-ker}(\iota - \Phi_{0})$, with the class of the unit in $K_{0}$ corresponding to the class of $1_{J_{R}}$ in $\text{co-ker}(\iota - \Phi_{0})$.
\par
By Proposition \ref{k1j} and Proposition \ref{transfer0}, we have $\text{ker}(\iota - \Phi_{0})\oplus\text{co-ker}(\iota - \hat{\otimes}_{1}[\mathcal{E}_{R,J_{R}}])\simeq \text{ker}([H_{R}])\oplus \mathbb{Z}/\omega_{R}\mathbb{Z}\oplus\mathbb{Z}^{f_{R}-1}$ and $\text{ker}(\iota - \hat{\otimes}_{1}[\mathcal{E}_{R,J_{R}}])\oplus \text{co-ker}(\iota - \Phi_{0})\simeq \mathbb{Z}^{(f_{R} -1) + (c_{R,J} - k_{R,J})}\oplus\mathbb{Z}^{h_{R}}\oplus\text{co-ker}([H_{R}])$, with the class of $1_{J_{R}}$ in $\text{co-ker}(\iota - \Phi_{0})$ corresponding to the class of $u$ in $\text{co-ker}([H_{R}])$.
\par
Since $(H_{R})_{c,Q} = (H_{R})_{d,Q} = ([H_{R}])_{[c],Q}$ for all $Q$ in $\mathcal{H}_{R}$ and $(H_{R})_{c,u} = (H_{R})_{d,u} = ([H_{R}])_{[c],u}$ for every $c\sim d$, we have that $\text{ker}(H_{R}) = \text{ker}([H_{R}])$ and $\text{co-ker}(H_{R})\simeq\mathbb{Z}^{c_{R,J}-k_{R,J}}\oplus \text{co-ker}([H_{R}])$, with the class of $u$ in $\text{co-ker}(H_{R})$ corresponding to the class of $u$ in $\text{co-ker}([H_{R}])$. Therefore, $\text{ker}([H_{R}])\oplus \mathbb{Z}/\omega_{R}\mathbb{Z}\oplus\mathbb{Z}^{f_{R}-1}\simeq \text{ker}(H_{R})\oplus\mathbb{Z}/\omega_{R}\mathbb{Z}\oplus\mathbb{Z}^{|f_{R}-1|}$ and $\mathbb{Z}^{(f_{R} -1) + (c_{R,J} - k_{R,J})}\oplus\mathbb{Z}^{h_{R}}\oplus\text{co-ker}([H_{R}])\simeq \text{co-ker}(H_{R})\oplus\mathbb{Z}^{|f_{R} +h_{R} -1|}$ with the class of $u$ in $\text{co-ker}[(H_{R}])$ corresponding to the class of $u$ in $\text{co-ker}(H_{R})$.
\end{proof}

\section{Applications}
\label{app}
\subsection{A conjugacy invariant for rational functions}
\label{appconj}

By the \textit{Fatou cycle length data of $R$} we shall mean the tuple $L_{R} = (|P|)_{P\in \mathcal{F}_{R}}$, where the entries are ordered in non-decreasing order. Similarly, the \text{Herman cycle length data of $R$} shall mean the the tuple $T_{R} = (|Q|)_{Q\in\mathcal{H}_{R}}$ with entries also ordered in non-decreasing order.  We show that $L_{R}$ and $T_{R}$ are conjugacy invariants for $R:J_{R}\mapsto J_{R}$ amongst all rational functions. This is not surprising due to the rigid nature of rational dynamics, but it is not clear how to prove it directly by a dynamical argument.
\par
Since $J_{R^{\circ n}} = J_{R}$ for any $n$ in $\mathbb{N}$ (\cite[Lemma~4.4]{Milnor:Dynamics_in_one_complex_variable}), $R$ and $S$ being conjugate on their Julia sets also implies $R^{\circ n}$ and $S^{\circ n}$ are conjugate on their Julia sets, for all $n$ in $\mathbb{N}$. Therefore, $\text{co-ker}(\iota-\hat{\otimes}_{1}[\mathcal{E}_{R^{\circ n},J_{R}}])\simeq \text{co-ker}(\iota-\hat{\otimes}_{1}[\mathcal{E}_{S^{\circ n},J_{S}}])$ and $\text{ker}(\text{id}-\Phi_{R^{\circ n}})\simeq \text{ker}(\text{id}-\Phi_{S^{\circ n}})$ for all $n$ in $\mathbb{N}$. By Proposition \ref{k1j} and Proposition \ref{k0nc}, it follows that $f_{R^{\circ n}} = f_{S^{\circ n}}$ and $h_{R^{\circ n}} = h_{S^{\circ n}}$ for all $n$ in $\mathbb{N}$. 
\par
We will now show that the sequences $\{f_{R^{\circ n}}\}_{n\in\mathbb{N}}$ and $\{h_{R^{\circ n}}\}_{n\in\mathbb{N}}$ are equivalent to the Fatou and Herman cycle length data of $R$. respectively.
\par
The first observation to make is that every cycle of $R^{\circ n}$ must be contained in a cycle $P$ for $R$ of the same type, and that the number of distinct cycles of $R^{\circ n}$ contained in $P$ in $\mathcal{F}_{R}$ is equal to the greatest common divisor between $|P|$ and $n$, denoted $(|P|,n)$. Moreover, each such cycle has length $\frac{|P|}{(|P|, n)}$.
Hence, $f_{R^{\circ n}} = \sum_{P\in\mathcal{F}_{R}}(|P|,n)$ and $h_{R^{\circ n}} = \sum_{Q\in\mathcal{H}_{R}}(|Q|,n)$.
\par
The rest of the argument showing equivalency is contained in a relevant lemma about elementary number theory. First, define 
$$\mathcal{A} = \{(a_{1},...,a_{k})\in\mathbb{N}^{k}:k\in \mathbb{N}\text{ }, a_{i}\leq a_{i+1}\text{ }\forall i\leq k-1\}.$$ 
For $A = (a_{1},...,a_{k})$ in $\mathcal{A}$ and $n$ in $\mathbb{N}$, define $(A,n) = \sum_{i=1}^{k}(a_{i},n)$.
\begin{lemma}
\label{number}
Suppose $A$, $B$ are tuples in $\mathcal{A}$. If $(A,n) = (B,n)$ for all $n$ in $\mathbb{N}$, then $A = B$.
\end{lemma}
\begin{proof}
Note that the tuple length of $A$ is equal to $(A,1)$. We will prove this lemma via induction on the tuple length $k = (A,1) = (B,1)$.
\par
If $k = 1$, then $a_{1} = \max_{n\in\mathbb{N}}(A,n) = \max_{n\in\mathbb{N}}(B,n) = b_{1}$, and the lemma is proved.
\par
Now, suppose $k>1$ and we know the lemma is true for all $m\leq k-1$. We prove for $m = k$, but first we will set some notation.
\par
For $a,b$ in $\mathbb{N}$, let $a|b$ mean $a$ divides $b$, and $a\not| \; b$ mean $a$ doesn't divide $b$. For $C$ in $\mathcal{A}$ such that $C = (c_{1},...,c_{k})$, let $\mathcal{P}_{C}$ denote the set of all primes appearing in the numbers $\{c_{i}\}_{i=1}^{k}$. For each $i\leq k$ and $p$ in $\mathcal{P}_{C}$, let $n_{p,C,i} = \text{max}\{l\in\mathbb{N}_{0}:p^{l}|c_{i}\}$. Hence, $c_{i} = \Pi_{p\in \mathcal{P}_{C}}p^{n_{p,C,i}}$ for all $i\leq k$.
\par
Let $n_{p,C} = \max_{i\leq k}n_{p,C,i}$. The smallest number $l$ for which $(C,l) = \max_{n\in\mathbb{N}}(C,n) = \sum_{i=1}^{k}c_{i}$ is denoted $m_{C}$, and it is easy to see it is equal to $\Pi_{p\in \mathcal{P}_{C}}p^{n_{p,C}}$.
\par
Now, suppose $A$, $B$ in $\mathcal{A}$ are such that $(A,n) = (B,n)$ for all $n$ in $\mathbb{N}$ and $(A,1) = (B,1) = k$. Then,
$\Pi_{p\in \mathcal{P}_{A}}p^{n_{p,A}} = m_{A} = m_{B} = \Pi_{p\in \mathcal{P}_{B}}p^{n_{p,B}}$, so $\mathcal{P}_{A} = \mathcal{P}_{B}=:\mathcal{P}$ and, for every $p$ in $\mathcal{P}$, we have $n_{p,A} = n_{p,B}=:n_{p}$.
\par
Now, for $C = A$ and $C = B$, for every $m$ in $\mathbb{N}$, and $p$ in $\mathcal{P}$, we have 
\begin{equation*}
\begin{split}
&(C, p^{n_{p}}m) - (C, p^{n_{p}-1}m) =\\
&\sum_{i=1}^{k}p^{n_{p, C,i}}(\frac{c_{i}}{p^{n_{p,C,i}}}, m) - \sum_{i=1}^{k}p^{\min\{n_{p,C,i},n_{p}-1\}}(\frac{c_{i}}{p^{\min\{n_{p,C,i},n_{p}-1\}}}, m) =\\
&\sum_{i:p^{n_{p}}|c_{i}}p^{n_{p}}(\frac{c_{i}}{p^{n_{p}}}, m) - \sum_{i:p^{n_{p}}|c_{i}}p^{n_{p}-1}(\frac{c_{i}}{p^{n_{p}-1}}, m)
\end{split}
\end{equation*}

So, if we let $C_{p}$ be the ordered tuple $(\frac{c_{i}}{p^{n_{p}}})_{i:p^{n_{p}}|c_{i}}$, and for every $m$ in $\mathbb{N}$, let $m_{p} = \frac{m}{p^{l}}$, where $l = \text{max}\{s\in\mathbb{N}_{0}:p^{s}|m\}$, then
$$ \frac{(C,p^{n_{p}}m_{p}) - (C, p^{n_{p}-1}m_{p})}{p^{n_{p}}-p^{n_{p}-1}} = (C_{p},m),\text{ for every }m \text{ in }\mathbb{N}.$$
Thus, $(A_{p},m) = (B_{p},m)$ for all $p$ in $\mathcal{P}$ and $m$ in $\mathbb{N}$. There are two cases.
\par
The first case is that $k = (A_{p},1) = (B_{p},1)$ for all $p$ in $\mathcal{P}$, in which case $a_{i} = m_{A} = m_{B} = b_{i}$ for all $i\leq k$, and the lemma is proved.
\par
The second case is that $k > (A_{p},1) = (B_{p},1)$ for some $p$ in $\mathcal{P}$. By induction, it follows that $A_{p} = B_{p}$, and hence $A'_{p} := (a_{i})_{i:p^{n_{p}}|a_{i}} = (b_{i})_{i:p^{n_{p}}|b_{i}} =: B'_{p}$. For $C = A$ and $C = B$, denote $C^{p} = (c_{i})_{i:p^{n_{p}}\not\text{ } | \text{ }c_{i}}$. Then, for every $m$ in $\mathbb{N}$, we have
$$(A^{p},m) = (A,m) - (A'_{p},m) = (B,m) - (B'_{p},m) = (B^{p},m).$$
By induction, it follows that $A^{p} = B^{p}$. Hence, $A = B$.
\end{proof}
\begin{cor}
\label{appconjboot}
Let $R$ and $S$ be rational functions. If $R$ and $S$ are conjugate on their Julia sets, then $L_{R} = L_{S}$ and $T_{R} = T_{S}$.
\end{cor}
\begin{proof}
As remarked above Lemma \ref{number}, $R$ and $S$ conjugate on $J$ implies $(L_{R},n) = (L_{S},n)$ and $(T_{R},n) = (T_{S},n)$, for all $n$ in $\mathbb{N}$. Lemma \ref{number} then implies that $L_{R} = L_{S}$ and $T_{R} = T_{S}$.
\end{proof}
\begin{cor}
\label{apphgroup}
Let $R$ and $S$ be rational functions. If $R$ and $S$ are topologically conjugate on their Julia sets, then $\text{ker}(H_{R})\simeq \text{ker}(H_{S})$ and there is a bijection $b:C_{R,J_{R}}\cup\{u\}\mapsto C_{S,J_{S}}\cup\{u\}$ mapping $u$ to $u$ and inducing an isomorphism $b:\text{co-ker}(H_{R})\mapsto \text{co-ker}(H_{S})$.
\end{cor}
\begin{proof}
 We will denote by $\iota - \Phi_{0}^{R} := \iota - \Phi_{0}:C_{0}(J_{R}\setminus C_{R,J_{R}},\mathbb{Z})\mapsto C(J_{R},\mathbb{Z})$. Let $\varphi:J_{R}\mapsto J_{S}$ be a homeomorphism such that $S\circ \varphi = \varphi\circ R$. Then, $(\iota - \Phi_{0}^{R})\circ \varphi^{*} = \varphi^{*}\circ (\iota - \Phi_{0}^{S})$ so Proposition \ref{transfer0} implies $\text{ker}([H_{S}])\simeq \text{ker}(\iota - \Phi_{0}^{S})\simeq \text{ker}(\iota - \Phi_{0}^{R})\simeq \text{ker}([H_{R}])$. Since $\text{ker}([H_{R}])\simeq\text{ker}(H_{R})$ for any rational function $R$, this proves the first claim.
 \par
 $\varphi:J_{R}\mapsto J_{S}$ restricts to a bijection $b:C_{R,J_{R}}\mapsto C_{S,J_{S}}$, so if $Y_{b(c)}$ is a connected component of $J_{S}$ containing only $[b(c)]$, for $c$ in $J_{R}$, then $Y_{c}:=\varphi^{-1}(Y_{b(c)})$ is a connected component of $J_{R}$ containing only $[c]$. It follows then by the definition of the map $\partial^{R} = \partial:V_{R}\mapsto\text{co-ker}(\iota - \Phi^{R}_{0})$ in the proof of Proposition \ref{transfer0} that $\varphi^{*}\circ \partial^{S}([b(c)]) = \varphi^{*}(\overline{1_{Y_{b(c)}} - \Phi(1_{Y_{b(c)}})}) = 1_{Y_{c}} - \Phi(Y_{c}) = \partial^{R}([c])$. Hence, $\varphi^{*}(\text{im}(\partial^{S})) = \text{im}(\partial^{R})$, and if we identify $\text{im}(\partial^{S})$ and $\text{im}(\partial^{R})$ with $\text{co-ker}([H_{S}])$ and $\text{co-ker}([H_{R}])$ as in the proof of Proposition \ref{transfer0}, the isomorphism $\varphi^{*}:\text{im}(\partial^{S})\mapsto \text{im}(\partial^{R})$ becomes $b^{-1}:\text{co-ker}([H_{S}])\mapsto \text{co-ker}([H_{R}])$. It is routine to see this isomorphism lifts to an isomorphism $b^{-1}:\text{co-ker}(H_{S})\mapsto \text{co-ker}(H_{R})$.
\end{proof}
\subsection{$K$-theory for polynomials}
\label{kpol}
When $R$ is a polynomial, the $K$-theory takes on an especially simple form. We first describe some general properties of polynomial dynamics.
\par
If $R$ is a degree $d$ polynomial, then there are constants $\lambda, M>0$ such that $R(z) > \lambda|z|^{d}$ for $|z|>M$, so $\mathcal{A}_{\infty} = \{z\in\hat{\mathbb{C}}:\lim_{n\to\infty}R^{\circ n}(z) = \infty\}$ is a non-empty open set of $F_{R}$ containing the critical point $\infty$. It is called the \textit{attracting basin at infinity}. By \cite[Lemma~9.4]{Milnor:Dynamics_in_one_complex_variable}, $\mathcal{A}_{\infty}$ is connected, and it is easy to see $R^{-1}(\mathcal{A}_{\infty}) = \mathcal{A}_{\infty}$. Therefore, the attracting basin at infinity is a Fatou cycle of cycle length one. Hence, $f_{R}\geq 1$ and $\omega_{R} = 1$. 
\par
Also, $\mathcal{H}_{R} = \emptyset$ by the maximum modulus principle. When $\mathcal{H}_{R} = \emptyset$, $H_{R}$ is the mapping $\mathbb{Z}\mapsto \mathbb{Z}^{C_{R,J_{R}}}\oplus\mathbb{Z}$ sending $u$ to $(d-1)u + \sum_{c\in C_{R,J_{R}}}c$. So, $\text{ker}(H_{R}) = 0$, and $\text{co-ker}(H_{R})\simeq\mathbb{Z}^{c_{R,J}}$ if $c_{R,J} \neq 0$ and $\text{co-ker}(H_{R})\simeq\mathbb{Z}/(d-1)\mathbb{Z}$ otherwise, with the class of $u$ in both cases corresponding to a generator in a minimal generating set for the group. Gathering these calculations and applying Theorem \ref{jk} yields the following Corollary.
\begin{cor}
\label{apppol}
Let $P$ be a polynomial of degree $d>1$. Then, $K_{1}(\mathcal{O}_{P,J_{P}})\simeq\mathbb{Z}^{f_{R}-1}$. If $c_{R,J_{R}} = 0$, then $K_{0}(\mathcal{O}_{R,J_{R}})\simeq\mathbb{Z}/(d-1)\mathbb{Z}\oplus\mathbb{Z}^{f_{R}-1}$, with the class of $1_{J_{R}}$ generating the isotropy. Otherwise, $K_{0}(\mathcal{O}_{R,J_{R}})\simeq\mathbb{Z}^{c_{R,J} +f_{R}-1}$, with the class of the unit corresponding to a generator in a minimal generating set for $\mathbb{Z}^{c_{R,J} +f_{R}-1}$.
\end{cor}
By the above Corollary, we can describe the isomorphism type of $\mathcal{O}_{R,J_{R}}$ rather easily.
\begin{thm}
\label{appcharpol}
Let $R$ and $S$ be polynomials. Then, $\mathcal{O}_{R,J_{R}}$ is isomorphic to $\mathcal{O}_{S,J_{S}}$ if and only if either
\begin{itemize}
    \item $c_{R,J} = c_{S,J} = 0$, $\text{deg}(R) = \text{deg}(S)$, and $f_{R} = f_{S}$, or 
    \item $c_{R,J} = c_{S,J}\neq 0$ and $f_{R} = f_{S}$.
\end{itemize}
Consequently, for every $d>1$ the number of isomorphism types of $\mathcal{O}_{R,J_{R}}$ when the degree of $R$ is $d$ is bounded above by $d(2d-2)$.
\end{thm}
\begin{proof}
The characterization of the isomorphism types follows directly from the above Corollary. By \cite[Corollary~2]{Shishikura}, $f_{R}\leq 2d-2$. Similarily, $c_{R,J}\leq d-1$. Therefore, the number of possible isomorphism types of degree $d$ is bounded above by $(2d-2) + (d-1)(2d-2) = d(2d-2)$
\end{proof}
Let us show the above inequality is sharp in the case when $d =2$ and describe the four isomorphism types.
First, we will record the following lemma, which is a collection of several results in complex dynamics.
\begin{lemma}
\label{compcrit}
Assume $R$ is a rational function and $P$ is a Fatou cycle.
\begin{enumerate}
    \item If $P$ is an attracting cycle, then it contains a critical point of $R$.
    \item If $P$ is a parabolic cycle, then it contains a critical point of $R$.
    \item If $P$ is a Siegel cycle, then its boundary is contained in the closure of the set of forward orbits for the critical points of $R$.
    \item If $P$ is a Herman cycle, then its boundary is contained in the closure of the set of forward orbits for the critical points of $R$.
\end{enumerate}
\end{lemma}
\begin{proof}
The proofs in the four cases can (respectively) be found in the proofs of \cite[Lemma~8.5]{Milnor:Dynamics_in_one_complex_variable}, \cite[Theorem~10.15]{Milnor:Dynamics_in_one_complex_variable}, \cite[Theorem~11.17]{Milnor:Dynamics_in_one_complex_variable}, and \cite[Lemma 15.7]{Milnor:Dynamics_in_one_complex_variable}.
\end{proof}

The complement $\hat{\mathbb{C}}\setminus\mathcal{A}_{\infty}$ is denoted $K_{P}$ and is called the \textit{filled Julia set}. Note that 
$K_{P}$ consists of all points in $\mathbb{C}$ with a bounded forward $P$-orbit. By \cite[Theorem~9.5]{Milnor:Dynamics_in_one_complex_variable}, $K_{P}$ is connected if and only if $C_{P,\mathbb{C}}\subseteq K_{P}$, and $K_{P}$ is connected if and only if $J_{P}$ is connected. Also, $\partial K_{P} = J_{P}$.
\par
It is easy to see that any quadratic map is conjugate to one of the form $f_{c}(z) = z^{2}+c$, for some $c$ in $\mathbb{C}$. Its critical points are thus $0$ and $\infty$. Denote its Julia set by $J_{c}$, its filled Julia set by $K_{c}$, and the Cuntz-Pimsner algebra $\mathcal{O}_{f_{c},J_{c}}$ by $\mathcal{O}_{c,J}$.
\par
We show that the $K$-theory of a quadratic $f_{c}$ depends only on the location of $c$ in $\mathbb{C}$ relative to the filled Julia set $K_{c}$ of $f_{c}$. Equivalently, the $K$-theory depends only on the type of bounded Fatou cycle $f_{c}$ admits.
\par
\begin{cor}
\label{appclassquad1}
Let $f_{c}(z) = z^{2} + c$. Then there are four isomorphism types for $\mathcal{O}_{c,J}$, dependent on the location of $0$ relative to the filled Julia set.
\begin{enumerate}
    \item[] \textbf{Case 0} $(0\notin K_{c}):$ Then, $K_{1}(\mathcal{O}_{c,J}) = K_{0}(\mathcal{O}_{c,J}) = 0$.
    \par
    \item[] \textbf{Case 1}  $(0\in\text{int}(K_{c})):$ Then, $K_{1}(\mathcal{O}_{c,J})\simeq K_{0}(\mathcal{O}_{c,J})\simeq\mathbb{Z}$ and $[1_{J_{c}}] = 0$.
    
    \item[] \textbf{Case 2} $(0\in \partial K_{c} = J_{c}$,  $\text{int}(K_{c})\neq\emptyset):$ Then $K_{1}(\mathcal{O}_{c,J})\simeq\mathbb{Z}$, $K_{0}(\mathcal{O}_{c,J})\simeq\mathbb{Z}^{2}$ and $[1_{J_{c}}]$ is a generator in a minimal generating set for $\mathbb{Z}^{2}$.
    
    \item[] \textbf{Case 3} $(0\in \partial K_{c} =J_{c}$, $\text{int}(K_{c}) = \emptyset):$ Then $K_{1}(\mathcal{O}_{c,J}) = 0$, $K_{0}(\mathcal{O}_{c,J})\simeq\mathbb{Z}$ and $[1_{J_{c}}]$ is a generator.
    
\end{enumerate}
\end{cor}
\begin{proof}
\textbf{Case 0}: If $0$ is not in $K_{c}$, then $0$ can't be in a bounded Fatou component, and the only limit point of $\{f_{c}^{\circ n}(0)\}_{n\in\mathbb{N}}$ is $\infty$. Thus, Lemma \ref{compcrit} implies $K_{c} = J_{c}$. Hence $c_{f_{c},J_{c}} = 0$, $p_{f_{c}} = 1$, and $\omega_{f_{c}} = 1$. Case 2 of Theorem \ref{jk} then implies $K_{1}(\mathcal{O}_{c,J}) = 0 = K_{0}(\mathcal{O}_{c,J})$.
\par
\textbf{Case 1}: If $0$ is in $\text{int}(K_{c})$, then $c_{f_{c},J_{c}} = 0$, and the Fatou set $F_{c}$ contains a cycle which is distinct from the attracting basin at $\infty$. Hence, $p_{f_{c}}\geq 2$. By \cite[Corollary~2]{Shishikura}, a rational function of degree $d$ has at most $2d-2$ distinct Fatou cycles. Therefore, $p_{f_{c}} = 2$, and $c_{f_{c},J_{c}} = 0$. Now, case 2 of Theorem \ref{jk} applies to show $K_{1}(\mathcal{O}_{c,J})\simeq\mathbb{Z}$, $K_{0}(\mathcal{O}_{c,J})\simeq\mathbb{Z}/\mathbb{Z}\oplus\mathbb{Z} = \mathbb{Z}$, and $[1_{J_{c}}] = 0$.
\par
\textbf{Case 2}: If $0$ is in $J_{c}$ and $\text{int}(K_{c})\neq\emptyset$, then $c_{f_{c},J_{c}} = 1$, and $p_{f_{c}} = 2$ by the same reasoning as before. Case 3 of Theorem \ref{jk} applies to show $K_{1}(\mathcal{O}_{c,J})\simeq\mathbb{Z}$, $K_{0}(\mathcal{O}_{c,J})\simeq\mathbb{Z}^{2}$, and $[1_{J_{c}}]$ is a generator in a minimal generating set for $\mathbb{Z}^{2}$.
\par
\textbf{Case 3}: If $0$ is in $J_{c}$ and $\text{int}(K_{c}) = \emptyset$, then the only Fatou cycle is the attracting basin at infinity. Hence, $c_{f_{c},J_{c}} = 1 = p_{f_{c}}$. Case $3$ of Theorem \ref{jk} applies to show $K_{1}(\mathcal{O}_{c,J}) = 0$, $K_{0}(\mathcal{O}_{c,J})\simeq\mathbb{Z}$, and $[1_{J_{c}}]$ is a generator.
\end{proof}
We now characterize the four isomorphism types of $\mathcal{O}_{c,J}$ above by the dynamics of $f_{c}$ on its Fatou set $F_{c}$.
\begin{cor}
\label{appclassquad2}
Let $f_{c}(z) = z^{2} + c$. Then there are four isomorphism types for $\mathcal{O}_{c,J}$, dependent on the dynamics of $f_{c}$ on the Fatou set.
\begin{enumerate}
    \item[] \textbf{Case 0}: If $\mathcal{A}_{\infty}$ is not simply connected, then $K_{1}(\mathcal{O}_{c,J}) = K_{0}(\mathcal{O}_{c,J}) = 0$.
    \par
    \item[] \textbf{Case 1}: If $F_{c}$ contains a hyperbolic or parabolic cycle of bounded Fatou components, then $K_{1}(\mathcal{O}_{c,J})\simeq K_{0}(\mathcal{O}_{c,J})\simeq\mathbb{Z}$ and $[1_{J_{c}}] = 0$.
    
    \item[] \textbf{Case 2}: If $F_{c}$ contains a cycle of irrational rotations, then $K_{1}(\mathcal{O}_{c,J})\simeq\mathbb{Z}$, $K_{0}(\mathcal{O}_{c,J})\simeq\mathbb{Z}^{2}$ and $[1_{J_{c}}]$ is a generator.
    
    \item[] \textbf{Case 3}: If $F_{c} = \mathcal{A}_{\infty}$ and is simply connected, then $K_{1}(\mathcal{O}_{c,J}) = 0$, $K_{0}(\mathcal{O}_{c,J})\simeq\mathbb{Z}$ and $[1_{J_{c}}]$ is a generator.
    
\end{enumerate}
\end{cor}
\begin{proof}
Note that $\mathcal{A}_{\infty}$ is simply connected if and only if $K_{c} = \hat{\mathbb{C}}\setminus\mathcal{A}_{\infty}$ is connected, which is true if and only if $0$ is in $K_{c}$. So, case $0$ above is equivalent to case $0$ of Corollary \ref{appclassquad1}.
\par
The only other case that doesn't follow immediately from Corollary \ref{appclassquad1} and Lemma \ref{compcrit} is \textbf{Case 2}, which we show. By Theorem \ref{jk}, it suffices to show $0$ is in $J_{c}$.
\par
Suppose the contrary. Then, Lemma \ref{compcrit} $(3)$ implies $0$ must be in a bounded Fatou component. If $P$ is the Siegel cycle, then it follows that there is some $k$ in $\mathbb{N}$ such that \{$f_{c}^{\circ n}(0)\}_{n\geq k}\subseteq U_{P}$. Since $f_{c}^{\circ |P|}:U_{x}\mapsto U_{x}$ is conjugate to an irrational rotation, for all $x$ in $P$, it follows that $\overline{\{f_{c}^{\circ n}(0)\}_{n\geq k}}\subseteq U_{P}$. But Lemma \ref{compcrit} $(3)$ then implies $\partial U_{P}\subseteq \overline{\{f_{c}^{\circ n}(0)\}_{n\geq k}}\subseteq U_{P}$, a contradiction.
\end{proof}
Case 0,1,3 above are realized by $c = 1, 0, -2$, respectively, while case 2 is realized by the quadratic $z^{2} + e^{2\pi i\varphi}z$, where $\varphi$ is the golden ratio.
\par
A complex number $c$ is said to be \textit{hyperbolic} (\textit{parabolic}) if $\{f_{c}^{n}(0)\}_{n\in\mathbb{N}}$ converges to an attracting (parabolic) periodic orbit. We get the following $K$-theory characterization of when $c$ is ``bolic''.
\begin{cor}
\label{bolicchar}
$c$ in $\mathbb{C}$ is hyperbolic or parabolic if and only if $[1_{J_{c}}] = 0$ in $K_{0}(\mathcal{O}_{c,J})$.
\end{cor}
\begin{proof}
If $[1_{J_{c}}] = 0$ then $f_{c}$ is a quadratic covered in Case 0 or Case 1 of Corollary \ref{appclassquad2}. 
\par
In Case 0, $0$ is in $\mathcal{A}_{\infty}$ and hence $\{f_{c}^{n}(0)\}_{n\in\mathbb{N}}$ converges to the attracting fixed point $\infty$. 
\par
In case $1$, $F_{c}$ contains either an attracting or parabolic cycle, and so Lemma \ref{compcrit} $(1)$ and $(2)$ imply $\{f_{c}^{n}(0)\}_{n\in\mathbb{N}}$ is eventually contained in a cycle of hyperbolic or parabolic Fatou components. This proves one direction of the Corollary. The converse is straightforward as well.
\end{proof}
Let us describe the corresponding $C^{*}$-algebra to each case. Case $0$ corresponds to the Cuntz algebra $\mathcal{O}_{2}$ considered first in \cite{C77}. This is the universal $C^{*}$-algebra generated by two isometries $s_{1}$, $s_{2}$ such that $s_{1}s^{*}_{1} + s_{2}s^{*}_{2} = 1$. The $C^{*}$-algebra $\mathcal{O}_{2}$ and its generalizations for $n = 3,...\infty$ have played a central role in the classification theory of purely infinite $C^{*}$-algebras \cite{Phillips:classification} and to the general theory of $C^{*}$-algebras. For instance, a seperable $C^{*}$-algebra is exact if and only if it embeds into $\mathcal{O}_{2}$ (see \cite{KP00}). The Cuntz algebras have also found applications in wavelet theory \cite{BJ97} and were influential in the discovery of the Doplicher-Roberts Theorem \cite{DE89}, which characterizes the representations of a compact group as an abstract catgory.
\par
Case $1$ corresponds to the 2-adic ring $C^{*}$-algebra $\mathcal{Q}_{2}$ studied in \cite{LL12}, and the author thanks Chris Bruce for this identification. It also appears in other contexts; see \cite[remark~3.2]{LL12}. It is shown in \cite{LL12} that $\mathcal{Q}_{2}$ is the universal $C^{*}$-algebra generated by a unitary $u$ and isometry $s$ satisfying $su = u^{2}s$ and $ss^{*} + uss^{*}u^{*} = 1$. Note that the isometries $s$ and $us$ generate $\mathcal{O}_{2}$ as a sub-$C^{*}$-algebra of $\mathcal{Q}_{2}$. The representations of $\mathcal{O}_{2}$ that extend to $\mathcal{Q}_{2}$ are characterized in \cite{LL12}. From this characterization, the authors of \cite{LL12} motivate viewing $\mathcal{Q}_{2}$ as a symmetrized version of $\mathcal{O}_{2}$.
\par
Case $3$ corresponds to the Cuntz algebra $\mathcal{O}_{\infty}$, see the above remarks concerning Cuntz algebras. This is the universal $C^{*}$-algebra generated by isometries $\{s_{i}\}^{\infty}_{i=1}$ satisfying $\sum_{i=1}^{n} s_{i}s_{i}^{*}\leq 1$ for every $n$ in $\mathbb{N}$. Like $\mathcal{O}_{2}$, it is special amongst the Cuntz algebras. Kirchberg showed that a simple, seperable, unital and nuclear $C^{*}$-algebra $A$ is purely infinite if and only if $A\otimes \mathcal{O}_{\infty}\simeq A$ (see \cite{KP00}).
\par
Not much is known about the $C^{*}$-algebra representing case $2$. It is isomorphic to the $C^{*}$-algebra of the partial dynamical system $z^{2}:S^{1}\setminus \{1\}\mapsto S^{1}$. This follows by an application of the theory developed in Section \ref{morphcorr}. From this description, it follows that it is the universal $C^{*}$-algebra generated by a unitary $u$ and isometry $s$ satisfying $su = u^{2}s$ and $ss^{*} + uss^{*}u^{*} = \frac{u + u^{*}}{2}$. We shall denote it by $\mathcal{Q}_{2,\infty}$, as it shares properties of both $\mathcal{Q}_{2}$ and $\mathcal{O}_{\infty}$.
\par
In all $4$ cases, the $C^{*}$-algbera is isomorphic to a graph $C^{*}$-algebra. We follow the southern convention for graph $C^{*}$-algebras as in \cite{R05}. The below figure shows the graph corresponding to each case (the numbers in the graphs represent multiple edges):
\begin{table}[hbt!]
\label{table}
        \begin{tabular}{l | l | l}
        Parameter $c$ & $C^{*}$-algebra & Graph \\
        \hline \hline
        Outside $\mathcal{M}$ & $\mathcal{O}_{2}$ & $\begin{tikzcd}
\bullet \arrow["2"', loop, distance=2em, in=215, out=145]
\end{tikzcd}$\\
        Hyperbolic or parabolic in $\mathcal{M}$ & $\mathcal{Q}_{2}$ & $\begin{tikzcd}
\bullet \arrow["2"', loop, distance=2em, in=215, out=145] \arrow[r, bend left] & \bullet \arrow[l, bend left] \arrow["2"', loop, distance=2em, in=35, out=325]
\end{tikzcd}$\\
        In $\mathcal{M}$ and $\mathbb{C}\setminus J_{c}$ connected & $\mathcal{O}_{\infty}$ & $\begin{tikzcd}
&[12pt]& \bullet \arrow["\infty"', loop, distance=2em, in=35, out=325]
\end{tikzcd}$\\
        $F_{c}$ contains a Siegel cycle & $\mathcal{Q}_{2,\infty}$ & $\begin{tikzcd}
\bullet \arrow[r, bend left] \arrow["2"', loop, distance=2em, in=215, out=145] \arrow[d] & \bullet \arrow["\infty"', loop, distance=2em, in=35, out=325] \arrow[l, bend left] \arrow[ld] \\
\bullet \arrow[loop, distance=2em, in=215, out=145] \arrow[ru, bend right=49]            &                                                                                              
\end{tikzcd}$
       \end{tabular}
\end{table}
\par

Recall that the Mandelbrot set is 
$\mathcal{M} := \{c\in\mathbb{C}:0\in K_{c}\}$. Equivalently, by Corollary \ref{appclassquad1}, we have that
$\mathcal{M} = \{c\in\mathbb{C}:K_{0}(\mathcal{O}_{c,J})\neq 0\}$.
\par
One of the most important open problems in holomorphic dynamics is the Density of Hyperbolicity Conjecture, stated below.
\begin{conjec}
\label{mandelconjec}
$\mathcal{H}:= \{c\in\mathcal{M}:f_{c}(z) = z^{2} + c\text{ is hyperbolic}\}$ is dense in $\mathcal{M}$.
\end{conjec}
This conjecture can be stated as a conjecture about the $K$-theory for quadratics in the following way.
\begin{cor}
\label{mandelconjeccor}
The Density of Hyperbolicity Conjecture is true if and only if 
$\mathcal{H}':=\{c\in\mathcal{M}: [1_{J_{c}}] = 0 \text{ in }K_{0}(\mathcal{O}_{c,J})\} = \{c\in \mathcal{M}:\mathcal{O}_{c,J}\simeq\mathcal{Q}_{2}\}$ is dense in $\mathcal{M}$.
\end{cor}
\begin{proof}
By Corollary \ref{bolicchar}, the set $\mathcal{H}'$ is precisely the parameters $c$ in $\mathcal{M}$ for which $f_{c}$ is either parabolic or hyperbolic. By \cite[Lemma~6.1]{Milnor:Mandelbrot} and \cite[Lemma~6.2]{Milnor:Mandelbrot}, every parabolic parameter lies on the boundary of the open set $\mathcal{H}$ of hyperbolic parameters in $\mathcal{M}$. Therefore, the set of parabolic parameters is a nowhere dense set in $\mathcal{M}$. It follows that density of $\mathcal{H}'$ is equivalent to density of $\mathcal{H}$.
\end{proof}


\begin{thebibliography}{20}

\bibitem{A94} Daniel S. Alexander, \emph{A history of complex dynamics: from Schröder to Fatou and Julia}, Aspects of Mathematics, \textbf{24}, Vieweg, Heidelberg, 1994.

\bibitem{AIR11} Daniel S. Alexander, Felice Iavernaro, and Alessandro Rosa, \emph{Early day in complex dynamics: a history of complex dynamics in one variable during 1906-1942}, Hist. Math. \textbf{38}, A.M.S., Providence, 2011.

\bibitem{B76} I. N. Baker, \emph{An entire function which has wandering domains}, J. Austral. Math. Soc. \textbf{22}(Series A), 173-176, (1976).

\bibitem{B99} Detlef Bargmann, \emph{Simple proofs of some fundamental properties of the Julia set.} Ergod. Th. Dynam. Sys. \textbf{19}, 553–558, (1999).

\bibitem{pcf} Robert Benedetto, Patrick Ingram, Rafe Jones, and Alon Levy, \emph{Attracting cycles in p-adic dynamics and height bounds for postcritically finite maps,} Duke Math. J. \textbf{163}(13) 2325-2356, (2014).

\bibitem{Benini:MLC_survey} A. M. Benini, \textit{A survey on MLC, rigidity and related topics,} preprint, arXiv:1709.09869 [math.DS], (2018).

\bibitem{B93} Walter Bergweiler, \emph{Iteration of meromorphic functions}, Bull. (New Ser.) Am. Math. Soc. \textbf{29}(2), 151-188, (1993).

\bibitem{BJ97} Ola Bratteli and Palle E.T. Jorgensen, \emph{Isometries, shifts, Cuntz algebras and multiresolution analyses of scale N} , Integral
Equations Operator Theory 28 (1997) 382–443.

\bibitem{C94} Alain Connes, \emph{Noncommutative geometry}, Academic Press, London and San Diego, 1994.

\bibitem{C77} J. Cuntz, Simple \emph{C*-algebras generated by isometries}, Comm. Math. Phys. \textbf{57}, 173–185, (1977)

\bibitem{CS86} J. Cuntz and G. Skandalis, \textit{Mapping cones and exact sequences in KK-theory}, J. Oper. Theory, \textbf{15}, 163-180, (1986).

\bibitem{DE89} Sergio Doplicher and John E. Roberts, \emph{A new duality theory for compact groups}, Invent. Math. \textbf{98}(1), 157-218, (1989).

\bibitem{DH85} Adrien Douady and John H. Hubbard. \emph{On the dynamics of polynomial-like mappings}. Ann. Sci. Ec. Norm. Sup., \textbf{18}(2), 287-343, (1985).

\bibitem{F19} Pierre Fatou, \emph{Sur les équations fonctionelles}, Bull. Soc. Math. France \textbf{47}, 161–271 (1919), \textbf{48}, 33–94 (1920), and \textbf{48}, 208–314 (1920).

\bibitem{Exel:newlook} Ruy Exel, \emph{A new look at the crossed-product of a $C^{*}$-algebra by an endomorphism,} Ergod. Th. Dyn. Syst. \textbf{23}, 1733-1750, (2003).

\bibitem{IzumiKW:KMS_states_branched} Masaki Izumi, Tsuyoshi Kajiwara and Yasuo Watatani, \emph{KMS states and branched points,}
Ergod. Th. Dynam. Sys. \textbf{27}, 1887-1918 (2007).

\bibitem{Jensen} Kjeld Knudsen Jensen and Klaus Thomsen, \emph{Elements of $KK$-Theory}, Birkhäuser, Boston, 1991.

\bibitem{J18} Gaston Julia, \emph{Memoire sur l'iteration des fonctions rationnelles}, J. Math.
Pure Appl. \textbf{8}, 47-245, (1918).

\bibitem{I23} Kei Ito, \emph{Cartan subalgebras of C*-algebras associated with complex dynamical systems} arXiv:2303.14860 [math.OA], (2023).

\bibitem{KW:C*-algebras_associated_with_complex_systems}
Tsuyoshi Kajiwara, Yasuo Watatani, \emph{C*-algebras associated with complex dynamical systems,} Indiana Uni. Math. J. \textbf{54}(3), 755–778 (2005).

\bibitem{KW17} Tsuyoshi Kajiwara, Yasuo Watatani, \emph{Maximal abelian subalgebras of C*-algebras associated with complex dynamical systems and self-similar maps}, J. Math. Anal. Appl. \textbf{455}, 1383-1400, (2017).

\bibitem{Karoubi:K-theory} Max Karoubi, \emph{$K$-Theory: An Introduction,} Springer-Verlag Berlin, Heidelberg, 1978.

\bibitem{KP00} Eberhard Kirchberg and N. Christopher Phillips, \emph{Embedding of exact C*-algebras in the Cuntz algebra $\mathcal{O}_{2}$}, J. reine angew. Math. \textbf{525}, 17-53, (2000).

\bibitem{Lance} E. Christopher Lance, \emph{Hilbert $C^{*}$-modules: A toolkit for operator algebraists}, Cambridge Uni. Press, Cambridge, 1995.

\bibitem{LL12} Nadia S. Larsen and Xin Li, \emph{The 2-adic ring C*-algebra of the integers and its
representations}, J. Funct. Anal. \textbf{262}, 1392-1426, (2012).

\bibitem{L18} Samuel Lattès, \emph{Sur l'iteration des substitutions rationelles et les fonctions
de Poincare}, C. R. Acad. Sci. Paris \textbf{16}, 26-28, (1918).

\bibitem{MSS83} R. Mañé, P. Sad, D. Sullivan, \emph{On the dynamics of rational maps}, Ann. Sci. Ec. Norm. \textbf{16}(2), 193-217, (1983).

\bibitem{M14} Curtis T. McMullen, \emph{Frontiers in complex dynamics}, Bull. Amer. Math. Soc. \textbf{31}(2), 155–172, (1994).

\bibitem{MS19} Ralf Meyer, Camila F. Sehnem, \emph{A bicategorical interpretation for relative Cuntz-Pimsner algebras,} Math. Scand. \textbf{124}(1), 84-112, (2019).

\bibitem{Milnor:Dynamics_in_one_complex_variable} John Milnor, \emph{Dynamics in One Complex Variable}, 3rd edition, Princeton Uni. Press, Princeton, 2006.

\bibitem{Milnor:Mandelbrot} John Milnor, \emph{Periodic orbits, externals rays and the Mandelbrot set: an expository account}, Géométrie complexe et systèmes dynamiques (Orsay, 1995), Astérisque, \textbf{261}, 277-333 (2000).

\bibitem{MT05} Paul S. Muhly and Mark Tomforde, \emph{Topological quivers}, Int. J. Math. \textbf{16}(7), 693–755 (2005).

\bibitem{Nek:C-alg_and_self_similar} Volodymyr Nekrashevych, \emph{$C^{*}$-algebras and self-similar groups,} J. reine angew. Math. \textbf{630}, 59-123 (2009).

\bibitem{Phillips:classification} N. Christopher Phillips, \emph{A classification Theorem for nuclear purely infinite simple $C^{*}$-algebras}, Doc. Math. \textbf{5}, 49-114 (2000).

\bibitem{Pimsner:Generalizing_Cuntz-Krieger} Michael V. Pimsner, \emph{A class of $C^*$-algebras
    generalizing both Cuntz-Krieger
    algebras and crossed products by $\mathbb{Z}$}, Fields Institute Commun.  \textbf{12}, 189-212 (1997).

\bibitem{R05} Iain Raeburn, \emph{Graph Algebras}, CBMS Reg. Conf. Ser. Math. \textbf{103}, Amer. Math. Soc., Providence, 2005.

\bibitem{RW98} Iain Raeburn and Dana P. Williams, \emph{Morita equivalence and continuous-trace $C^{*}$-algebras}. Amer. Math. Soc., Providence, 1998.


\bibitem{R86} Mary Rees, \emph{Positive measure sets of ergodic rational maps}, Ann. Sci.
Ec. Norm. Sup., \textbf{19}(4), 383-407, (1986).

\bibitem{R20} Joseph Fels Ritt, \emph{On the iteration of rational functions}, Trans. Amer. Math.
Soc. \textbf{21}, 348-356, (1920).

\bibitem{Rordam:intro_K-theory} M. Rørdam, F. Larsen, N. Laustsen, \emph{An Introduction to $K$-Theory for $C^{*}$-Algebras}, Cambridge Uni. Press, New York, 2000.

\bibitem{S98} Dierk Schleicher, \emph{Dynamics of entire functions,} In: Proceedings of the CIME Summer School, Springer-Verlag, 295-339, (2008).

\bibitem{Shishikura} Mitsuhiro Shishikura, \emph{On the quasiconformal surgery of rational functions,} Ann. Sci. Ec. Norm. Sup. 4th series, \textbf{20}(1), 1-29 (1987).

\bibitem{S89} Mitsuhiro Shishikura, \emph{Trees associated with the configuration of Herman rings}, Ergod. Th. Dyn. Syst. \textbf{9}, 543-560, (1989).


\bibitem{NWDSullivan} Dennis Sullivan, \emph{Quasiconformal Homeomorphisms and Dynamics I. Solution of the Fatou-Julia Problem on
Wandering Domains,} Ann. Math. \textbf{122}(2), 401-418 (1985).


\bibitem{W07} Dana P. Williams, \emph{Crossed products of C*-algebras}, Math. Surv. Monogr. \textbf{134}, Providence, 2007.

\end{thebibliography}
\end{document}